\documentclass[12pt]{article}
\usepackage{a4wide}
\usepackage{amsmath, amssymb, amsthm}
\usepackage{parskip}
\usepackage{graphicx,float} % Required for inserting images
\usepackage{rotating} % for rotatebox
\usepackage{array}
\usepackage{multirow}
\newcommand{\rowlabel}[1]{%
    \rotatebox{90}{\parbox{2.2cm}{\centering \footnotesize #1}}%
}

\usepackage{bm}
\newtheorem{definition}{Definition}
\newtheorem{lemma}{Lemma}
\newtheorem{theorem}{Theorem}

\usepackage{hyperref}
\usepackage{biblatex} %Imports biblatex package
\usepackage{algorithm}
\numberwithin{equation}{section}
\usepackage{algpseudocode}
\usepackage{color}
\begin{document}

\title{Joint Near-Field Holotomography Reconstruction with a Phase-Guided Bregman TV Regularization}
\date{}

\author{Jin Liu\thanks{{Helmholtz Imaging, Deutsches Elektronen-Synchroton DESY, Notkestraße 85, 22607 Hamburg, Germany.} \texttt{jin.liu1@desy.de}}
\and
 Johannes Hagemann\thanks{{Centre for X-Ray and Nano Science CXNS, Deutsches Elektronen-Synchrotron DESY, Notkestraße 85, 22607 Hamburg, Germany.}}
\and
Martin Burger$^\ast$\thanks{Department of Mathematics, University of Hamburg, Bundesstraße 55, 20146 Hamburg, Germany.}
}

\maketitle     

\textbf{Abstract:} Near-field holotomography combines coherent diffraction imaging with tomographic acquisition to recover the three-dimensional complex refractive index of a specimen. Since the measured diffraction intensities are generated by a nonlinear object transmission and wave propagation process, the resulting inverse problem is intrinsically nonlinear and ill-posed. We study a direct variational reconstruction framework based on a fully nonlinear wave propagation model that avoids both intermediate phase retrieval and linearization under the weak-object approximation. We analyze the forward operator in appropriate Banach spaces, establish its Fr\'echet differentiability, and derive explicit gradient expressions for variational reconstruction.

To improve quantitative reconstruction of weak absorption features, we further develop a phase-guided Bregman TV regularization framework that exploits structural correlations between phase and absorption components. This enables multi-material reconstruction without imposing a globally fixed ratio between the two components. We perform numerical studies on synthetic phantoms and experimental data. The results demonstrate stable three-dimensional reconstructions and improved recovery of the absorption contrast compared to state-of-the-art methods.

\section{Introduction}

\subsection{Holotomography}
Modern synchrotron radiation sources provide highly brilliant and spatially coherent X-ray beams, enabling high-resolution phase-contrast imaging of materials and biological specimens \cite{snigirev1995possibilities}.
As X-rays propagate through matter, they undergo both attenuation and phase shifts determined by the refractive properties of the specimen \cite{ als2011elements, paganin2006coherent}. 
Conventional absorption imaging exploits only the attenuation signal, whereas phase-contrast imaging additionally utilizes the phase shifts, which are often substantially stronger for weakly absorbing specimens and therefore provide richer structural information \cite{Mayo:03}.

To convert phase shifts into measurable intensity patterns, several phase-contrast imaging techniques have been developed, including crystal interferometry, analyzer-based diffractometry, grating interferometry, and propagation-based imaging \cite{cloetens1996phase, davis1995phase, takeda1995phase, weitkamp2005x}. In this work, we focus on near-field propagation-based imaging, where Fresnel diffraction converts phase information into detectable interference fringes during free-space propagation without requiring interferometric optics \cite{ediss27478, snigirev1995possibilities, williams2006fresnel}.

While propagation-based imaging provides access to phase information with high sensitivity, a single near-field diffraction image contains only projected information accumulated along the beam direction. As a result, it cannot determine the three-dimensional internal structure of the specimen. Recovering volumetric information therefore requires tomographic acquisition from multiple viewing angles \cite{kak2001principles}. This leads naturally to near-field holotomography \cite{cloetens1999holotomography, kim2024holotomography, KODAMA199636}, which extends propagation-based phase imaging to the quantitative three-dimensional reconstruction imaging from near-field diffraction measurements.

A typical near-field holotomography experiment employs a spatially coherent X-ray beam, a tomographic rotation stage, and a downstream detector separated from the specimen by a finite propagation distance. During acquisition, the sample is rotated over multiple projection angles while near-field diffraction patterns are recorded at the detector. In some experimental settings, additional measurements acquired at multiple propagation distances provide complementary information and further improve reconstruction quality \cite{cloetens1999holotomography}.

From these measured diffraction patterns, the goal of holotomography is to reconstruct the three-dimensional internal structure of the specimen. Unlike conventional absorption tomography, where the measured data are commonly modeled as linear Radon projections of the attenuation coefficient, holotomographic reconstruction is formulated as a nonlinear inverse problem coupling wave propagation, phase retrieval, and tomography. This nonlinear structure gives rise to fundamental challenges in uniqueness, stability, and reconstruction efficiency \cite{bronnikov1999reconstruction, doi:10.1137/22M1474382, fannjiang20233d}. These challenges have motivated the development of reconstruction algorithms for holographic and holotomographic imaging.

\subsection{Reconstruction Algorithms in Holotomography}

Reconstruction algorithms for near-field holotomography can be viewed from two complementary perspectives: how the imaging physics is modeled and how the inverse problem is formulated. The latter naturally determines the physical quantity chosen for reconstruction.

With respect to the forward model, existing approaches range from linearized approximations to fully nonlinear formulations. Linearized methods, such as the transport-of-intensity equation (TIE) \cite{gureyev1996phase, paganin2006coherent, krenkel2013transport} and contrast transfer function (CTF) methods \cite{cloetens2002quantitative, villanueva2017contrast}, derive computationally efficient inversion formulas under weak-object assumptions. In contrast, fully nonlinear projection algorithms, such as Gerchberg--Saxton (GS) \cite{gerchberg1972practical} and alternating projections (AP) \cite{hagemann2018phase}, retain the complete wave-propagation model by iteratively enforcing consistency between propagated wavefields and measured diffraction intensities \cite{fienup1982phase, marchesini2007invited, shechtman2015phase}. 

A second distinction concerns the formulation of the inverse problem. Most existing reconstruction methods adopt a sequential reconstruction strategy, in which phase retrieval is first performed independently for each diffraction pattern and the recovered projections are subsequently reconstructed tomographically using analytical or algebraic reconstruction methods \cite{gureyev2006phase, andersen1984simultaneous, ruhlandt2014three}. Within this framework, conventional holographic phase-retrieval methods typically reconstruct the complex transmission function or wavefield immediately behind the sample \cite{bauschke2002phase, fienup1982phase, huhn2022fast}. 
However, such formulations often suffer from phase-wrapping ambiguities, particularly
for thick or strongly refracting specimens, where the recovered phase is only defined modulo $2\pi$ \cite{guizar2011phase}. To alleviate this limitation, several refractive phase retrieval frameworks have been proposed that reconstruct the projected refractive index directly instead of the complex transmission function \cite{chowdhury2019high, ruhlandt2014three, kahnt2019coupled}. Since the projected refractive index is linearly related to the underlying material properties and is not restricted to a wrapped phase interval, these formulations avoid explicit phase unwrapping and provide improved quantitative stability. Nevertheless, these methods remain projection-based and therefore still treat phase retrieval and tomography as two separate inverse problems. Although computationally convenient, such sequential reconstruction strategies do not fully exploit the global consistency shared across projection angles and may propagate phase retrieval errors into the final three-dimensional reconstruction.

Several works therefore moved beyond projection-wise reconstruction and incorporated diffraction modeling directly into volumetric reconstruction frame works \cite{stockmar2015x, ruhlandt2016three, thompson2023three}. 
Unlike sequential approaches, direct formulations reconstruct the three-dimensional refractive-index distribution \cite{born2013principles} directly from the measured diffraction data, allowing all projection angles to contribute simultaneously to a single volumetric reconstruction.
In particular, Ruhlandt \cite{ruhlandt2016three} proposed a three-dimensional near-field
tomography framework that incorporates Fresnel diffraction directly into the tomographic reconstruction process. 
However, the formulation relies on linearization under the optically weak-object approximation and therefore remains restricted to weakly scattering samples.

Motivated by these developments, we formulate near-field holotomography as a direct nonlinear variational inverse problem posed in the three-dimensional refractive-index domain. Unlike conventional sequential pipelines that decouple phase retrieval and tomography, the proposed framework reconstructs the volumetric refractive-index distribution directly from measured diffraction intensities through a unified nonlinear forward operator combining tomographic projection, exponential transmission, wave propagation, and intensity formation. By avoiding intermediate projection-domain reconstruction and weak-object linearization, the formulation enables globally consistent optimization across all projection angles while remaining applicable in strongly refractive regimes.

Moreover, appropriate physical assumptions are more naturally imposed in the three-dimensional object domain, enabling more effective regularization strategies as we investigate in this paper. In particular, the phase shift and absorption components of the refractive index are closely related within each material and typically share the same structural features. This relationship provides useful prior information for stabilizing their joint reconstruction. For multi-material specimens, however, the relative phase and absorption contrasts can vary between different materials, making a globally fixed relationship between the two components restrictive \cite{momose2005recent, paganin2002simultaneous, beltran20102d}. 

To stabilize the resulting ill-posed inverse problem, we first develop a Bregman TV regularization framework \cite{osher2005iterative,rasch2018joint} operating directly on the two refractive-index components. Since the absorption component is generally much weaker and more sensitive to measurement noise than the phase shift component in propagation-based imaging, we further propose a phase-guided Bregman TV regularization strategy, in which structural information recovered from the phase shift component is used to guide the absorption reconstruction. In this way, the proposed coupling exploits the shared material structure without enforcing a globally fixed ratio between the two components, making it suitable for multi-material near-field holotomography. 

\subsection{Contributions and Organization of the Paper}
The main contributions of this work are threefold. First, we formulate near-field holotomography as a direct nonlinear inverse problem for the three-dimensional complex refractive index and establish its Fréchet differentiability and local convergence properties. Second, we introduce a phase-guided Bregman TV regularization that exploits structural correlations between the phase and absorption components, enabling reconstruction of heterogeneous multi-material specimens. Third, we validate the proposed framework on synthetic single- and multi-distance data and experimental measurements, demonstrating improved recovery of weak absorption features.

The remainder of this paper is organized as follows. 
Section 2 establishes the forward model of near-field holotomography. Section 3 analyzes the continuity and differentiability of the forward model in appropriate Banach spaces.
Section 4 formulates the variational inverse problem, derives the analytical gradient expressions and provides a convergence analysis based on the nonlinear Landweber iteration.
Section 5 introduces the numerical reconstruction framework, incorporating the phase-guided Bregman TV regularization, and presents both simulated and experimental validation results.
Finally, Section 6 summarizes the main conclusions and discusses potential extensions to limited-angle reconstruction and uncertainty quantification.

\section{Forward Model} \label{sec: Forward Model}

The near-field holotomographic imaging process can be described as a sequence of four physical operations:
(i) projection of the relative refractive index along the X-ray beam,
(ii) nonlinear transmission through the specimen,
(iii) Fresnel diffraction during free-space propagation, and
(iv) intensity detection at the sensor plane.

Let $
n = 1 - \delta + i \beta
$
be the complex refractive index of the material in 3D,
where $ \delta $ denotes the refractive phase decrement and $ \beta $ the absorption index.  
It is often convenient to define the \emph{relative refractive index} relative to the vacuum as
$$
f = n - 1 = -\delta + i\beta,
$$
so that $ f $ directly represents the spatial distribution of optical properties of the sample to be reconstructed.

In the following, we introduce the different steps of the image formation process as separate operators and subsequently formulate their composition mapping the object to the detected intensity.

\subsection{Projection}
Let the bounded domain $\Omega\subset\mathbb{R}^3$ with coordinates $(x,y,z)$ be the support of the object, where $z$ denotes the direction of beam propagation. 
During tomographic acquisition, the specimen is rotated clockwise (in a right-handed coordinate system) about the $y$-axis by an angle $\theta\in[0,\pi)$. The rotated relative refractive index becomes
$$
f_\theta(x,y,z)
= f(x\cos\theta - z\sin\theta,\, y,\, x\sin\theta + z\cos\theta).
$$
Under the parallel-beam projection approximation, the projection operator $\tilde O$ is defined by the accumulated contrast along the beam path
\begin{equation}\label{eq:projection}
\widetilde{O}(f)(x,y,\theta)
= \int_{\mathbb{R}} f_\theta(x,y,z)\,dz, \quad (x,y)\in \Omega',
\end{equation}
where $\Omega' \subset \mathbb{R}^2$ is the detector-plane orthogonal to the beam direction.
Equation~\eqref{eq:projection} represents the complex line integral through the sample $f$. 
Collecting these projections for all rotation angles $\theta$ yields the tomographic sinogram.

The effect of the object on an incident X-ray beam is described by its \emph{complex transmission function}
$$
O(f) = e^{ i k \, \widetilde{O}(f) },
$$
with wave number $k$. 
The exponential form arises naturally from the solution of the Helmholtz equation under the \emph{paraxial approximation} \cite{goodman2005introduction, grella1982fresnel}.

In near-field holography, one may reconstruct either the transmission
$O(f)$ or the refractive projection $\widetilde O(f)$  \cite{Wittwer:22}.
Reconstructing $\widetilde O(f)$ avoids the phase wrapping problem \cite{chowdhury2019high}, since the projection is linearly related to the underlying relative refractive index,
which in turn naturally motivates direct optimization over the volumetric relative refractive index $f$.

\subsection{Illumination}
To model the illumination realistically, one must account for the fact that the incident X-ray beam is generally not an ideal plane wave of uniform amplitude and flat phase. 
Instead, its intensity and phase vary laterally across the beam cross-section due to finite source size, imperfect beamline optics, and coherence effects. 
These variations are represented by the \emph{probe function}  
$
P: \mathbb{R}^2\to \mathbb{C},
$
a complex-valued function of the transverse coordinates $\mathbf{x}:=(x,y) \in \mathbb{R}^2$ (in the 2D sense for a given projection). 
It encodes both the amplitude envelope (brightness profile) and the phase distortions of the incoming beam at the object plane. 
The \emph{exit wavefield} immediately behind the sample is therefore obtained by multiplying the transmission function by this probe profile, i.e.,
$$
\Psi(f) = P\, O(f) \;=\; P \, e^{\, i k \widetilde{O}(f)}.
$$
Thus, the probe–object interaction becomes a nonlinear, pointwise transformation acting on the projected refractive index $\widetilde{O}(f)$.

\subsection{Fresnel Propagation}
The free-space propagation of an exit wavefield $ \Psi(\mathbf{x}) $ over a distance $ d $ is given by the Fresnel diffraction integral, which corresponds to a convolution with the Fresnel kernel $h_d$ \cite{goodman1969introduction}:
$$
\mathcal D_{\mathrm{Fr}}(\Psi)(\mathbf{x}) 
= (h_d * \Psi)(\mathbf{x})
= \int_{\mathbb{R}^2} h_d(\mathbf{x}-\mathbf{x}')\, \Psi(\mathbf{x}')\, d\mathbf{x}'.
$$
Let $\mathcal F$ be the 2D Fourier transform with inverse $\mathcal F^{-1}$.
By the convolution theorem, 
\begin{equation} \label{eq:Fresnel_kernel}
\mathcal D_{\mathrm{Fr}}(\Psi) 
= \mathcal F^{-1}\!\left[\,H_d(\boldsymbol{\xi}) \cdot \mathcal F(\Psi)\,\right],
\end{equation}
where the Fresnel kernel in Fourier space is given by
$$
H_d(\boldsymbol{\xi}) 
= \exp\!\left(-\,i \pi \frac{|\boldsymbol{\xi}|^2}{2\,\mathrm{Fr}}\right),
$$
where $\boldsymbol{\xi} \in \mathbb{R}^2$ is the frequency coordinate and
$
\mathrm{Fr}
=
\frac{\Delta x^2}{\lambda d}
$
is the dimensionless Fresnel number, with $\Delta x$ the effective pixel size, $\lambda$ the X-ray wavelength, and $d$ the sample-to-detector propagation distance. The Fresnel number characterizes the diffraction regime during free-space propagation, with smaller values corresponding to stronger Fresnel diffraction.

Gathering all steps together, the forward operator $F$ of near-field holographic tomography reads
\begin{equation}\label{eq:forward_operator}
f \mapsto\;  I_m 
     = \left|\, \mathcal D_{\mathrm{Fr}} (\Psi(f))                 \right|^2
     = \left|\, \mathcal{F}^{-1}\!\left[ H_d \cdot
       \mathcal{F}\!\left( P e^{ik\tilde{{O}}(f)} \right)
       \right] \right|^2 =: F(f).
\end{equation}
This sequence of mappings describes the complete imaging process (for optically thin objects justifying separate projection and propagation), starting from the three-dimensional refractive index distribution, proceeding through projection, probe modulation, Fresnel propagation, and finally leading to the intensity recorded at the detector plane.

\subsection{Noise and Likelihood Modeling}
Having the forward operator $F$ that maps the complex relative refractive index
$f$ to the detector intensity, we now consider the associated inverse problem: recovering $f$ directly from experimental measurements $I_m$. 
While image formation in near-field tomography is strongly influenced by the Fresnel propagation geometry, measurement noise is often
approximated by Poisson counting statistics in photon-limited settings \cite{homann2015phase}. Each detector pixel records a discrete number of incident photons or electrons during a finite exposure time, and the recorded counts follow a Poisson distribution
$
I_m\sim \mathrm{Poisson}(I_e),
$
where $I_e$ denotes the model-estimated intensity.
The corresponding negative log-likelihood functional reads
$$
  J_P(f) =\int_0^\pi\!\int_{\Omega'} \!\!\left[\, I_e(f)(x,y,\theta)
      - I_m\,\log I_e(f)(x,y,\theta) \right] d\mathbf{x} \, d\theta ,  
$$
up to constants independent of $f$. 

This Poisson log-likelihood accurately models low-dose or photon-limited regimes where the variance of the noise depends on the signal itself. For high photon counts, the Poisson distribution approaches a Gaussian with variance $\mathrm{Var}(I_m)=I_e$. Expanding $J_P$ to second order around $I_e = I_m$ yields
$$
  J_P(f) \approx
\frac{1}{2}\int_{0}^\pi\!\int_{\Omega'}
\frac{(I_e(f) - I_m)^2}{I_m}\,d\mathbf{x} \, d\theta.
$$
We work with the corresponding amplitude representation
$$
J(f) = \frac{1}{2}\int_{0}^\pi\!\int_{\Omega'}
\big(\sqrt{I_e(f)}-\sqrt{I_m}\big)^2 d\mathbf{x} \, d\theta.  
$$
Based on these arguments we will consider a variational model that can be viewed as the Gaussian approximation of the Poisson likelihood, valid for sufficiently high photon counts, while still grounded in the more general statistical model above \cite{thibault2012maximum}.

Define $y := \sqrt{I_m}$ and the amplitude forward operator $$ \qquad G(f) := \sqrt{F(f)}=\Big|\mathcal \mathcal D_{\mathrm{Fr}}\big(P e^{ik \widetilde O(f)}\big)\Big|. $$ 
% which is non-differentiable at points where $\mathcal D_{\mathrm{Fr}}(P e^{ik\widetilde O(f)}) = 0$. 
Then the inverse problem can be written compactly as the nonlinear least squares problem
\begin{equation}\label{eq: full nonlinear}
     \min_{f} J(f) = \tfrac{1}{2}\|G(f)-y\|^2.   
\end{equation}

\section{Analysis of the Forward Operator}
In this section, we investigate the properties of the forward operator. The goal is to prove that $F$ (and thus $G$) is well-defined, continuous, and Fr\'echet differentiable in appropriate Banach spaces.

We begin with the notation and assumptions.
Let $\Omega \subset \mathbb{R}^3$ be a bounded domain for the object, 
$\Omega' \subset \mathbb{R}^2$ be the bounded measurement window (the illuminated field of view) on the detector side,
and $\Omega'_\pi=\Omega'\times [0,\pi)$. 
We assume that the complex relative refractive index satisfies
$$
f = -\delta + i\beta \in L^2(\Omega;\mathbb{C}),
$$
and the probe function $P\in L^\infty(\mathbb R^2;\mathbb C)$ has a uniform positive lower bound in modulus.
The wavenumber $k>0$ is a fixed constant.
Physically admissible values are restricted to the cone
$$
\mathcal K = \{\, f \in L^2(\Omega;\mathbb{C}) : \Re f \le 0,\ \Im f \ge 0 \text{ a.e. in } \Omega \,\},
$$
where $\Re f $ and $\Im f$ denote the real and imaginary parts of $f$ \cite{paganin2006coherent, hagemann2017probe}.
For simplicity of notation, whenever there is no risk of confusion, we write $L^p(\Omega)$, $1\leq p\leq \infty$, for the complex-valued Lebesgue space $L^p(\Omega;\mathbb{C})$. 

The forward operator $F$ defined in \eqref{eq:forward_operator} consists of four parts, namely,
$$
F = \mathcal I \circ\mathcal D_{\mathrm{Fr}} \circ \mathcal N_g \circ \tilde { O},
$$
where $\tilde { O}$ is the line-integral projection, $\mathcal N_g$ is a \emph{Nemytskii operator} (also known as a superposition operator \cite{appell1990nonlinear})  
$$
\mathcal{N}_g(u)(\mathbf{x}) = P(\mathbf{x})\, e^{i k u(\mathbf{x})},
$$
induced by the Carath\'{e}odory kernel $g(\mathbf{x},\mathbf{z}) = P(\mathbf{x})\, e^{ i k \mathbf{z}}$,
$\mathcal D_{\mathrm{Fr}}$ is the Fresnel propagation that acts as a linear unitary operator, and $\mathcal I$ is the intensity formation operator that maps a complex field to its pointwise modulus-square.

In what follows, we examine each stage individually.
Let us begin with the boundedness of the projection operator, whose proof is straightforward.

\begin{lemma}[Boundedness of projection \cite{natterer2001mathematics}]
\label{lem:projection}
Let $L<\infty$ be the maximal path length through $\Omega$ along the beam direction.  
Then the parallel-beam projection defined in \eqref{eq:projection}
$$
\tilde{ O} : L^2(\Omega;\mathbb C) \longrightarrow L^2(\Omega'_\pi;\mathbb C)
$$
is a bounded linear operator satisfying
$$
\|\tilde{ O}(f)\|_{L^2(\Omega'_\pi)} \le \sqrt{L \pi}\, \|f\|_{L^2(\Omega)}.
$$
Moreover, the sign of the real and imaginary part are preserved, i.e. for $f \in \mathcal K$, we have $\Re(\widetilde{O}(f)) \leq 0 $ and $\Im(\widetilde{O}(f)) \geq 0 $.
\end{lemma}

We now proceed to the analysis of the Nemyskii operator  $\mathcal N_g$, which needs a specialization of the general theory. The proof is given in the appendix for completeness.  In order to verify the Fr\'echet differentiability in a suitable setting, we define an extension of the map from $\mathcal{K}$ to $L^2(\Omega_\pi')$, since the set of functions with nonnegative imaginary part do not include open sets. For this sake we introduce a monotone continuously differentiable function $S \in C^1(\mathbb{R})$ such that $S(t)=t$ for $t \geq 0$ and $S(t) \geq - \frac{1}k$ for $t < 0$. Then, by slight abuse of notation we define 
$$ \hat{e}^{ikz} := e^{ik(\Re(z) + i S(\Im(z))} $$
and it is easy to see that $\hat{e}^{ikz} = e^{ikz}$  for $\Im(z) \geq 0$  as well as 
$$ 0 \leq \hat{e}^{ikz} \leq e \qquad \forall z \in \mathbb{C}.$$

\begin{lemma}[Nemytskii operator]
\label{lem:nemytskii}
Let $1\le p \le r\le\infty$. 
For a fixed constant $k>0$ and $P\in L^\infty(\Omega'_\pi;\mathbb C)$, define
$$
g:\Omega'_\pi\times\mathbb C\to\mathbb C,\qquad g(x,z)=P(x) \hat{e}^{ikz},
$$
and the associated operator
$$
\mathcal N_g: \mathcal U \to L^p(\Omega'_\pi;\mathbb C),\qquad
\mathcal N_g(u)(x)=g(x,u(x))=P(x)\hat{e}^{ik u(x)},
$$
with the admissible set
$$
\mathcal U:=\{u\in L^r(\Omega'_\pi;\mathbb C):\Im u\ge 0\ \text{a.e. in }\Omega'_\pi\}.
$$
Then the following statements hold.
\begin{enumerate}
\item (Well-definedness). $\mathcal N_g$ is a Nemytskii operator and 
$$\|\mathcal N_g(u)\|_{L^p(\Omega'_\pi)}\le\|P\|_{L^p(\Omega'_\pi)}\le\|P\|_{L^\infty(\Omega'_\pi)}|\Omega'_\pi|^{1/p}.
$$
\item (Lipschitz continuity). $\mathcal N_g$ is Lipschitz continuous on $\mathcal U$. For any $u,v\in\mathcal U$,
$$
\|\mathcal N_g(u)-\mathcal N_g(v)\|_{L^p(\Omega'_\pi)} \le k\,\|P\|_{L^\infty(\Omega'_\pi)}\,\|u-v\|_{L^p(\Omega'_\pi)}.
$$
\item (Fr\'echet differentiability). If $p < r$, then the extension of $\mathcal N_g$ to $L^p(\Omega;\mathbb{C})$ is Fr\'echet differentiable 
with
$$
{\mathcal N}_g'(u)[h] = ik\,P\,{\hat e}^{ik u}\,(\Re(h)+ i S'(\Im(z))\Im(h)),
$$
for any  $h\in L^r(\Omega'_\pi;\mathbb C)$. For $u \in \mathcal U$ it holds that
${\mathcal N}_g'(u) = ik\,P\,{\hat e}^{ik u}$.
\end{enumerate}
\end{lemma}

The next step is the Fresnel propagator, which is indeed a unitary operator in $L^2$.

\begin{lemma}[Fresnel propagator]
\label{lem:fresnel}
Define
$$
\mathcal D_{\mathrm{Fr}} u \;=\; \mathcal F^{-1}\!\big[H_d\cdot\mathcal F \,u\big],\qquad
H_d(\xi)=\exp\!\Big(-i\pi \frac{|\xi|^2}{2\,\mathrm{Fr}}\Big).
$$
Then for any dimension $m\ge 1$,
\begin{enumerate}
\item $\mathcal D_{\mathrm{Fr}}$ is unitary on $L^2(\mathbb{R}^m)$, i.e., $\|\mathcal D_{\mathrm{Fr}}u\|_{L^2(\mathbb{R}^m)}=\|u\|_{L^2(\mathbb{R}^m)}$ for any $u\in L^2(\mathbb{R}^m)$.
\item $\mathcal D_{\mathrm{Fr}}$ is a bounded linear operator
$$
 \mathcal D_{\mathrm{Fr}}:L^p(\mathbb{R}^m)\to L^{p'}(\mathbb{R}^m),\qquad \text{for any }1\le p\le 2,\quad p'=\frac{p}{p-1}.
$$
\end{enumerate}
\end{lemma}

\begin{proof}   
1. Since the Fourier transform $\mathcal F$ is unitary on $L^2(\mathbb{R}^m)$ and $|H_d(\xi)|=1$ for any $\xi\in \mathbb{R}^m$,
we have, for any $u\in L^2(\mathbb{R}^m)$, 
{\footnotesize
$$
\|\mathcal D_{\mathrm{Fr}} u\|_{L^2(\mathbb{R}^m)}
=\left\|\mathcal F^{-1}\!\big[H_d\cdot\mathcal F \,u\big]\right\|_{L^2(\mathbb{R}^m)}
=\left \|H_d\cdot\mathcal F \,u\right\|_{L^2(\mathbb{R}^m)}
=\left\|\mathcal F \,u\right\|_{L^2(\mathbb{R}^m)}
=\|u\|_{L^2(\mathbb{R}^m)}.
$$
}
2.
By the Fourier convolution theorem \cite{Grafakos2014},
$$
\mathcal D_{\mathrm{Fr}}(u) = \mathcal{F}^{-1}\!\left[H_d \cdot \mathcal{F}(u)\right]
            = h_d * u,
$$
where $h_d = \mathcal{F}^{-1}H_d$ is the Fresnel kernel in real space. Since $H_d(\xi)=\exp(-i\pi|\xi|^2/(2 \, \mathrm{Fr}))$ is a pure phase factor,
it is easy to verify that  
$h_d\in L^\infty(\mathbb{R}^m)$. It follows the conclusion.
\end{proof}
Finally, we study the intensity mapping.
\begin{lemma}[Intensity mapping]
\label{lem:intensity}
The intensity map $\mathcal I :\psi\mapsto |\psi|^2$ satisfies:
\begin{enumerate}
  \item If $\psi\in L^{2p}(\mathbb{R}^m)$, then $\mathcal I(\psi)\in L^p(\mathbb{R}^m)$ and
    $\|\mathcal I(\psi)\|_{L^p(\mathbb{R}^m)} = \|\psi\|_{L^{2p}(\mathbb{R}^m)}^2$.
  \item $\mathcal I: L^{2p}(\mathbb{R}^m)\to L^p(\mathbb{R}^m)$ is continuously Fr\'echet differentiable with derivative 
    $$
      \mathcal I'(\psi)[h] = 2\,\Re\big( \overline\psi\,h \big),
    $$
for any $\psi$, and $h$ is a small perturbation. 
\end{enumerate}
\end{lemma}

Combining the above results, we arrive at the main result of this section. 

\begin{theorem}
\label{thm:forward}
The forward operator
$$
F : \mathcal K \to L^2(\Omega'_\pi;\mathbb{R}), \qquad F(f) = \Big|\mathcal D_{\mathrm{Fr}}\big(P e^{ik \widetilde O(f)}\big)\Big|^2,
$$
is well-defined, continuous, and its extension to $L^2(\Omega;\mathbb{C})$ is Fr\'echet differentiable with derivative  
$$
F'(f)[h] =
2\,\Re\!\left(
\overline{\mathcal D_{\mathrm{Fr}}(P e^{ik\widetilde O(f)})}\;
\mathcal D_{\mathrm{Fr}}\big( i k P e^{ik\widetilde O(f)} \cdot \widetilde O (h)\big)
\right),
$$
for any $f\in \mathcal K$.%L^2(\Omega; \mathbb C)$.
\end{theorem}

\begin{proof}
\emph{Well-definedness.}  
Lemma~\ref{lem:projection} shows the projection $u=\tilde {O}(f)\in L^2(\Omega'_\pi)$.  
Lemma~\ref{lem:nemytskii} ensures $\mathcal N_g(u) \in  L^{p}(\Omega'_\pi)$ for any $1\le p\le\infty$.
There exists an extension of $\mathcal N_g(u)$ (still denoted by itself) such that 
$\mathcal N_g(u) \in  L^{p}(\mathbb{R}^2\times [0, \pi))$.
Taking $p = \frac{4}{3}$ in Lemma~\ref{lem:fresnel} gives that 
$\mathcal D_{\mathrm{Fr}}(\mathcal N_g(u))\in L^4(\mathbb{R}^2\times [0, \pi))$.
We then restrict back to $\Omega'$ to find $\mathcal D_{\mathrm{Fr}}(\mathcal N_g(u))\in L^4(\Omega'_\pi)$.
Applying Lemma~\ref{lem:intensity}, we obtain that $F(f)\in L^2(\Omega'_\pi)$.

\emph{Continuity.}  
All four components ($\tilde {O}$, $\mathcal N_g$, $\mathcal D_{\mathrm{Fr}}$, $\mathcal I$) are continuous between their respective domains; therefore their composition $F$ is continuous.  

\emph{Differentiability.}  
Each component is Fr\'echet differentiable (Lemmas~\ref{lem:projection}–\ref{lem:intensity}), so the chain rule applies. 
The resulting derivative is precisely the stated formula. 
\end{proof}

For the purposes of the subsequent analysis, we derive an explicit formula for the Fr\'echet derivative at interior points. The directional derivative applies for boundary points.
From a numerical perspective, the derivative formula given in Theorem \ref{thm:forward} can still be employed in practice, with boundary points handled by projected gradient methods.

\section{Iterative Regularization}\label{sec: regularization}

In this section, we consider the (iterative) regularization of inverse problems in holotomography. For this sake, we start from the least square problem established in \eqref{eq: full nonlinear}
where the relative refractive index $f = -\delta + i\beta$.
In order to minimize $J$ numerically, we naturally use iterative schemes based on gradient descent.
Denote by $G'(f)$ the Fr\'echet derivative of $G$ at $f$
and by $G'(f)^*$ its adjoint. We obtain that
$$
J'(f) = G'(f)^*(G(f)-y).
$$
We shall derive an explicit form for $J'(f)$ by the adjoint method in Appendix \ref{Sec: Gradient Derivation}.
The gradient descent iteration with step size $\omega>0$ reads
\begin{equation}\label{eq:gradient_descent}
f^{k+1} = f^k - \omega\, J'(f^k).
\end{equation}
This formulation coincides with the general framework of the nonlinear Landweber iteration for operator equations \cite{scherzer1995convergence}. In particular, the update is not only a gradient descent step but also an iterative regularization method in the sense of \cite{kaltenbacher2008iterative} for ill-posed inverse problems. In order to incorporate the constraints on $f$ as well as additional regularization we will rather consider mirror-type descent of the form 
\begin{equation}\label{eq:mirror_descent}
p^{k+1} = p^k - \omega\, J'(f^k), \qquad f^{k+1} \in \mathcal H(p^{k+1})
\end{equation}
for some (monotone) map $\mathcal H$. This can include e.g. the projection to the set $\mathcal K$ as well as iterative regularizations, where $\mathcal H$ is the subdifferential of the convex conjugate of a regularization functional \cite{benning2018modern}.
%(REF BENNING-BURGER)

In propagation-based holotomography, the absorptive component $\beta$ is generally more challenging to reconstruct than the refractive component $\delta$. For most materials in the X-ray regime, the refractive component gives rise to substantially stronger image contrast than the absorptive component, i.e., $\delta \gg \beta$, resulting in a more stable reconstruction of $\delta$. To improve the reconstruction of the weaker absorptive component, we propose a phase-guided Bregman iteration framework, in which the Bregman regularization for $\beta$ incorporates structural information recovered from $\delta$.
The motivation for introducing such a coupling is supported by the Paganin phase-retrieval framework \cite{beltran2018phase}, where a material prior based on a constant ratio $\delta/\beta$ links phase shift and absorption through a common physical structure. This suggests that the two quantities are not independent, but instead exhibit correlated structural interfaces.
In material mixtures, we expect several subdomains where $\beta$ and $\delta$ are proportional with some fixed (material-dependent) constant. 
Consequently, structural information recovered from the more stable phase reconstruction can serve as a useful prior for stabilizing the weaker and more unstable absorption reconstruction. The key idea is that 
$$ \nabla \delta \cdot \nabla \beta = | \nabla \delta | ~ |\nabla \beta| $$
almost everywhere in $\Omega$. Several structural regularizations approximating such a property have been proposed in the past \cite{kaipio1999inverse, ehrhardt2013vector, ehrhardt2015joint, knoll2016joint}, a particularly appealing one being the contrast-invariant Bregman distance  \cite{osher2005iterative} for total variation (Bregman TV)
$$
D_{TV}^p(\delta, \beta) = TV(\delta) - \langle p, \delta \rangle, \quad p \in \partial TV(\beta). 
$$
This expression is equivalent to the standard Bregman TV distance up to an additive constant independent of $\delta$. The TV functional \cite{rudin1992nonlinear} is given by 
$$
TV(u)
=
\sup
\left\{
\int_{\Omega}
u\,\operatorname{div}g\,dx
:
g \in C_c^1(\Omega,\mathbb{R}^n),
\ |g| \le 1
\right\}.
$$
In the case of $u \in W^{1,1}(\Omega)$ we simply have $TV(u) = \int_\Omega |\nabla u(x)|~dx$. Moreover, at points where $\nabla \beta \neq 0$, $p=-\operatorname{div}( \frac{\nabla \beta}{|\nabla \beta|})$. 

Thus, minimizing the Bregman distance means to locally match $|\nabla \delta|$ and $\frac{\nabla \beta \cdot \nabla \delta}{|\nabla \beta|}$,
thereby encouraging common edge locations while remaining insensitive to differences in contrast magnitude.
This property makes the Bregman distance particularly suitable for
transferring structural information from the more stable phase
reconstruction to the absorption reconstruction. 

In the following we will discuss an iterative regularization scheme to incorporate the phase-guiding idea for Bregman TV regularization. Subsequently, we discuss the nonlinearity of the forward operator in holotomography, which is motivated by the convergence analysis of the iterative scheme \eqref{eq:gradient_descent}.

\subsection{Phase-Guided Bregman Framework}

Regularization methods are commonly employed in  inverse problems to suppress noise and stabilize the reconstruction process. In our setting, we choose total variation regularization, respectively iterative regularizations based on Bregman distances \cite{osher2005iterative, benning2018modern}. The standard variational regularization approach in this setting corresponds to the minimization of 
\begin{equation}
  J_{\text{reg}}(\delta, \beta) = J(\delta, \beta) + \lambda_\delta \, \text{TV}(\delta) + \lambda_\beta \, \text{TV}(\beta),  
  \label{eq:full pb}
\end{equation}
with $ \lambda_\delta, \lambda_\beta > 0 $ are parameters controlling the strength of regularization for $ \delta $ and $ \beta $.
In order to perform structural regularization for the absorption based on the phase we want to include the Bregman distance between  $\delta$ and $\beta$. A simple idea would be to augment \eqref{eq:full pb} by such an additional regularization term. However, this results in a nonconvex problem being difficult to minimize with respect to $\beta$ (due to the complexity of the Bregman distance with respect to the second variable). Instead we apply a joint Bregman iteration including this distance as well, as studied in \cite{rasch2018joint}.

In order to efficiently deal with the nonlinearity and complexity of the forward operator, we make a further slight modification to the joint Bregman iteration, which resembles the Bregman-Landweber \cite{bachmayr2009iterative} or linearized Bregman method \cite{cai2009linearized}. More precisely, we employ operator-splitting methods \cite{glowinski2017some} and use a gradient descent for the first half step, leading to 
$$
\delta^{k + \frac{1}{2}} = \delta^k - \alpha_\delta \frac{\partial J}{\partial \delta} \bigg|_{(\delta^k,\beta^k)}, \quad 
´\beta^{k + \frac{1}{2}} = \beta^k - \alpha_\beta \frac{\partial J}{\partial \beta} \bigg|_{(\delta^k,\beta^k)}.
$$
In the second half step we regularize $\delta^{k + \frac{1}{2}}$ with the Bregman distance
\begin{equation}\label{eq: phase update}
\delta^{k+1} = \arg \min_{\delta} \left( \frac{1}{2} \|\delta - \delta^{k + \frac{1}{2}}\|_2^2 + \alpha_\delta \lambda_\delta D_{TV}^{p_\delta^k}(\delta, \delta^k) \right),
\end{equation}
where $ p_\delta^k \in \partial TV(\delta^k) $ is the subgradient at  the previous iterate.

The Bregman subgradient update $p_\delta^{k+1} $ is derived from the optimality condition of the above Bregman-regularized minimization problem. Namely, 
$$
(\delta^{k+1} - \delta^{k+\frac{1}{2}}) + \alpha_\delta \lambda_\delta (\partial TV(\delta^{k+1}) - p_\delta^k) = 0,
$$
which is equivalent to
$$
p_\delta^{k+1} = p_\delta^k- \frac{1}{\alpha_\delta \lambda_\delta} (\delta^{k+1} - \delta^{k+\frac{1}{2}}).
$$
Substituting it back to \eqref{eq: phase update}, we arrive at 
\begin{equation}\label{eq: delta reg}
   \delta^{k+1} = \arg \min_{\delta} \left( \frac{1}{2} \|\delta - \delta^{k + \frac{1}{2}} - \alpha_\delta \lambda_\delta p_\delta^k\|_2^2 + \alpha_\delta \lambda_\delta \text{TV}(\delta) \right). 
\end{equation}
The term $ \delta^{k + \frac{1}{2}} + \alpha_\delta \lambda_\delta p_\delta^k $ acts as a modified input, meaning the problem can be seen as a denoising step applied to a shifted iterate. We then use efficient TV denoising solvers, e.g., Chambolle’s projection method, split-Bregman, Primal-Dual methods \cite{chambolle2004algorithm, goldstein2009split, chambolle2011first} for \eqref{eq: delta reg}.

For the absorption $\beta$, we use a modified update with a joint Bregman distance to the last iterates of $\delta$ and $\beta$, i.e. 
$$  \widetilde{D}_{TV}(\beta, \beta^k, \delta^{k+1}) = \gamma D_{TV}^{p_\beta^k}(\beta, \beta^k) + (1-\gamma) D_{TV}^{p_\delta^{k+1}}(\beta, \delta^{k+1}),$$
where $\gamma \in [0,1]$ balances the two terms. Now the update rule for $\beta^{k+1}$ based on the joint Bregman distance for TV is
\begin{equation}
\beta^{k+1} = \arg\min_{\beta} \left( \frac{1}{2} \|\beta - \beta^{k + \frac{1}{2}}\|_2^2 + \alpha_\beta \lambda_\beta \widetilde{D}_{TV}(\beta, \beta^k, \delta^{k+1}) \right). 
\label{eq: Breg beta}
\end{equation}
Noticing that
$$
\begin{aligned}
\widetilde{D}_{TV}(\beta, \beta^k, \delta^{k+1}) &= \gamma \left[ \text{TV}(\beta)  - \langle p_\beta^k, \beta \rangle \right] + (1 - \gamma) \left[ \text{TV}(\beta) - \langle p_\delta^{k+1}, \beta \rangle \right] \\
&= \text{TV}(\beta) - \gamma \langle p_\beta^k, \beta  \rangle - (1 - \gamma) \langle p_\delta^{k+1}, \beta \rangle,
\end{aligned}
$$
we arrive at the equivalent form
\begin{equation}\label{eq: beta}
  \beta^{k+1}= \arg \min_{\beta} \left( \frac{1}{2} \|\beta - \beta^{k + \frac{1}{2}} - \alpha_\beta \lambda_\beta \left( \gamma  p_\beta^k + (1-\gamma) p_\delta^{k+1} \right) \|_2^2 + \alpha_\beta \lambda_\beta\text{TV}(\beta) \right),  
\end{equation}
$$
p_\beta^{k+1} = \gamma \,p_\beta^k + (1-\gamma) \, p_\delta^{k+1} - \frac{1}{\alpha_\beta \lambda_\beta} (\beta^{k+1} - \beta^{k+\frac{1}{2}}).
$$
The resulting coupled regularization framework improves the robustness of the absorption reconstruction while preserving consistent structural information between the phase and absorption components.

\subsection{Convergence Analysis}
The convergence and regularization behavior of iterative regularization methods has been studied extensively in the past, starting from the nonlinear Landweber iteration
\eqref{eq:gradient_descent} in \cite{scherzer1995convergence}. Extensions to Bregman-type methods (cf. \cite{bachmayr2009iterative}) closely follow those arguments and are based on similar assumptions. The key properties of the forward operator needed for convergence are the local Fr\'echet differentiability (with local boundedness or Lipschitz continuity of the derivative) and nonlinearity conditions such as the tangential cone condition (TCC)
 $$
   \|F(x)-F(\tilde x)-F'(x)(x-\tilde x)\| \leq \eta \, \|F(x)-F(\tilde x)\| ,
   $$
for some constant $0 \leq \eta < \tfrac{1}{2}$ and $x,\tilde x$ in suitable neighborhood of the solution $x^*$. 

While it is notoriously difficult to understand the convergence properties for the full phase problem, we can at least investigate the implications of the additional exponential nonlinearity in the full holotomography problem (compared to the simple Fresnel propagators in standard holography). Thus, we verify in the following that the operator mapping to full Fresnel data,
$$
\widetilde G(f) = D_{\mathrm{Fr}} \big( P e^{i k \widetilde{O}(f)} \big) 
$$
with relative refractive index $f = -\delta + i\beta$ satisfies the assumptions for convergence of gradient based-iterative regularization methods.

We have established in Lemma \ref{lem:projection}--\ref{lem:fresnel} that the forward mappings, including projection $\widetilde{O}$, multiplication by the probe $P$, the exponential modulation $u\mapsto e^{ik u}$, and the Fresnel propagator $\mathcal D_{\mathrm{Fr}}$, are well defined and Fr\'echet differentiable on any physically admissible neighborhood. Under physically reasonable assumptions, e.g.,
bounded probe, finite propagation distance, and weak contrast,
the forward operator $\widetilde G$ is Fr\'echet differentiable and admits
a uniform local bound on its derivative,
$$
\sup_{f\in\mathcal M}\|\widetilde G'(f)\| \le L.
$$
Here, $\mathcal M \subset \mathcal K$ denotes a bounded neighborhood
of the true solution.
We therefore concentrate on the remaining and the most delicate hypothesis of the TCC.  
It is convenient to verify it by means of the stronger \emph{range-invariance condition} (RIC) \cite{kaltenbacher2008iterative,deuflhard1998convergence}, which implies further structure of the derivatives.

\begin{definition}(Range-invariance condition (RIC))
Let $\mathcal M$ be a neighborhood of the true solution in $X$. We say that a nonlinear operator $F$ satisfies a range invariance condition on $\mathcal M$ if for every $f,\tilde f\in\mathcal M$ there exists a bounded linear operator $R(f,\tilde f):Y\to Y$ with
$$
 F'(f) = R(f,\tilde f)\,F'(\tilde f),   
$$
and the family $R(f,\tilde f)$ satisfies
$$
 \sup_{f,\tilde f\in\mathcal M}\|R(f,\tilde f)-I\|_{L(Y,Y)} \le \eta <1.   
$$
\end{definition}
The implication from RIC to TCC is standard; see, e.g. \cite{deuflhard1998convergence}, for completeness we give the detailed result in the Appendix C. 

We now show that $\widetilde G$ admits the RIC structure with a small $\eta$ under physically natural assumptions. Consequently, TCC follows.

Let $\Omega \subset \mathbb{R}^3$ and $\Omega' \subset \mathbb{R}^2$ be bounded domains. Define $\widetilde G: L^2(\Omega;\mathbb C) \to  L^2(\Omega';\mathbb C)$
$$
\widetilde G(f)=\mathcal D_{\mathrm{Fr}}\circ \mathcal N_{g}\circ\widetilde{O},
\qquad \mathcal N_{g}(u)(x)=P(x)e^{ik u(x)}.
$$
By Lemma~\ref{lem:nemytskii}, the Nemytskii map $\mathcal N_g$ is Fr\'echet differentiable.  Applying the chain rule yields
\begin{equation}\label{eq:Gprime}
\widetilde G'(f)[h] \;=\;\mathcal D_{\mathrm{Fr}}\!\big( s_f\;\widetilde{O}(h)\big),
\qquad
s_f(x):=ik\,P(x)\,e^{ik\,\widetilde{O}(f)(x)}.
\end{equation}
For each $f$, define $M_{s_f}: L^2(\Omega';\mathbb C) \to L^2(\Omega';\mathbb C)$ 
to be the pointwise multiplication operator $v\mapsto s_f v$.
Notice that $P$ is uniformly bounded with a positive lower bound and $|e^{ik\widetilde{O}(f)}|\le 1$ when $\Im \widetilde{O}(f)\ge0$.
% We see that $M_{s_f}$ is invertible with bounded inverse.

\begin{theorem}
For a fixed $\tilde f\in \mathcal M$, define the bounded linear operator on $L^2(\Omega';\mathbb C)$
\begin{equation*}\label{eq:Rdef}
R(f,\tilde f) \;=\;\mathcal D_{\mathrm{Fr}}\circ M_{s_f}\circ M_{s_{\tilde f}}^{-1}\circ\mathcal D_{\mathrm{Fr}}^{-1}.
\end{equation*}
Then for any $f \in L^2(\Omega;\mathbb C)$
\begin{equation}
\label{RIC_1}
\widetilde G'(f) = R(f,\tilde f)\, \widetilde G'(\tilde f),   
\end{equation}
and the family $R(f,\tilde f)$ satisfies
\begin{equation}
\label{RIC_2}
 \sup_{f,\tilde f\in\mathcal M}\|R(f,\tilde f)-I\|_{L^2(\Omega')} \le \eta <1.   
\end{equation}
\end{theorem}

\begin{proof}
From equation (\ref{eq:Gprime}) we have
$$
\begin{aligned}
R(f,\tilde f)\,\widetilde G'(\tilde f)[h]
=\mathcal D_{\mathrm{Fr}}\big(s_f\,\widetilde{O}(h)\big)
= \widetilde G'(f)[h],
\end{aligned}
$$
for any $h\in L^2(\Omega;\mathbb C)$. Thus, equation(\ref{RIC_1}) holds.

By Lemma \ref{lem:fresnel}, $\mathcal {D}_{Fr}$ is bounded, invertible, and unitary. Then we have
$$R(f,\tilde f) - I =\mathcal D_{\mathrm{Fr}}  (M_{s_f} M^{-1}_{s_{\tilde f}} - I) \,\mathcal D_{\mathrm{Fr}}^{-1} .$$
As $M_{s_f}, M^{-1}_{s_{\tilde f}}$ is itself a multiplication operator
$$
M_{s_f} M^{-1}_{s_{\tilde f}} = M_{s_f s^{-1}_{\tilde f}}.
$$
Consequently,
$$
M_{s_f s^{-1}_{\tilde f}} - I = M_{s_f s^{-1}_{\tilde f}} - M_1 = M_{\frac{s_f}{s_{\tilde f}}-1},
\quad
\frac{s_f}{s_{\tilde f}}-1 = e^{ik\,(\widetilde{O}(f)- \widetilde{O}(\tilde f))} -1.$$
Using the fact $|e^{ix}-1|\leq |x|$ and Lemma \ref{lem:projection}, we have 
$$
\|e^{ik\,(\widetilde{O}(f)- \widetilde{O}(\tilde f))} -1\|_{L^2(\Omega')}
\leq k\|\widetilde{O}(f)- \widetilde{O}(\tilde f)\|_{L^2(\Omega')} 
\leq C \| f - \tilde f \|_{L^2(\Omega)}.
$$
Then
\begin{align*}
\|R(f,\tilde f) - I\|_{L^2(\Omega')}
\leq C\|\mathcal D_{\mathrm{Fr}}\| \|\mathcal D_{\mathrm{Fr}}^{-1}\| \| f - \tilde f \|_{L^2(\Omega)}  
\leq  C \| f - \tilde f \|_{L^2(\Omega)}.
\end{align*}
Let $\| f - \tilde f \|_{L^2(\Omega)} \leq \epsilon$.
We get
$$
\|R(f,\tilde f) - I\|_{L^2(\Omega')} \leq C\epsilon.
$$
Consequently,
$$
\sup_{f,\tilde f\in\mathcal M}\|R(f,\tilde f)-I\|_{L^2(\Omega')}  \;\le\; \eta<1,
$$
for some $\eta$ depending only on constant $C$.  
This gives the desired estimate (\ref{RIC_2}).
\end{proof}

\section{Numerical Experiments}  

In this section, we evaluate the proposed reconstruction framework in both projection-domain and direct volumetric reconstruction settings, allowing comparison between the conventional two-step reconstruction pipeline and the proposed direct nonlinear reconstruction framework for holotomography.

The performance of different reconstruction strategies is evaluated on both simulated phantom data and real experimental measurements. In particular, we investigate the effects of phase-guided Bregman-TV regularization, including the influence of the guiding parameter, robustness under different $\delta/\beta$ ratios, and reconstruction performance under both single-distance and multi-distance acquisition settings.

\subsection{Reference Solution: Two-step Reconstruction}

As a reference for our results we use the commonly applied two-step procedure of holographic reconstruction in the projection-domain (for each angle separately), followed by a tomographic reconstruction (with filtered backprojection). 
In the projection-domain formulation, the projected refractive quantity $\widetilde{O}_\theta$ is treated as the unknown to be reconstructed from the measured near-field intensity data
$$
I_{m,\theta} = \left| D_{\mathrm{Fr}} \left( P e^{i k \widetilde{O}_\theta} \right) \right|^2 = y_\theta^2.
$$
For each projection angle $\theta$, this leads to the minimization problem
\begin{equation}\label{eq:5.1}
\min_{\widetilde{O}_\theta} \;
\frac{1}{2}
\| G(\widetilde{O}_\theta) - y_\theta \|^2,
\end{equation}
where
$
G(\widetilde{O}_\theta)
=
\left|
\mathcal D_{\mathrm{Fr}}
\left(
P e^{ik\,\widetilde O_\theta}
\right)
\right|,
$
denotes the predicted amplitude measurement.

Within this formulation, the real and imaginary parts of $\widetilde{O}_\theta$ correspond to the projected phase and absorption components, respectively. 
An advantage of the proposed phase-guided Bregman TV framework is that it can be naturally applied to \eqref{eq:5.1} through a coupled regularization of $\Re(\widetilde{O}_\theta)$ and $\Im(\widetilde{O}_\theta)$, where structural information from $\Re(\widetilde{O}_\theta)$ is incorporated to guide the reconstruction of $\Im(\widetilde{O}_\theta)$.

\subsection{Numerical Settings}
All numerical experiments rely on a consistent discretization of the forward model introduced in Section \ref{sec: Forward Model}. The Fresnel operator \eqref{eq:Fresnel_kernel} is realized numerically as
$$
\mathcal D_{\mathrm{Fr}}^{\mathrm{disc}}(u)
   = \mathrm{FFT}^{-1}\!\left(H_d^{\mathrm{disc}} \odot \mathrm{FFT}(u)\right),
$$
where $\odot$ denotes element-wise multiplication and the superscript denotes the discrete approximation of the continuous Fresnel kernel.  
The quadratic phase factor in $H_d(\xi)=\exp\!\Big(-i\pi \frac{|\xi|^2}{2\,\mathrm{Fr}}\Big)$ is a chirp function whose oscillation speed increases as the Fresnel number $\mathrm{Fr}$ decreases.  
To represent this kernel faithfully on a discrete grid, the number of pixels $N$ in each transverse dimension must satisfy the Fresnel sampling condition
$
N \;\ge\; \frac{1}{\mathrm{Fr}},
$
which determines the minimum computational window required for accurately approximating the continuous propagator \cite{voelz2009digital}.  
If this condition is violated, the chirp kernel becomes undersampled and the propagated field suffers from aliasing.  
Since the FFT imposes periodic boundary conditions, such undersampling produces wrap-around artifacts where diffracted features re-enter the image from the opposite side.  
To avoid these nonphysical effects, the exit wave is embedded into a larger computational window using the padding strategy proposed by \cite{dora2024artifact}, consisting of mirror padding, constant extension, and smooth windowing before Fresnel propagation. This ensures that the Fresnel kernel is adequately sampled while the periodic replicas introduced by the FFT remain outside the region of interest.

For each projection angle~$\theta \in \{\theta_1,\dots,\theta_{N_\theta}\}$, the parallel-beam projection operator $\widetilde{O}_{\theta}$ is implemented numerically by integrating the rotated volume along the beam direction. In practice, this step is realized using a discrete Radon transform, yielding a two-dimensional complex projection $\widetilde{O}_{\theta}$, from which the transmission function is evaluated pointwise as
$$
O_{\theta} = \exp(i k\, \widetilde{O}_{\theta}).
$$
Then a probe multiplies this field.  
In all numerical experiments, we use a spatially constant probe~$P$ so that illumination structure does not obscure the comparison between reconstruction strategies.  
In physical measurements, the probe is typically a non-uniform complex field with a Gaussian-like amplitude envelope and a nontrivial phase profile.  
Probe reconstruction techniques, such as \cite{hagemann2017probe, nikitin2024x}, can estimate the illumination jointly with the object. The proposed algorithm is compatible with such extensions as the probe enters the model only through multiplicative factors.
The propagated field becomes
$$
v_{\theta}
    = \mathcal D_{\mathrm{Fr}}^{\mathrm{disc}}\!\bigl(P \odot O_{\theta}\bigr),
\qquad
G_{\theta}(f)=|v_{\theta}|.
$$   
The proximal minimization problems \eqref{eq: delta reg} \eqref{eq: beta} arising from Bregman TV formulations are solved using a standard primal--dual algorithm  \cite{chambolle2011first}.

\subsection{Simulation Studies on a synthetic 3D Phantom}
\subsubsection{Single-distance Reconstruction}

We first evaluate the proposed reconstruction framework on synthetic near-field data generated from a multi-sphere phantom of size $128\times 128\times 128$. The phantom is used to define the complex relative refractive index
$$
f(\mathbf r)
=
-\delta(\mathbf r)
+
i\beta(\mathbf r)
=
(-\delta_0+i\beta_0)\times\,\mathrm{phantom3D}(\mathbf r),
\quad 
\text{with}
\quad
\delta_0 = 10^{-6},
\qquad
\beta_0 = 10^{-8}.
$$
This choice corresponds to a weakly absorbing specimen, in which the phase contrast is approximately two orders of magnitude stronger than the absorption contrast.

The X-ray wavelength is set to
$\lambda = 10^{-10}\,\mathrm{m}$, with a propagation distance of $d = 0.01\,\mathrm{m}$ and a detector
pixel size of $\Delta x = 2.0\times10^{-7}\,\mathrm{m}$, corresponding to a Fresnel number
$\mathrm{Fr} = (\Delta x)^2/(\lambda d)\approx 0.04$, placing the setup in the near-field regime. A total of $90$ projection angles are uniformly sampled over $[0,\pi)$. We assume a constant complex probe and use noise-free intensity measurements.

Figure~\ref{fig:phantom_data} provides a visual overview of the synthetic multi-sphere phantom, representative refractive index components, and corresponding near-field measurements, illustrating both the geometric structure of the simulated object and the associated diffraction characteristics.

\begin{figure}[H]
\centering
\begin{minipage}[t]{0.25\linewidth}
\centering
\includegraphics[width=1\textwidth]{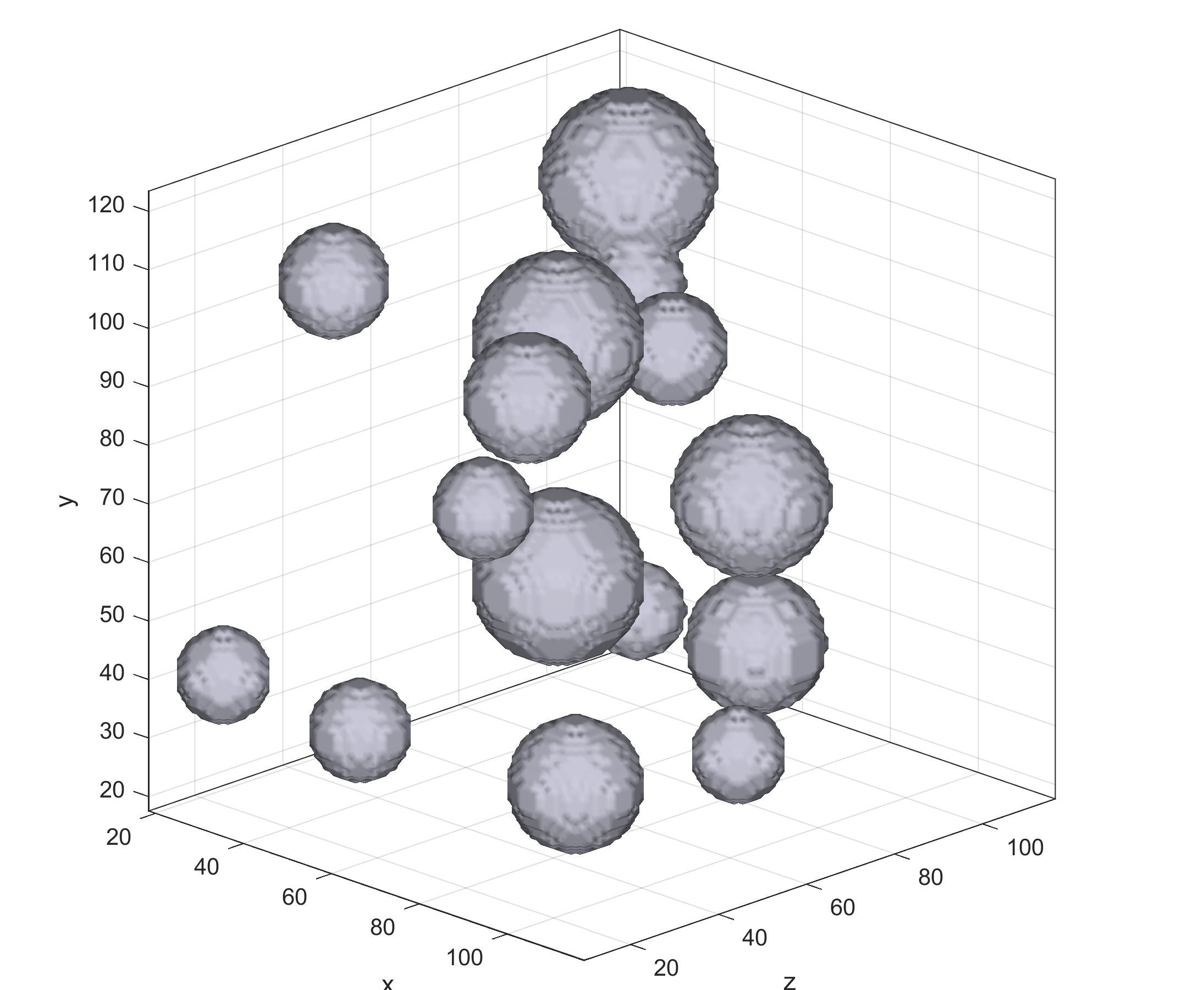}\\
\centerline{\footnotesize\text{Phantom3D}}
\includegraphics[width=1\textwidth]{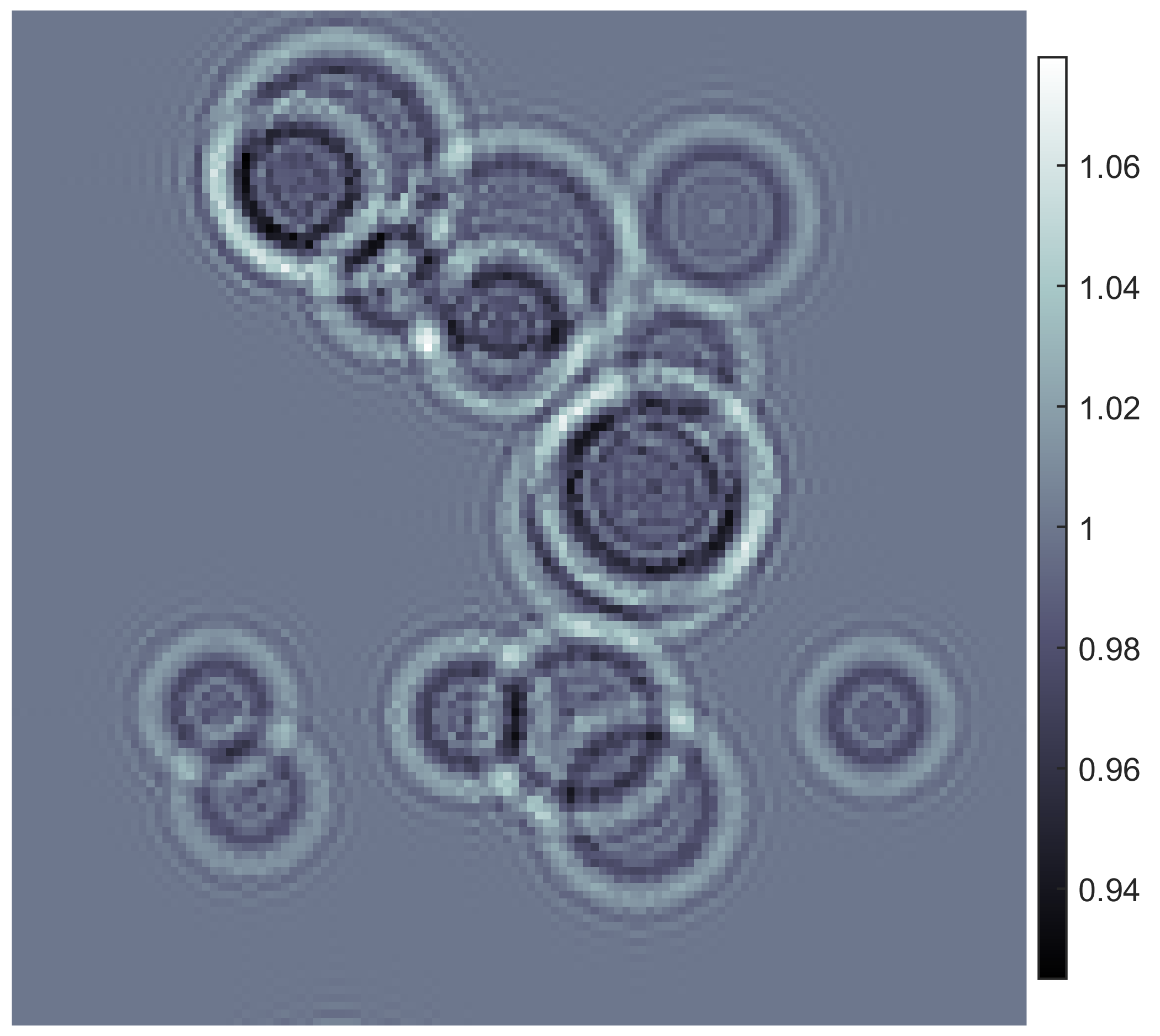}\\
\centerline{\footnotesize\text{$I_m$}}
\end{minipage}    
\begin{minipage}[t]{0.25\linewidth}
\centering
\includegraphics[width=1\textwidth]{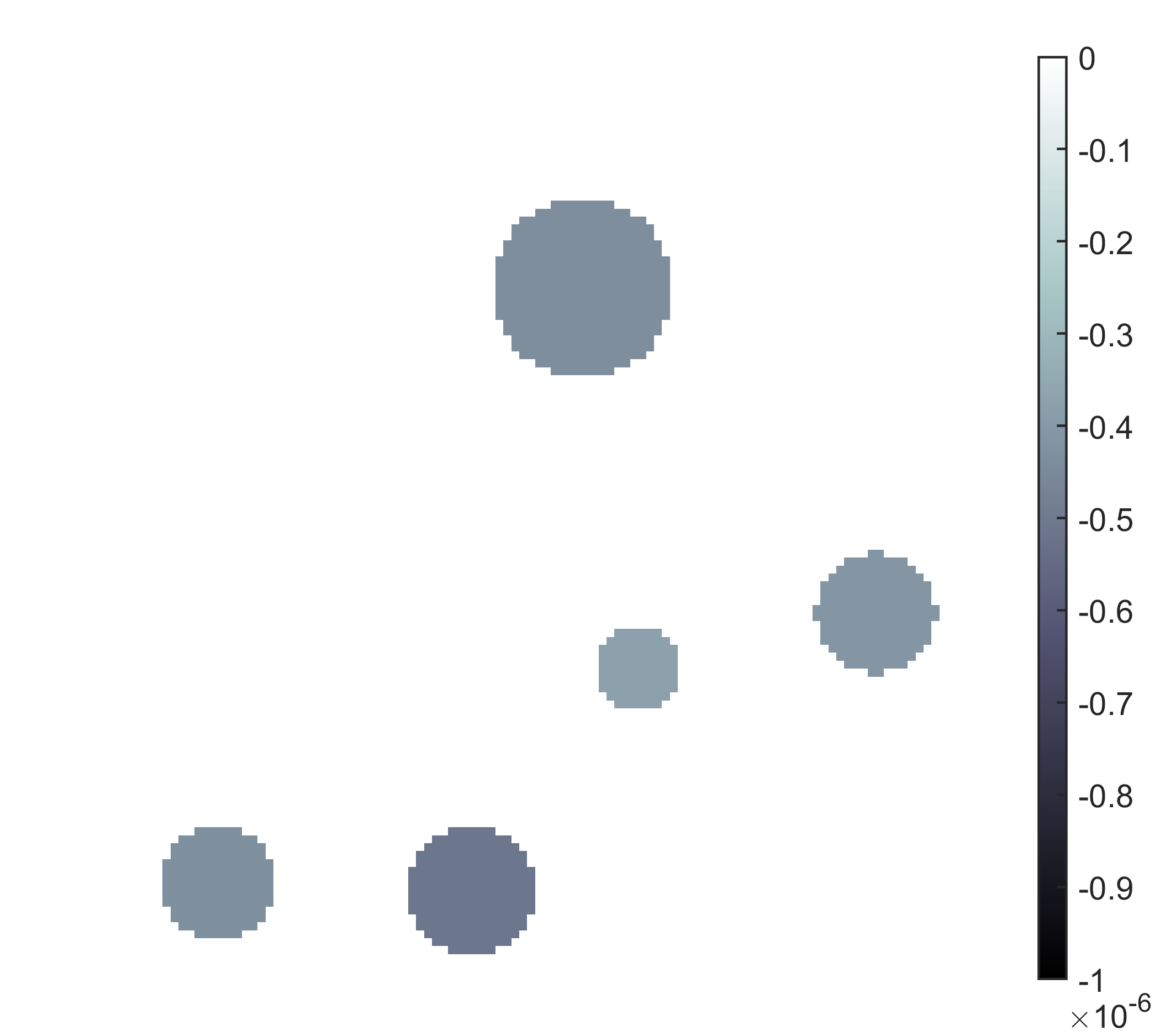}\\
\centerline{\footnotesize\text{$-\delta_{true}$}}
\includegraphics[width=1\textwidth]{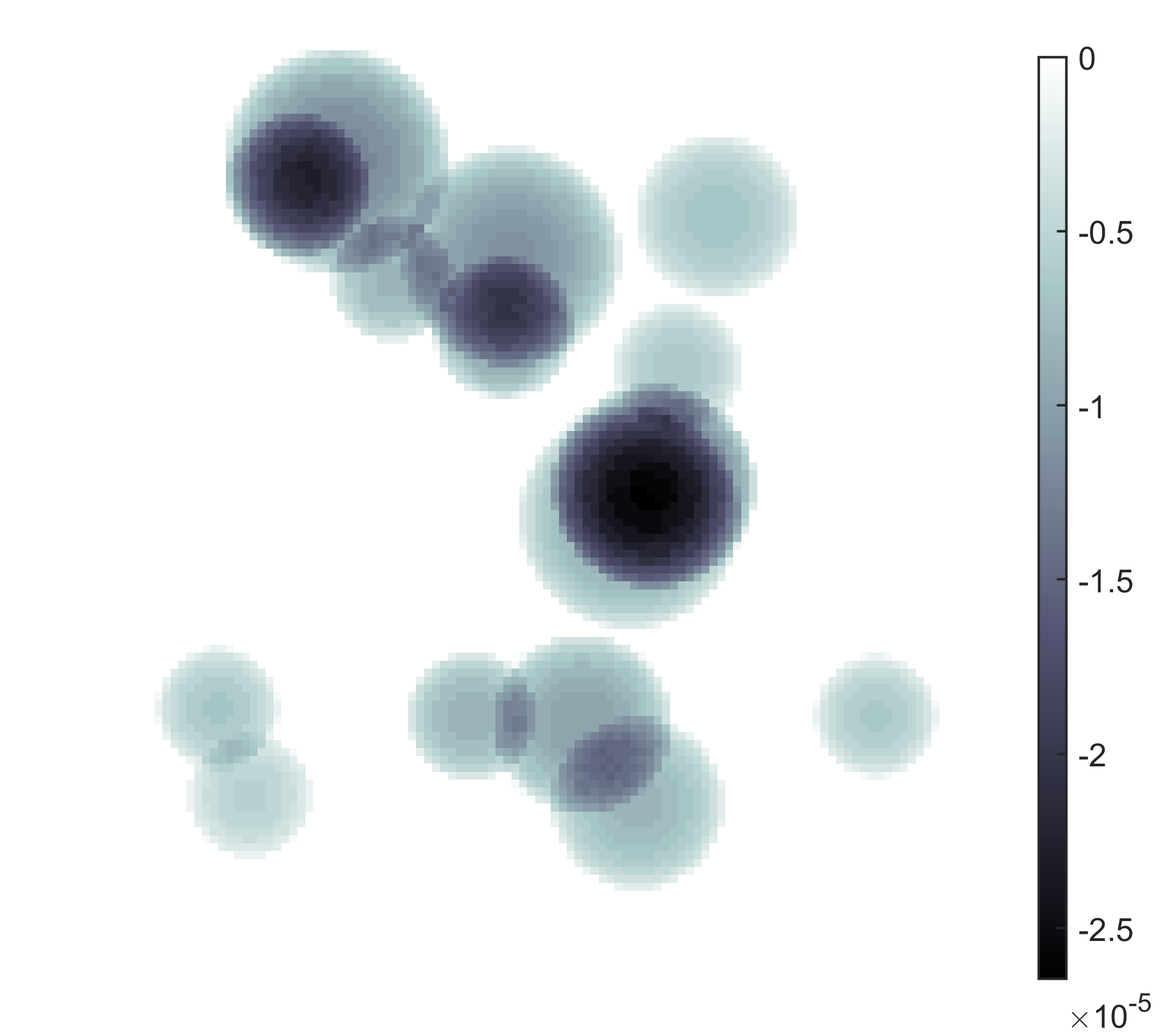}\\
\centerline{\footnotesize\text{$\Re(\widetilde O_{true})$}}
\end{minipage}
\begin{minipage}[t]{0.25\linewidth}
\centering
\includegraphics[width=1\textwidth]{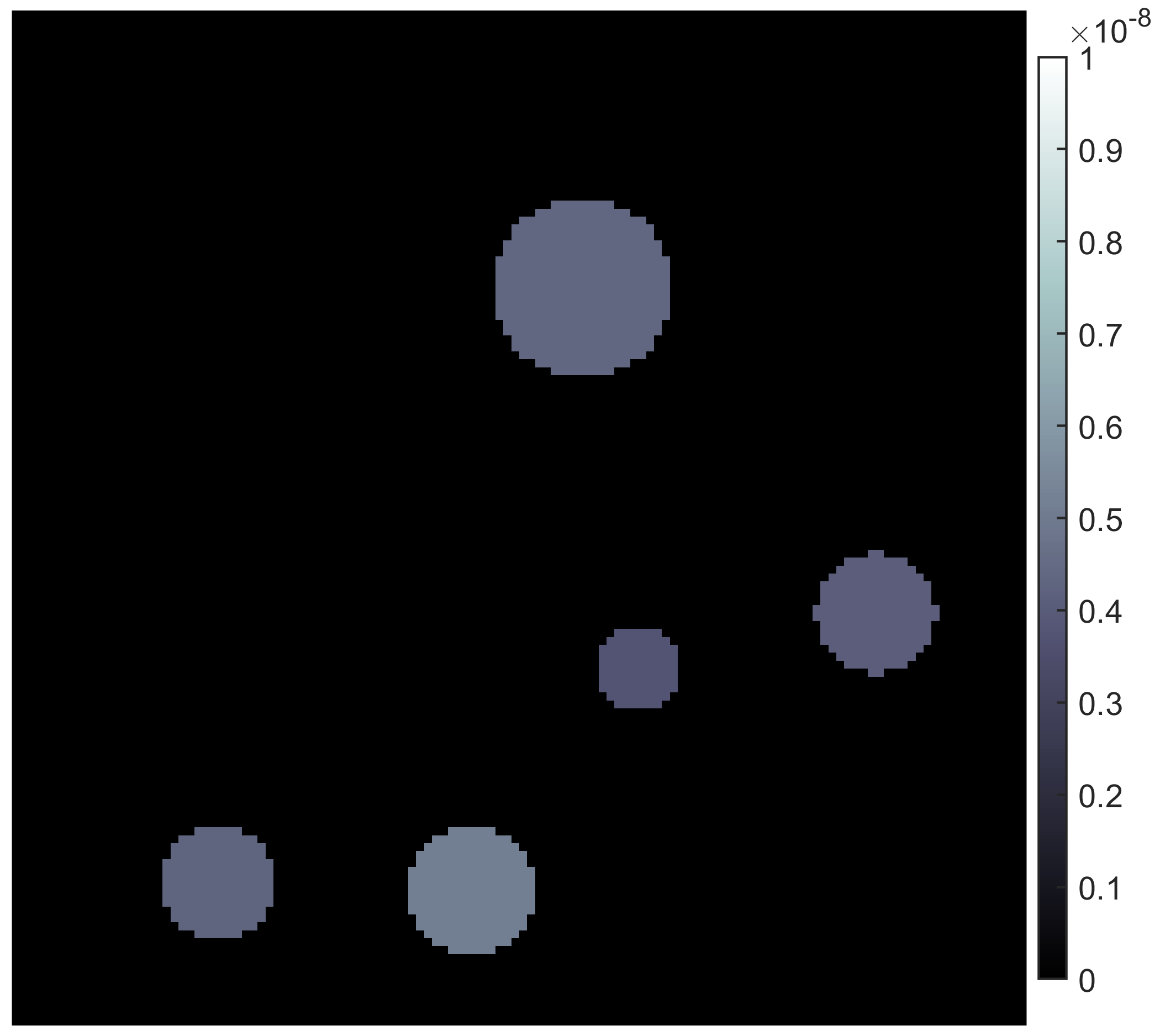}\\
\centerline{\footnotesize\text{$\beta_{true}$}}
\includegraphics[width=1\textwidth]{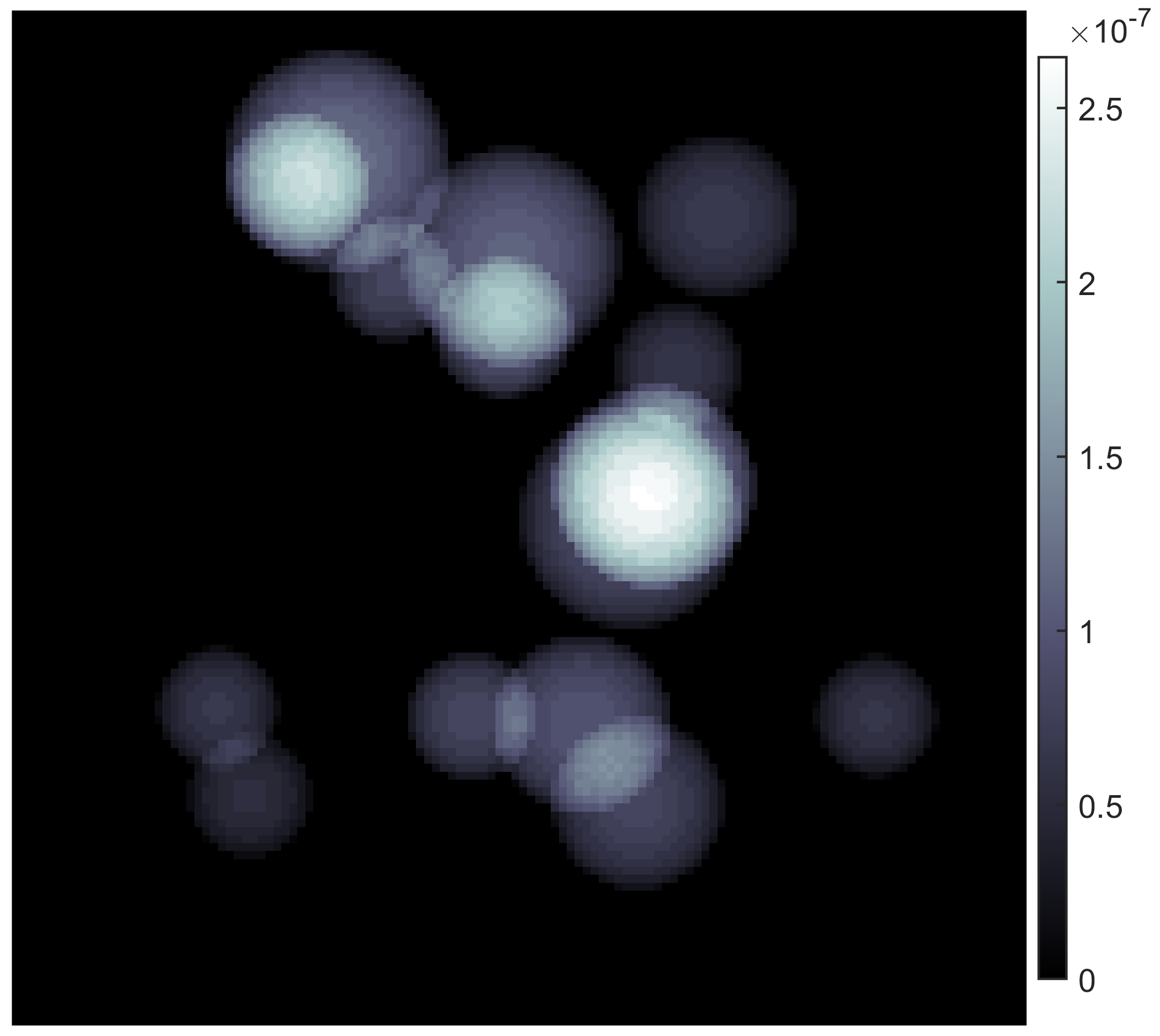}\\
\centerline{\footnotesize\text{$\Im(\widetilde O_{true})$}}
\end{minipage}
\caption{Simulated multi-sphere phantom and corresponding ground truth data used in the reconstruction experiment. Top left: 3D rendering of the phantom. Top middle: ground truth phase shift component on the slice $z=38$. Top right: ground truth absorption component on the slice $z=38$. Bottom left: measured near-field intensity $I_m$ at a representative projection angle $\theta = 0$. Bottom middle: projected phase shift $ \Re (\widetilde{O})$ at $\theta = 0$. Bottom right: projected absorption $ \Im (\widetilde{O})$ at $\theta = 0$.}
\label{fig:phantom_data}
\end{figure}
We first consider the projection-domain reconstruction with \eqref{eq:5.1}. Figure~\ref{fig:projection_recon} presents the reconstructed complex projected object $\widetilde{O}_\theta$ obtained using different reconstruction strategies. With gradient descent (GD) alone, the projected phase component $\Re(\widetilde{O}_\theta)$ is reconstructed reasonably well, with the main structures clearly recovered. In contrast, the projected absorption component $\Im(\widetilde{O}_\theta)$ exhibits significant degradation and loss of structural fidelity. This discrepancy indicates that the reconstruction of the absorption component is substantially more challenging and less stable under the basic gradient descent scheme.
\begin{figure}[t]
    \centering
    \setlength{\tabcolsep}{1pt}

    \begin{tabular}{c c c c}
        & \footnotesize GD 
        & \footnotesize Bregman TV 
        & \footnotesize Phase-guided \\

        \raisebox{16mm}{ \footnotesize $\Re(\widetilde O_\theta)$} &
        \includegraphics[width=0.25\textwidth]{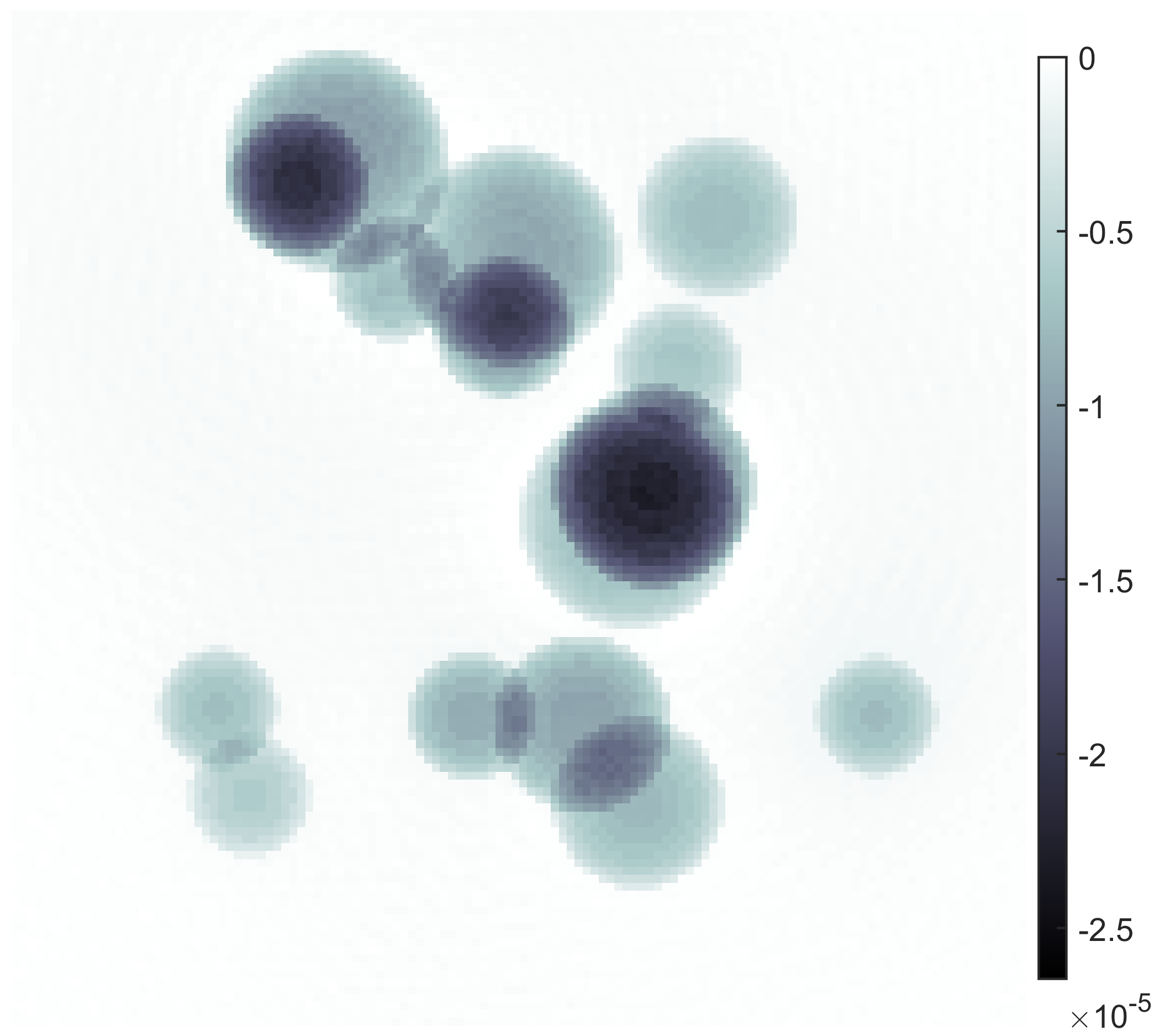} &
        \includegraphics[width=0.25\textwidth]{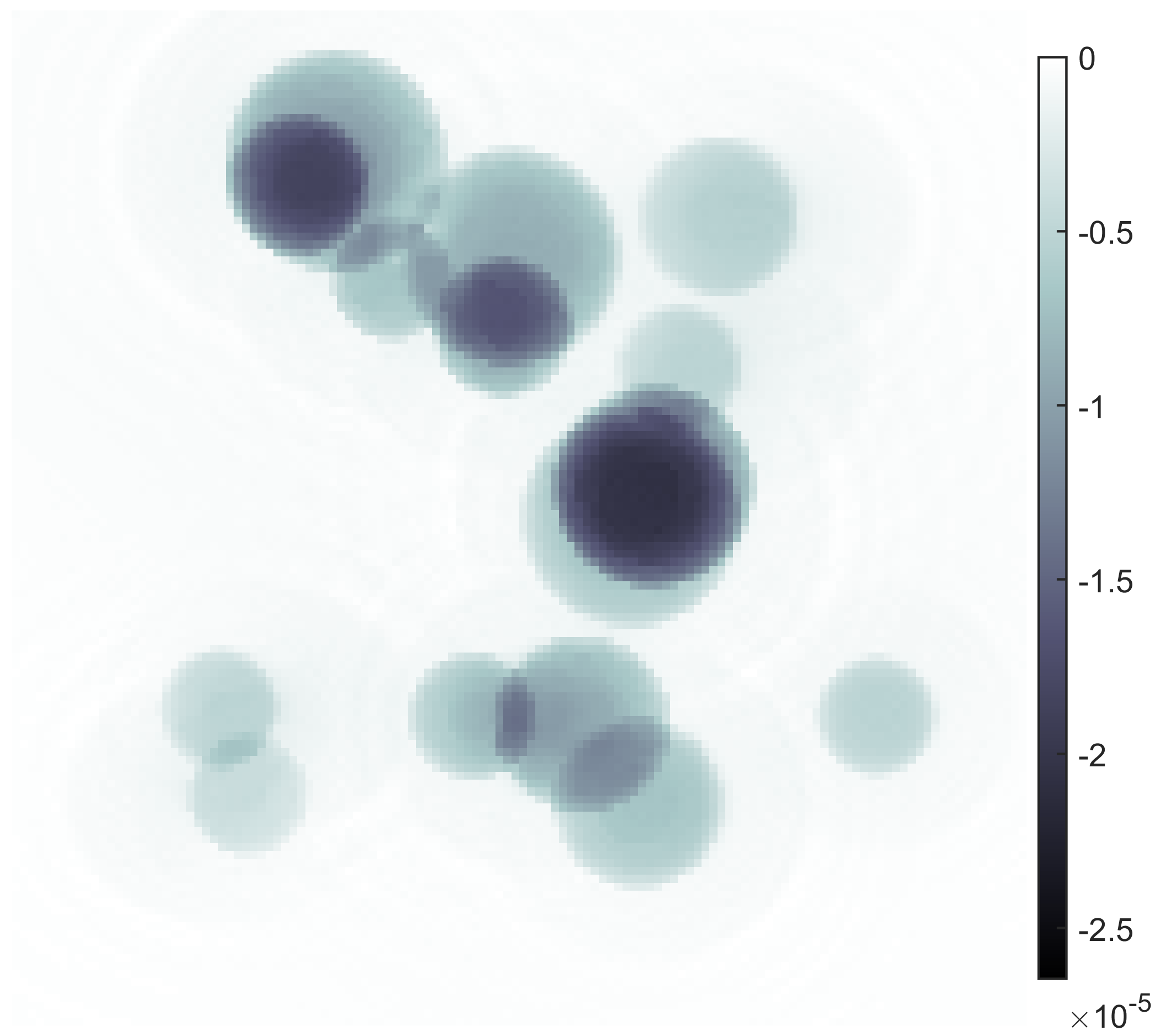} &
        \includegraphics[width=0.25\textwidth]{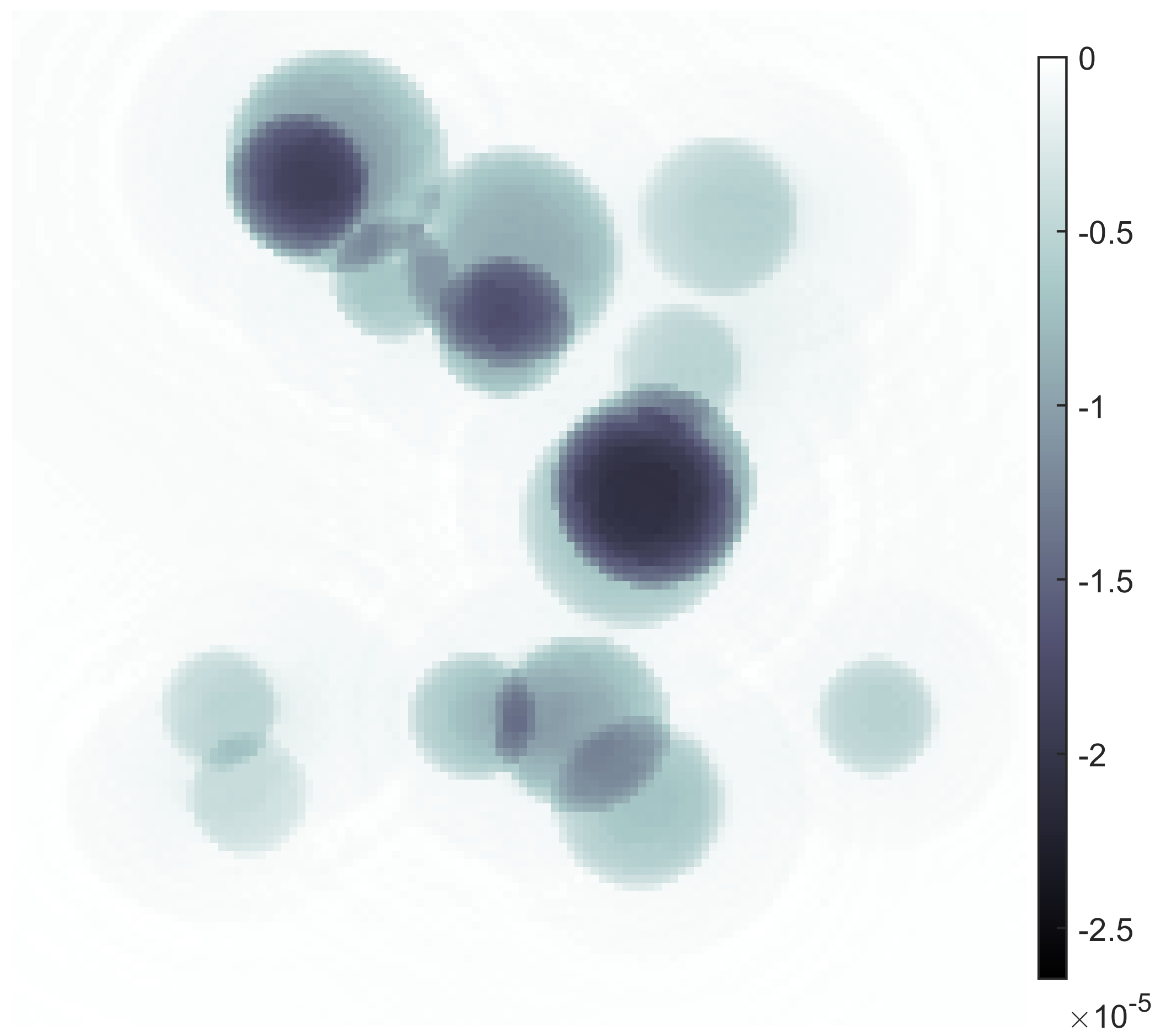} \\

        \raisebox{16mm}{ \footnotesize $\Im(\widetilde{O}_\theta)$} &
        \includegraphics[width=0.25\textwidth]{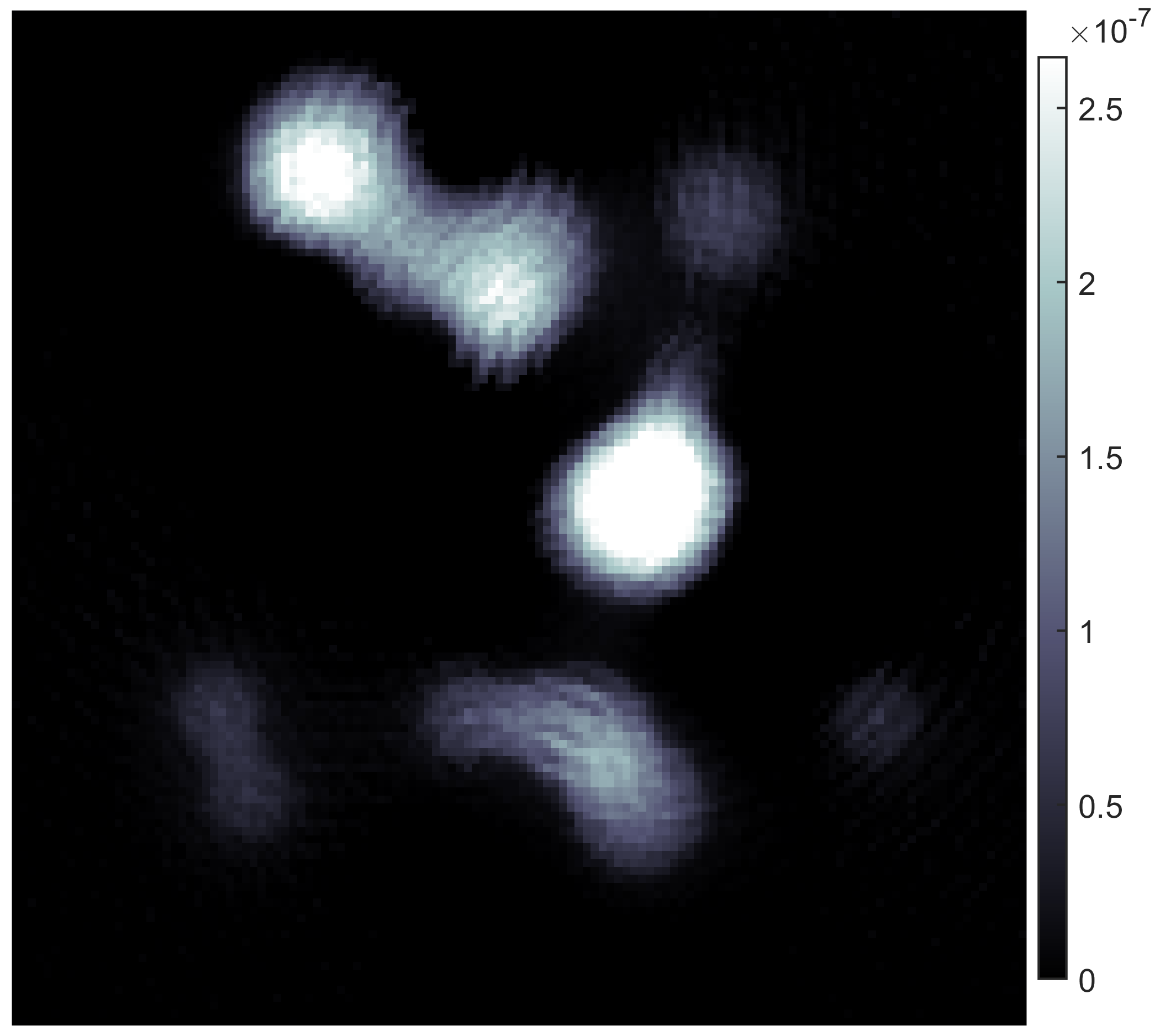} &
        \includegraphics[width=0.25\textwidth]{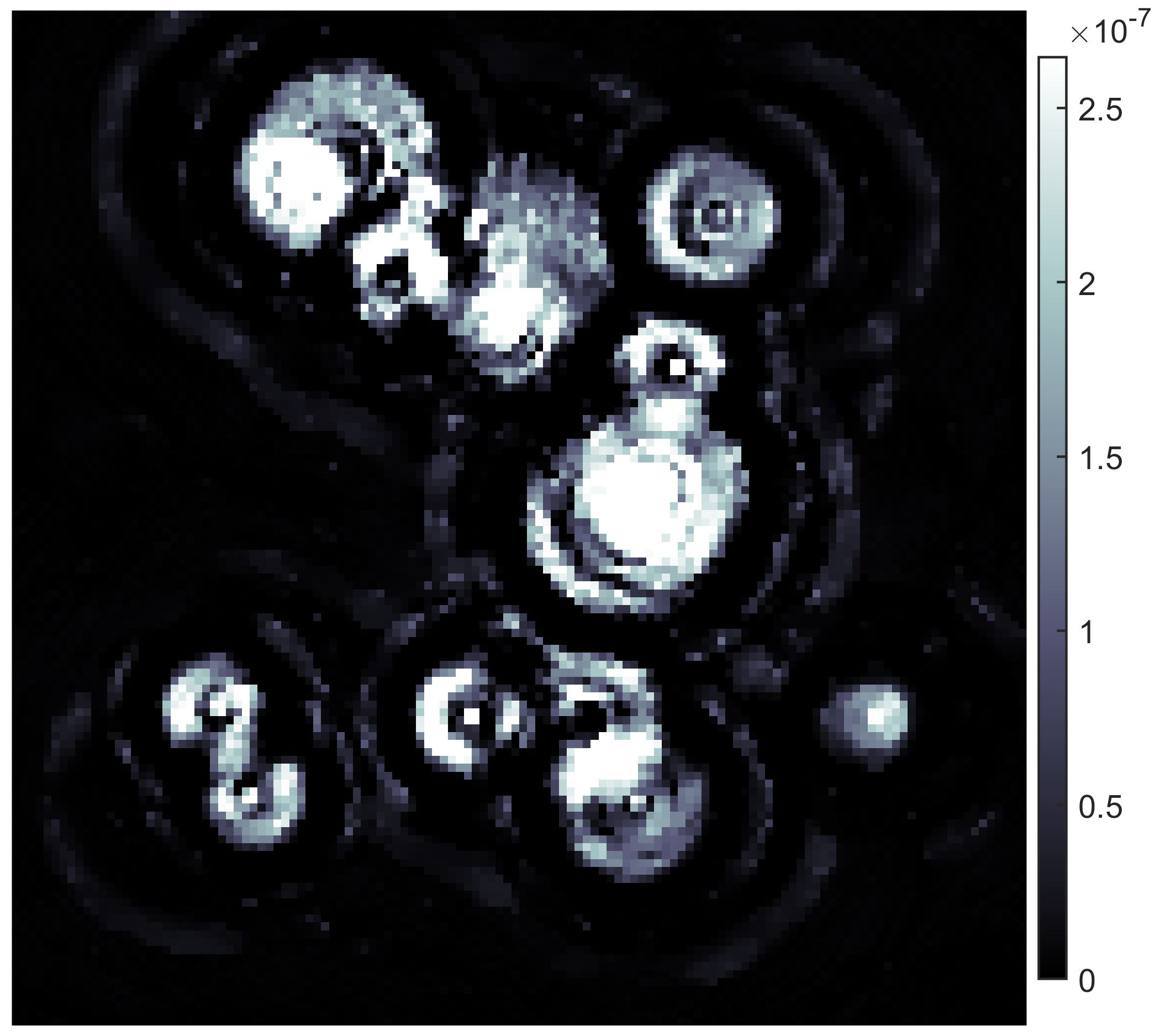} &
        \includegraphics[width=0.25\textwidth]{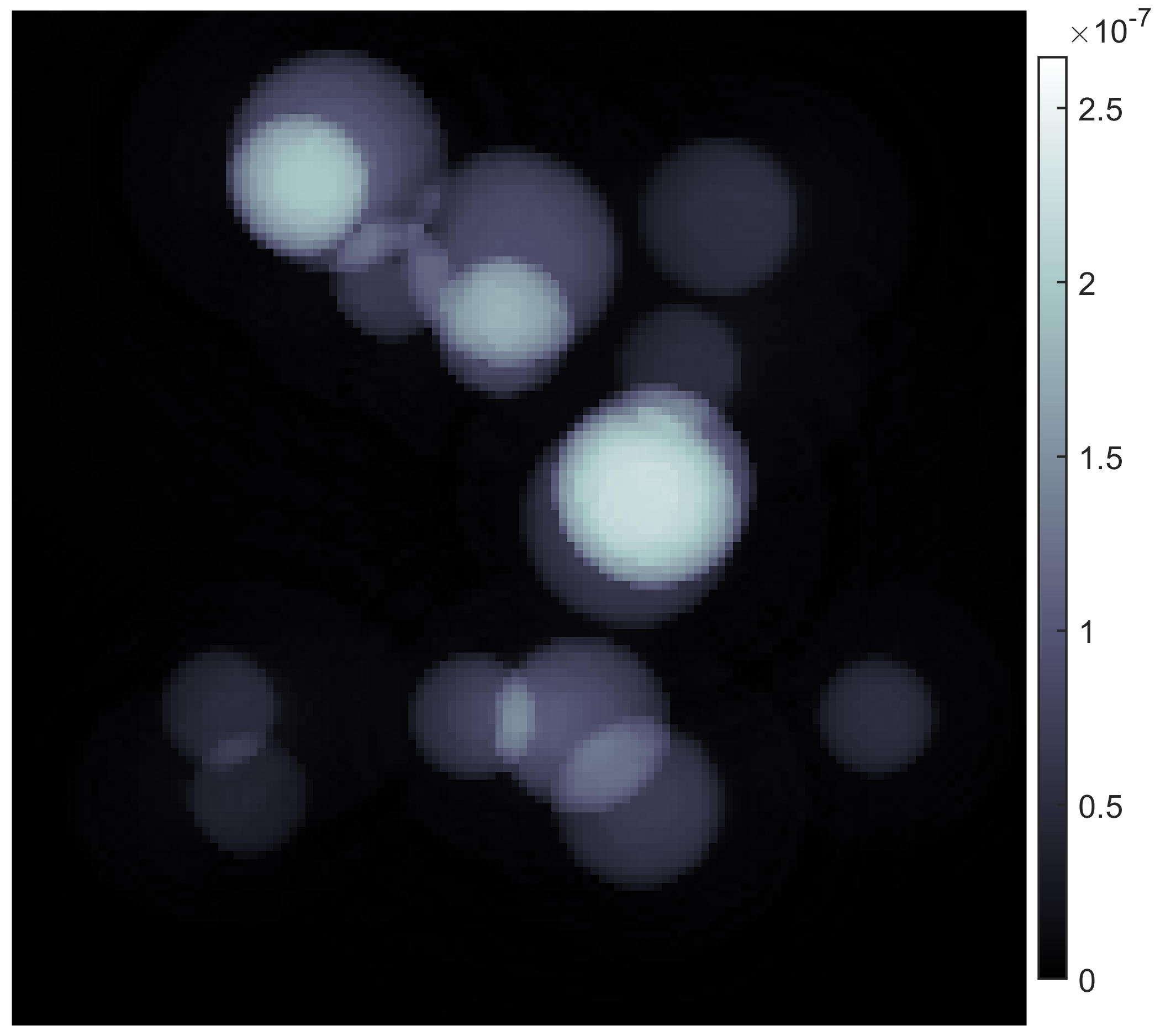} \\
    \end{tabular}

    \caption{Projection-domain reconstruction of the complex projected object $\widetilde{O}_\theta$ for a representative projection angle $\theta = 0$. 
    The first row shows the projected phase decrement $\widetilde{O}_\theta$, and the second row shows the projected absorption index $\widetilde{O}_\theta$. 
    Columns correspond to different reconstruction methods.}
    \label{fig:projection_recon}
\end{figure}

Introducing Bregman TV regularization improves the overall reconstruction stability. While the projected phase component changes only slightly, the projected absorption component shows visibly improved structural coherence.

Further improvement is obtained with phase-guided regularization ($\gamma = 0.1$). Compared with the Bregman TV reconstruction, the absorption component $\Im(\widetilde{O}_\theta)$ exhibits improved structural recovery and reduced artifacts. This indicates that structural information from the projected phase can provide effective guidance for the reconstruction of the absorption component. 

For comparison with direct 3D reconstruction, the reconstructed projections $\widetilde{O}_\theta$ obtained from the phase-guided projection-domain reconstruction are further processed using filtered backprojection (FBP) to recover the volumetric refractive index components $\delta$ and $\beta$. The corresponding FBP reconstructions and error maps are shown in the first column of Figure~\ref{fig:direct3d_recon}. The reconstructed $\delta$ accurately captures the overall structure of the object, with the main features and spatial distribution correctly recovered. The corresponding error map indicates that the reconstruction errors are primarily localized near object boundaries. Similar reconstruction characteristics are observed for $\beta$.
\begin{figure}[H]
    \centering
    \setlength{\tabcolsep}{1pt}

    \begin{tabular}{c c c c c}
        & \footnotesize Two-step FBP
        & \footnotesize GD 
        & \footnotesize Bregman TV 
        & \footnotesize Phase-guided \\
        
        \rowlabel{$-\delta_{reco}$} &
        \includegraphics[width=0.23\textwidth]{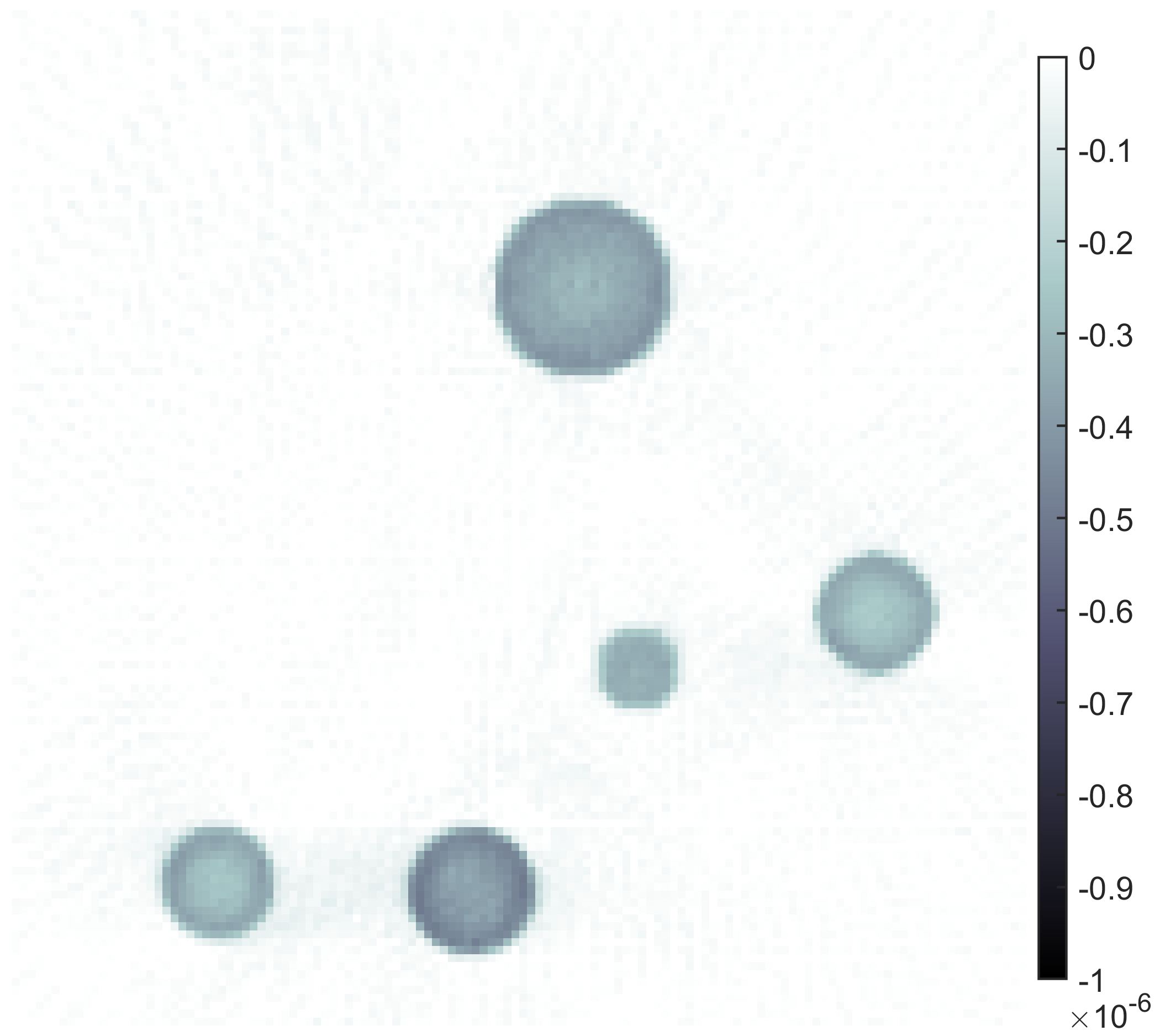} &
        \includegraphics[width=0.23\textwidth]{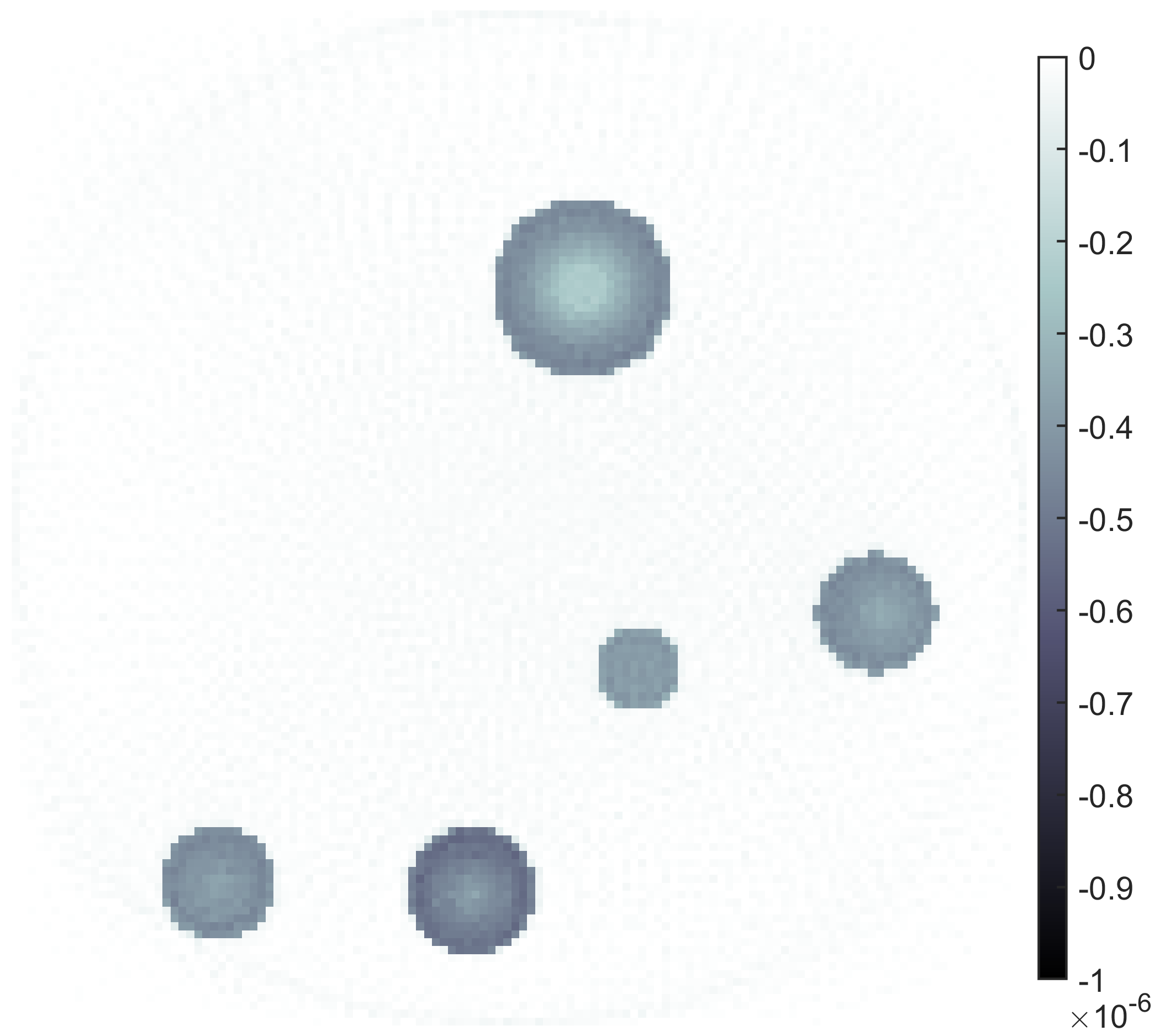} &
        \includegraphics[width=0.23\textwidth]{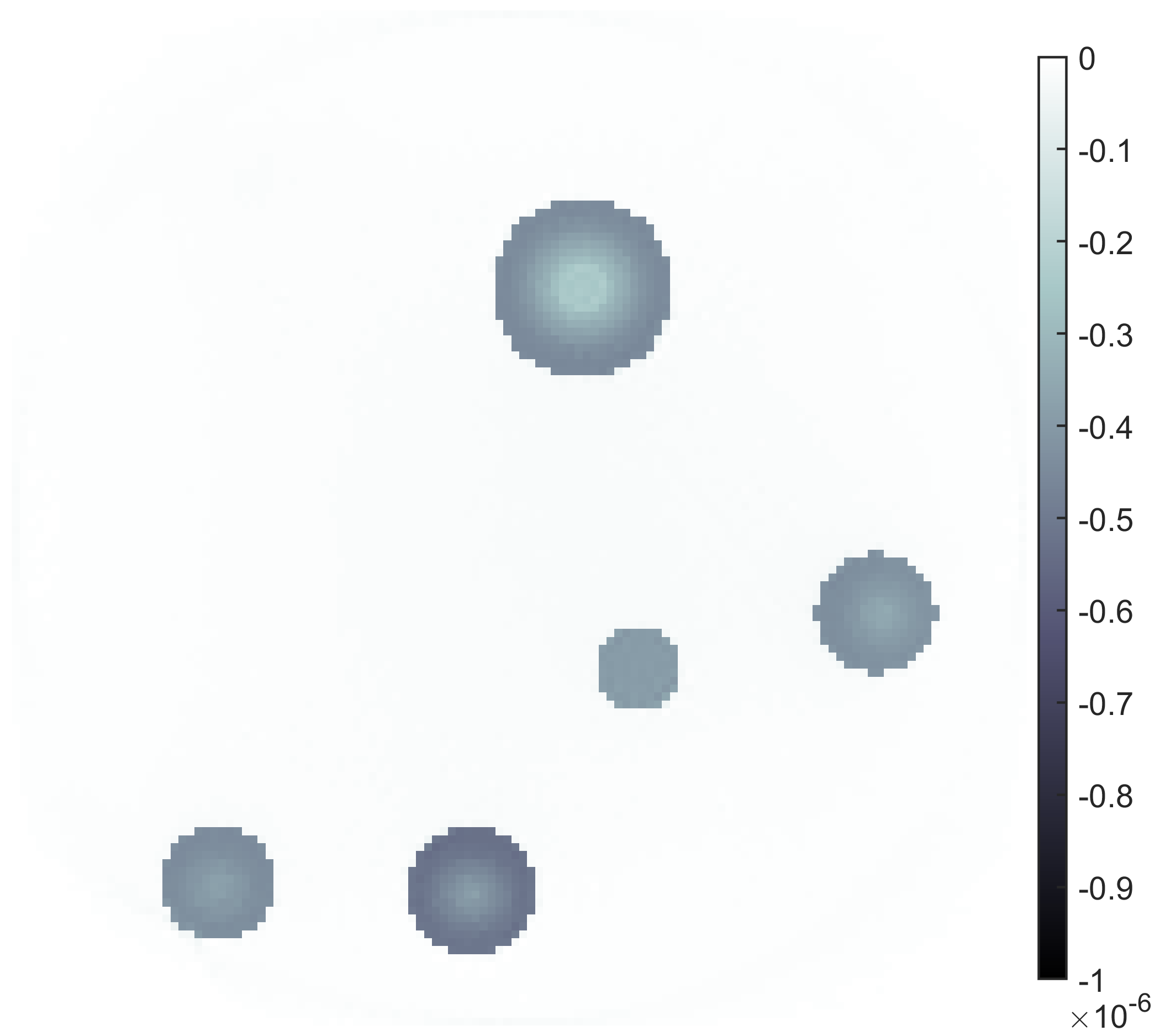} &
        \includegraphics[width=0.23\textwidth]{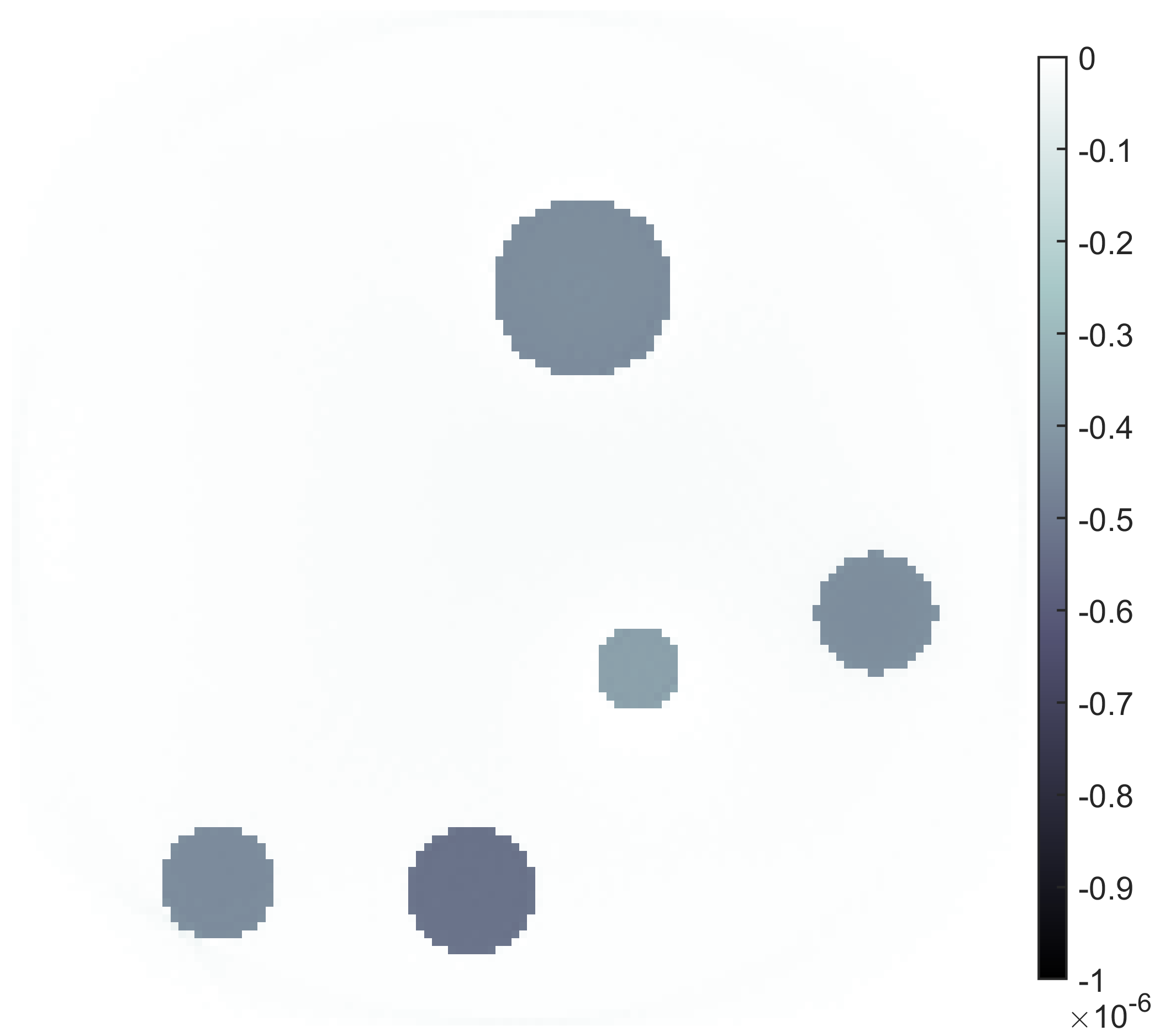} \\

        \rowlabel{diff $\delta$} &
         \includegraphics[width=0.23\textwidth]{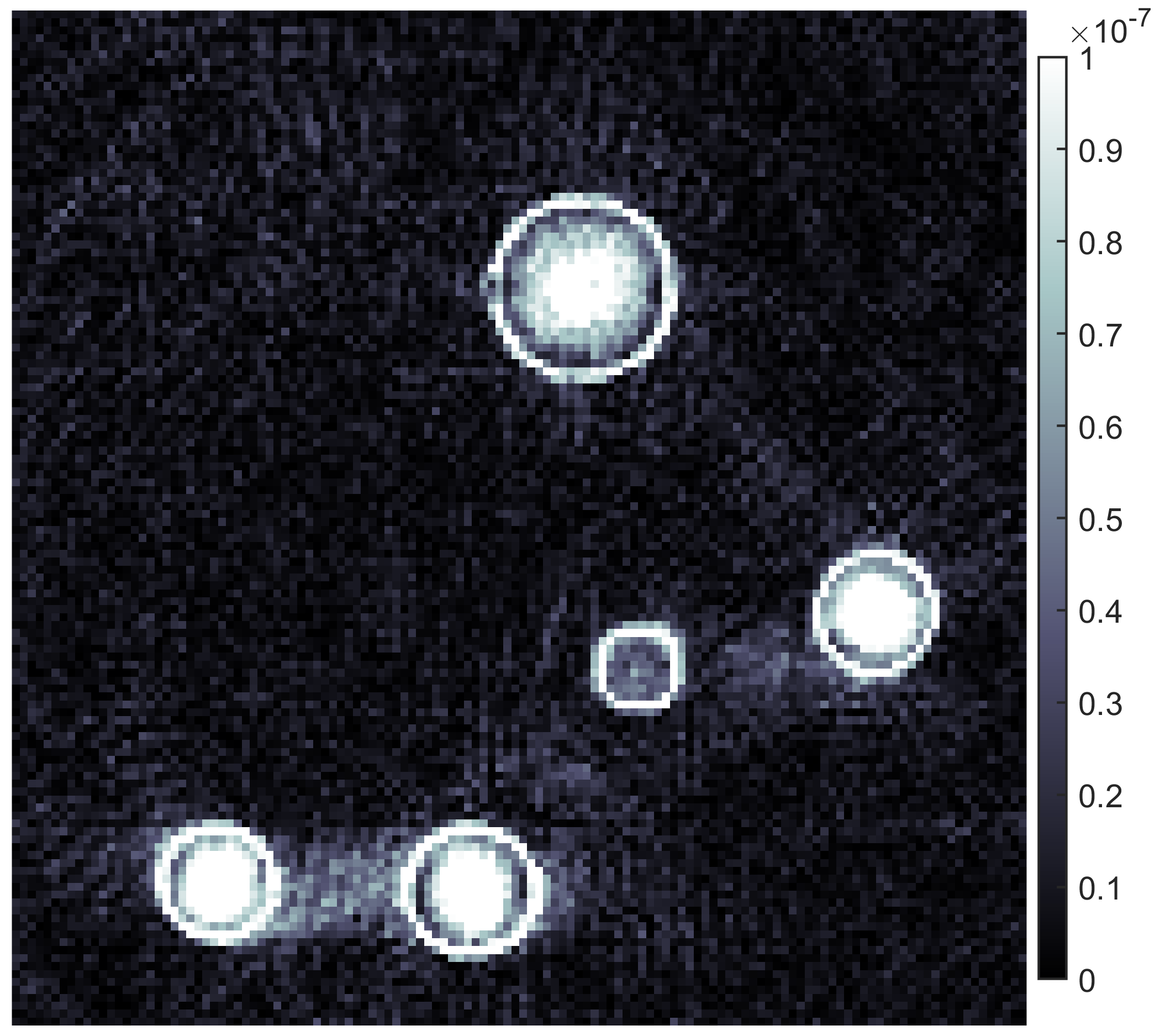} &
        \includegraphics[width=0.23\textwidth]{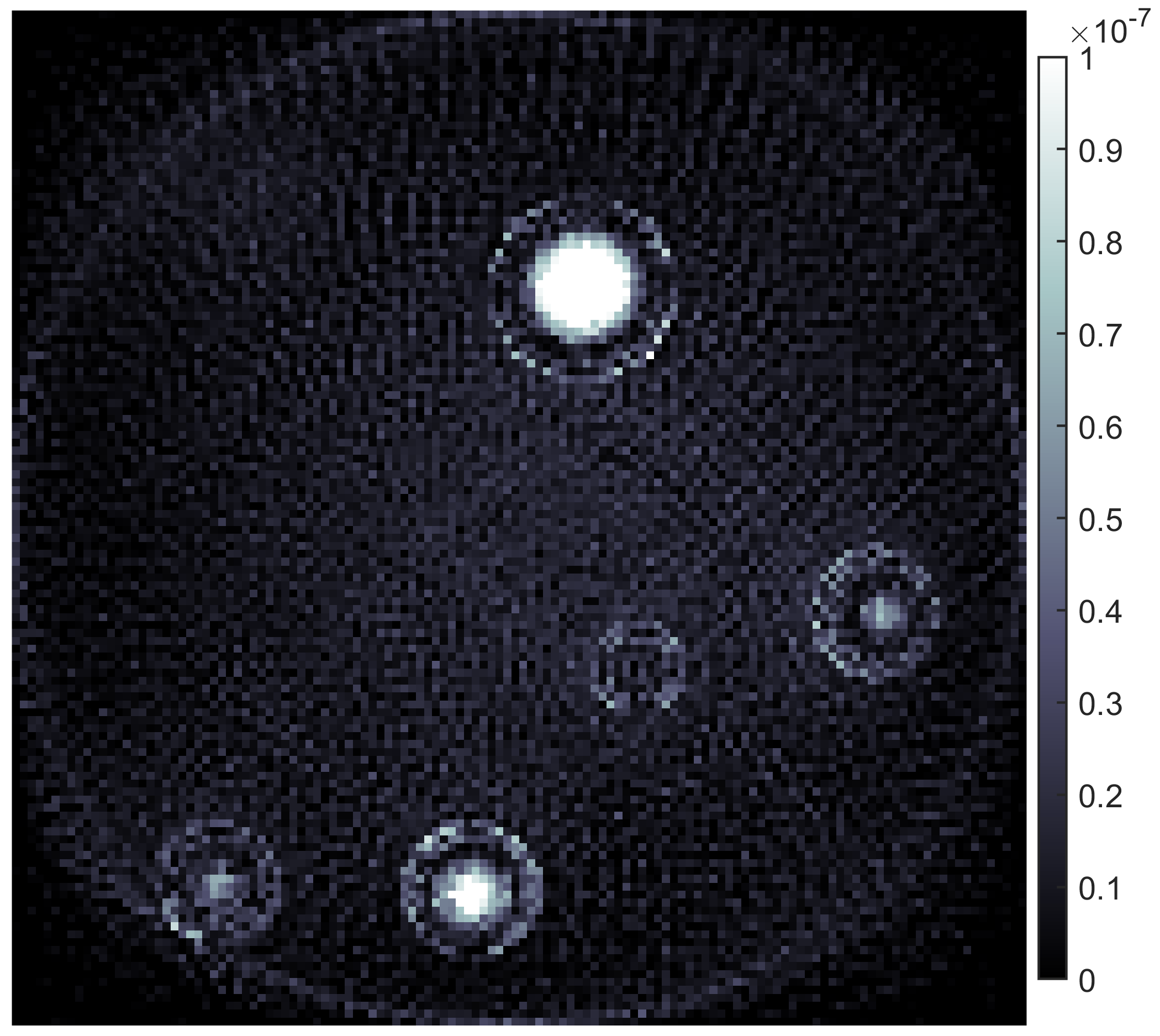} &
        \includegraphics[width=0.23\textwidth]{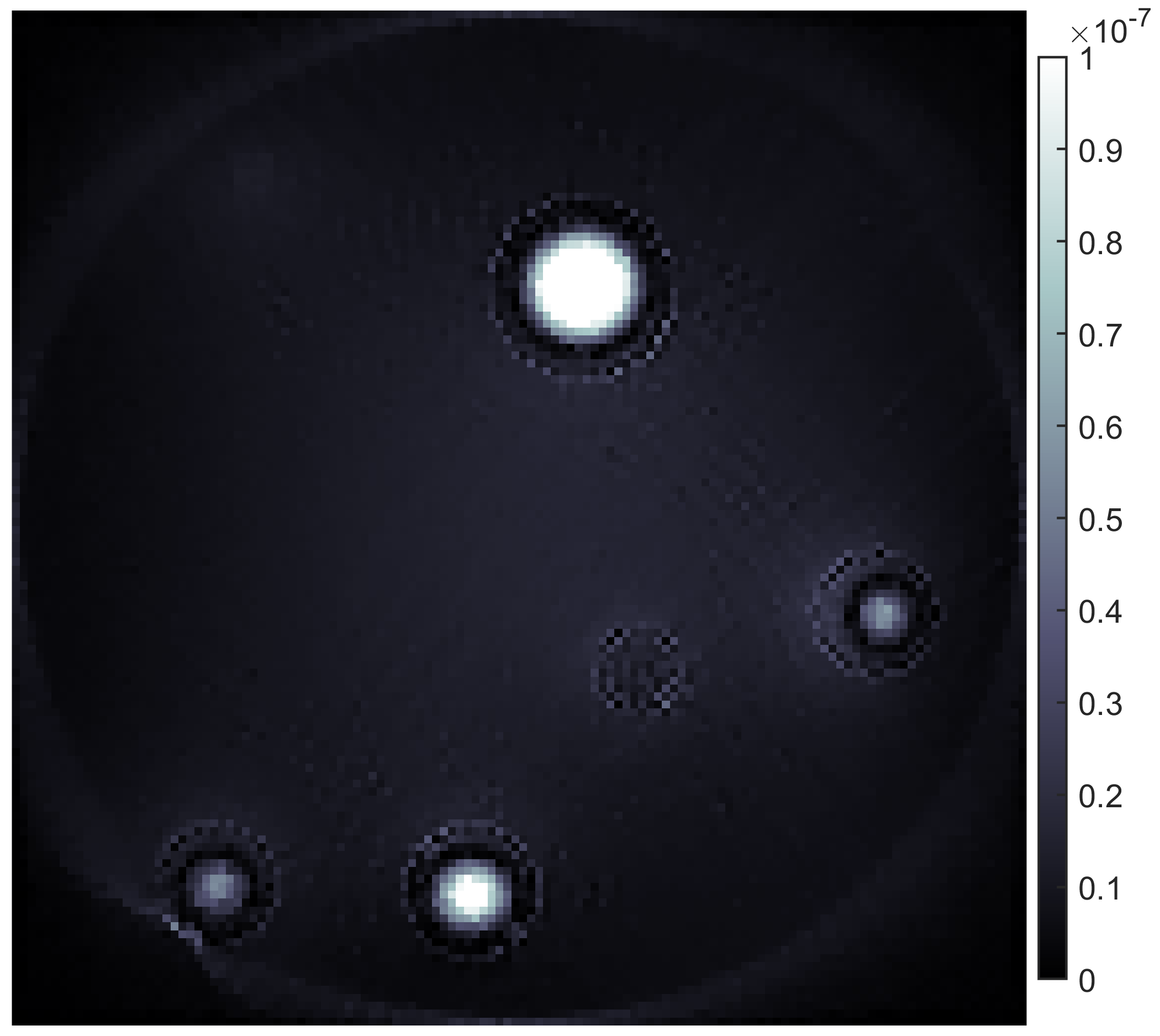} &
        \includegraphics[width=0.23\textwidth]{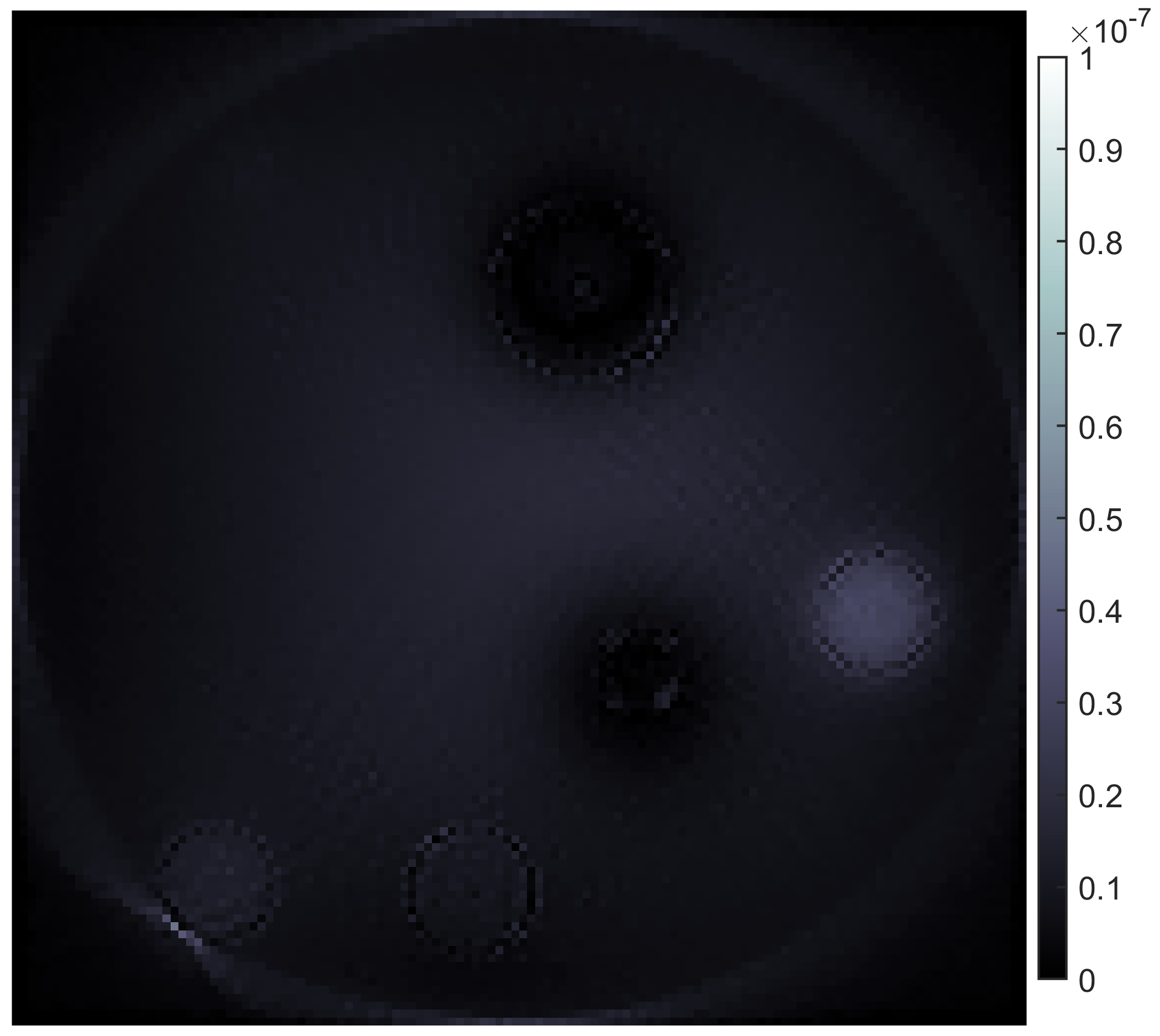} \\

        \rowlabel{$\beta_{reco}$} &
        \includegraphics[width=0.23\textwidth]{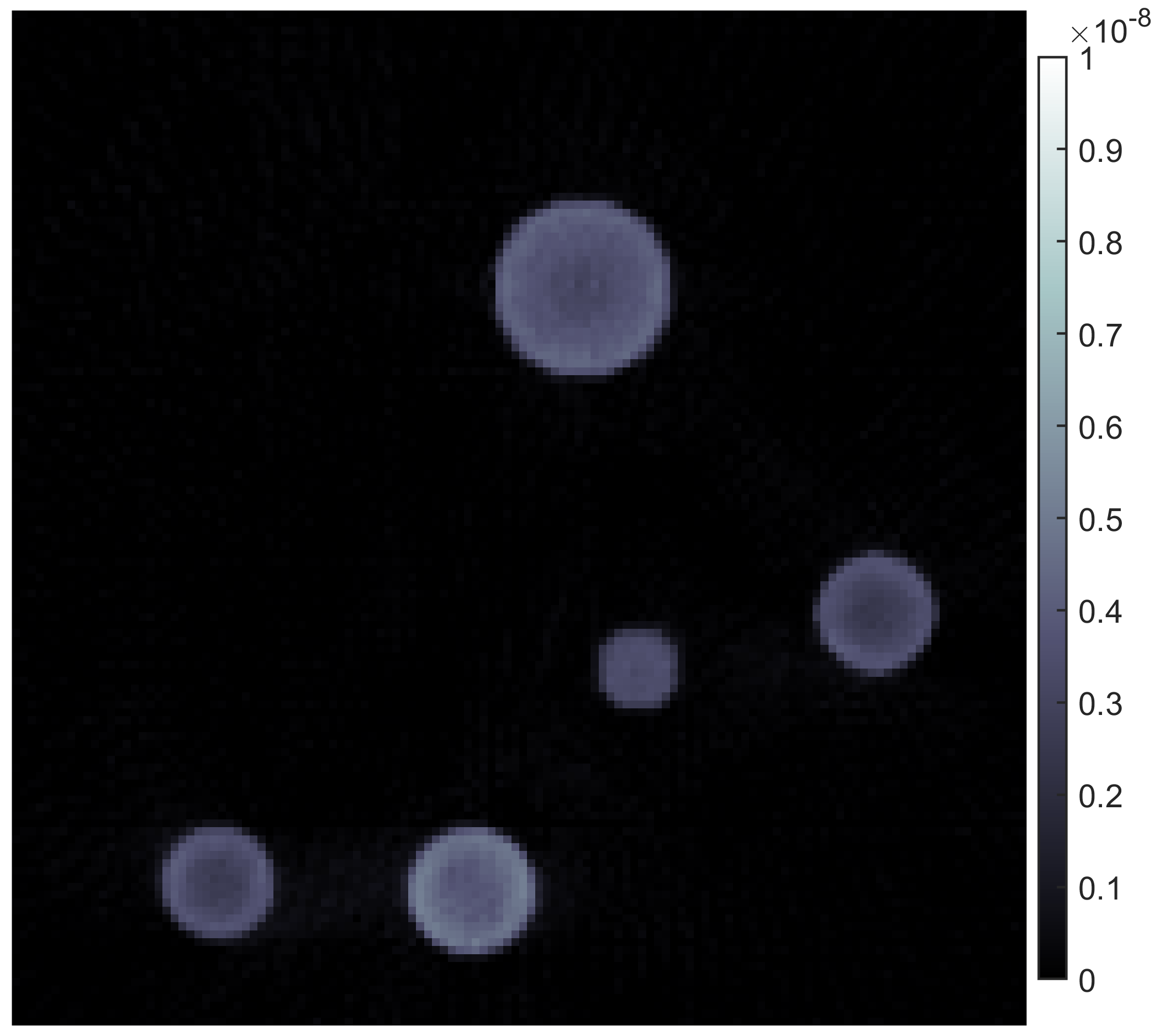} &
        \includegraphics[width=0.23\textwidth]{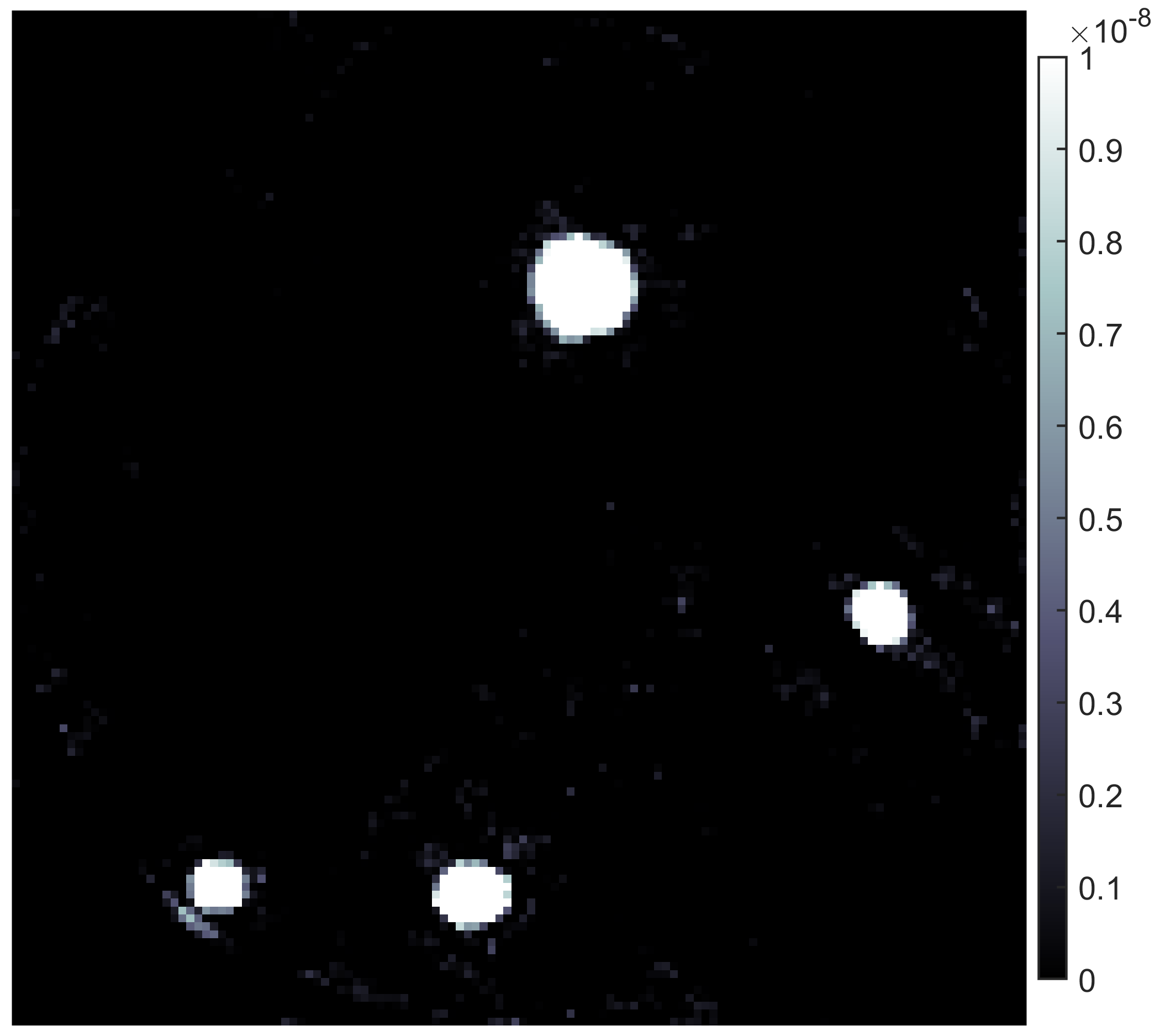} &
        \includegraphics[width=0.23\textwidth]{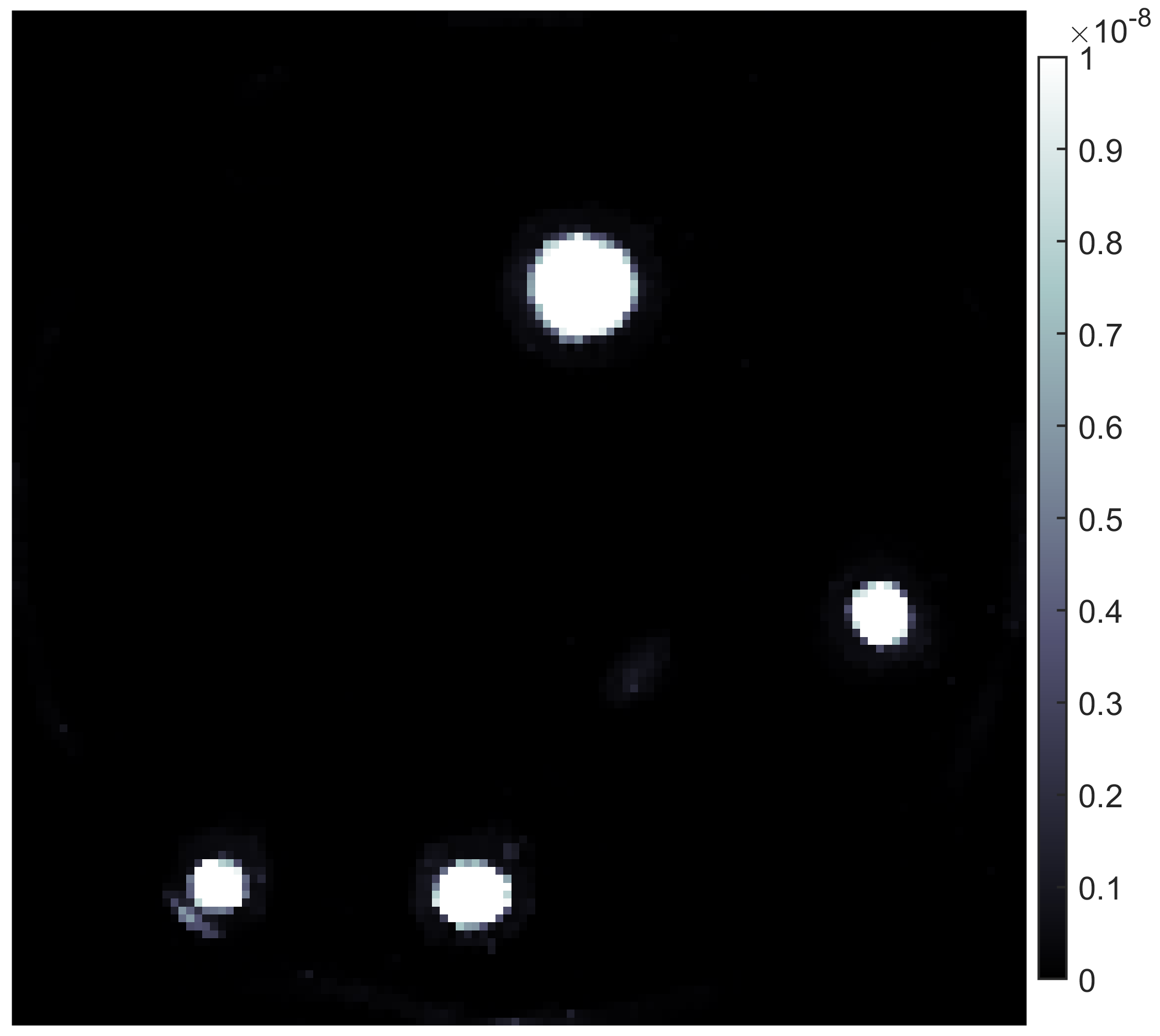} &
        \includegraphics[width=0.23\textwidth]{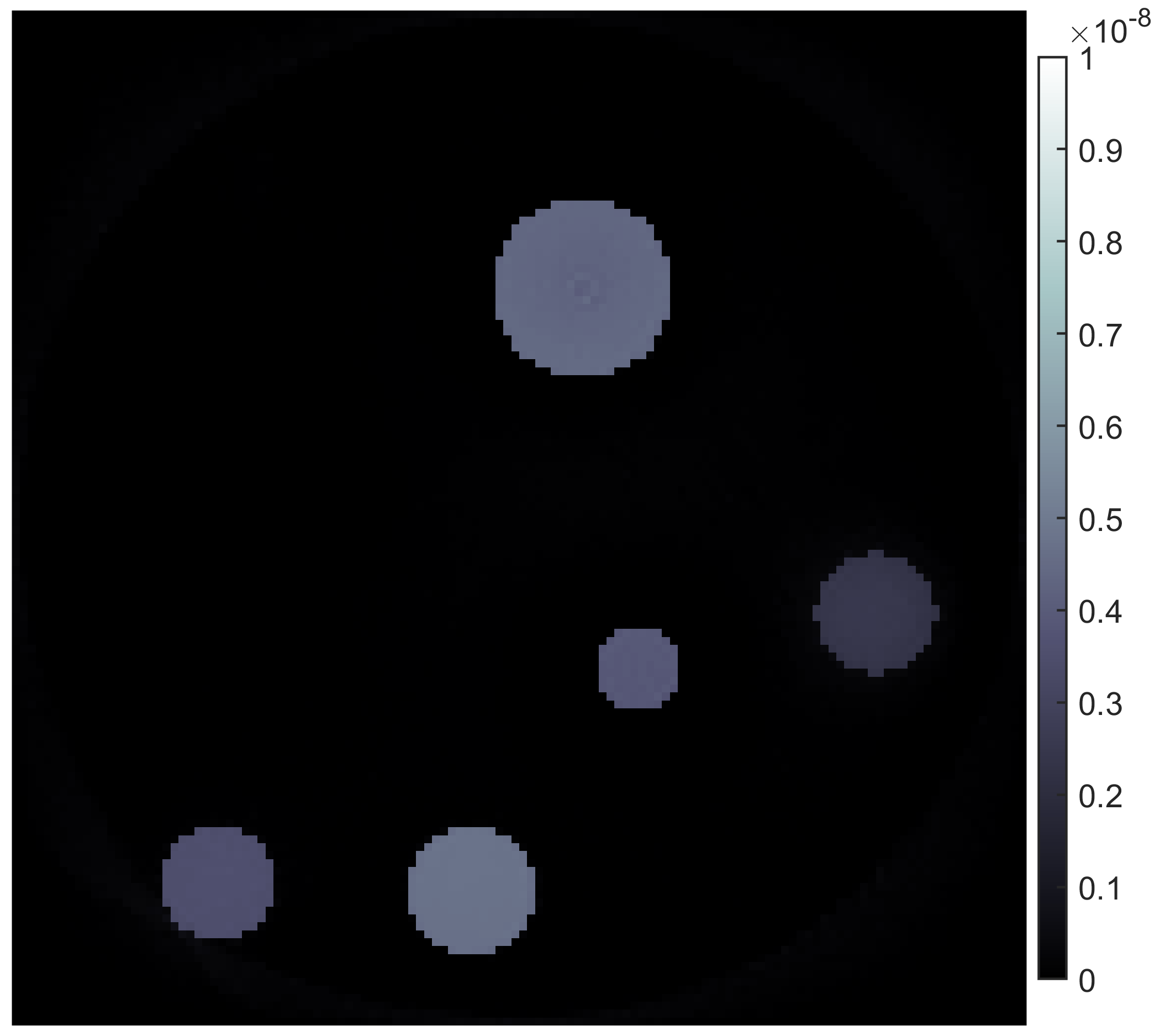} \\

        \rowlabel{diff $\beta$} &
        \includegraphics[width=0.23\textwidth]{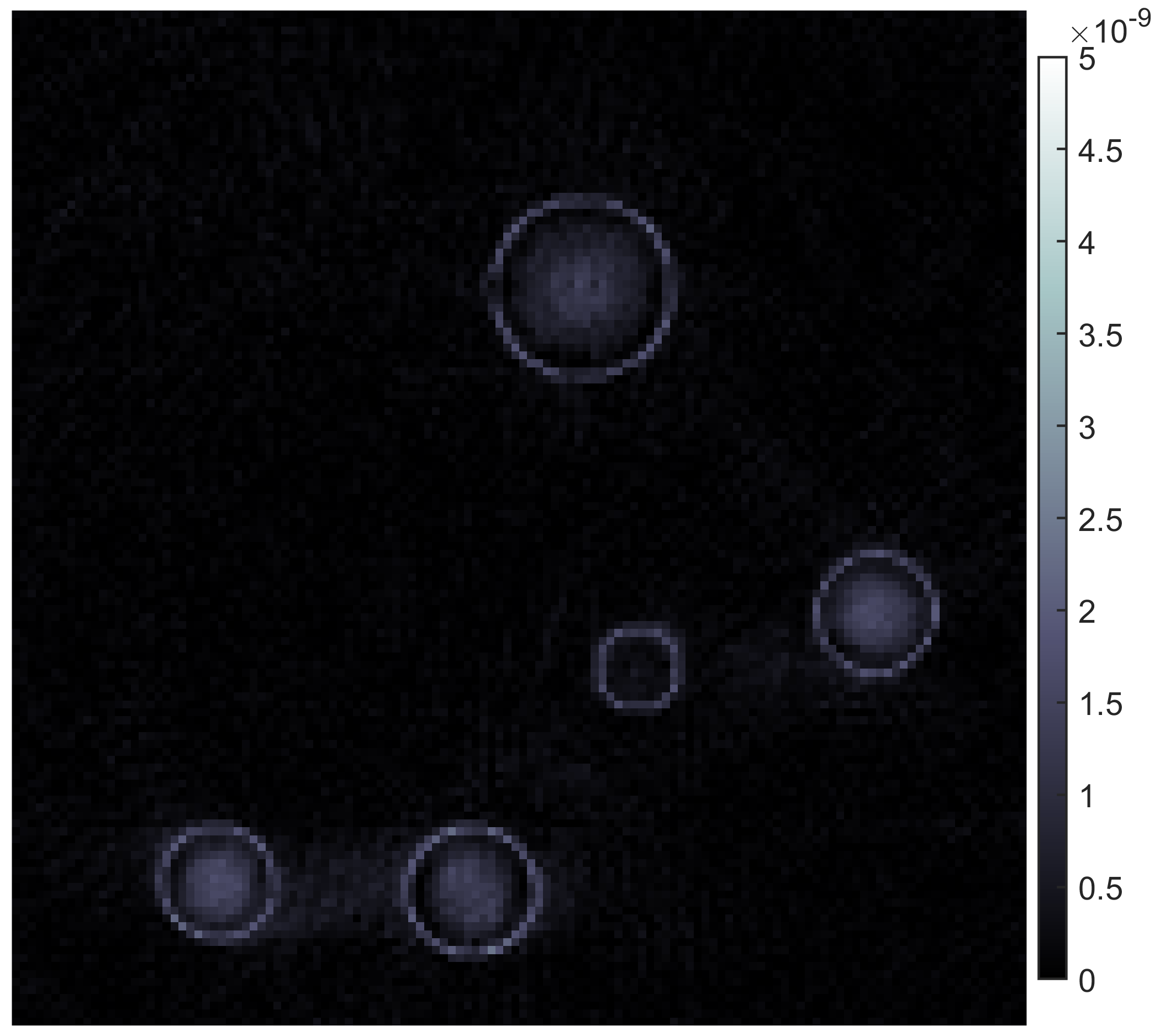} &
        \includegraphics[width=0.23\textwidth]{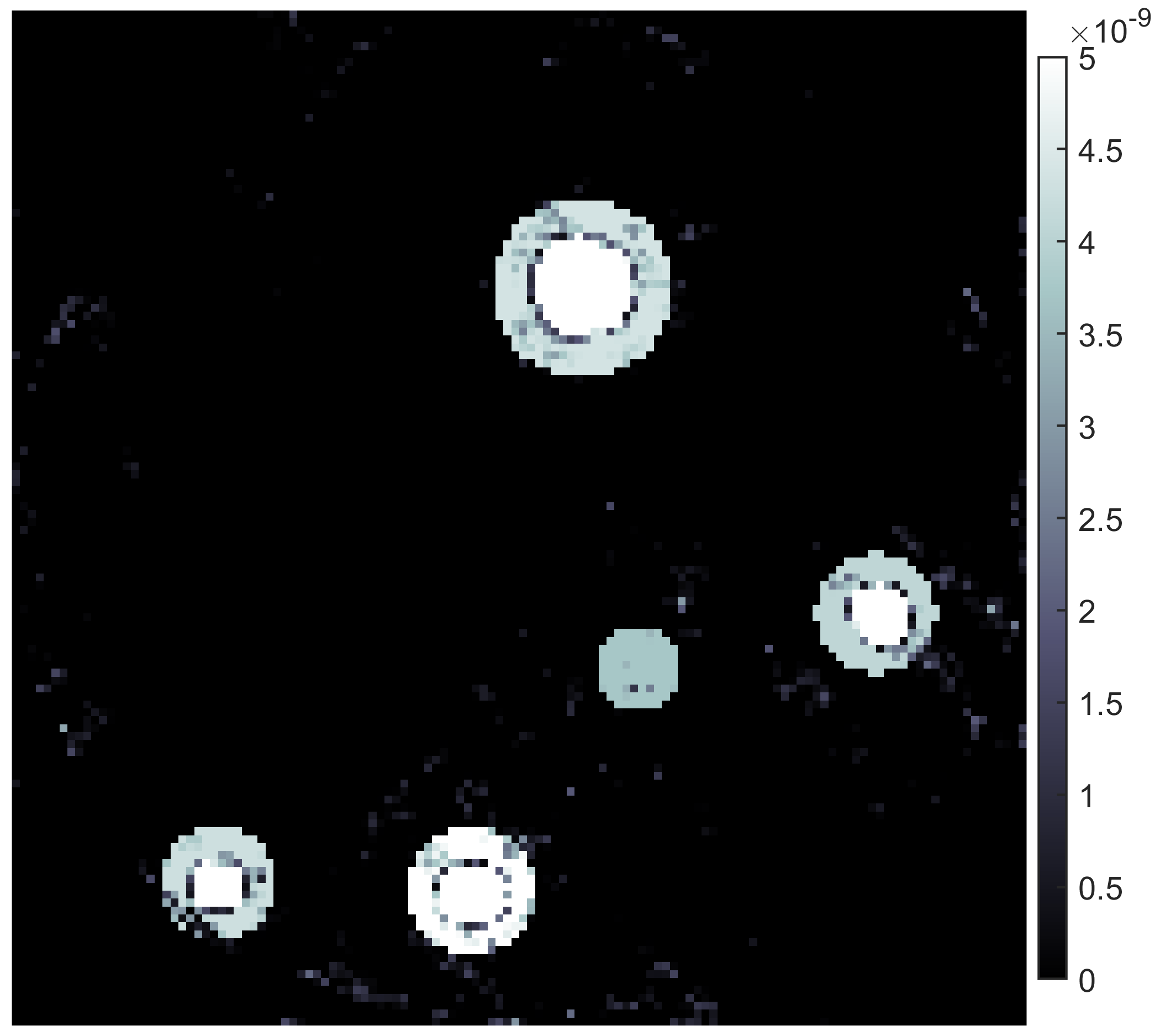} &
        \includegraphics[width=0.23\textwidth]{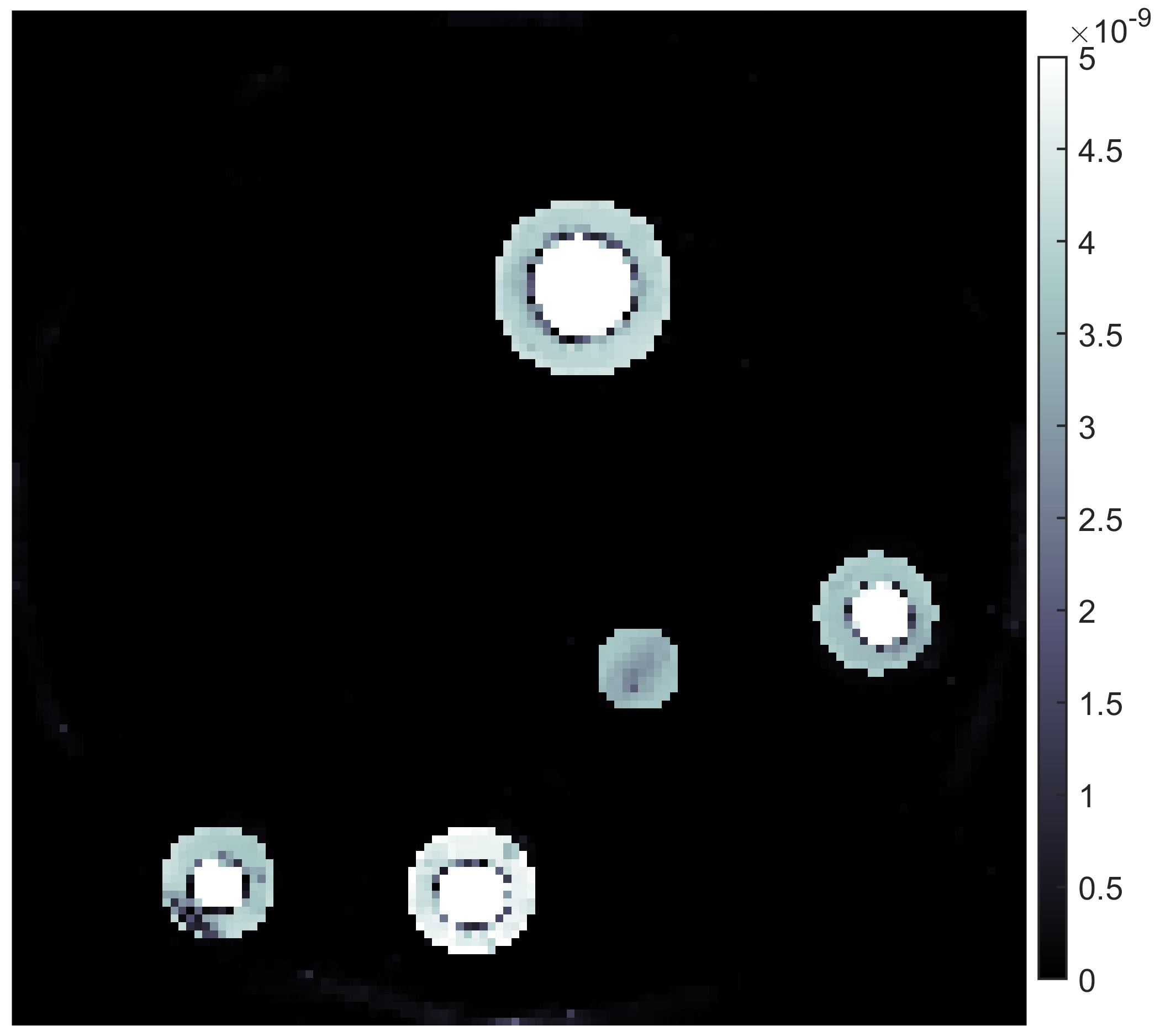} &
        \includegraphics[width=0.23\textwidth]{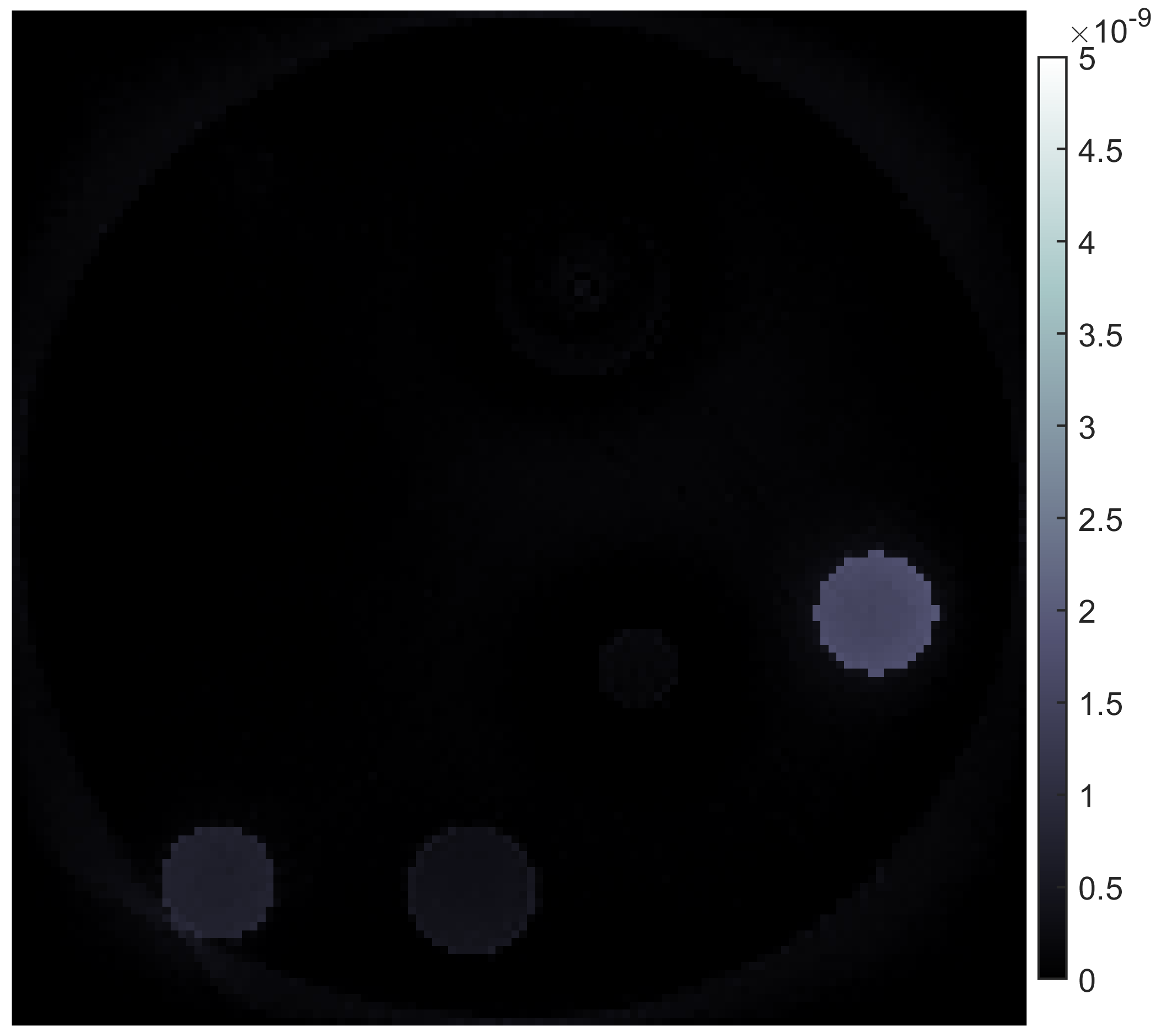} \\
    \end{tabular}
    \caption{3D reconstruction of the relative refractive index (at slice $z=38$). 
    Rows display the reconstructed phase decrement $-\delta$, absorption component $\beta$, and the corresponding reconstruction error maps. 
    Columns correspond to different reconstruction methods.
    Each row uses a common color scale across all reconstruction methods for direct visual comparison.}
\label{fig:direct3d_recon}
\end{figure}
In contrast to the two-step reconstruction pipeline, we now consider direct 3D reconstruction of the refractive index by solving the nonlinear inverse problem \eqref{eq: full nonlinear} with different regularization strategies. Representative reconstruction results are shown in the rest columns of Figure~\ref{fig:direct3d_recon}.

Using GD alone, the reconstruction of $\delta$ already captures the overall structure of the object. However, weak reconstruction artifacts are clearly visible, particularly inside the spherical structures, where the interior appears under-reconstructed with reduced contrast and noticeable difference from the ground truth. This behavior arises from the weak sensitivity of low spatial frequencies in the forward model, leading to slow convergence and insufficient recovery of interior values \cite{dora2024artifact}, as also reflected in the error maps. In contrast, the reconstruction of $\beta$ is significantly more degraded with limited structural fidelity, indicating the challenges of the absorption recovery in the absence of regularization.
\begin{figure}[t]
    \centering  
    \begin{minipage}[t]{0.25\linewidth}
    \centering
    \includegraphics[width=1\textwidth]{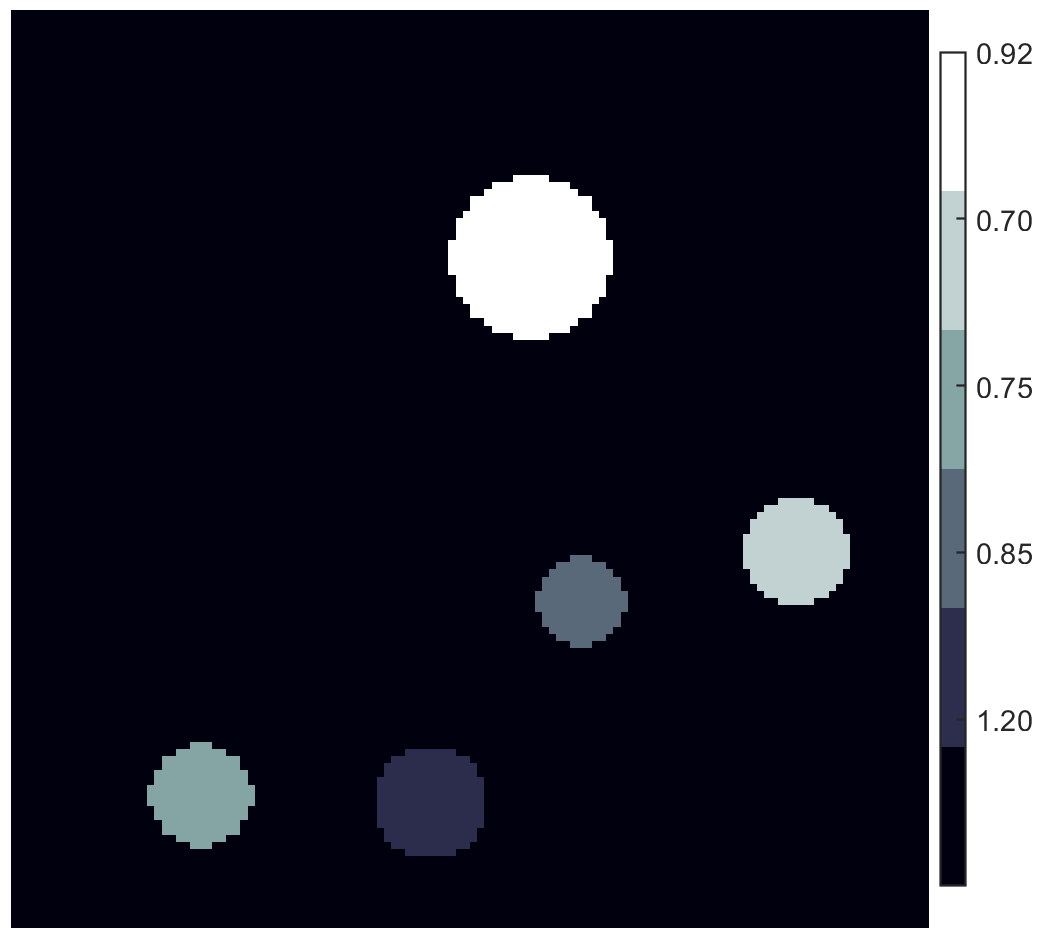}\\    \centerline{\footnotesize\text{$s_j$(Slice)}}
    \end{minipage}   
    \begin{minipage}[t]{0.25\linewidth}
    \centering
    \includegraphics[width=1\textwidth]{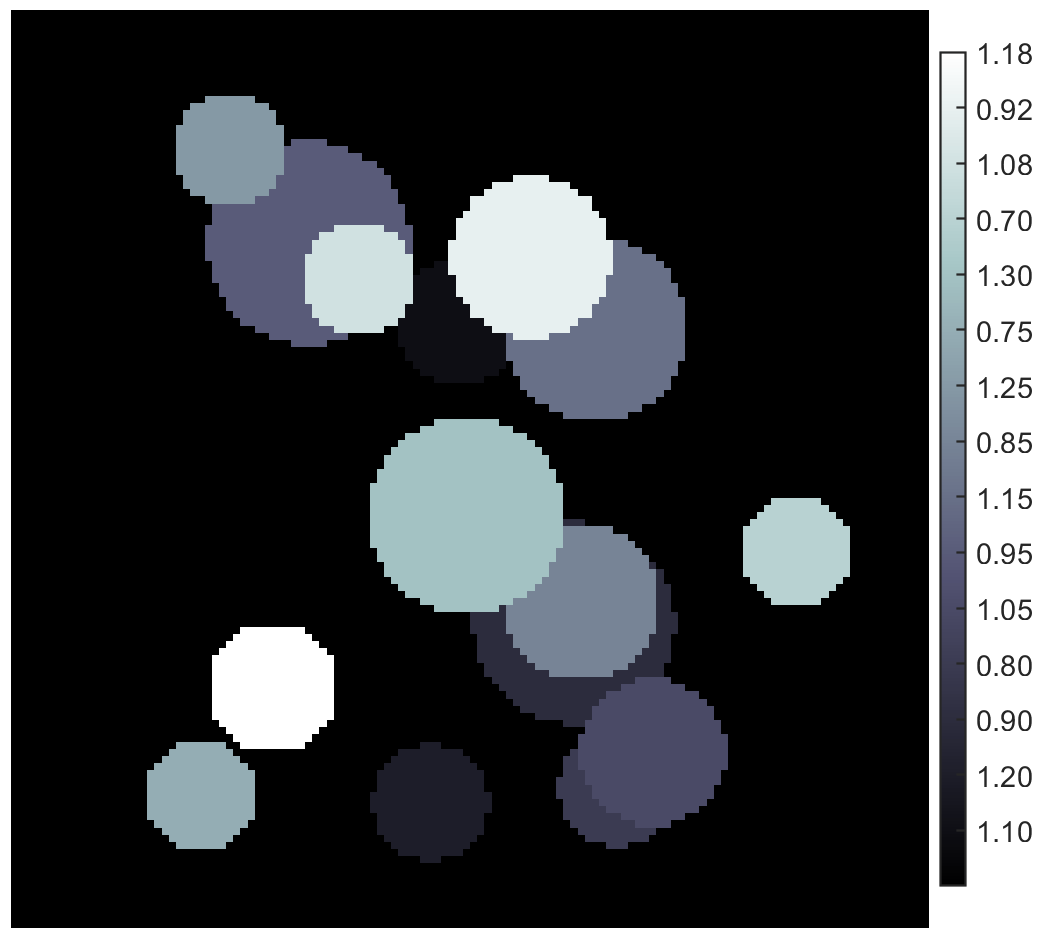}\\ \centerline{\footnotesize\text{$s_j$(Projection)}}
    \end{minipage}    
    \caption{Material-assignment maps for the phantom with material-dependent
    $\delta/\beta$ ratios. Left: slice-wise material assignment at $z=38$.
    Right: corresponding projection-view material assignment at $\theta=0$.}
    \label{fig:material_maps}
\end{figure}

\begin{figure}[H]
    \centering
    \setlength{\tabcolsep}{1pt}

    \begin{tabular}{c c c c}
        & \footnotesize Ground Truth 
        & \footnotesize Phase-guided 
        & \footnotesize Diff \\

        \raisebox{16mm}{ \footnotesize $-\delta$} &
         \includegraphics[width=0.25\textwidth]{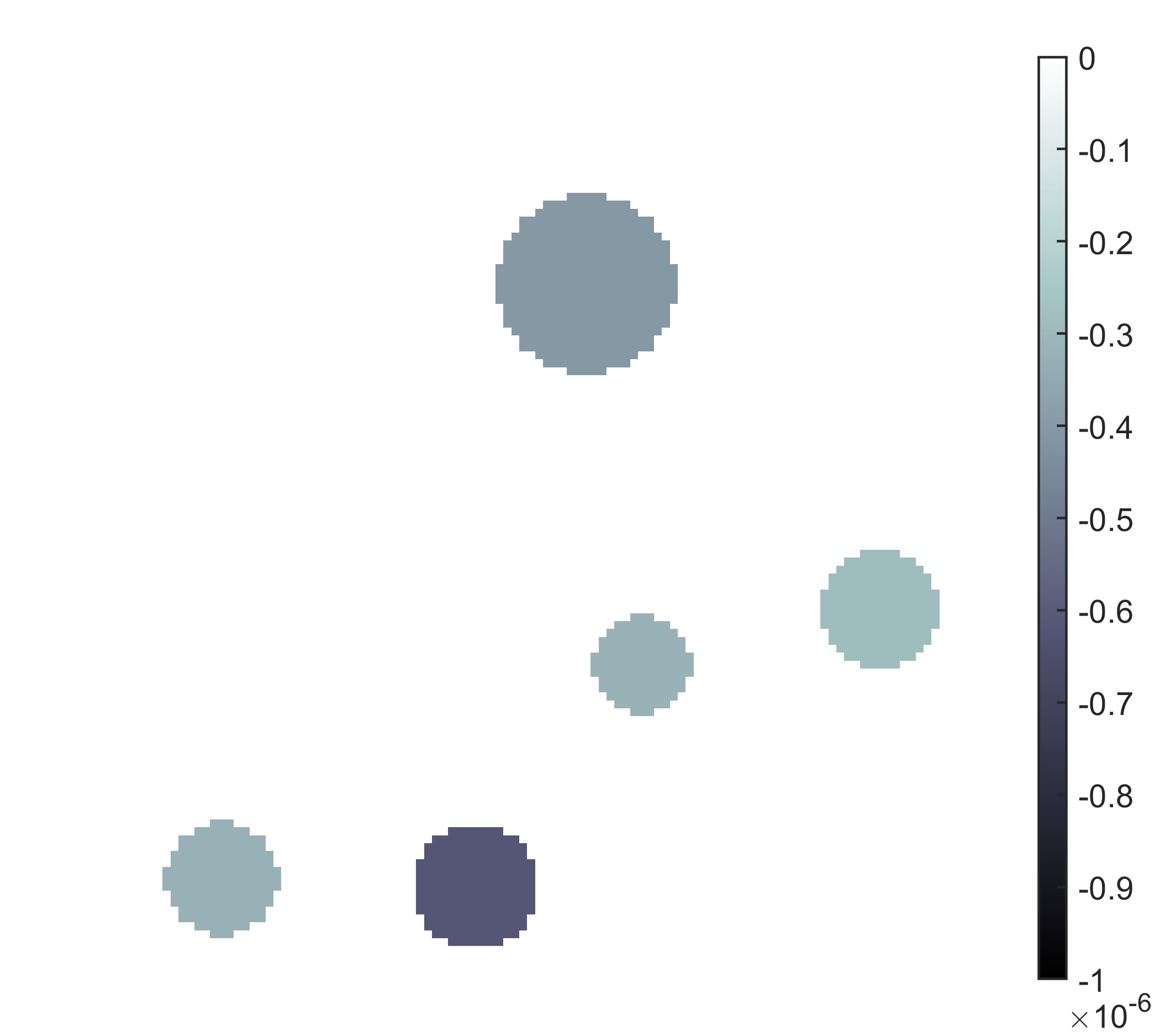} &
        \includegraphics[width=0.25\textwidth]{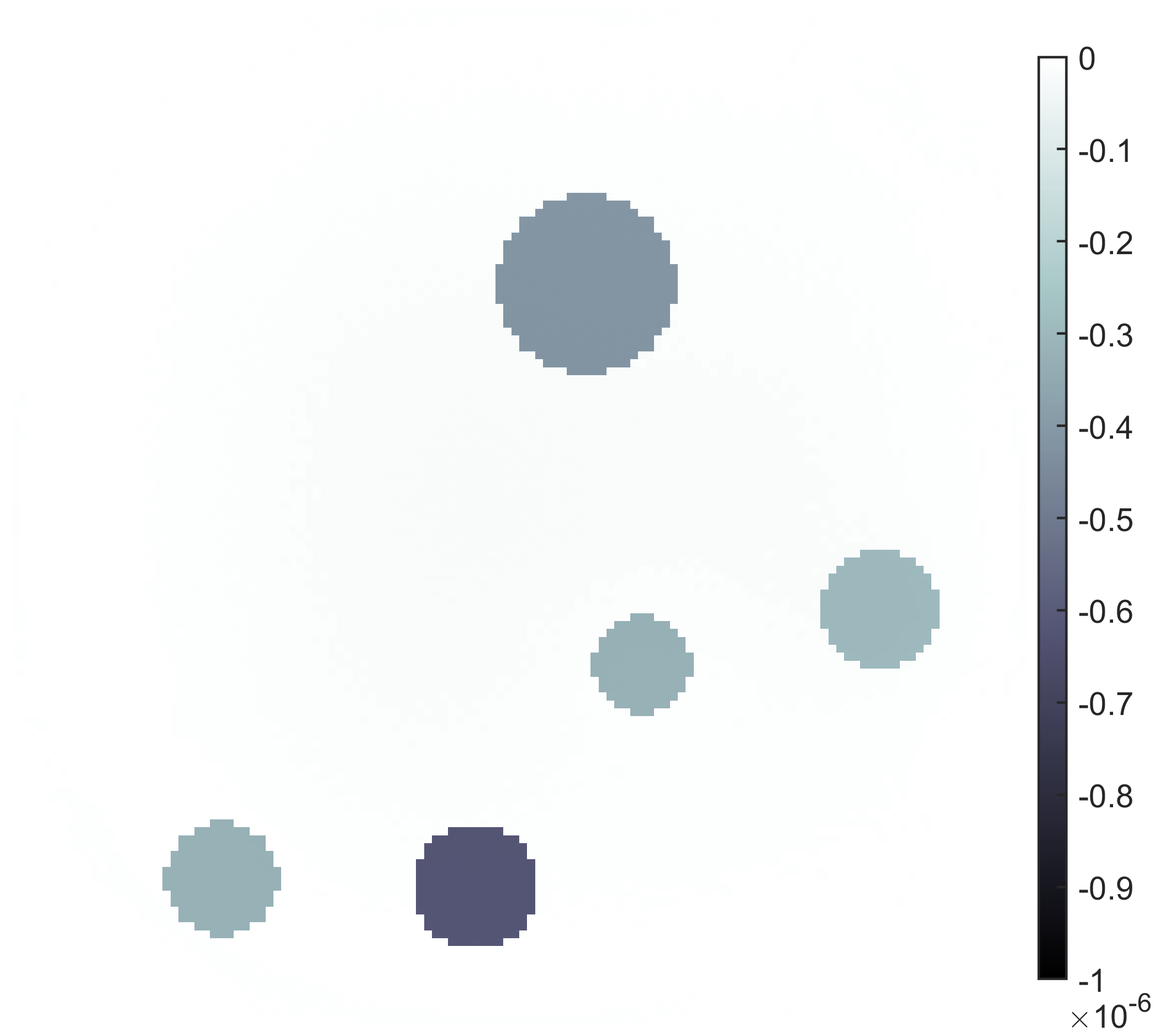} &
        \includegraphics[width=0.25\textwidth]{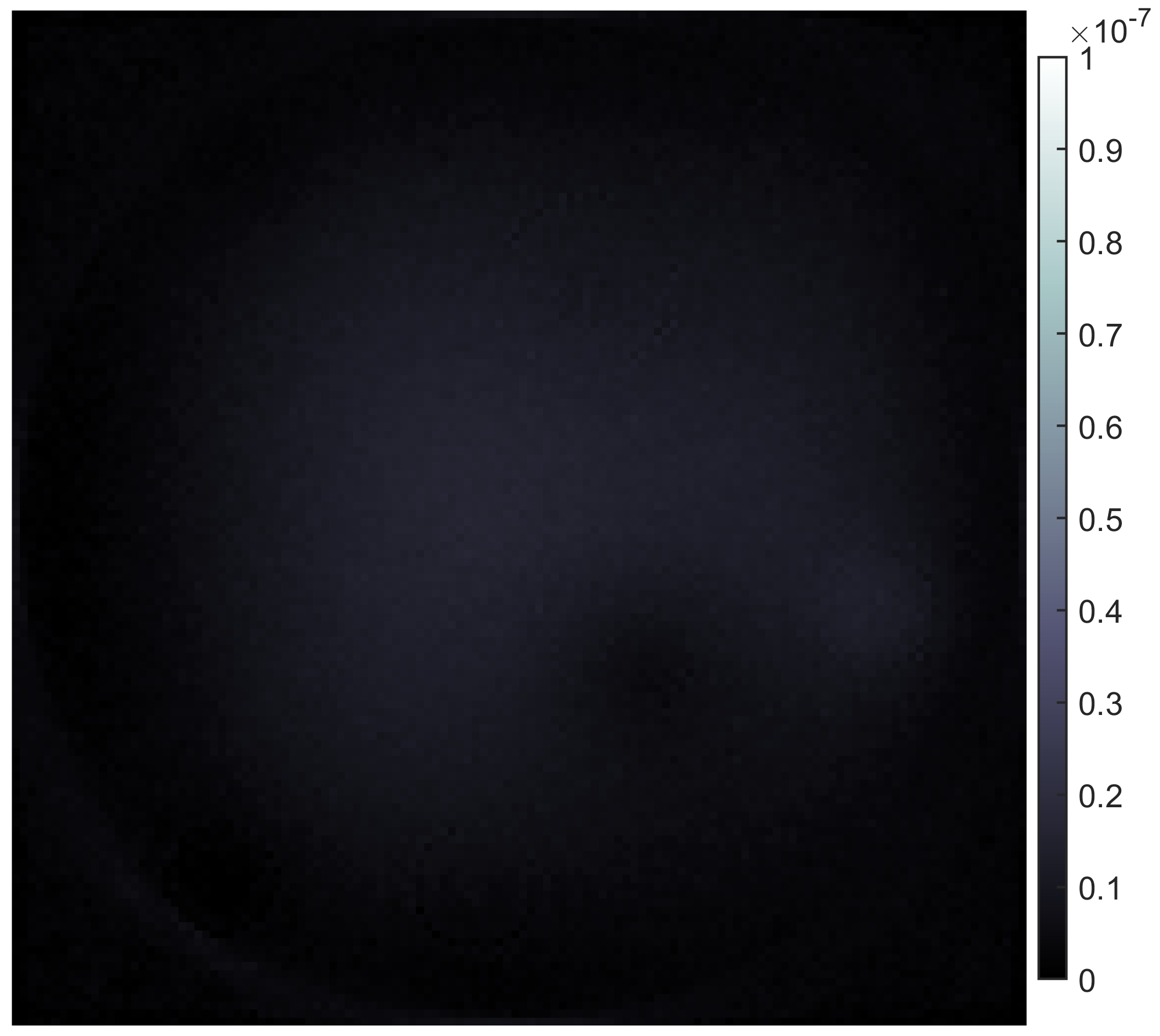} \\

        \raisebox{16mm}{ \footnotesize $\beta$} &
        \includegraphics[width=0.25\textwidth]{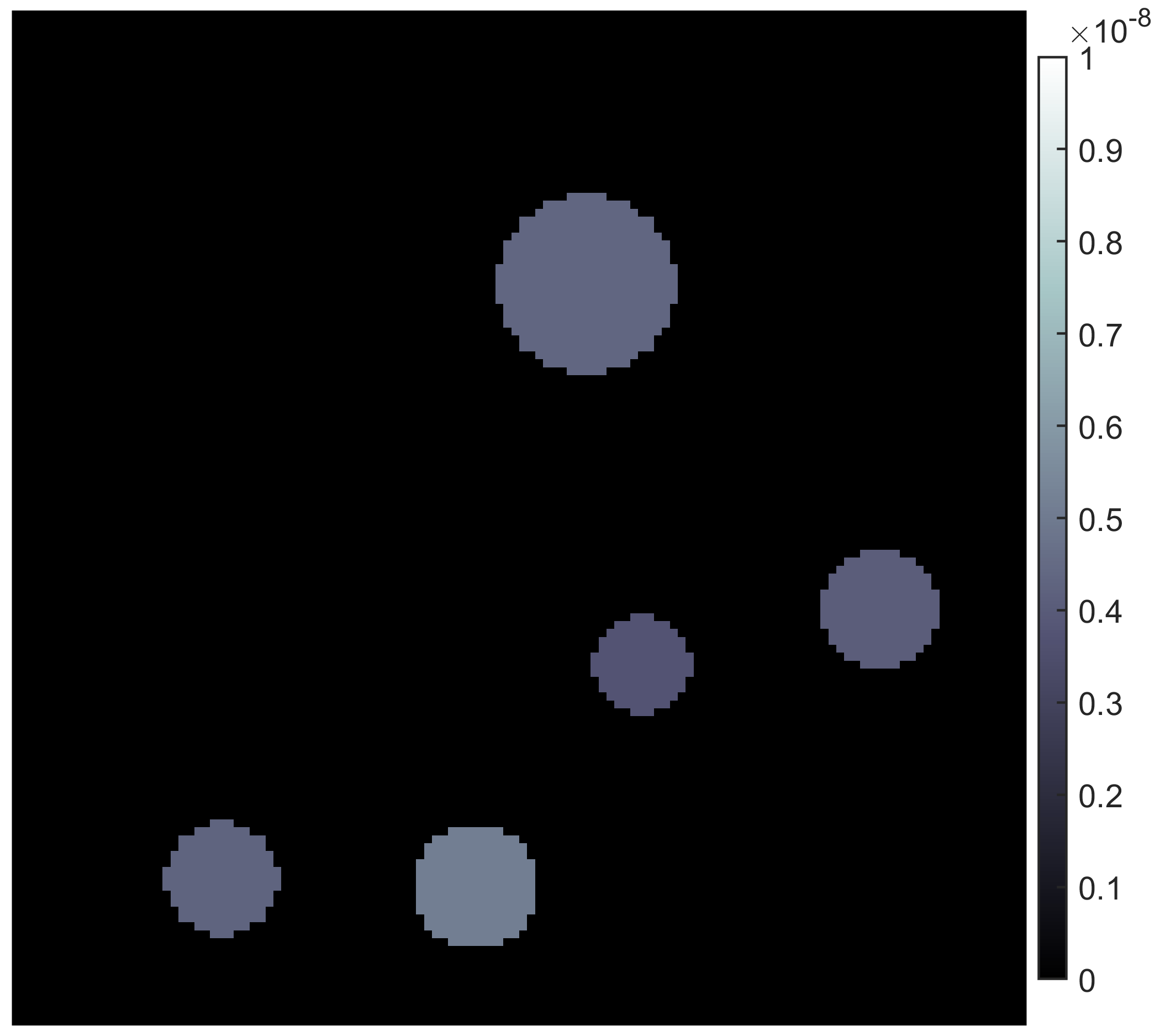} &
        \includegraphics[width=0.25\textwidth]{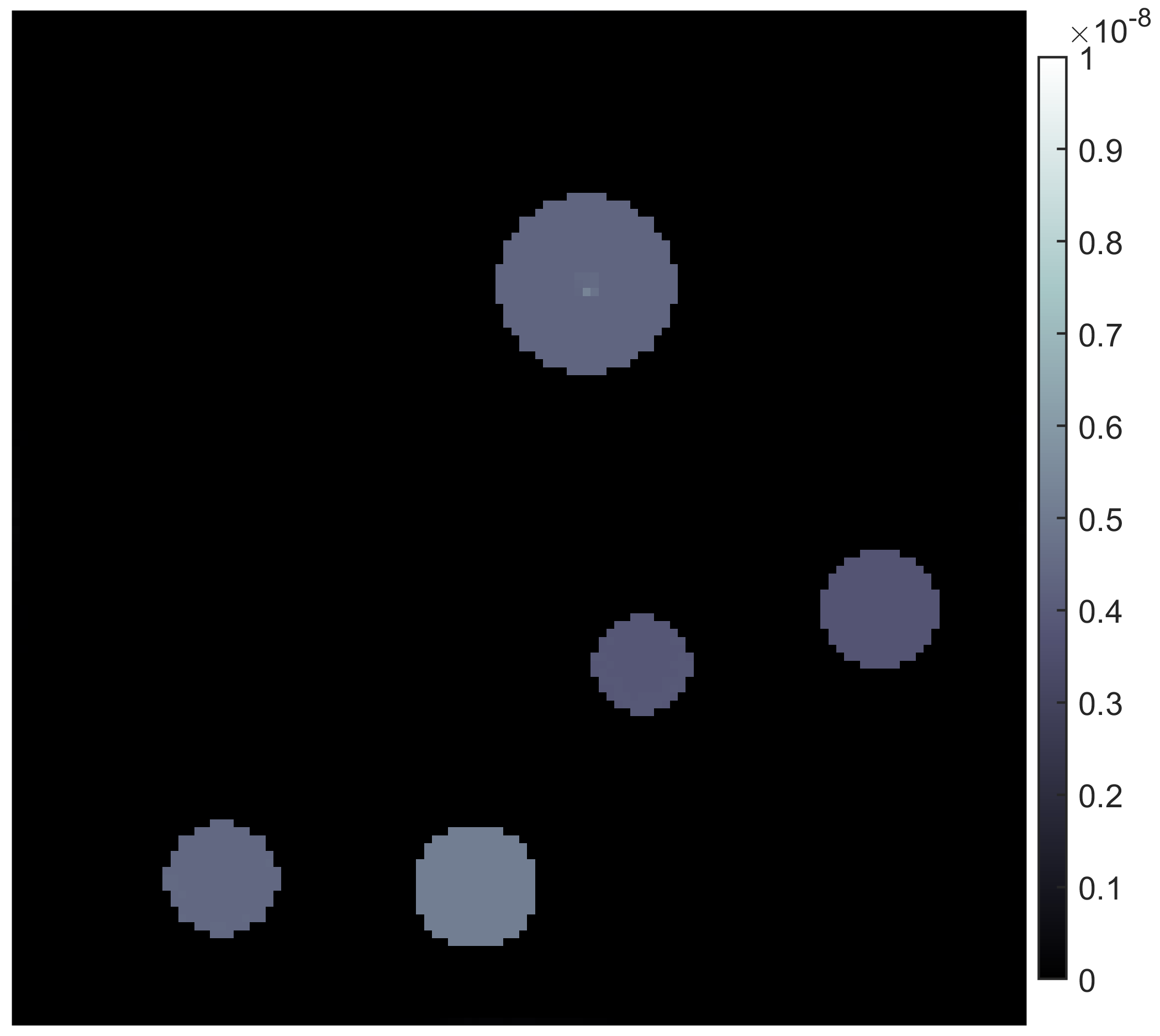} &
        \includegraphics[width=0.25\textwidth]{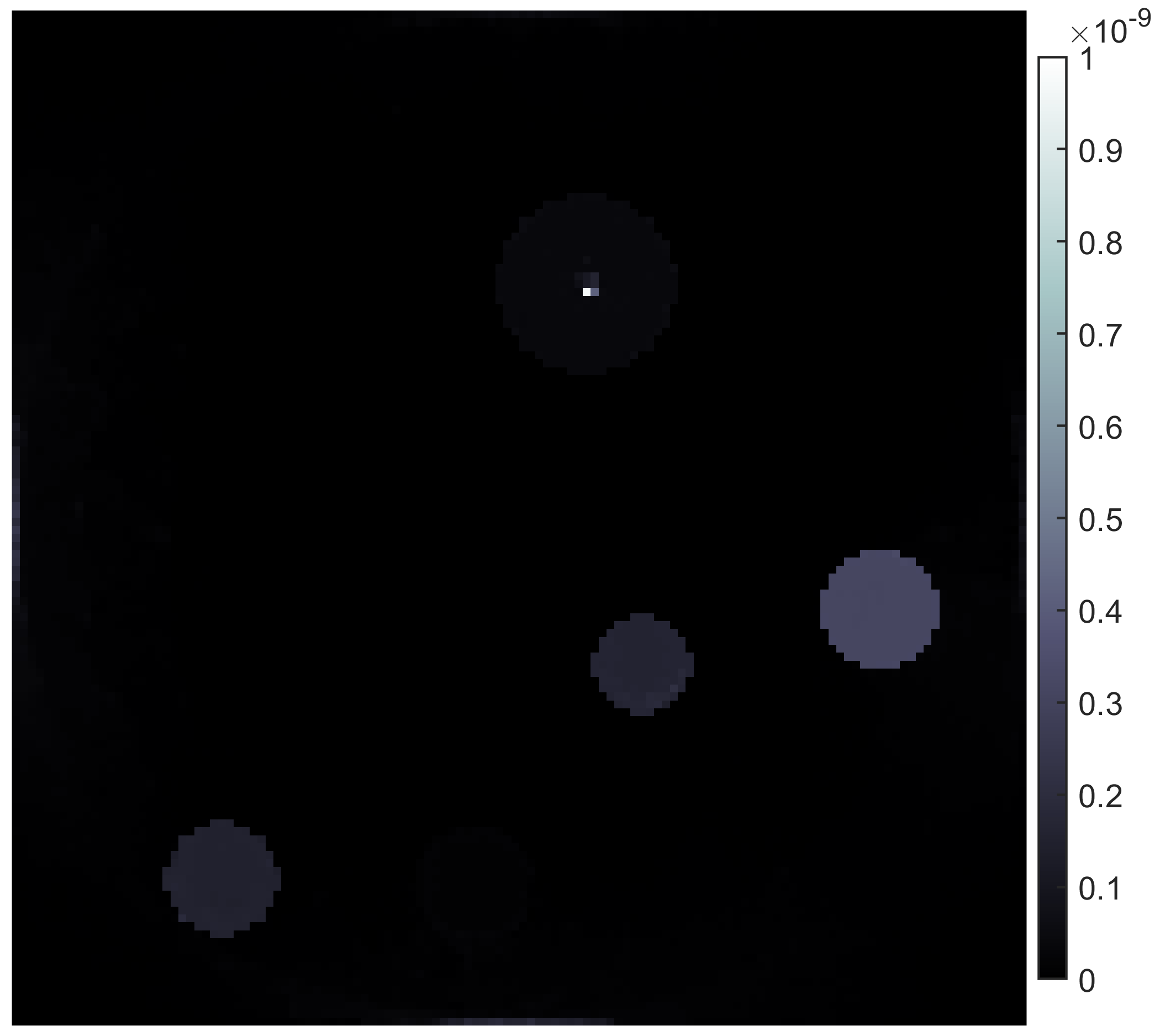} \\
    \end{tabular}

    \caption{Reconstruction results for the phantom with material-dependent
    $\delta/\beta$ ratios at slice $z=38$. Rows correspond to the phase decrement and absorption component, respectively.
    Columns show the ground truth, reconstruction, and corresponding
    absolute error.}
   \label{fig:materials_reconstruction}
\end{figure}
By incorporating TV regularization with Bregman iteration, the reconstruction becomes more stable. Noise and artifacts are effectively suppressed in both $\delta$ and $\beta$, leading to clearer overall structures. However, improvement in $\beta$ remains limited, as underlying structural deficiencies persist.

To evaluate the sensitivity to material-dependent phase--absorption ratios, each sphere was assigned a distinct $\delta/\beta$ ratio while preserving the same spatial morphology. The baseline ratio was set to $\delta_0/\beta_0 = 100$, and the ratio of the $j$-th sphere was defined as
$
\frac{\delta_j}{\beta_j}
=
s_j \frac{\delta_0}{\beta_0},
$
where $s_j$ varied between $0.70$ and $1.30$. 
Figure~\ref{fig:material_maps} illustrates the resulting material assignments for a representative slice and the corresponding projection view, showing the spatial variation of $s_j$ among the different spheres. 
The corresponding reconstruction results are presented in Figure~\ref{fig:materials_reconstruction}. The results show that the method performs reliably when the $\delta/\beta$ ratio varies spatially across the object.

Further improvements are obtained with phase-guided regularization ($\gamma = 0.1$). Both $\delta$ and $\beta$ exhibit more accurate structural recovery, with improved contrast consistency and visibly reduced reconstruction errors. In particular, the interior regions of the reconstructed objects appear substantially more homogeneous, while many of the residual artifacts observed in the GD and Bregman TV reconstructions are suppressed. This improvement can be explained by the coupling introduced between $\delta$ and $\beta$ through phase-guided regularization. By leveraging structural information from the phase component, the reconstruction of $\beta$ is better constrained. At the same time, this coupling improves the conditioning of the overall inverse problem, particularly for low-frequency components that are weakly constrained by the data term alone. As a result, both $\delta$ and $\beta$ benefit from improved low-frequency recovery, leading to more accurate and physically consistent reconstructions.

The influence of the guiding parameter $\gamma$ on the reconstruction is further investigated in Figure~\ref{fig:gamma_analysis}.
When $\gamma=1$, the phase-guided regularization is absent, and the reconstruction reduces to the unguided case. Consequently, the reconstruction quality deteriorates, with pronounced errors and weak recovery of homogeneous regions. 
As $\gamma$ decreases (e.g., $\gamma=0.7$ and $\gamma=0.4$), the strength of the guidance increases, leading to a gradual improvement in reconstruction quality. This is reflected by decreasing errors, particularly in the interior regions and low-frequency components. 
For smaller values of $\gamma$ (e.g., $\gamma=0.1$), where the phase-guided regularization is strong, the reconstruction errors of $\delta$ are further reduced. However, for $\beta$, overly strong guidance introduces slight over-regularization effects. These results suggest that an intermediate value of $\gamma$ provides a good balance between effective guidance and avoiding over-regularization.

We also investigate its robustness with respect to different $\delta$/$\beta$ ratios, as shown in Figure~\ref{fig:ratio_analysis}. 
As the relative magnitude of $\beta$ decreases, the absorption signal becomes weaker, making its reconstruction increasingly challenging. This leads to more pronounced artifacts and reduced contrast in $\beta$, while $\delta$ remains comparatively stable due to its stronger influence on the measured intensity.

Despite these variations, the proposed phase-guided reconstruction remains robust across different $\delta$/$\beta$ ratios. The structural features are consistently recovered, and the reconstruction errors remain well controlled, demonstrating that the method effectively compensates for the weak sensitivity of $\beta$ by leveraging the structural information from $\delta$.

\begin{figure}[H]
    \centering
    \setlength{\tabcolsep}{1pt}

    \begin{tabular}{c c c c c}
        & \footnotesize $\gamma = 1$ 
        & \footnotesize $\gamma = 0.7$ 
        & \footnotesize $\gamma = 0.4$ 
        & \footnotesize $\gamma = 0.1$  \\

        \rowlabel{$-\delta_{reco}$} &
        \includegraphics[width=0.23\textwidth]{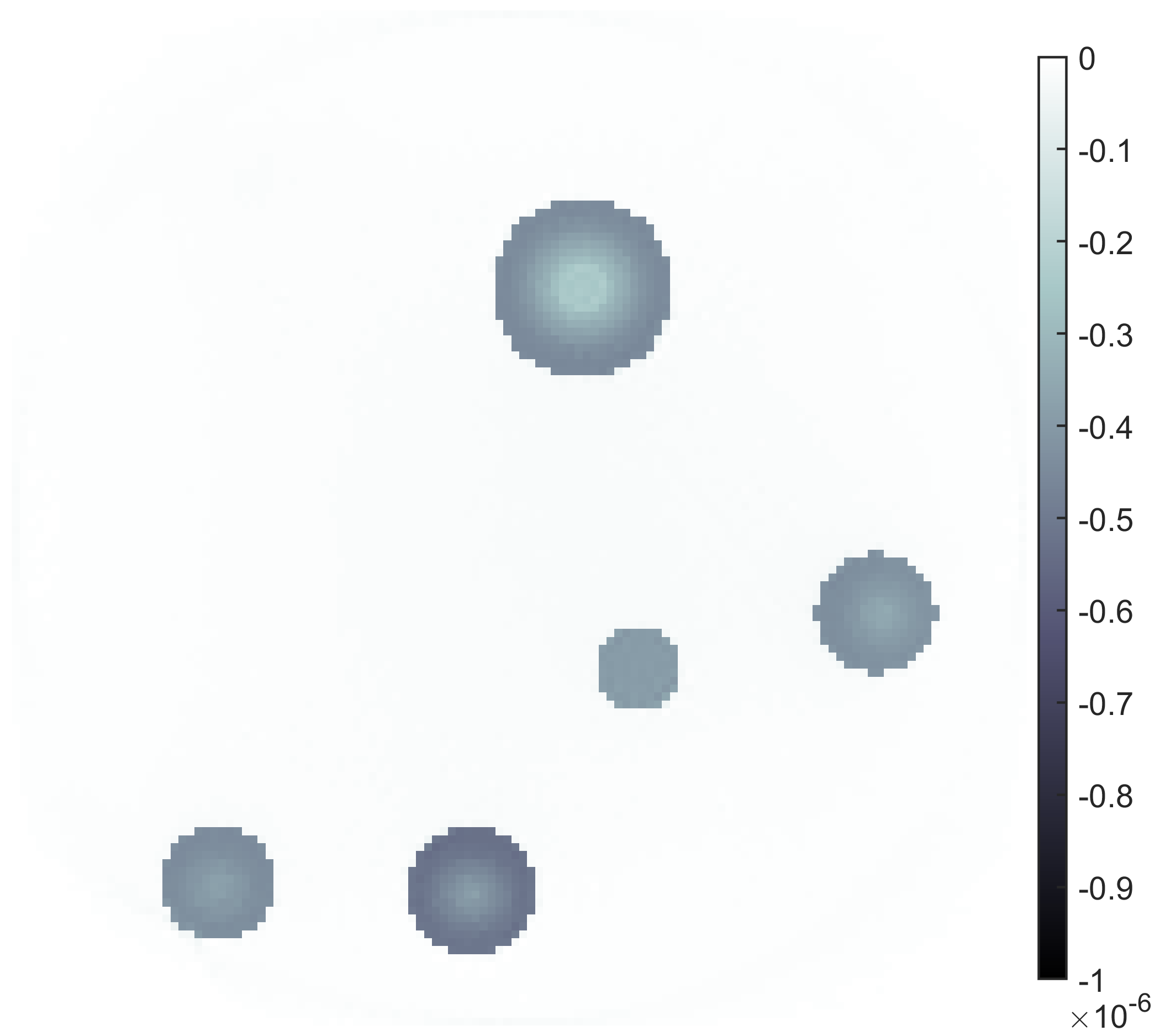} &
        \includegraphics[width=0.23\textwidth]{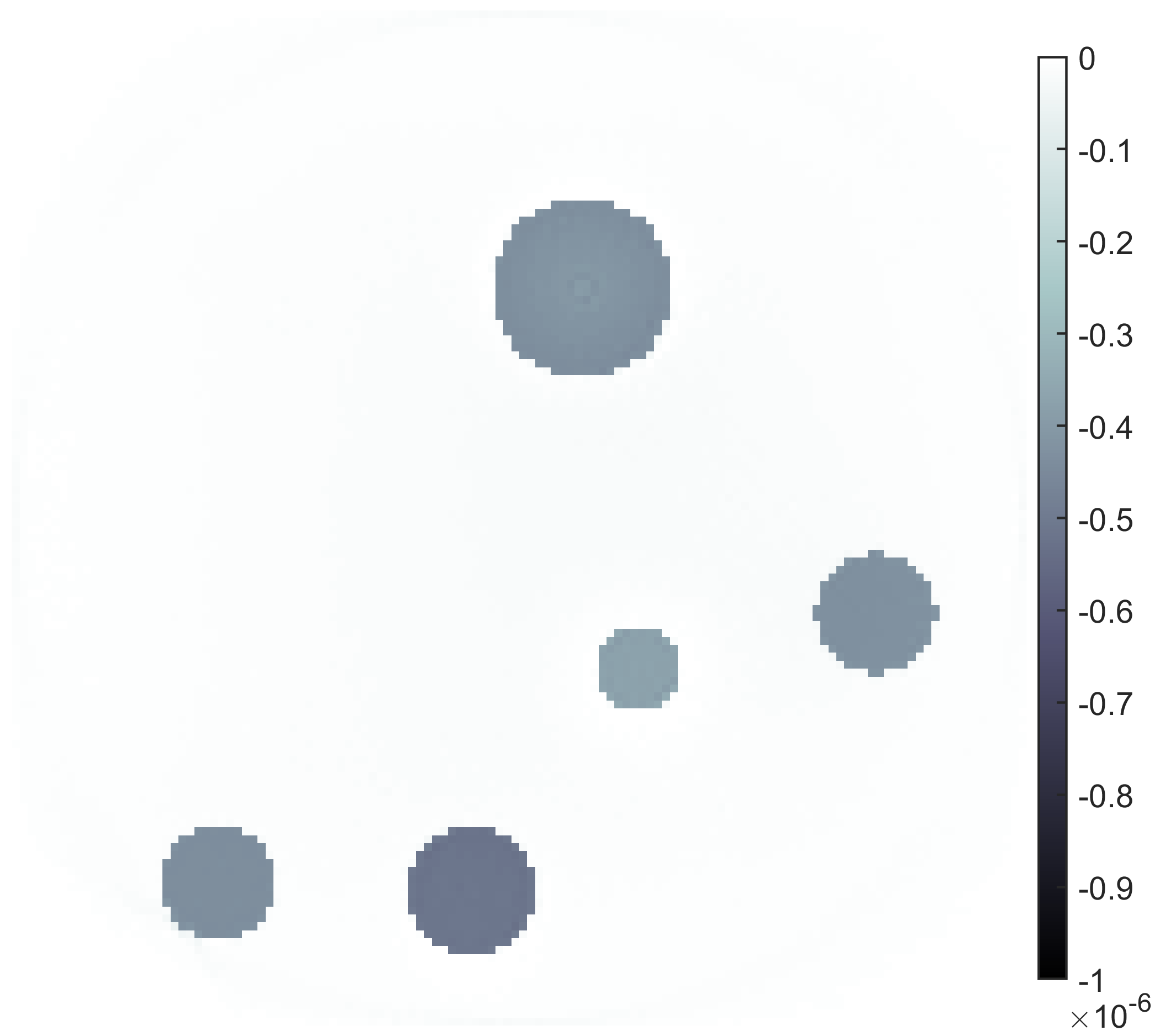} &
        \includegraphics[width=0.23\textwidth]{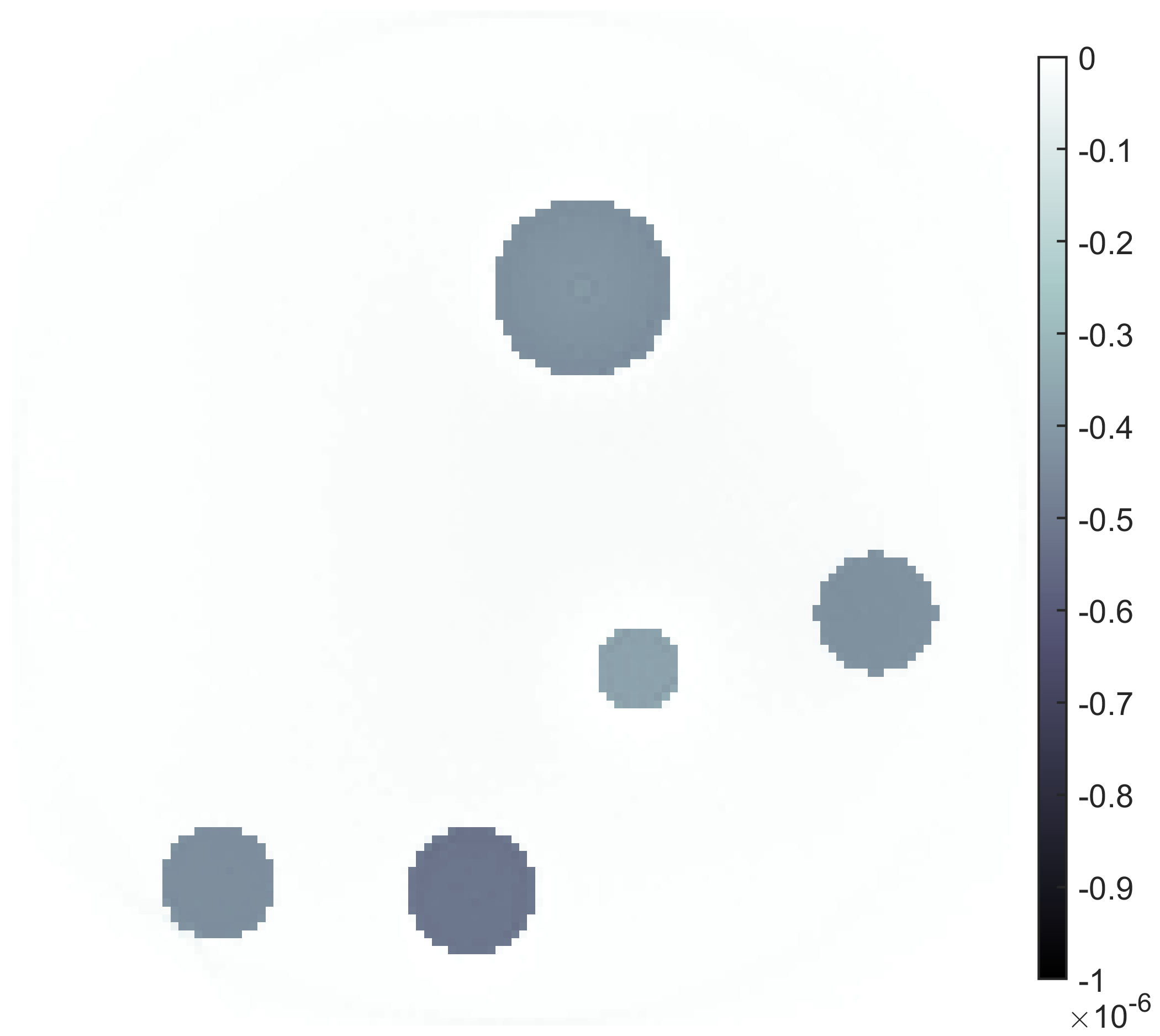} &
        \includegraphics[width=0.23\textwidth]{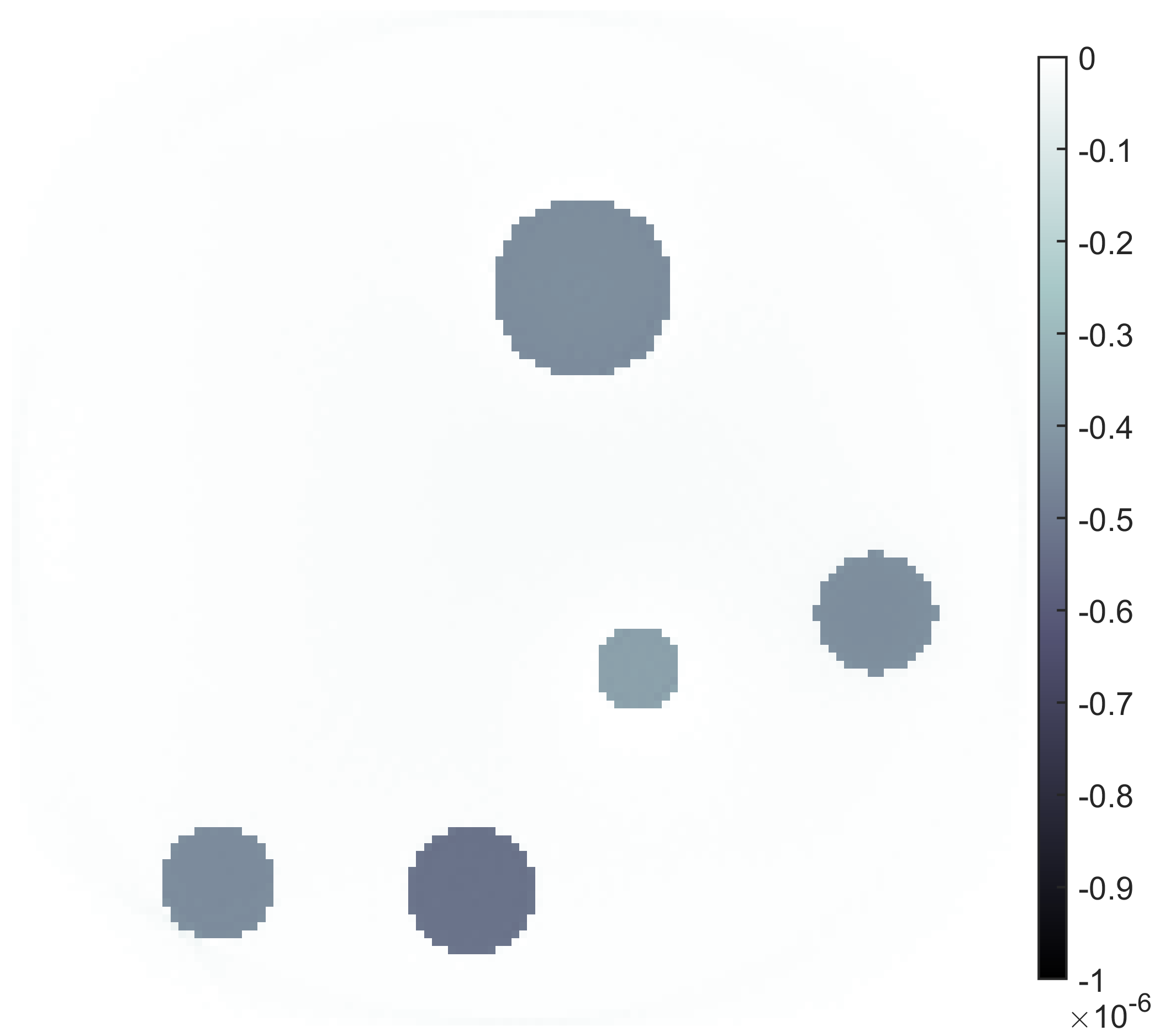} \\
        
        \rowlabel{diff $\delta$} &
        \includegraphics[width=0.23\textwidth]{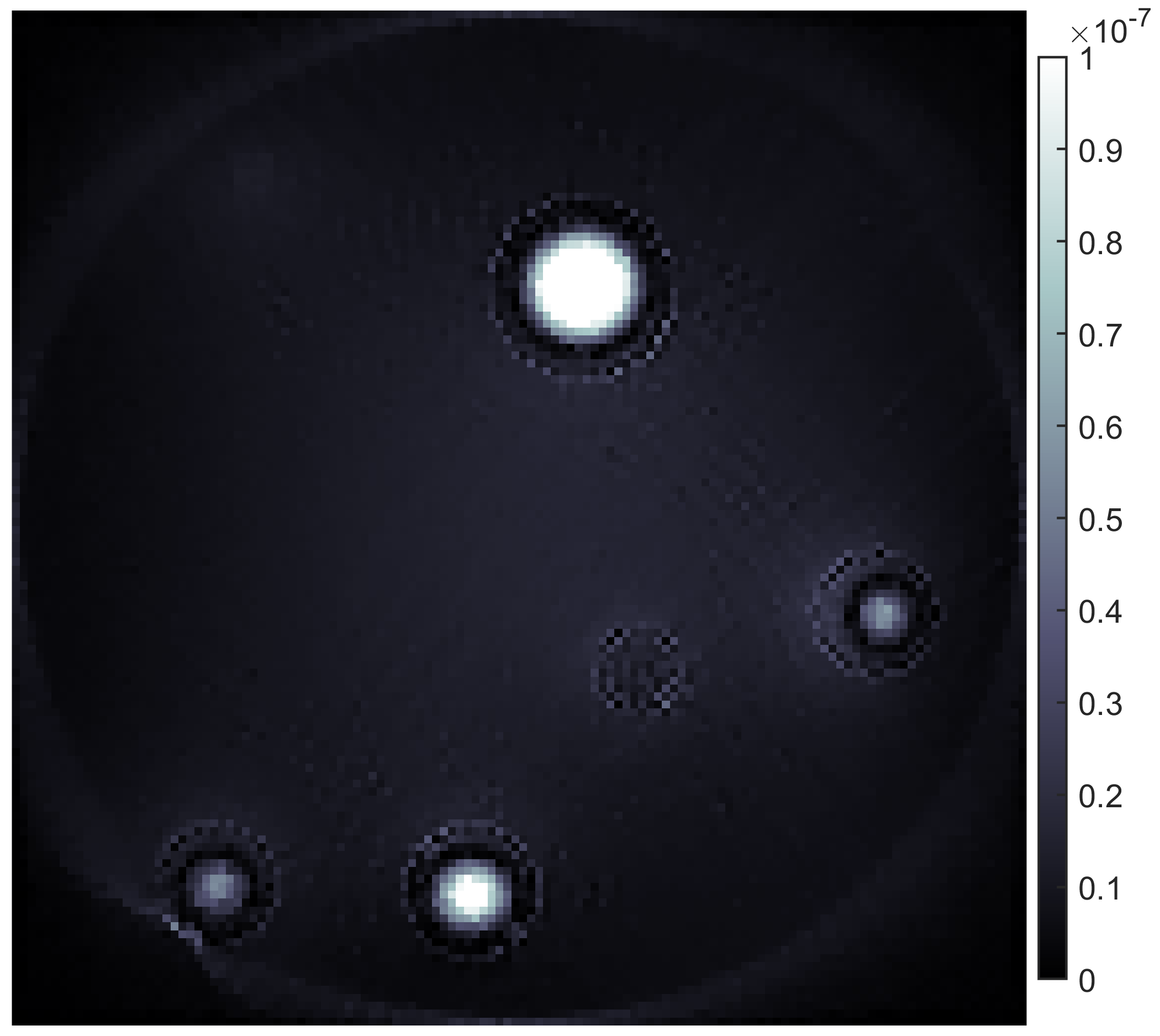} &
        \includegraphics[width=0.23\textwidth]{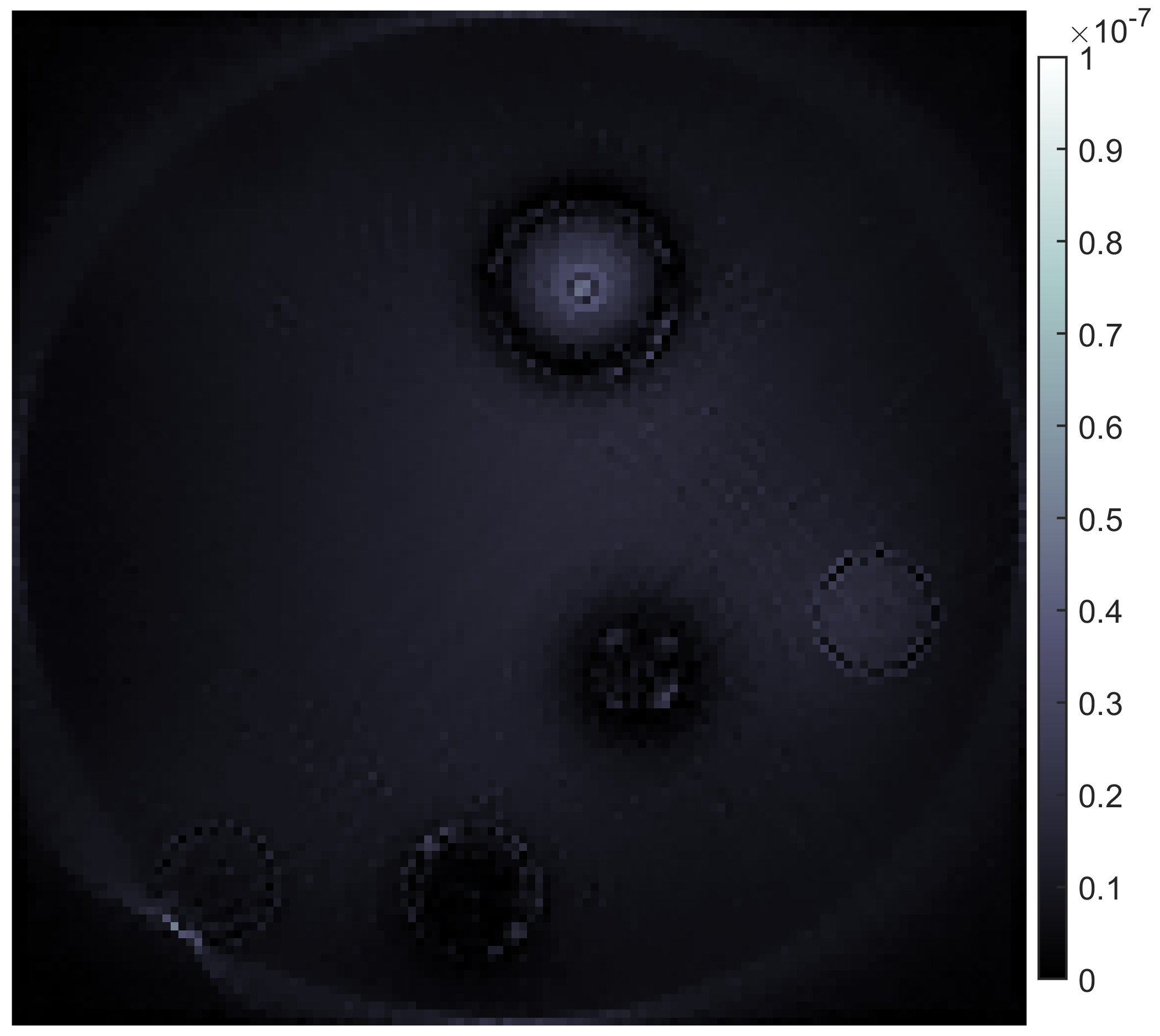} &
        \includegraphics[width=0.23\textwidth]{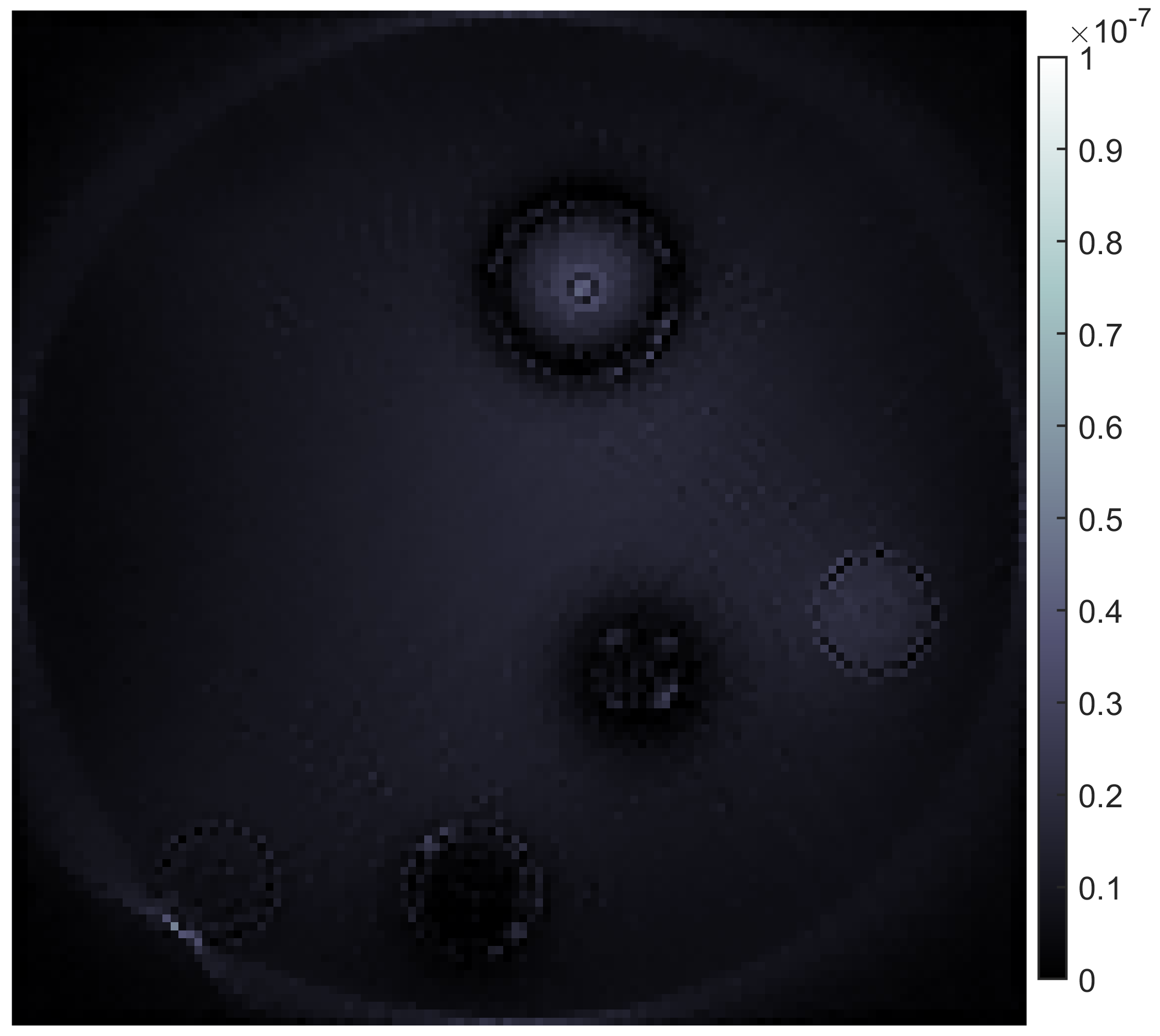} &
        \includegraphics[width=0.23\textwidth]{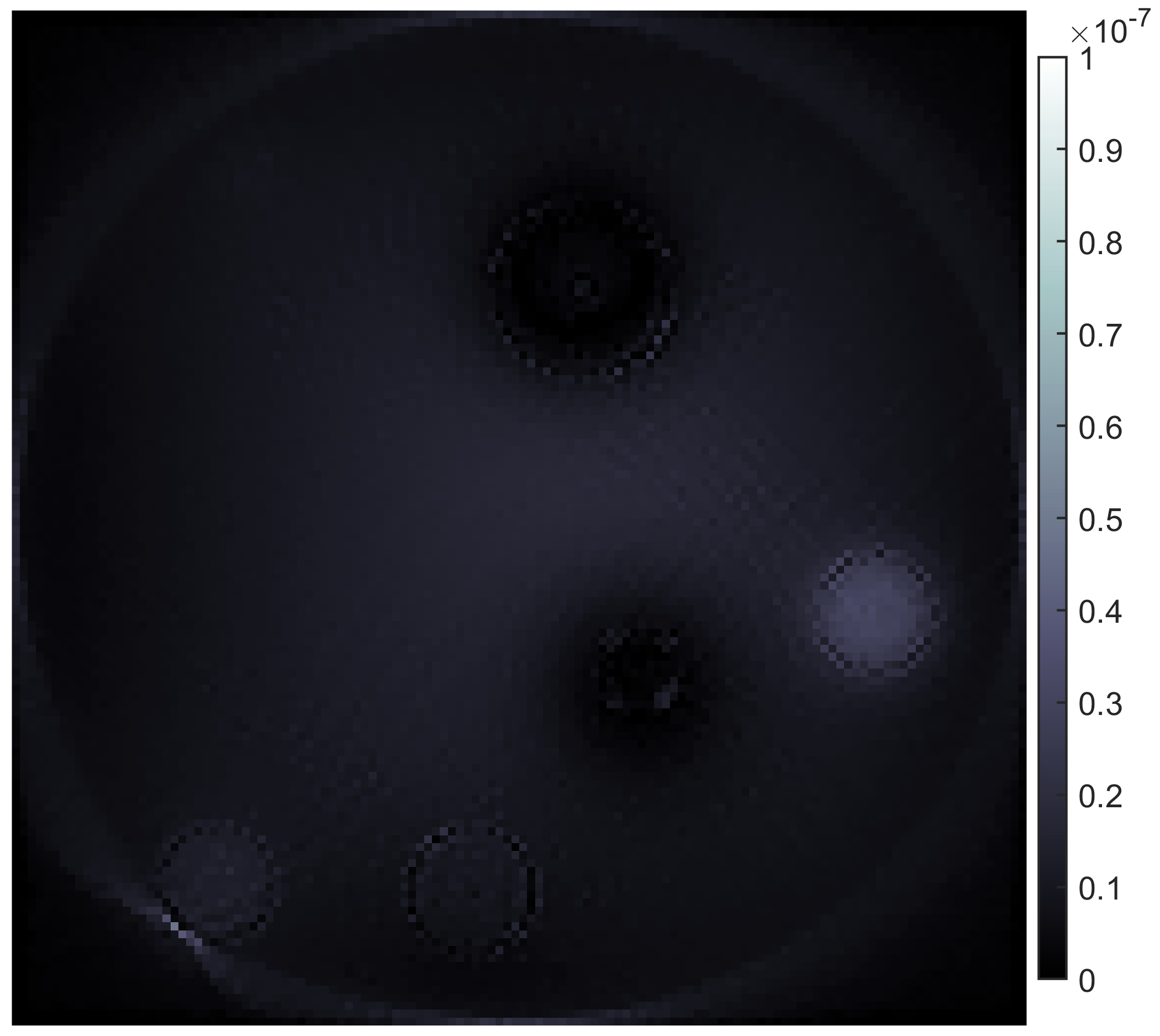} \\
        
       \rowlabel{$\beta_{reco}$} &
        \includegraphics[width=0.23\textwidth]{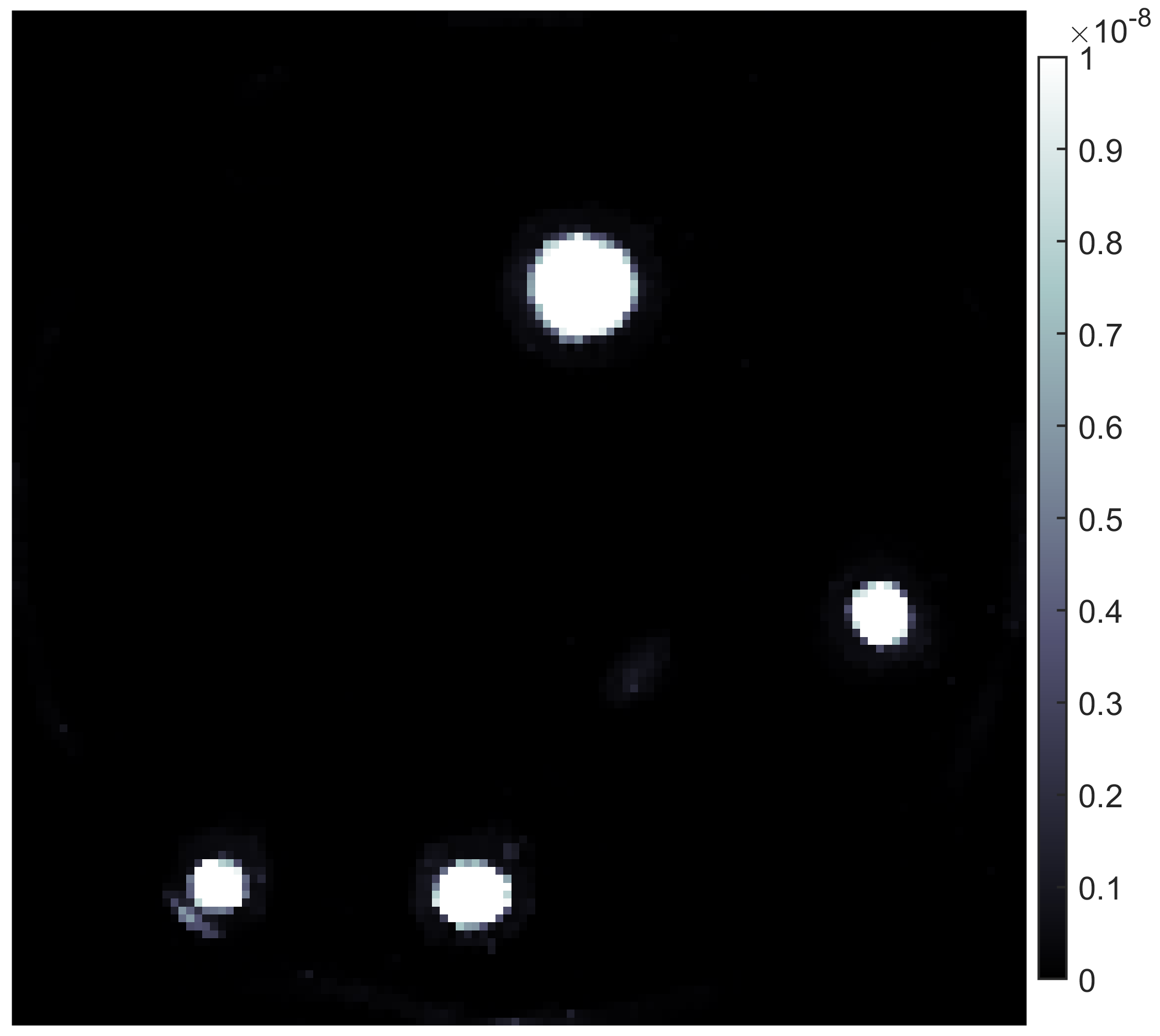} &
        \includegraphics[width=0.23\textwidth]{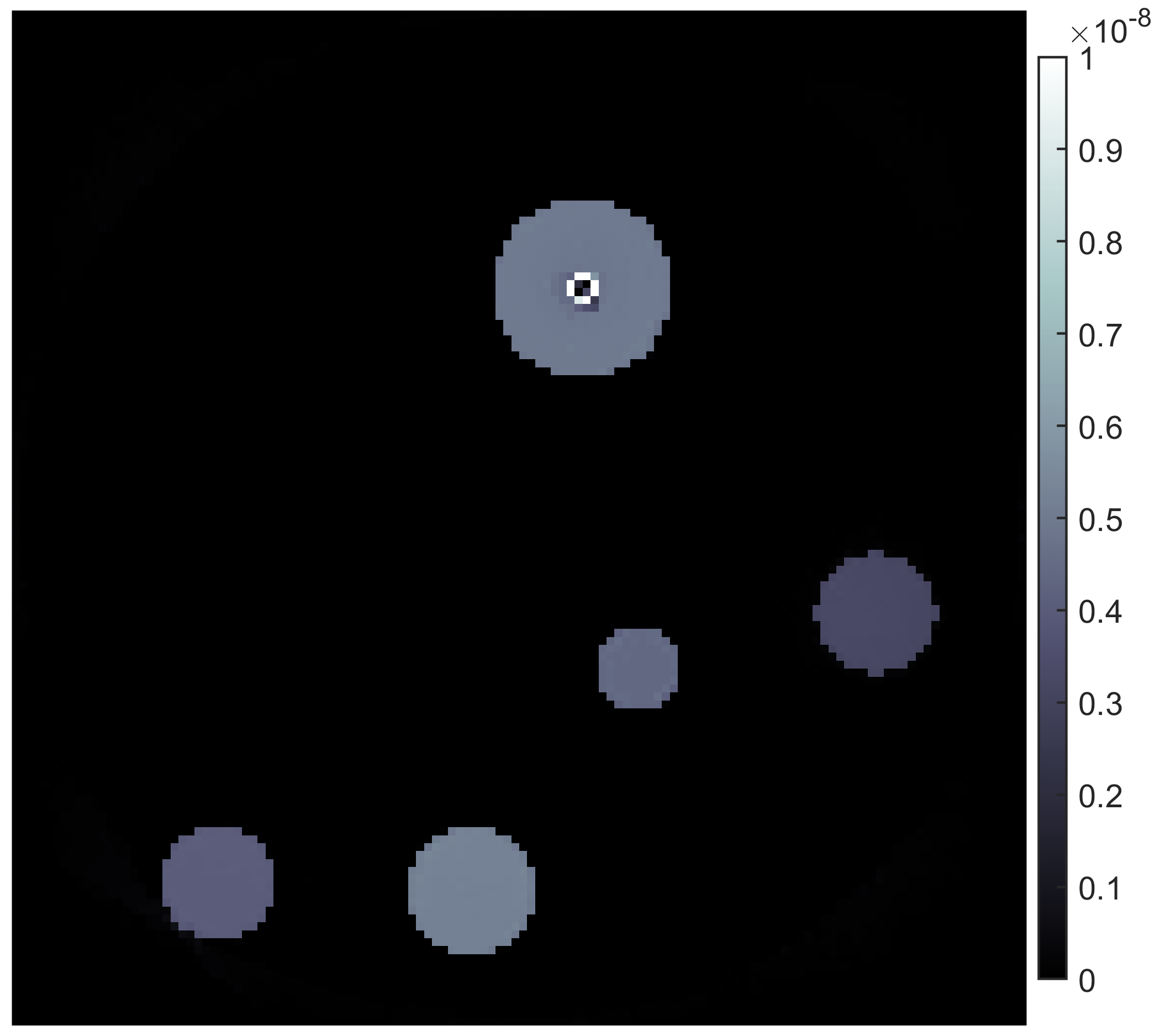} &
        \includegraphics[width=0.23\textwidth]{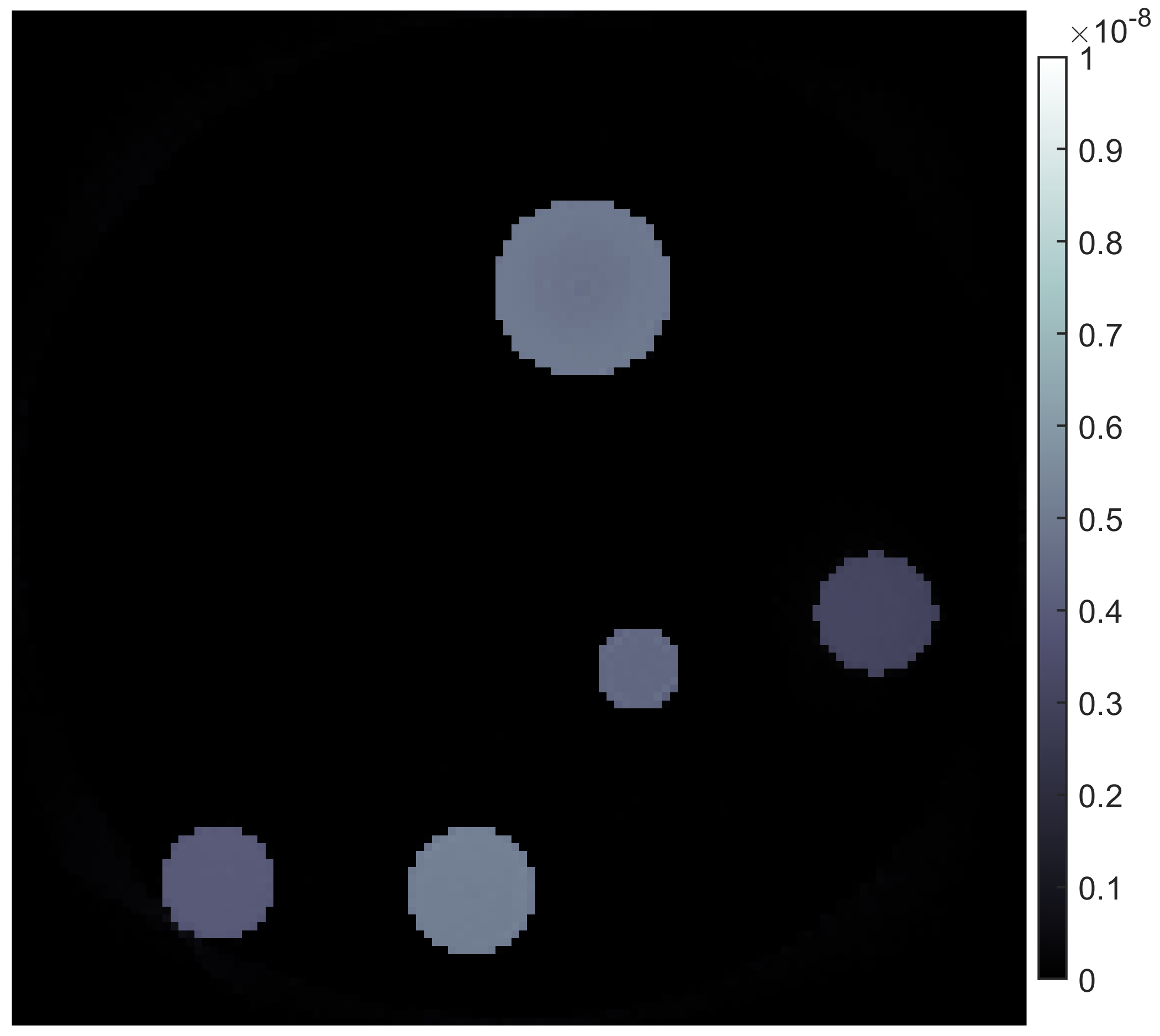} &
        \includegraphics[width=0.23\textwidth]{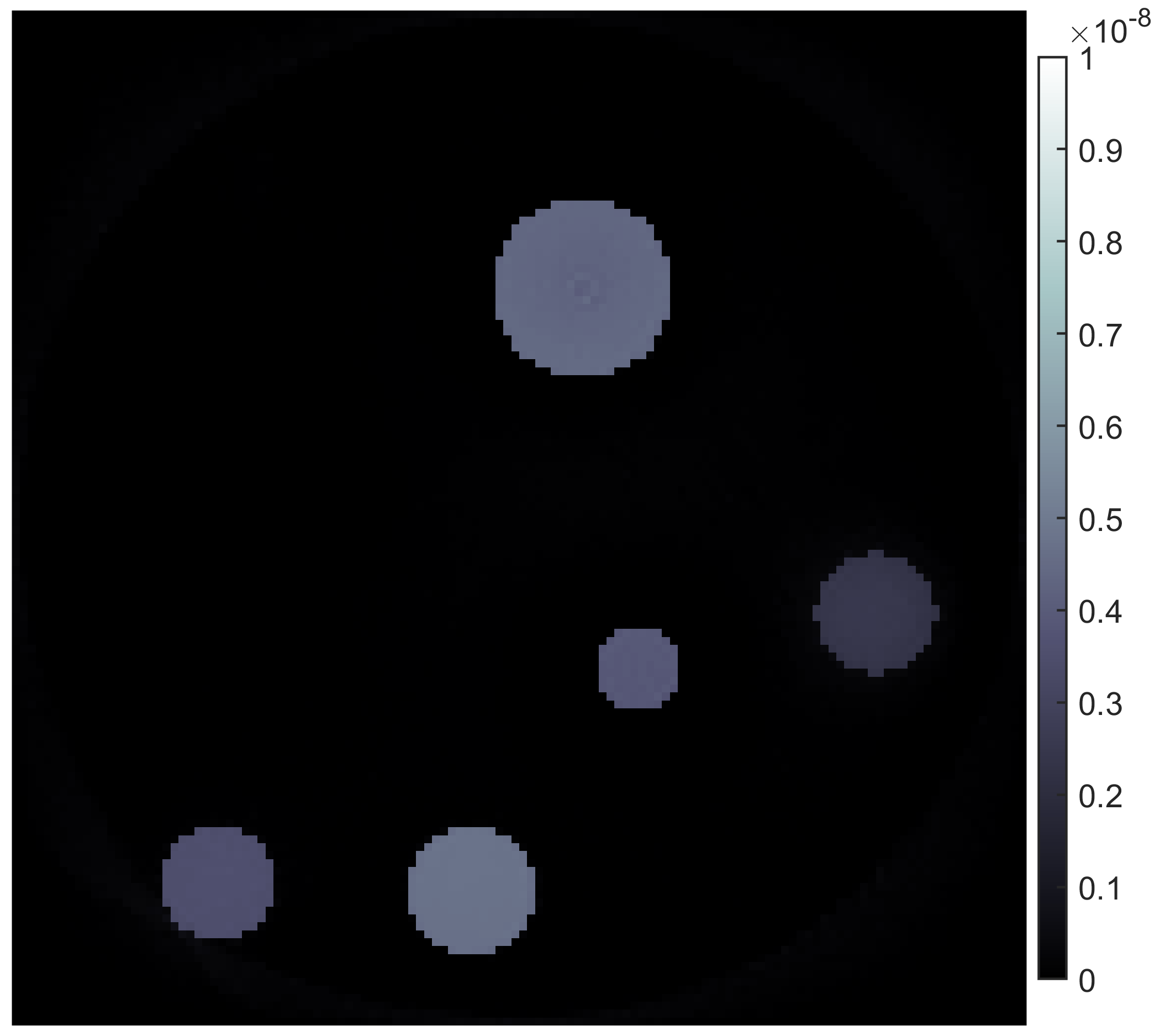} \\
        
        \rowlabel{diff $\beta$} &
        \includegraphics[width=0.23\textwidth]{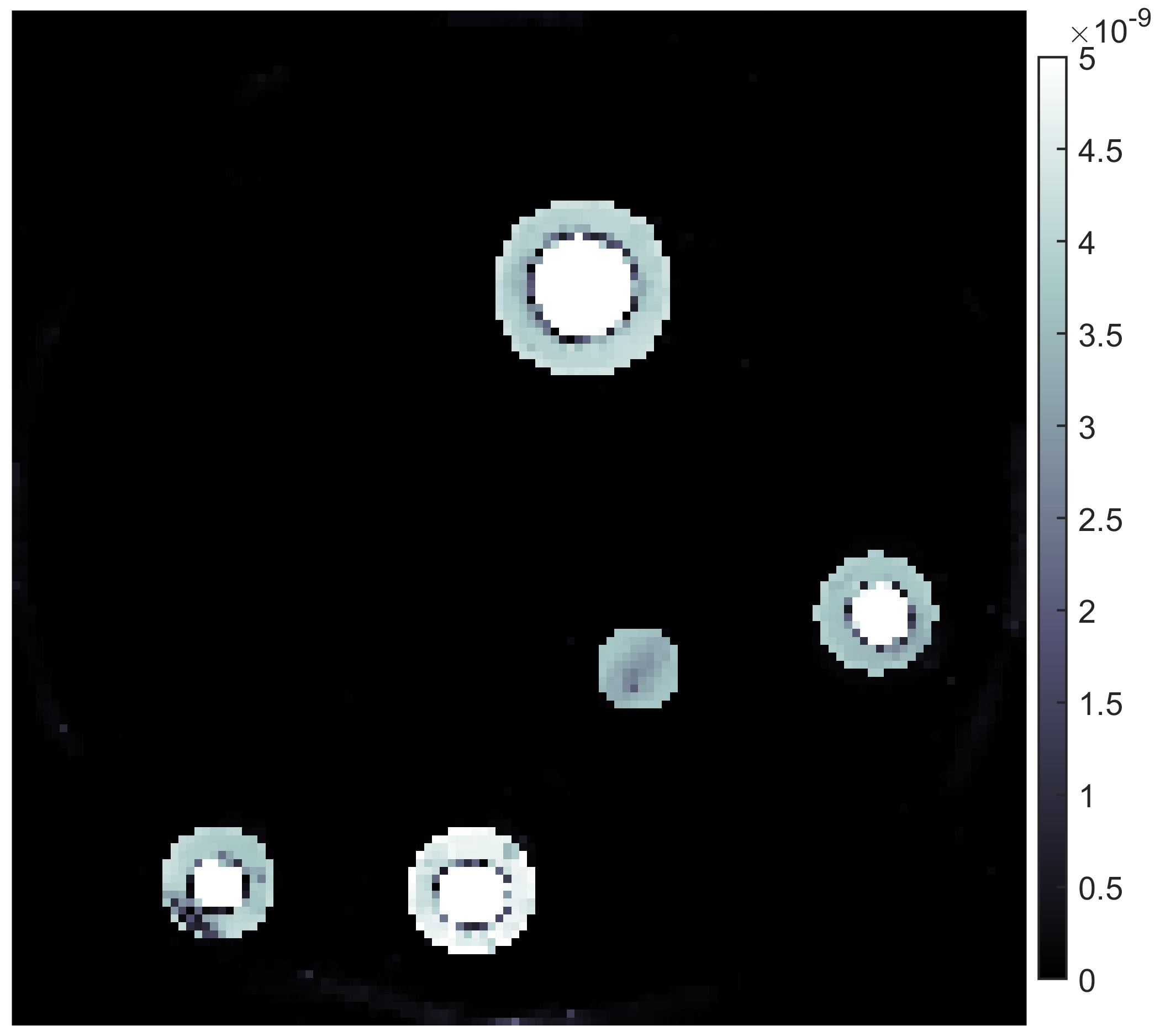} &
        \includegraphics[width=0.23\textwidth]{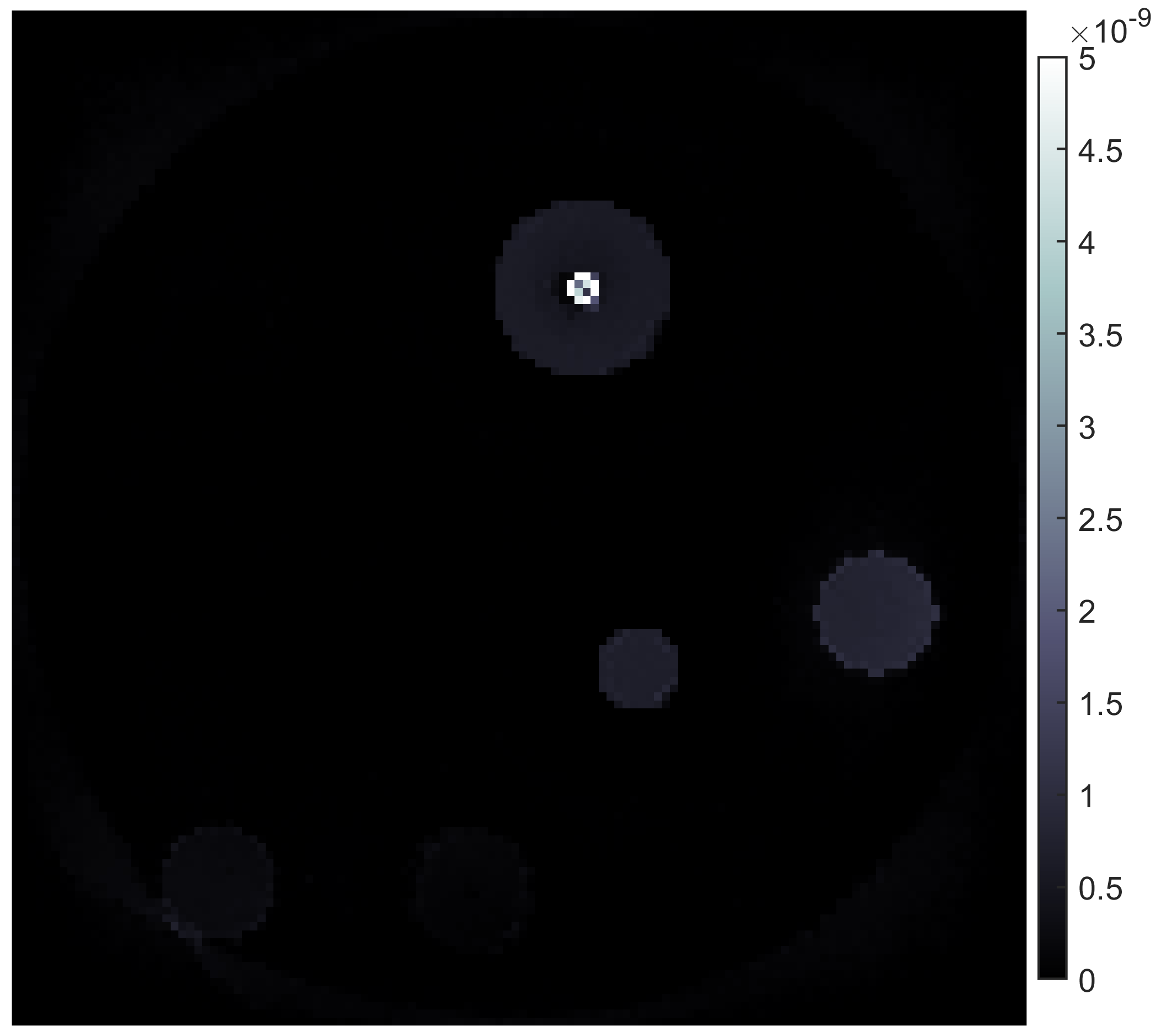} &
        \includegraphics[width=0.23\textwidth]{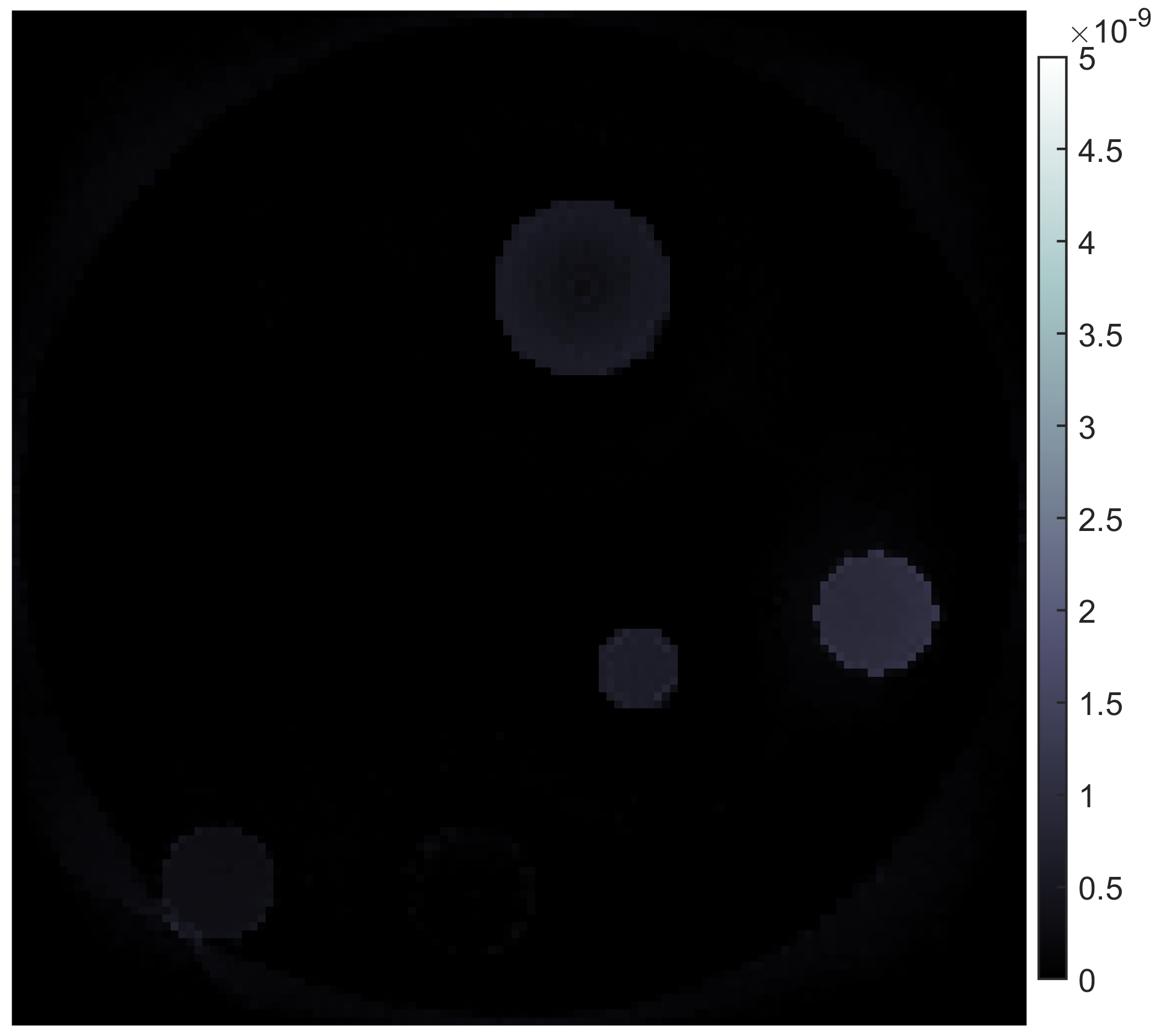} &
        \includegraphics[width=0.23\textwidth]{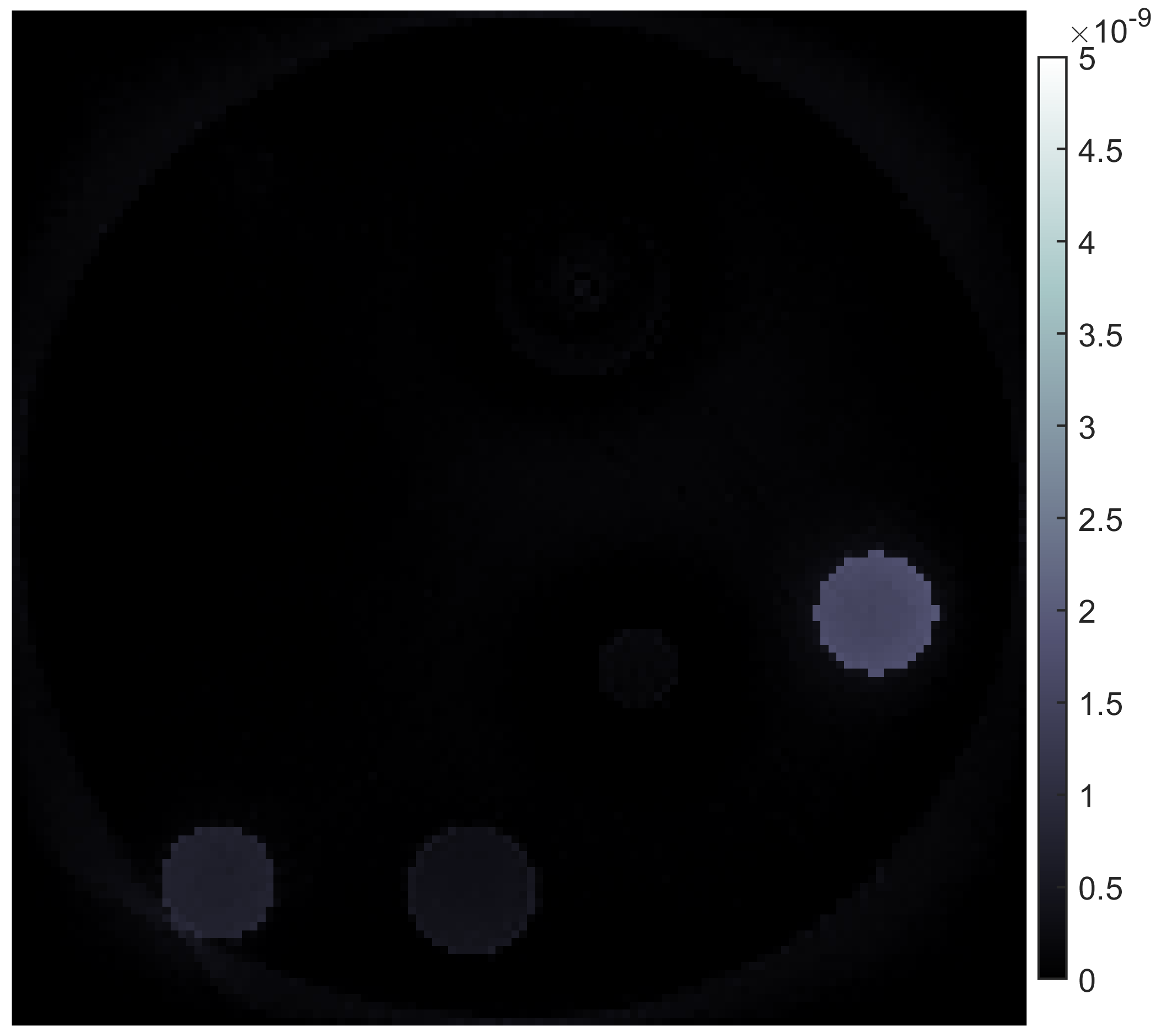} \\
    \end{tabular}
  \caption{3D reconstruction of the relative refractive index (at slice $z=38$) with the phase-guided Bregman TV regularization.
  Rows display the reconstructed phase decrement $-\delta$, absorption component $\beta$, and the corresponding reconstruction error maps. 
  Columns correspond to different values of $\gamma$.
  }
  \label{fig:gamma_analysis}
\end{figure}
\begin{figure}[H]
    \centering
    \setlength{\tabcolsep}{1pt}
     \resizebox{0.95\textwidth}{!}{%
    \begin{tabular}{c c c c c}
        & \footnotesize GD 
        & \footnotesize Phase-guided
        & \footnotesize GD diff
        & \footnotesize Guided diff \\

        \rowlabel{$\delta/\beta=10^{2}$} &
        \includegraphics[width=0.23\textwidth]{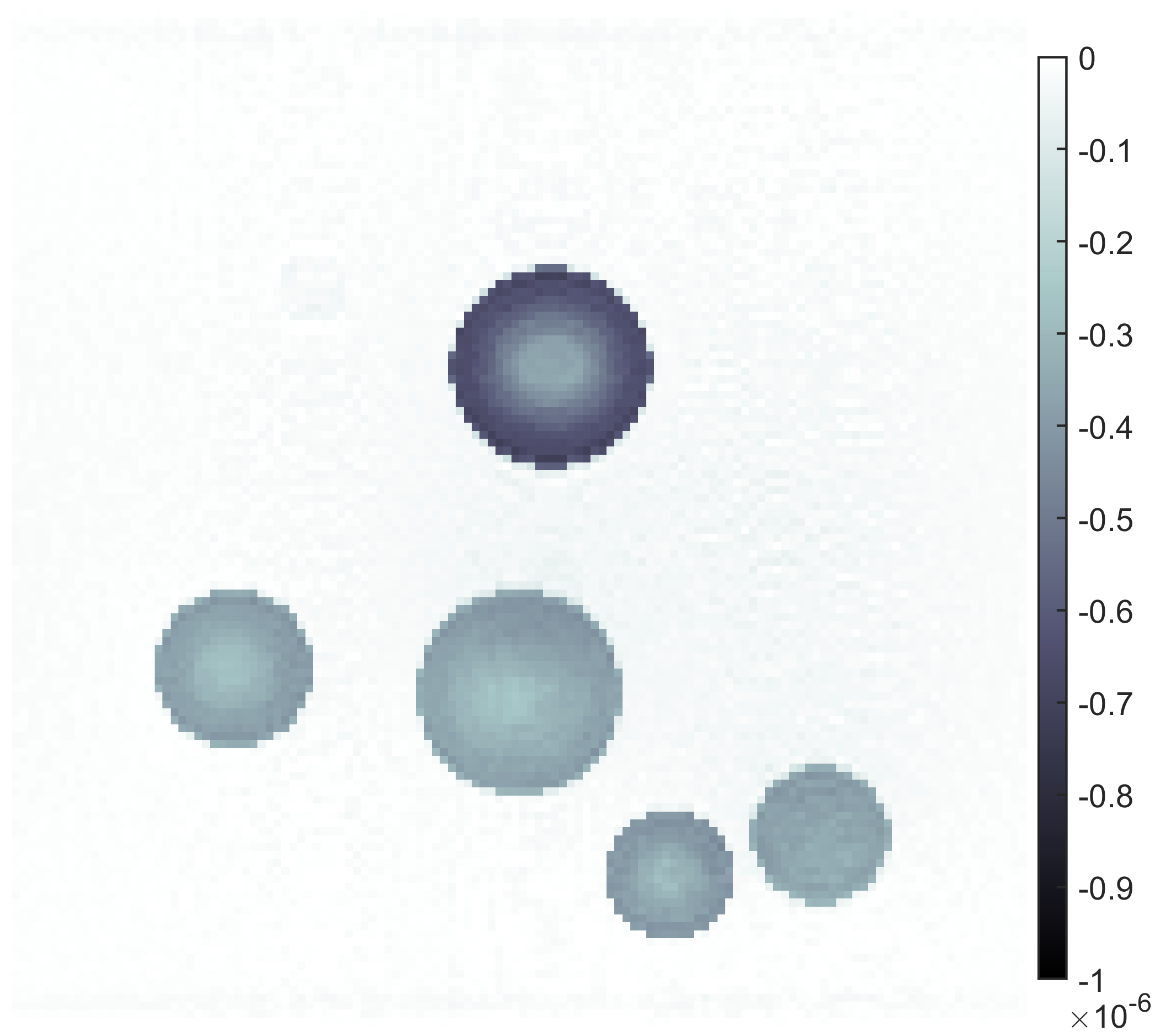} &
        \includegraphics[width=0.23\textwidth]{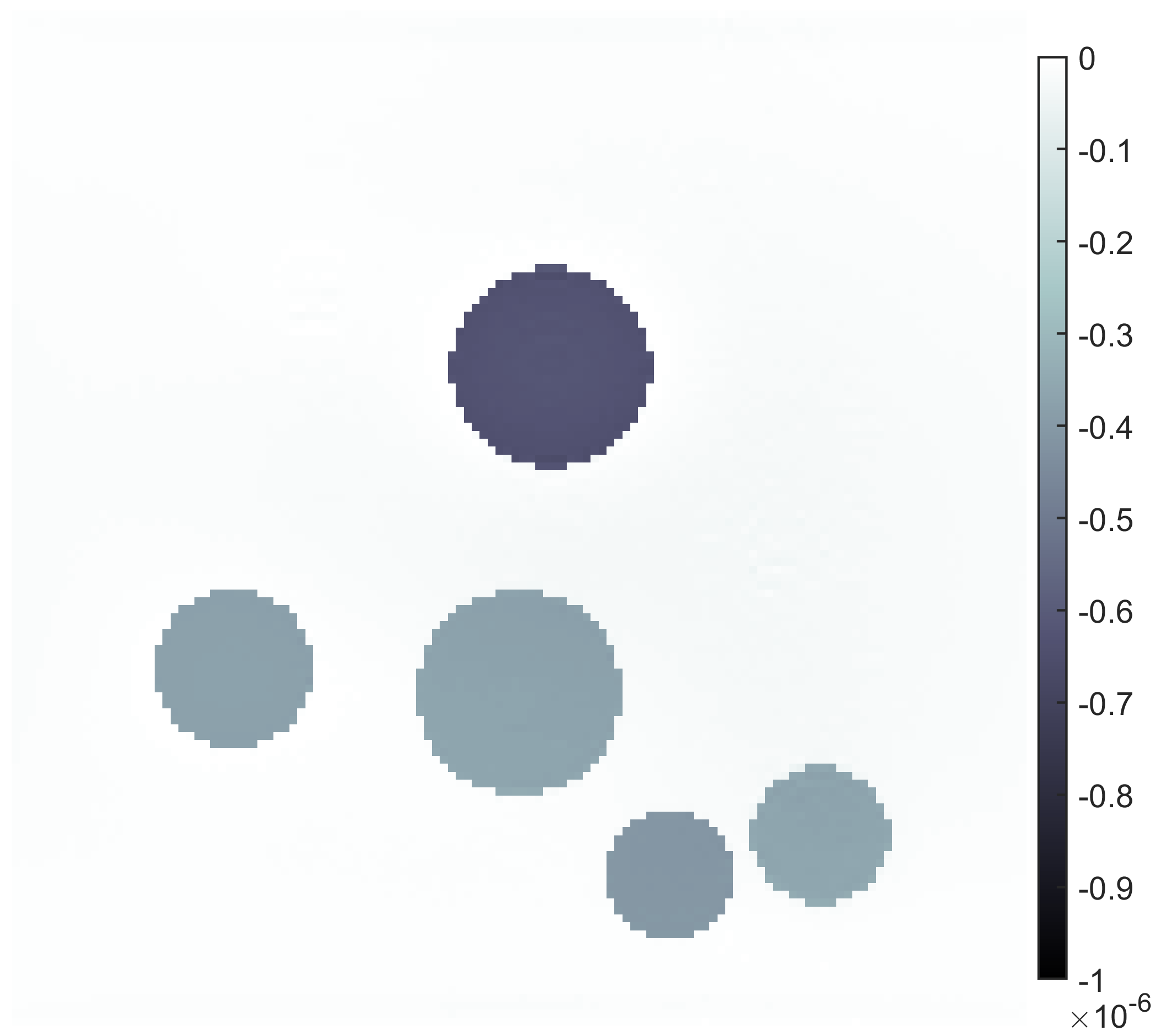} &
        \includegraphics[width=0.23\textwidth]{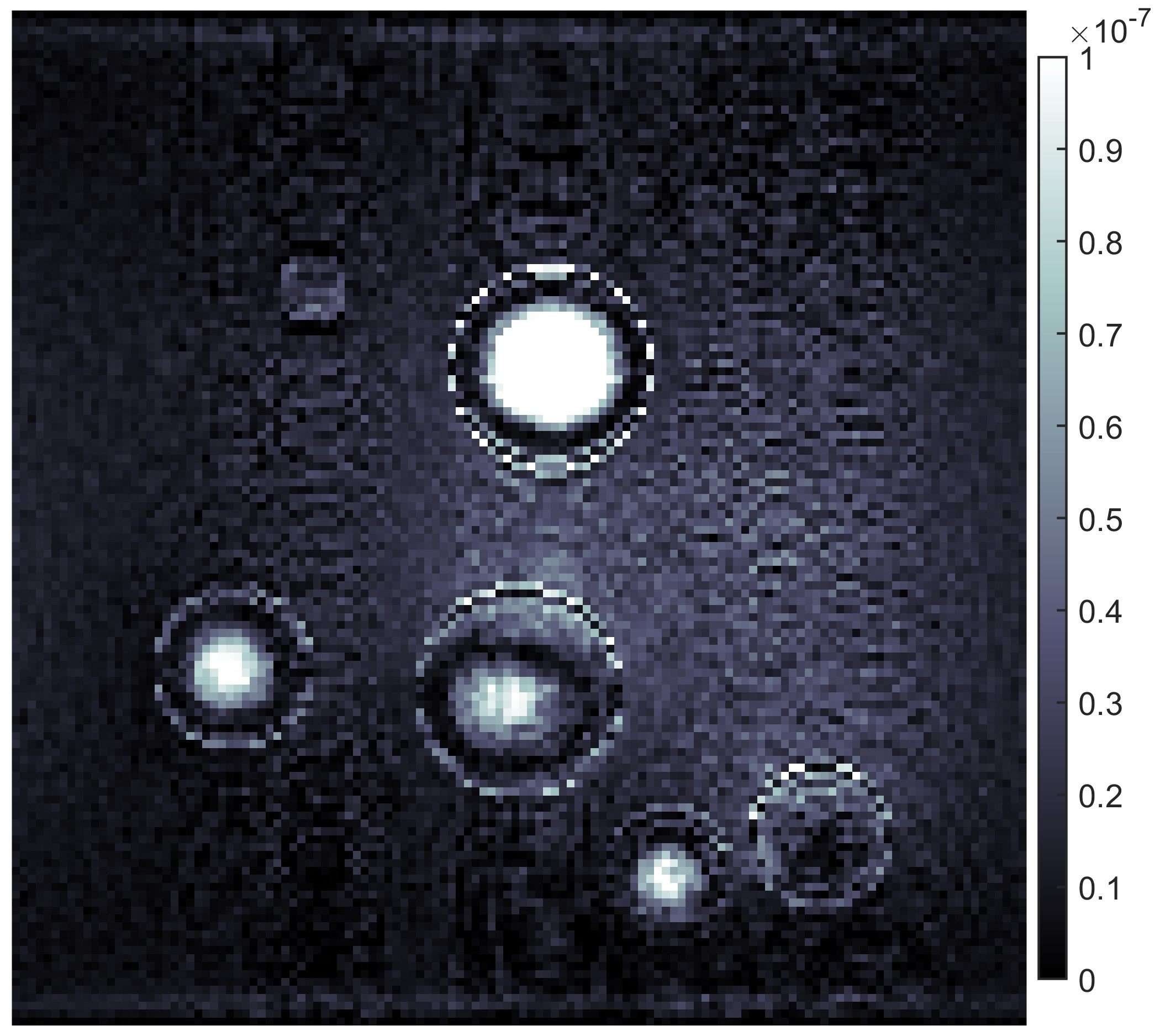} &
        \includegraphics[width=0.23\textwidth]{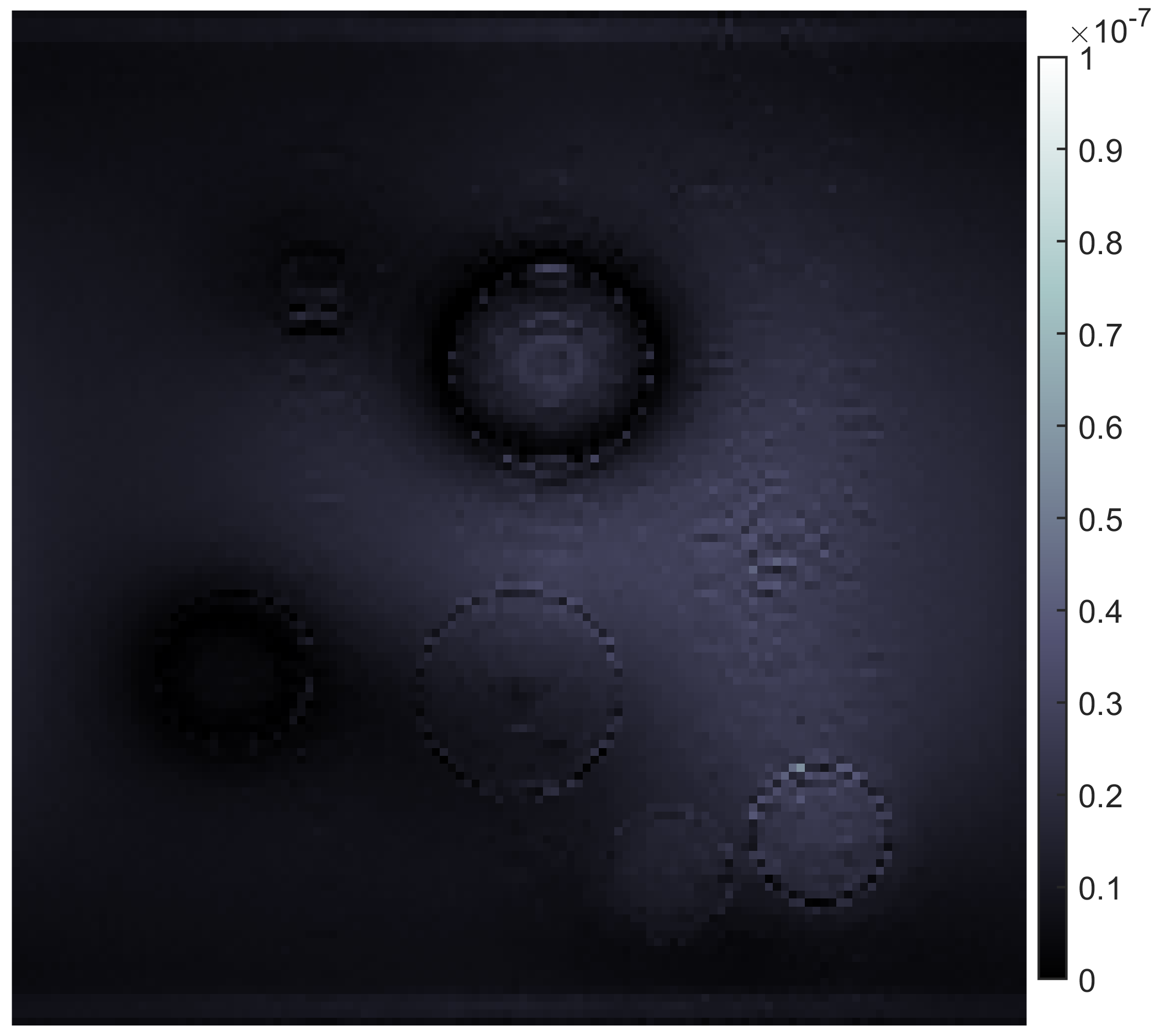} \\
        
        \rowlabel{$\delta/\beta=10^{3}$} &
        \includegraphics[width=0.23\textwidth]{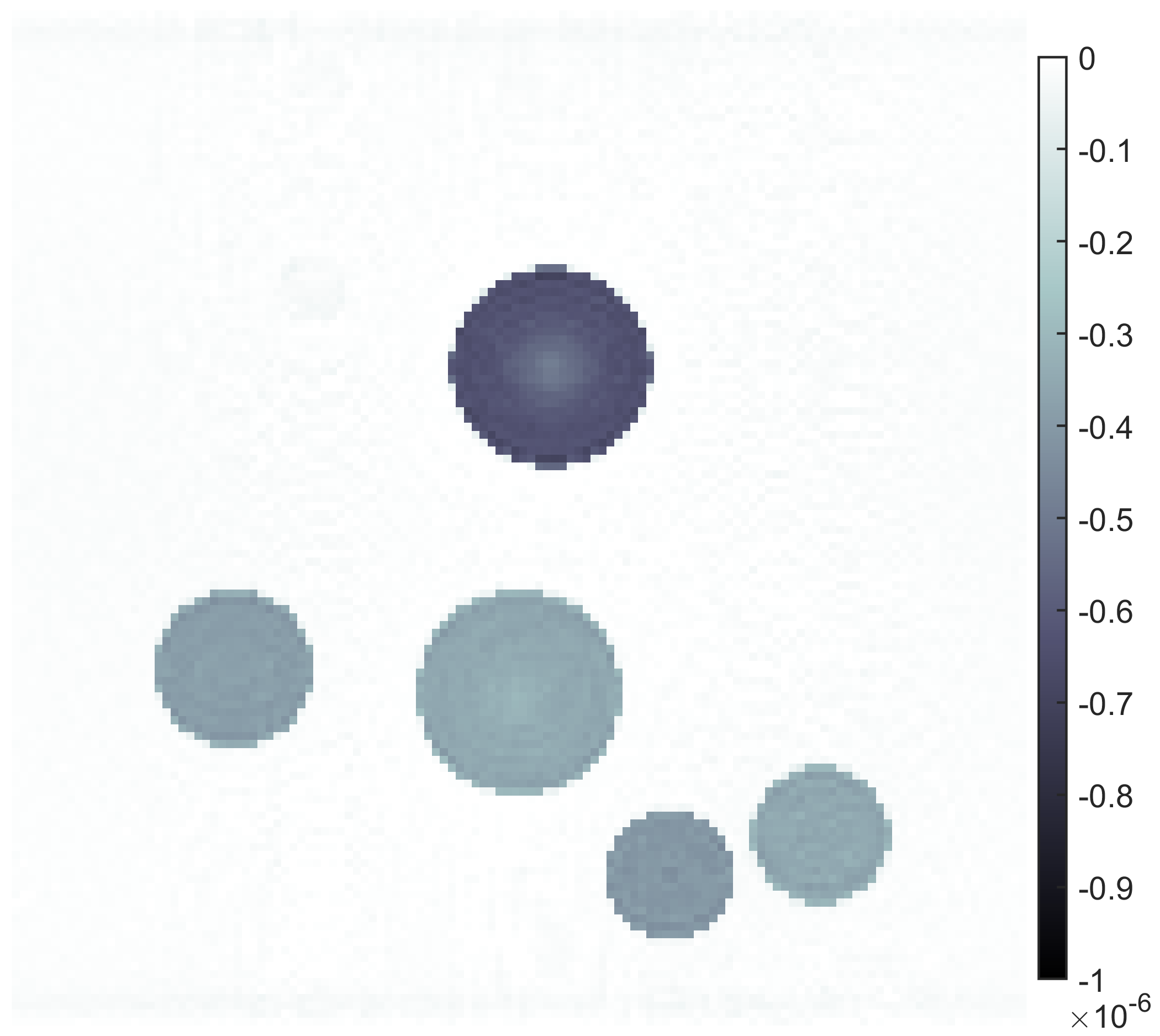} &
        \includegraphics[width=0.23\textwidth]{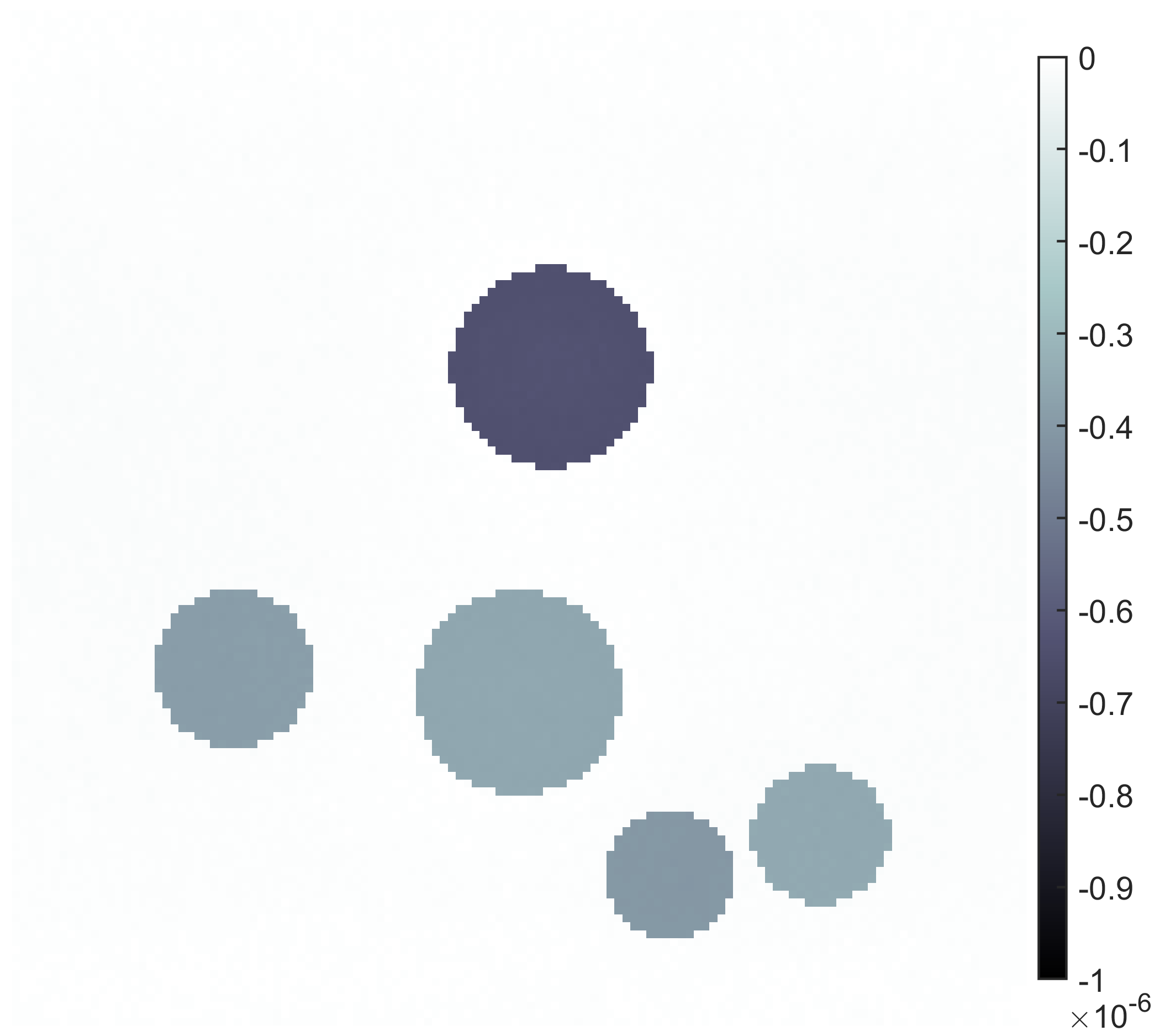} &
        \includegraphics[width=0.23\textwidth]{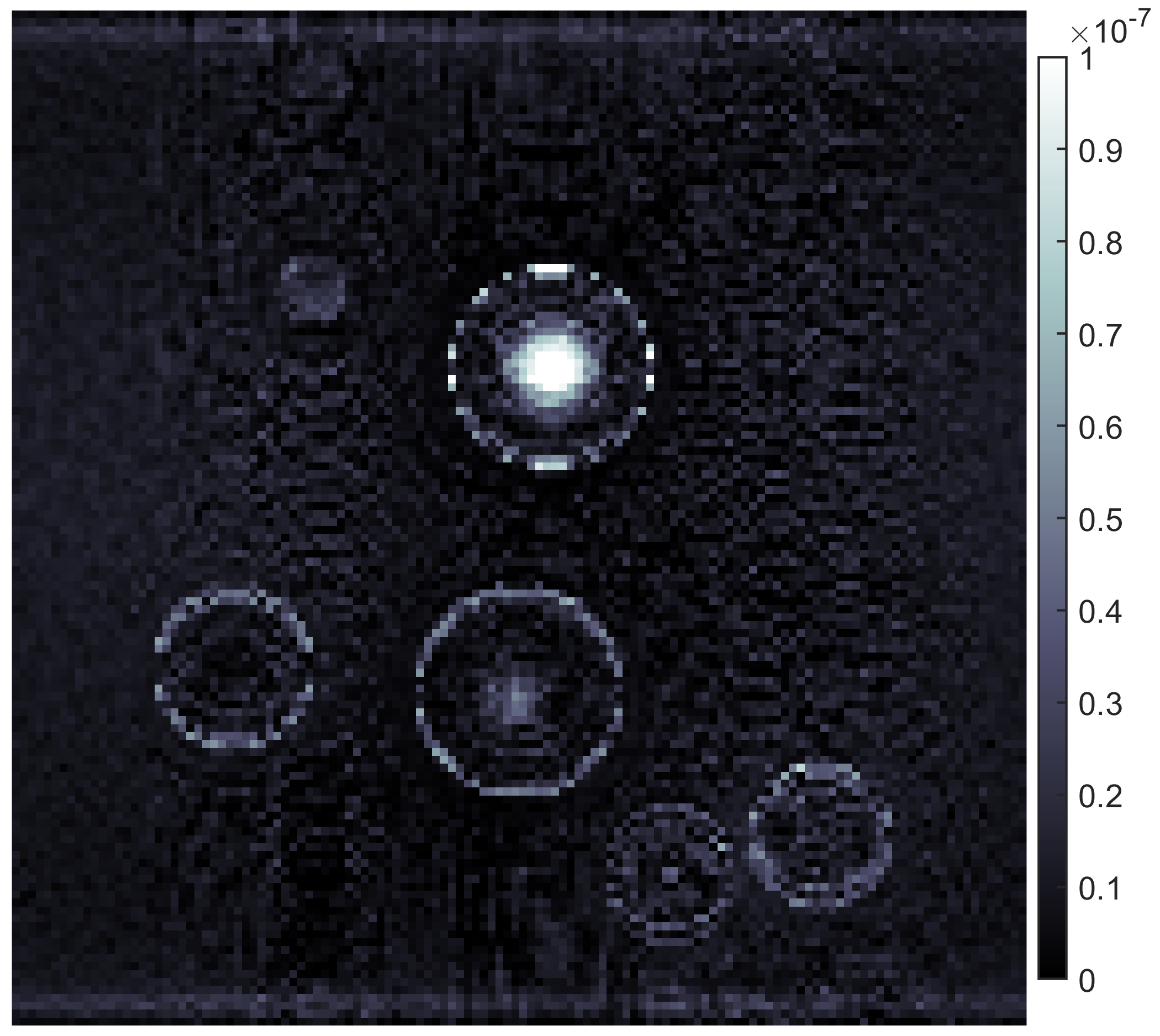} &
        \includegraphics[width=0.23\textwidth]{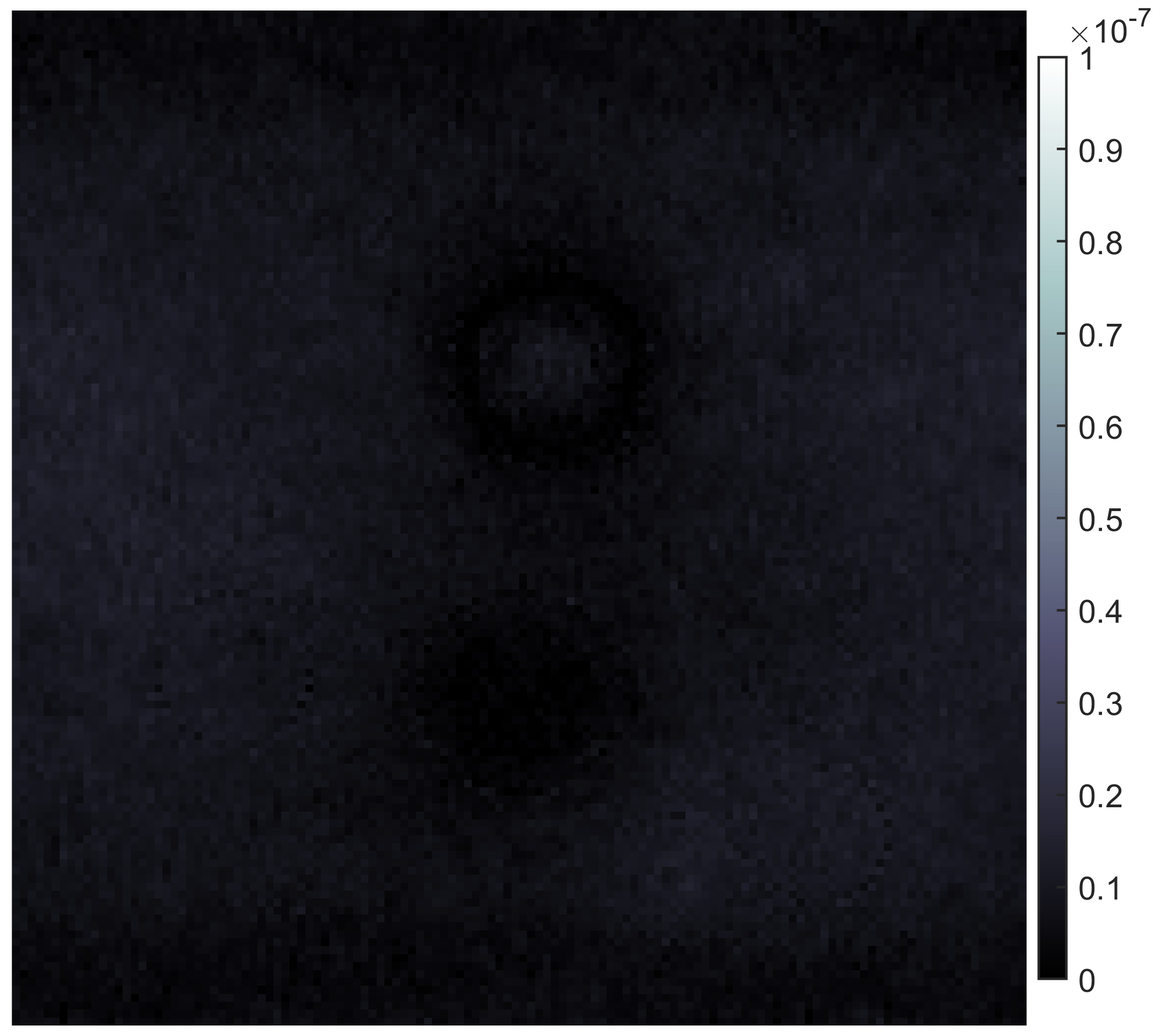} \\

       \rowlabel{$\delta/\beta=10^{4}$} &
       \includegraphics[width=0.23\textwidth]{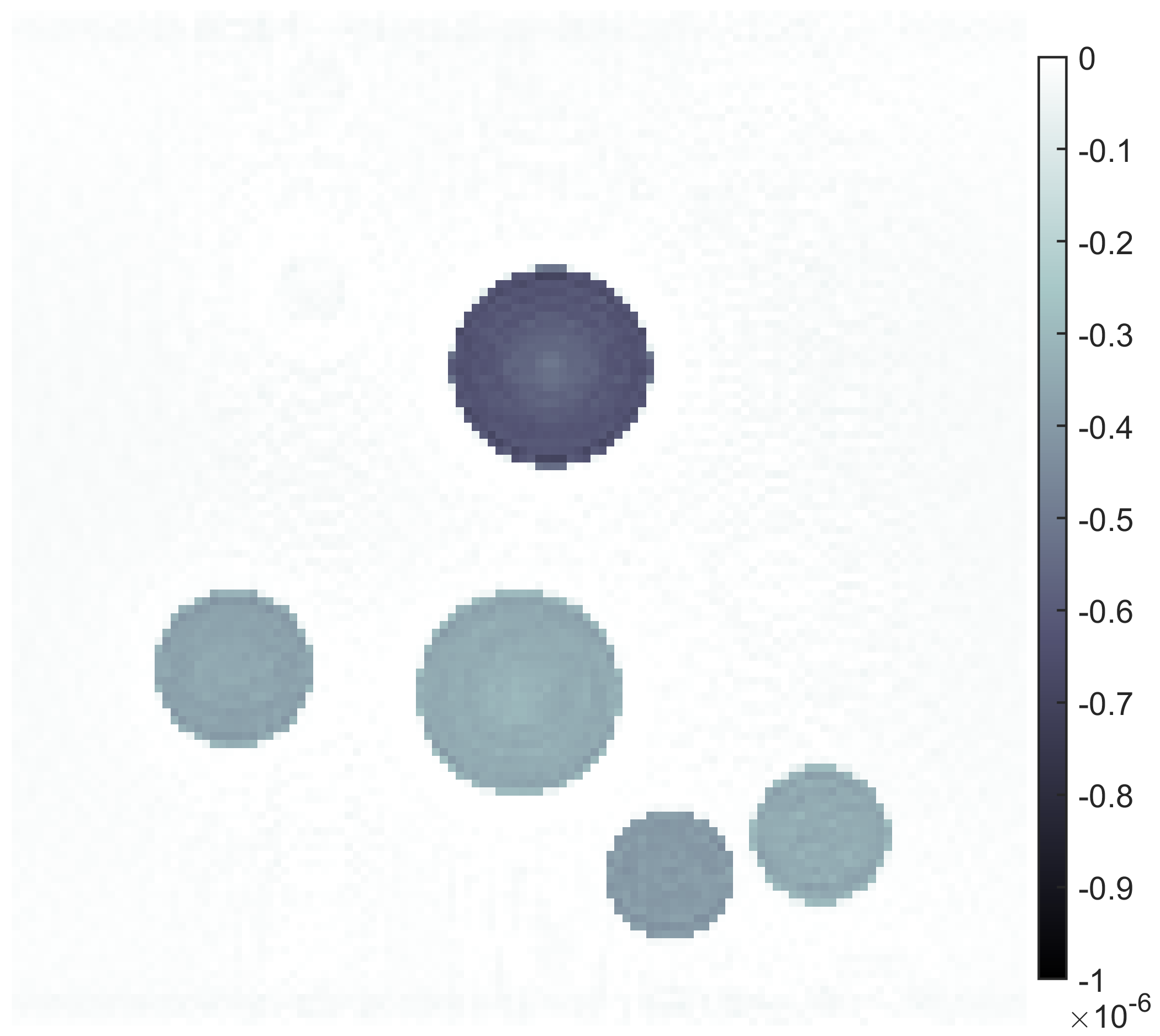} &
        \includegraphics[width=0.23\textwidth]{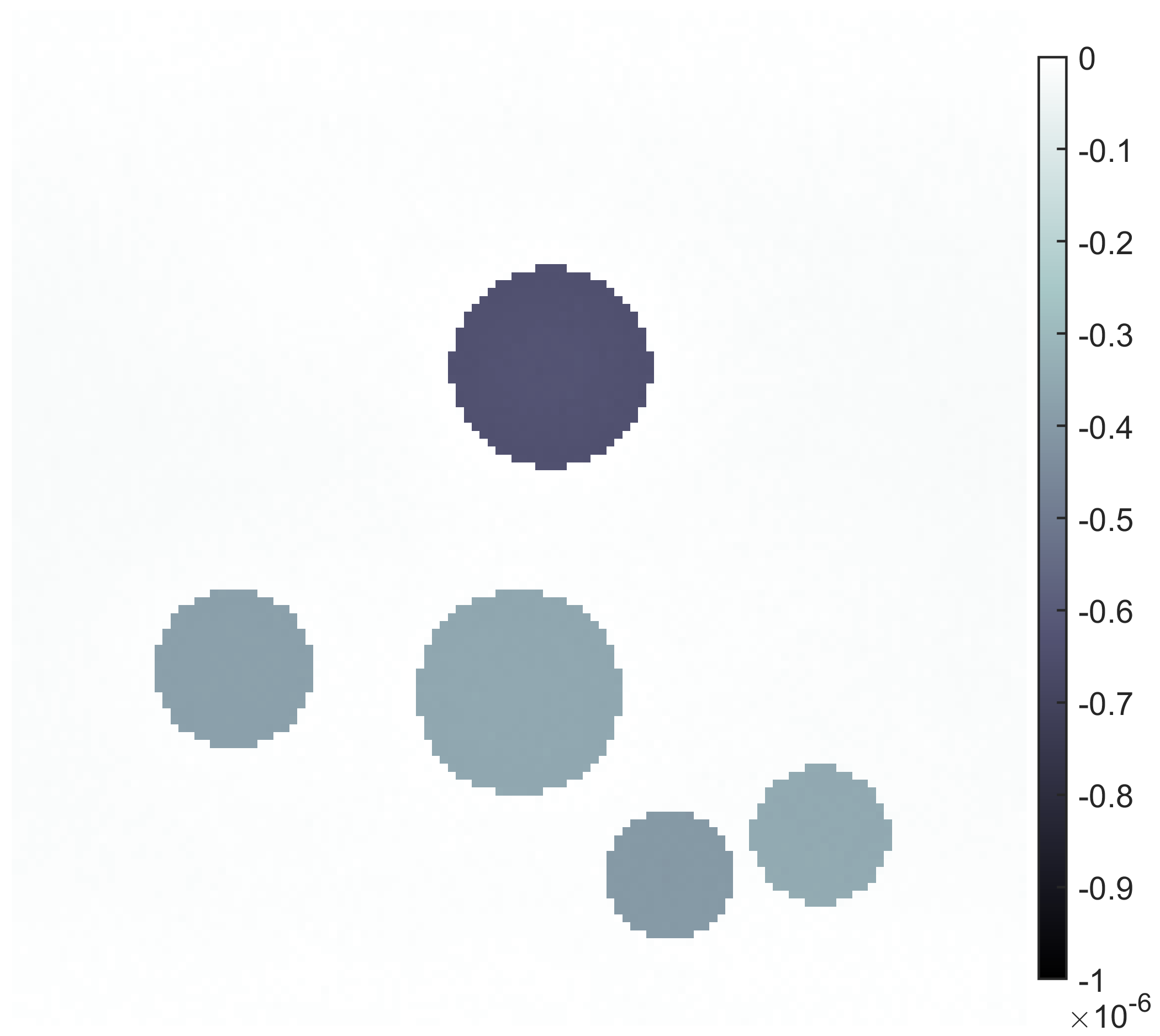} &
        \includegraphics[width=0.23\textwidth]{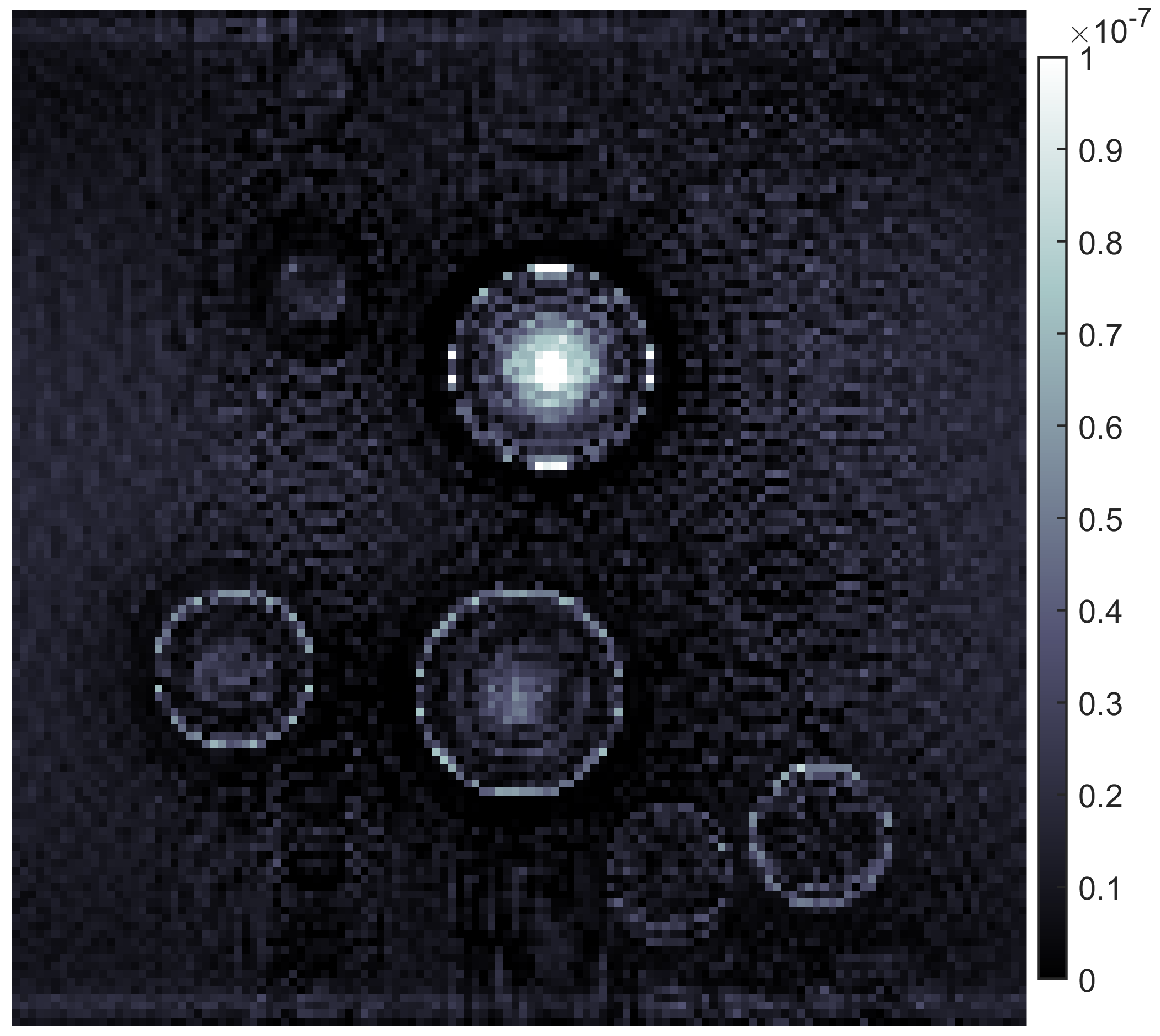} &
        \includegraphics[width=0.23\textwidth]{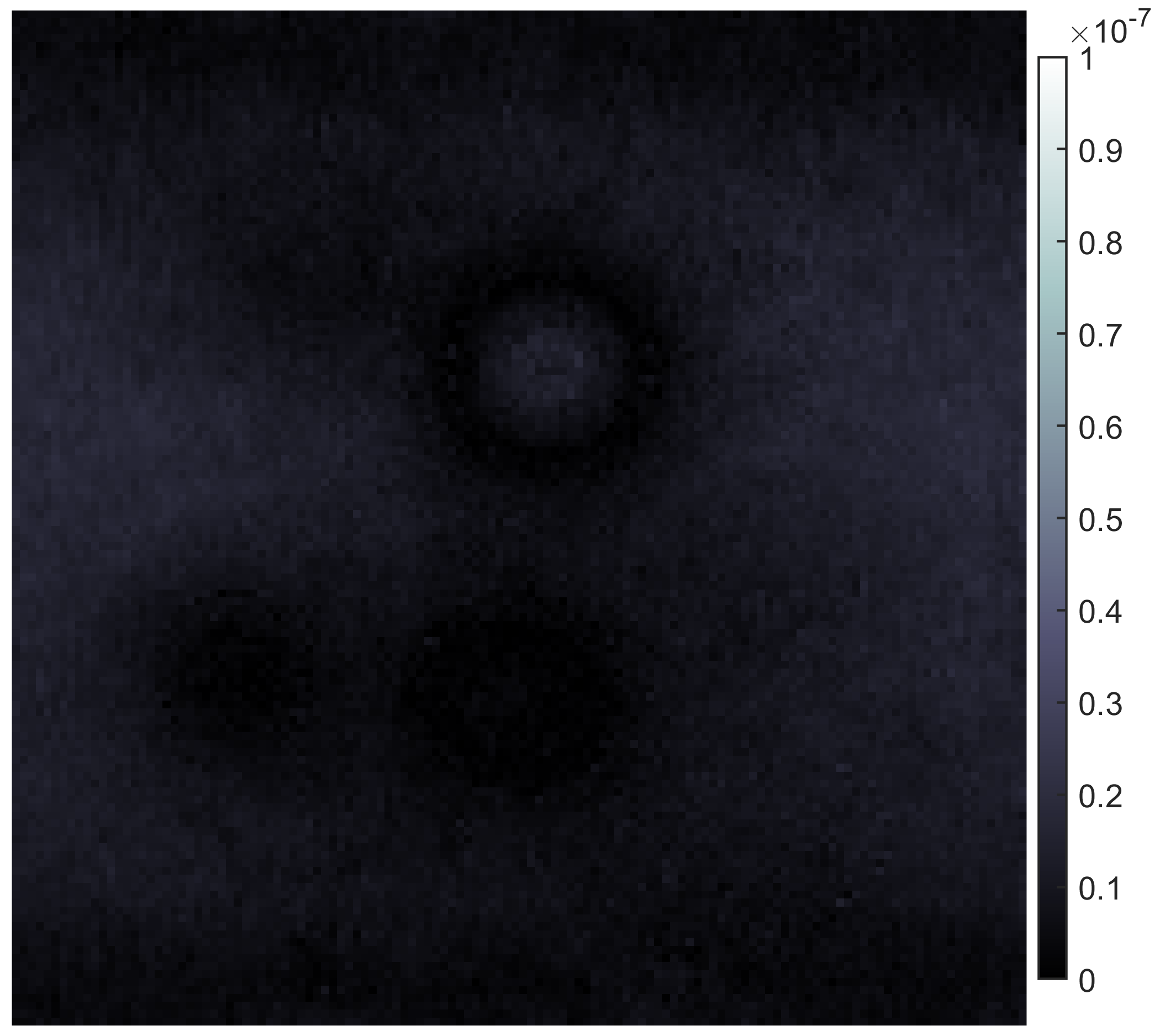} \\

        \rowlabel{$\delta/\beta=10^{2}$} &
        \includegraphics[width=0.23\textwidth]{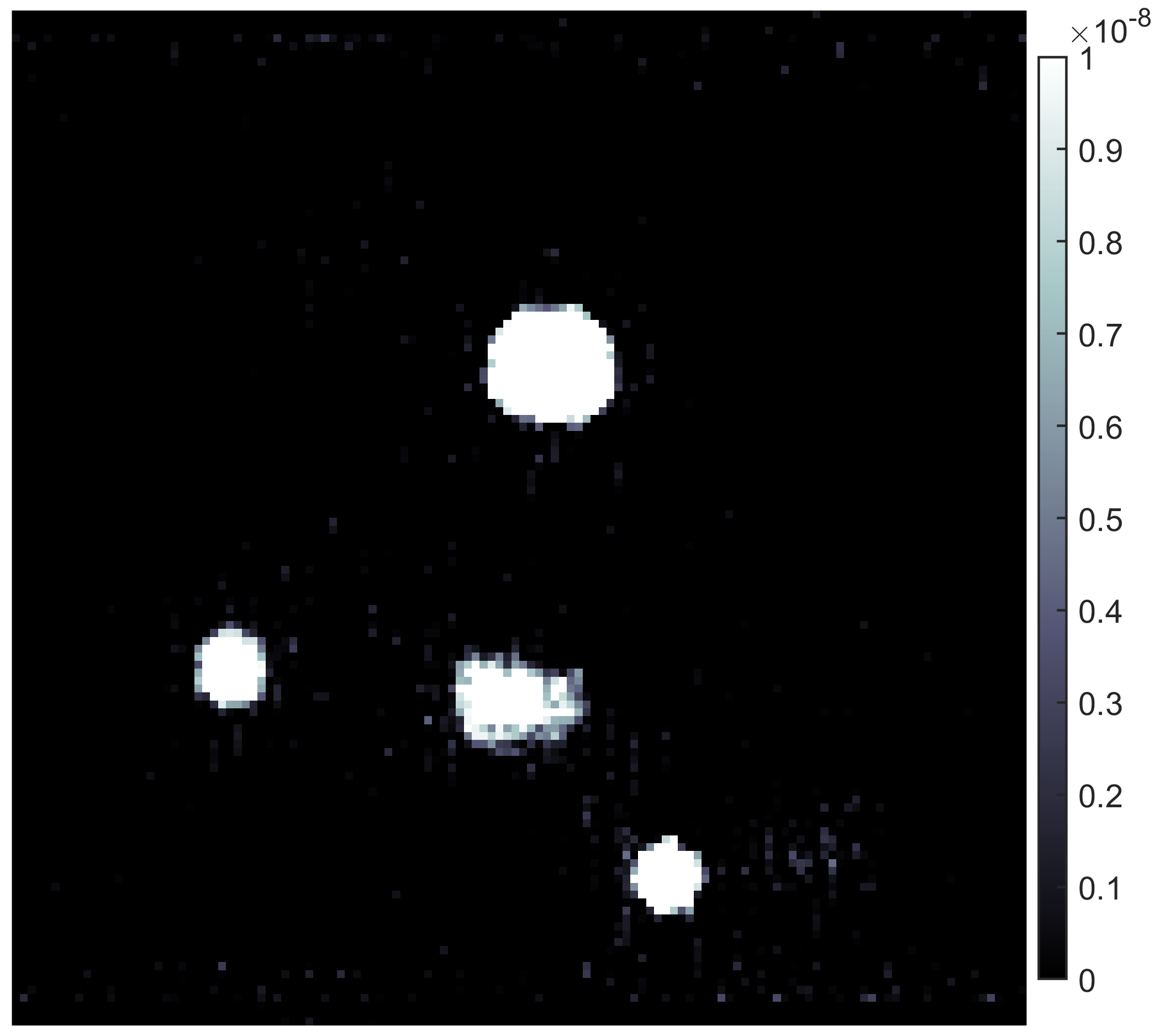} &
        \includegraphics[width=0.23\textwidth]{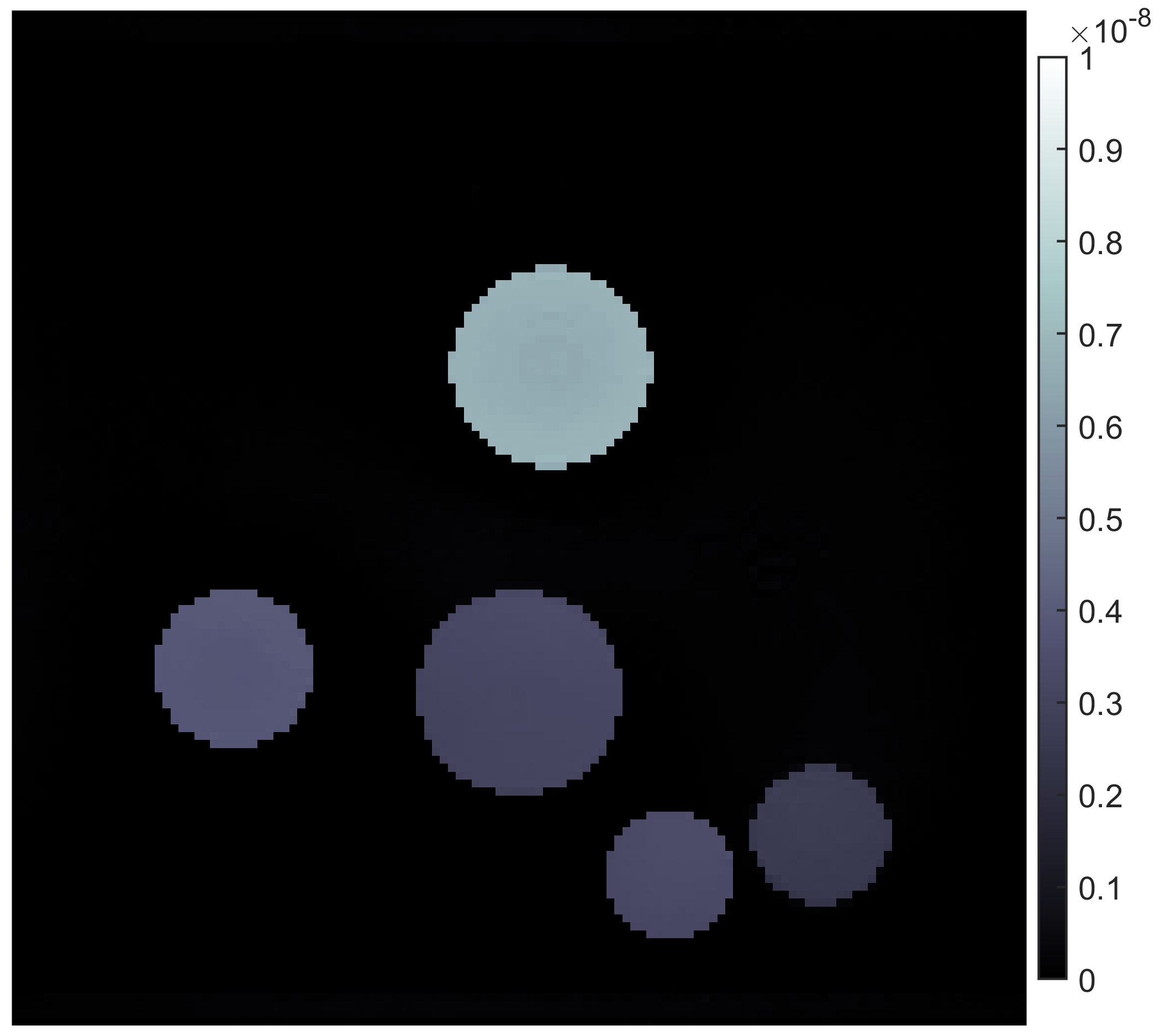} &
        \includegraphics[width=0.23\textwidth]{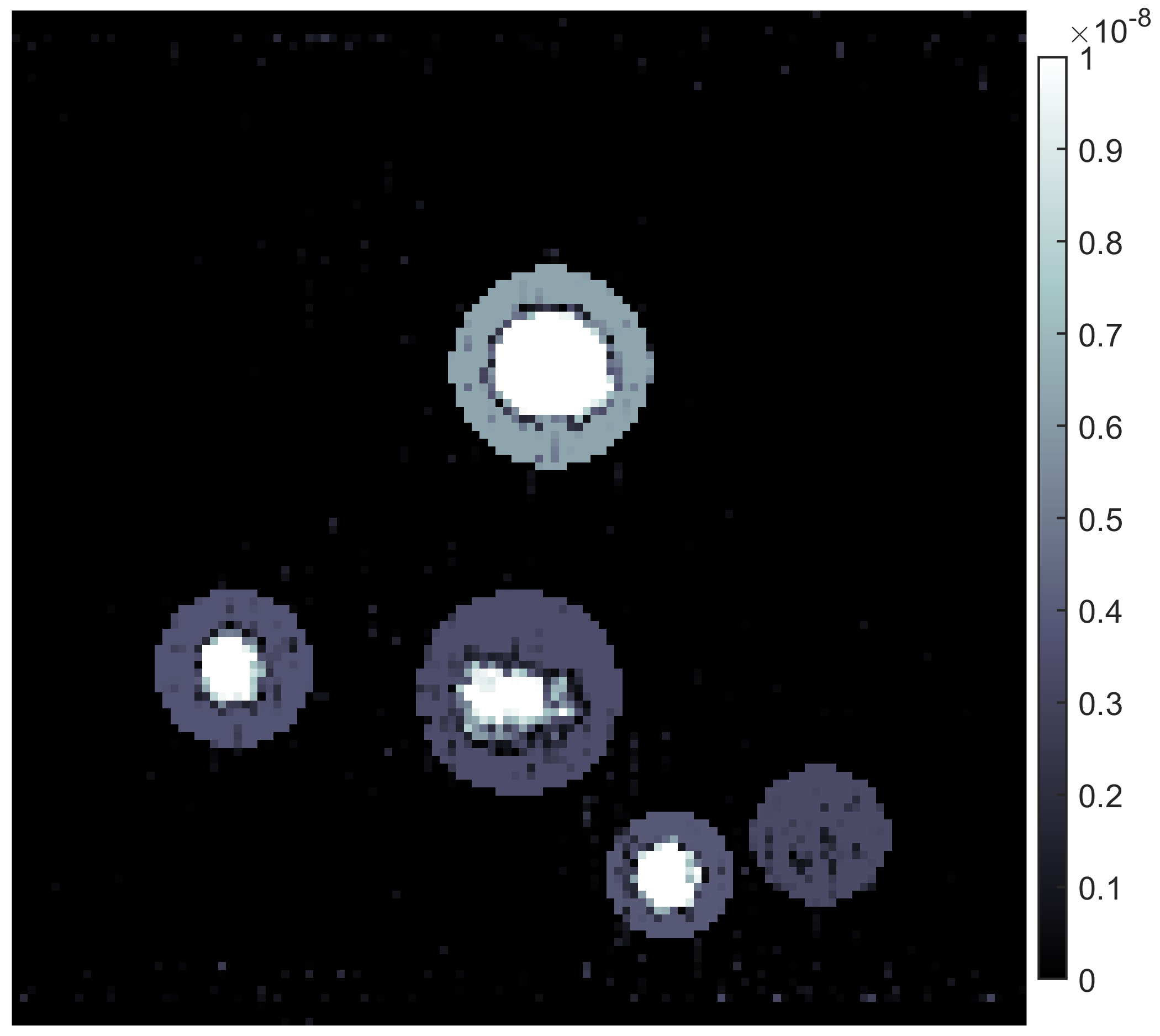} &
        \includegraphics[width=0.23\textwidth]{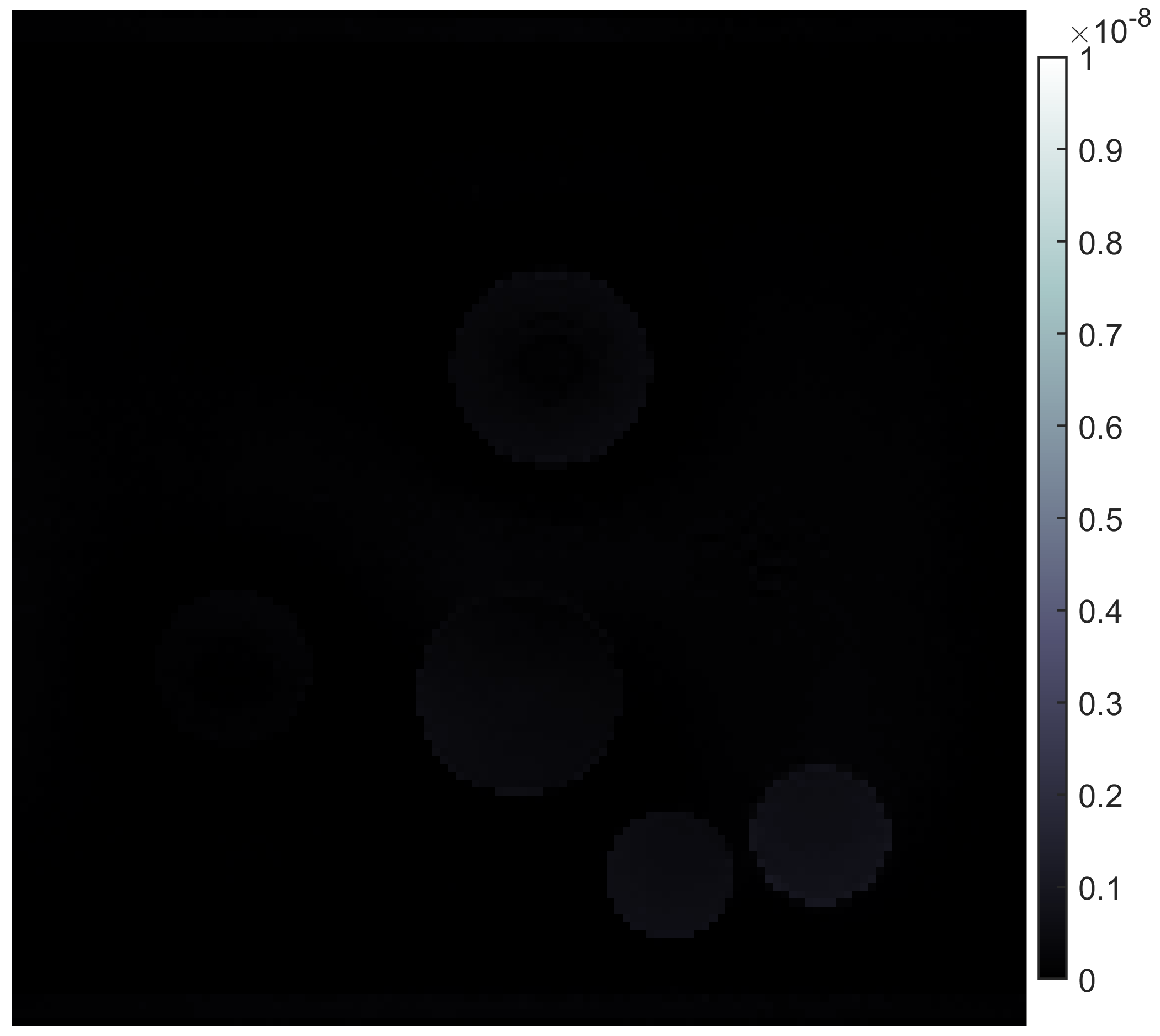} \\
        
        \rowlabel{$\delta/\beta=10^{3}$} &
        \includegraphics[width=0.23\textwidth]{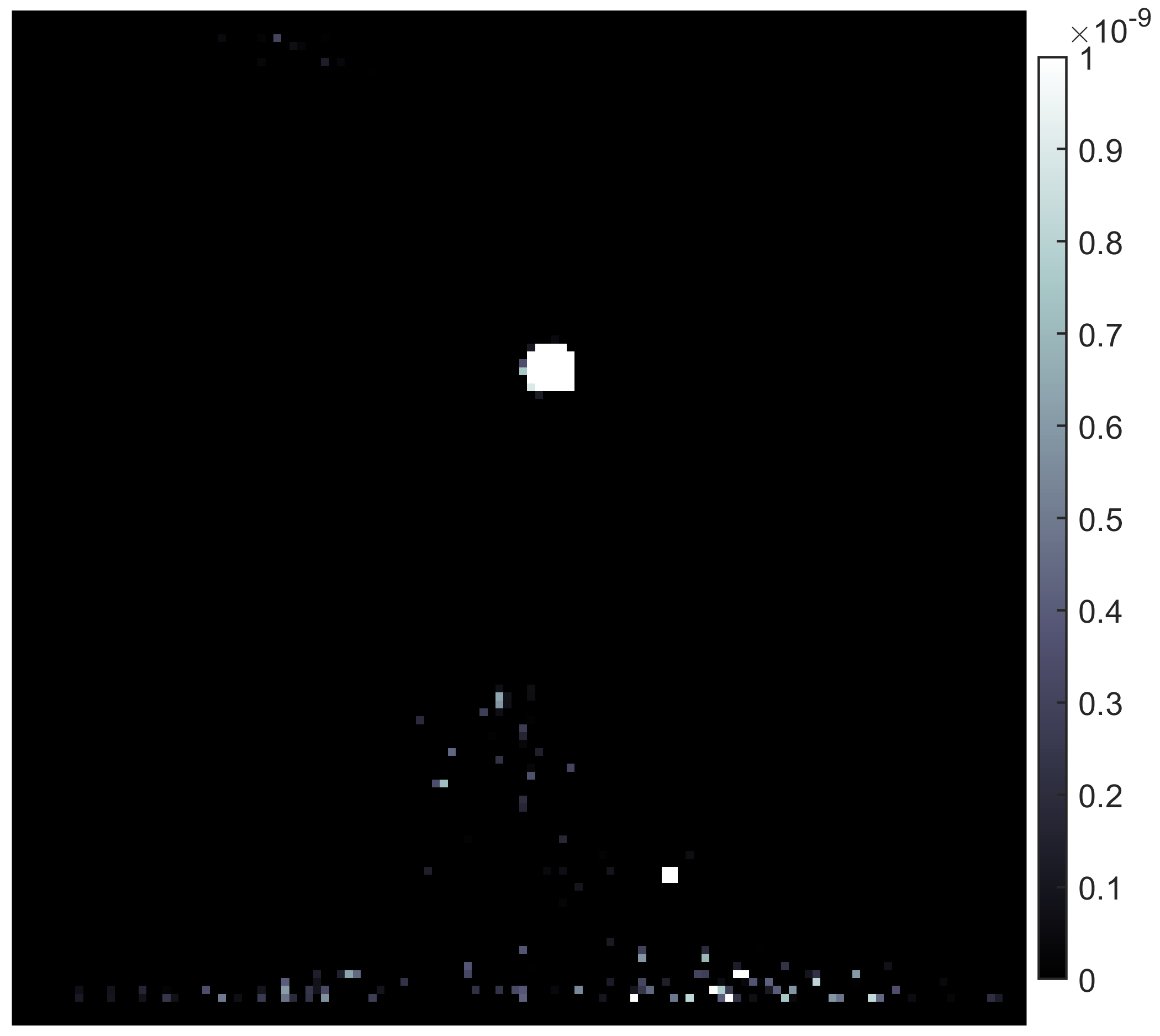} &
        \includegraphics[width=0.23\textwidth]{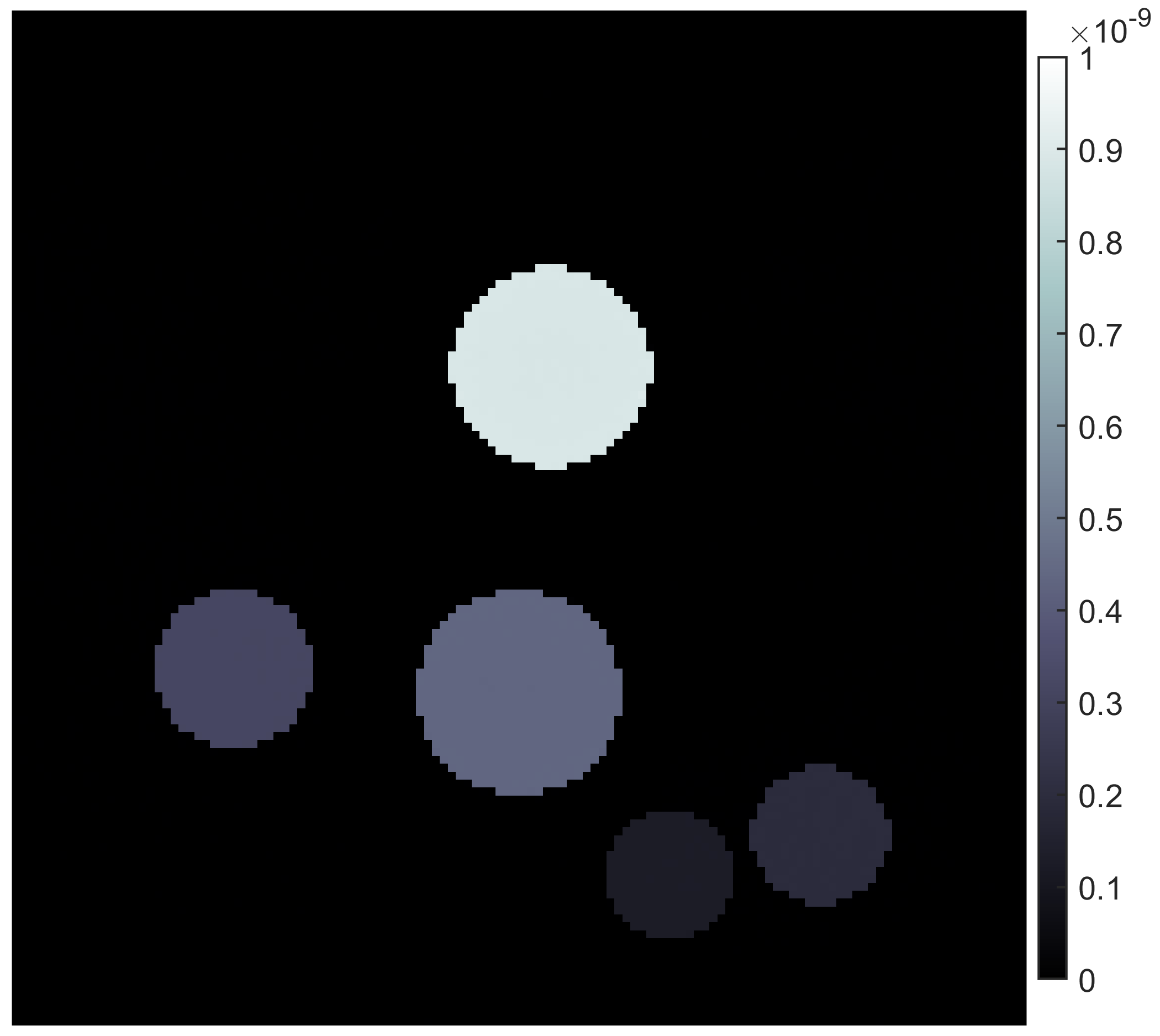} &
        \includegraphics[width=0.23\textwidth]{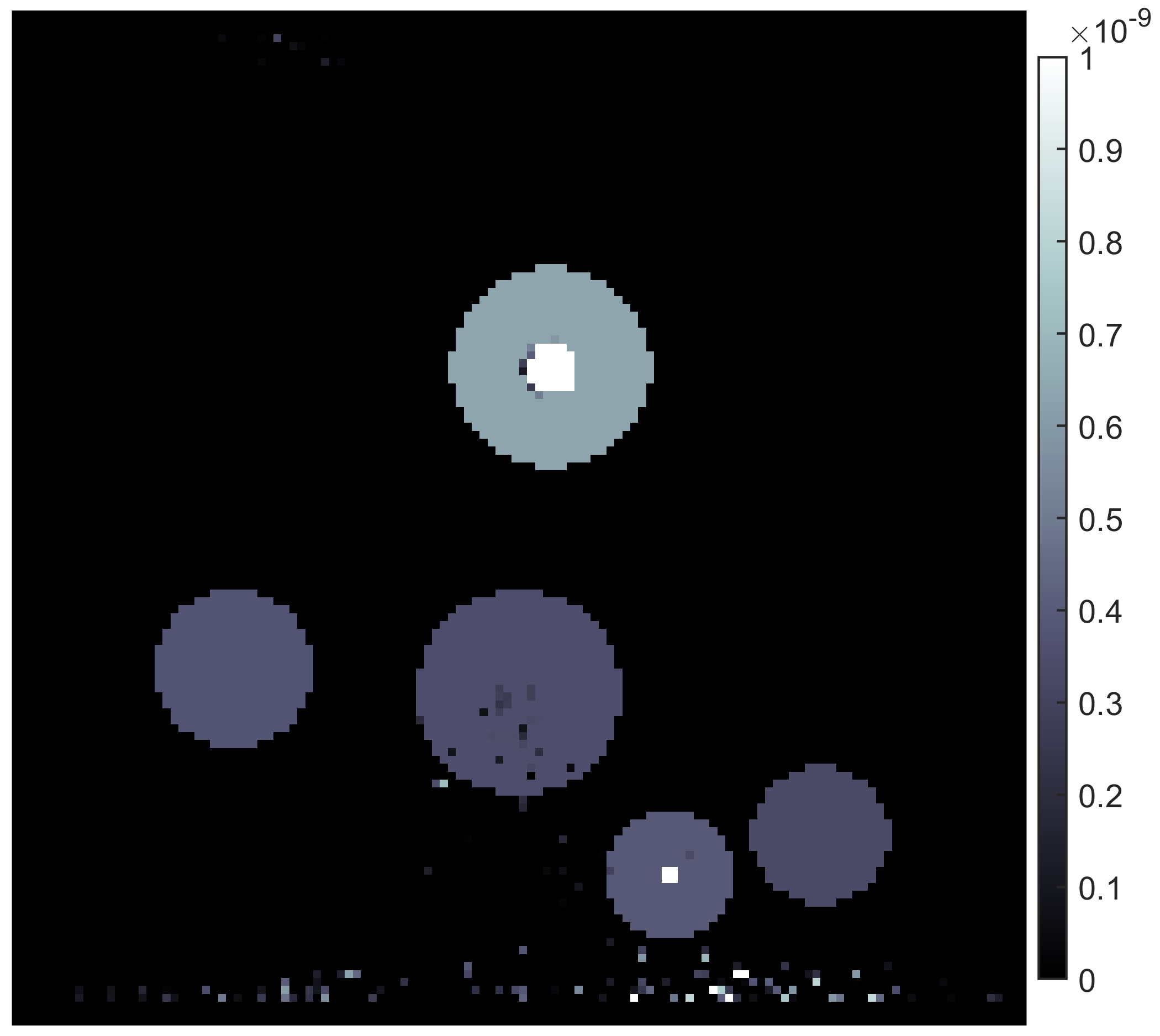} &
        \includegraphics[width=0.23\textwidth]{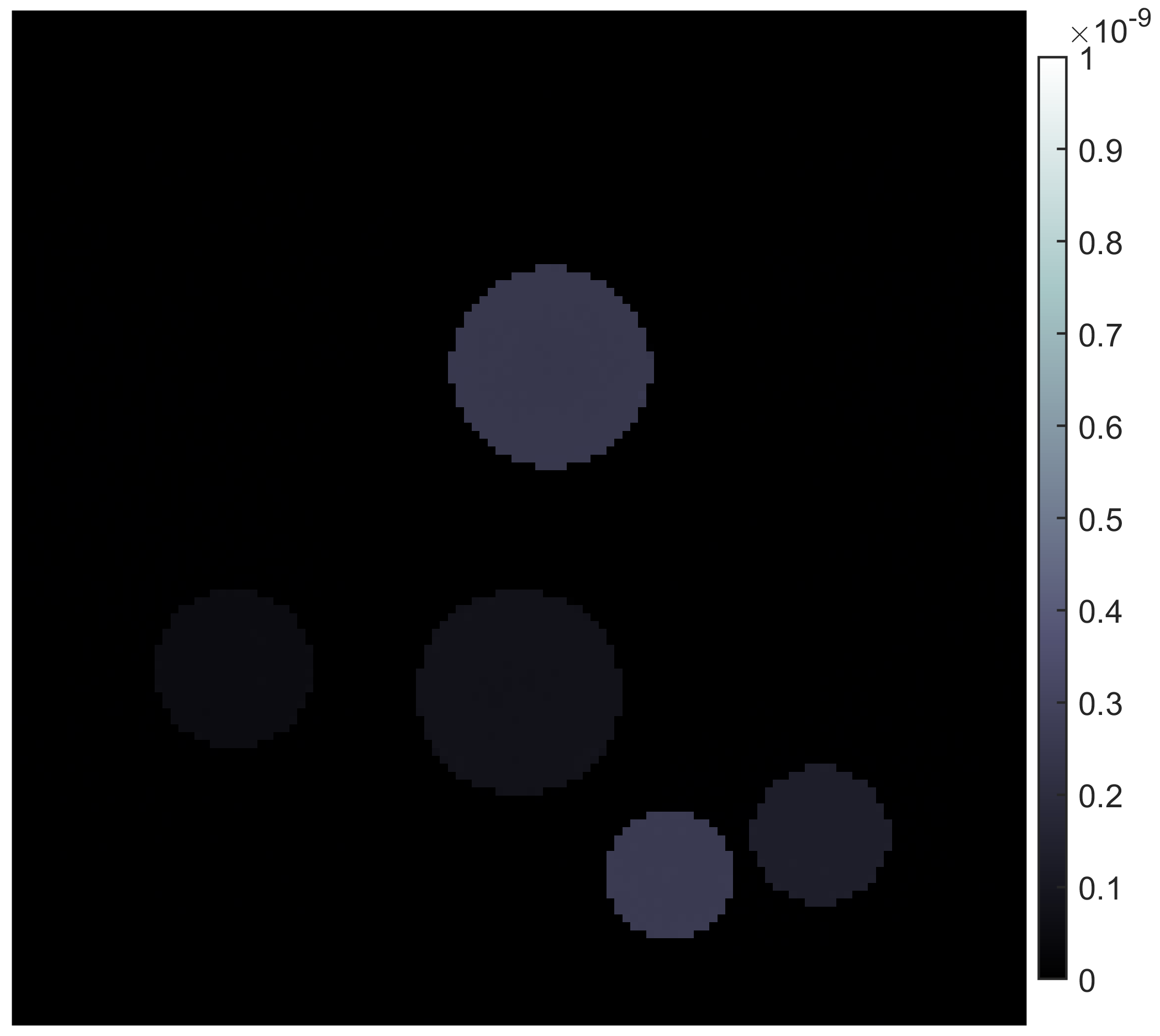} \\

       \rowlabel{$\delta/\beta=10^{4}$} &
       \includegraphics[width=0.23\textwidth]{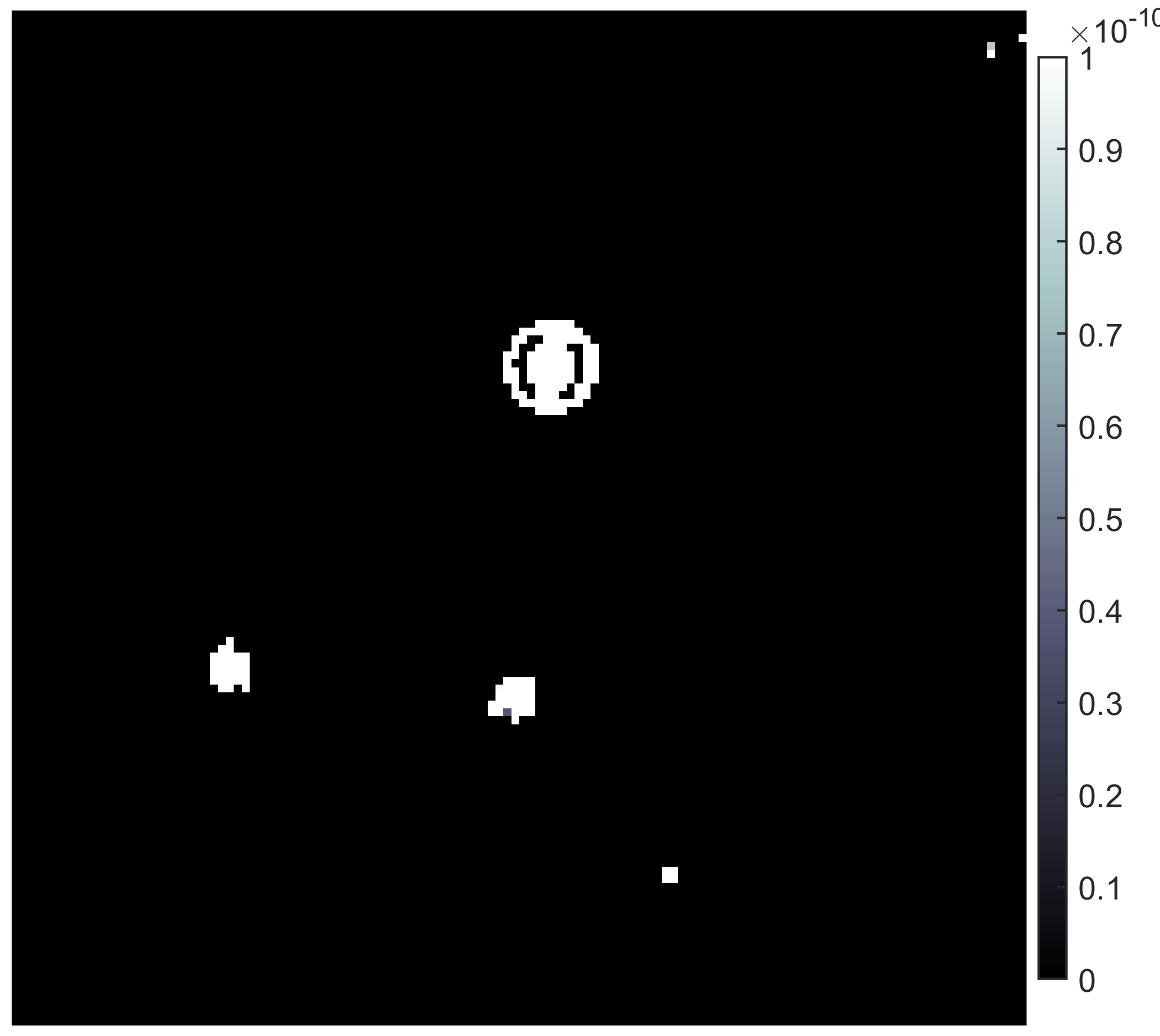} &
        \includegraphics[width=0.23\textwidth]{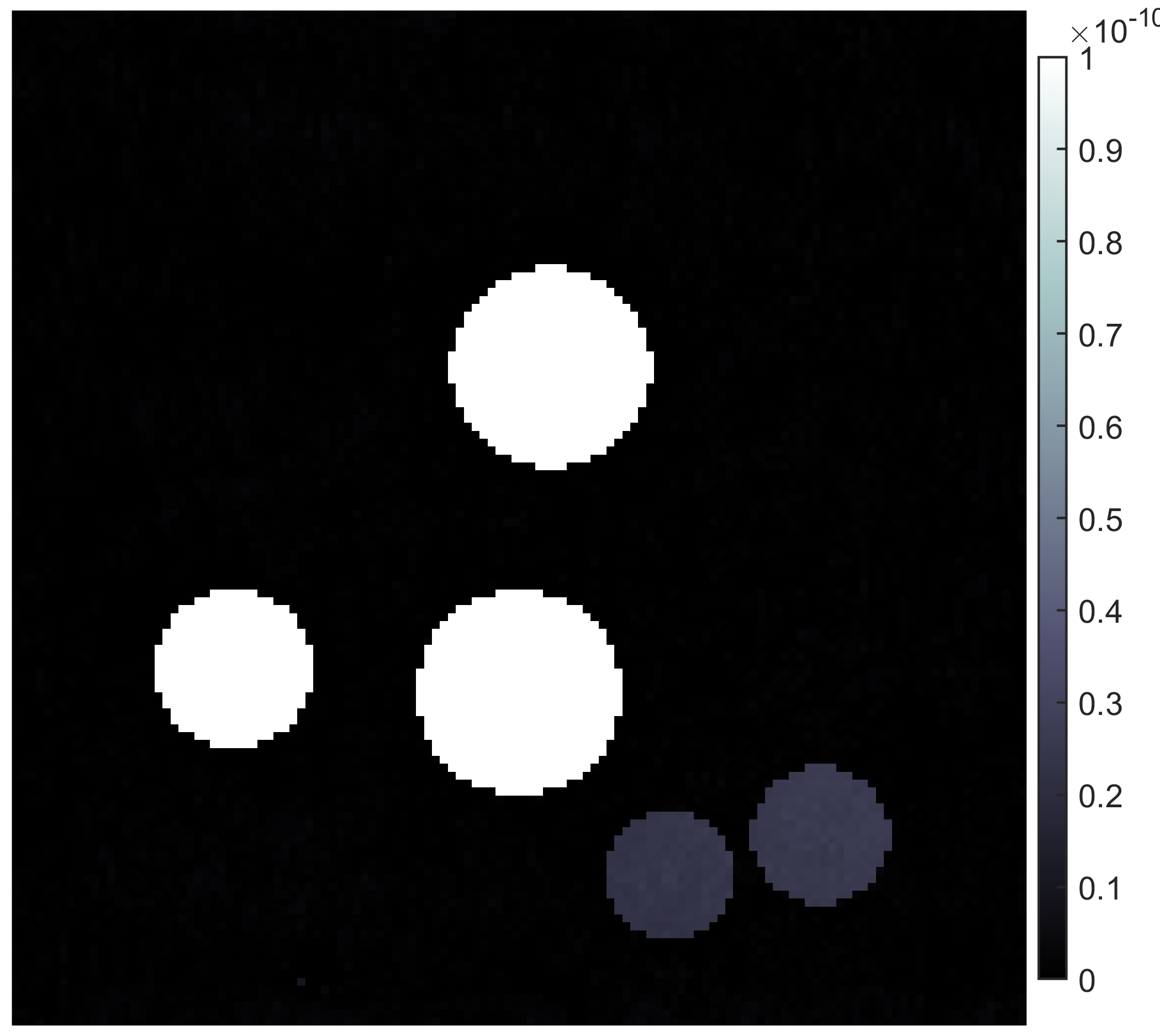} &
        \includegraphics[width=0.23\textwidth]{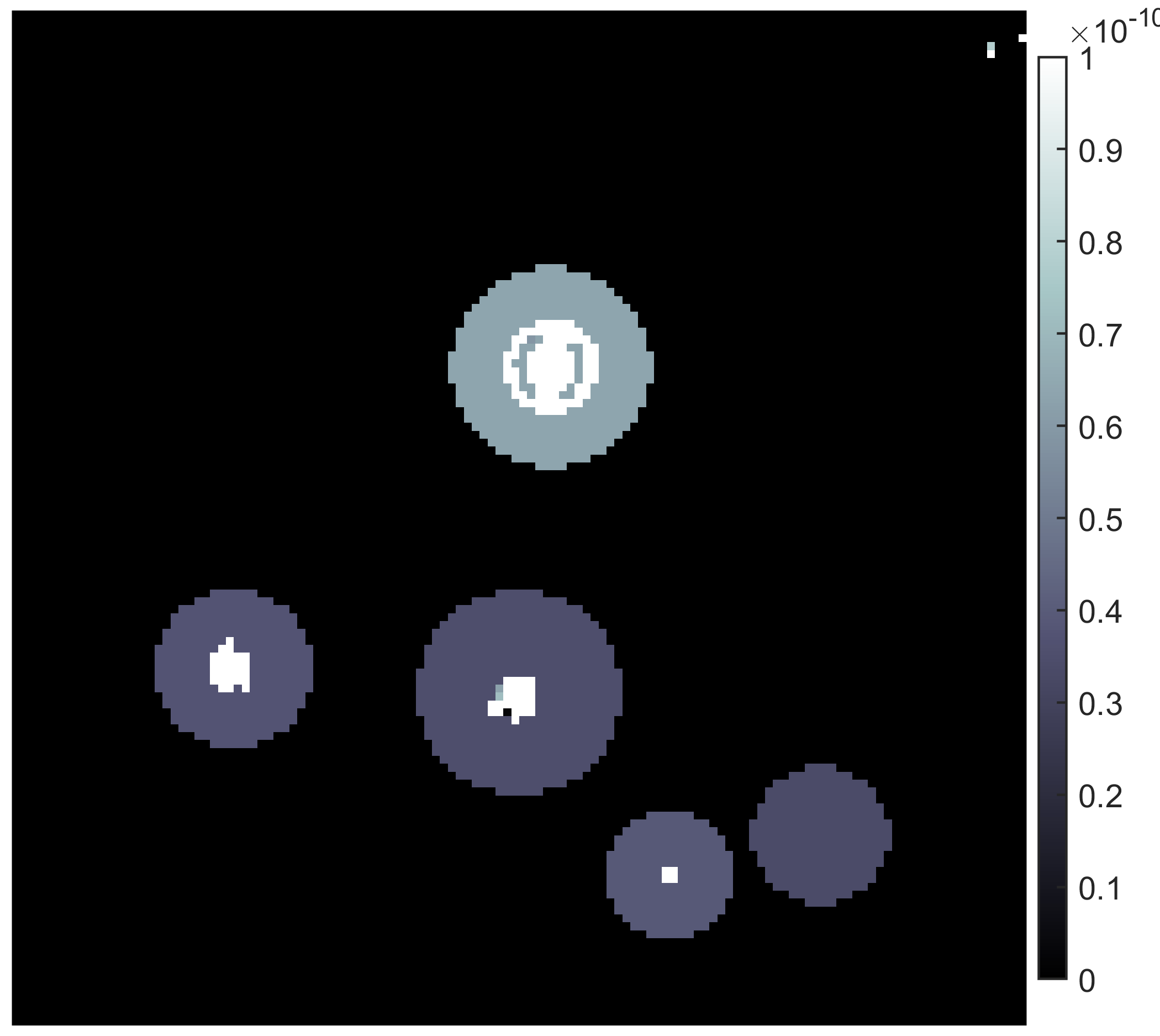} &
        \includegraphics[width=0.23\textwidth]{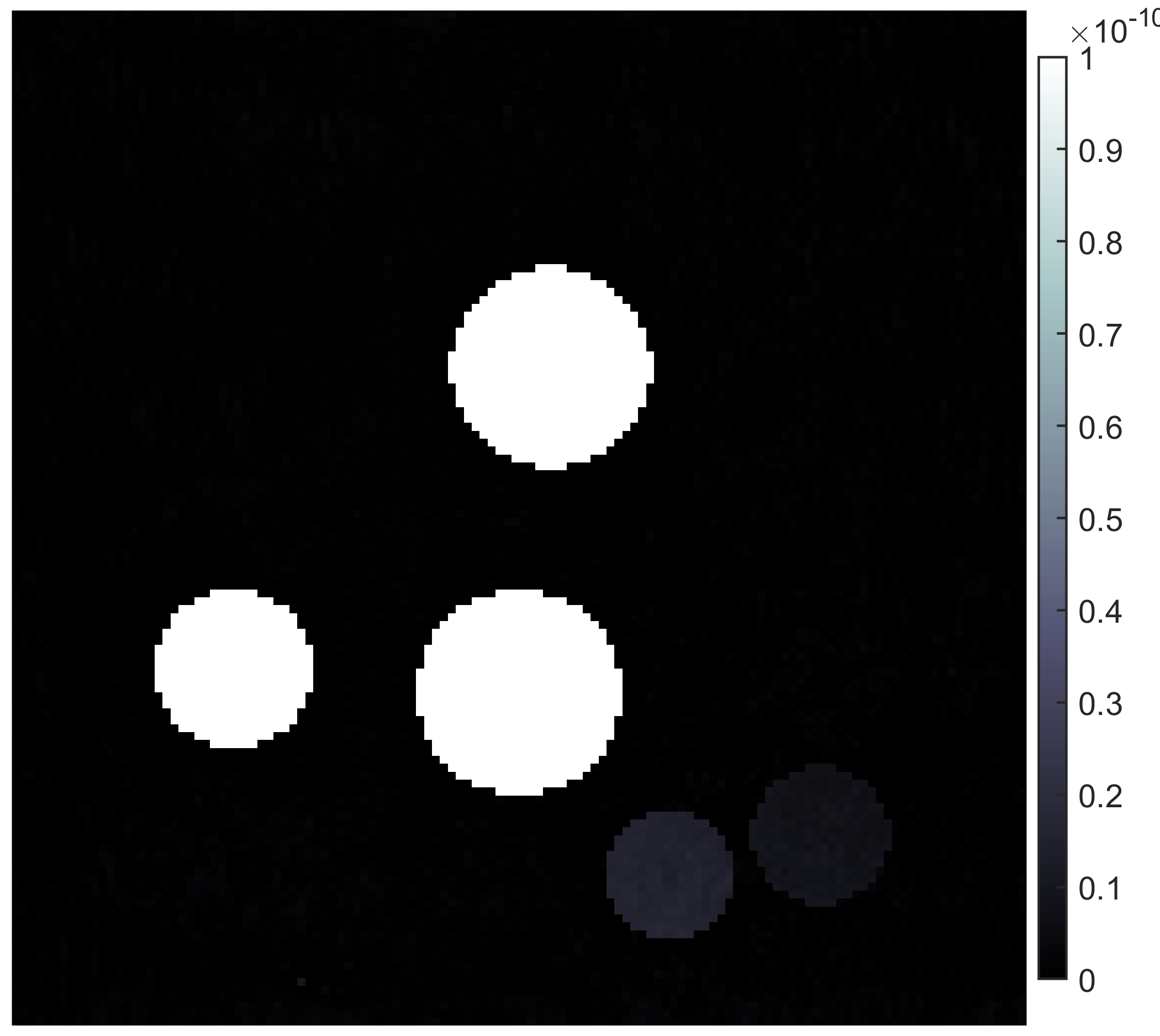} \\
        
    \end{tabular}
    }
    \caption{Reconstruction results (at slice $y=84$) under different $\delta/\beta$ ratios with $\delta=
        10^{-6}\times\,\mathrm{phantom3D}$.
    %, obtained by varying $\beta_0$ while keeping $\delta_0$ fixed. 
    % Rows correspond to different ratios.
    The first three rows show the reconstructions of $-\delta$. 
    The last three rows show the corresponding results for $\beta$.}
  \label{fig:ratio_analysis}
\end{figure}
\subsubsection{Multi-distance Reconstruction}
Multi-distance acquisition provides additional propagation diversity by recording intensity measurements at several sample-to-detector distances \cite{mokso2007nanoscale, cloetens1999holotomography, cloetens2006quantitative, nugent1996quantitative}. 
Compared with single-distance acquisition, these measurements contain complementary Fresnel diffraction information, which helps reduce reconstruction ambiguities. 
Nevertheless, despite the increased amount of measured information, many existing reconstruction methods still rely on additional material priors, often in the form of an assumed or approximately constant $\delta/\beta$ ratio, to achieve stable reconstructions.

For multi-distance reconstruction, the forward operator is evaluated at each
propagation distance $d_j$. We denote the corresponding amplitude operator by
$$
G_{d_j}(f)
=
\left|
\mathcal D_{\mathrm{Fr},d_j}
\left(
P e^{ik\widetilde{O}(f)}
\right)
\right|,
$$
where $\mathcal D_{\mathrm{Fr},d_j}$ is the Fresnel propagator associated with
the distance $d_j$. The multi-distance data-fidelity term is then defined as
$$
J(f)
=
\sum_{j=1}^{N_d}
\frac12
\left\|
G_{d_j}(f)-y_{d_j}
\right\|_2^2,
$$
where $y_{d_j}$ denotes the measured amplitude data at distance $d_j$.

In the present simulation, the same object is propagated to multiple distances using the forward model, and all corresponding intensity measurements are jointly included in the data-fidelity term. 
Physically, such measurements can be acquired by moving the sample or detector to different propagation distances. 
However, this is experimentally more demanding, since it requires accurate alignment, registration, and calibration across different distances.

Figure~\ref{fig:multidistance_intensity} shows representative forward simulated near-field
intensity measurements at a projection angle for four propagation distances.
As the propagation distance increases, the recorded
intensity patterns exhibit progressively stronger Fresnel diffraction fringes, reflecting the
transition from weak to more pronounced near-field diffraction regimes.

\begin{figure}[H]
\centering
\begin{minipage}[t]{0.23\linewidth}
\centering
\includegraphics[width=1\textwidth]{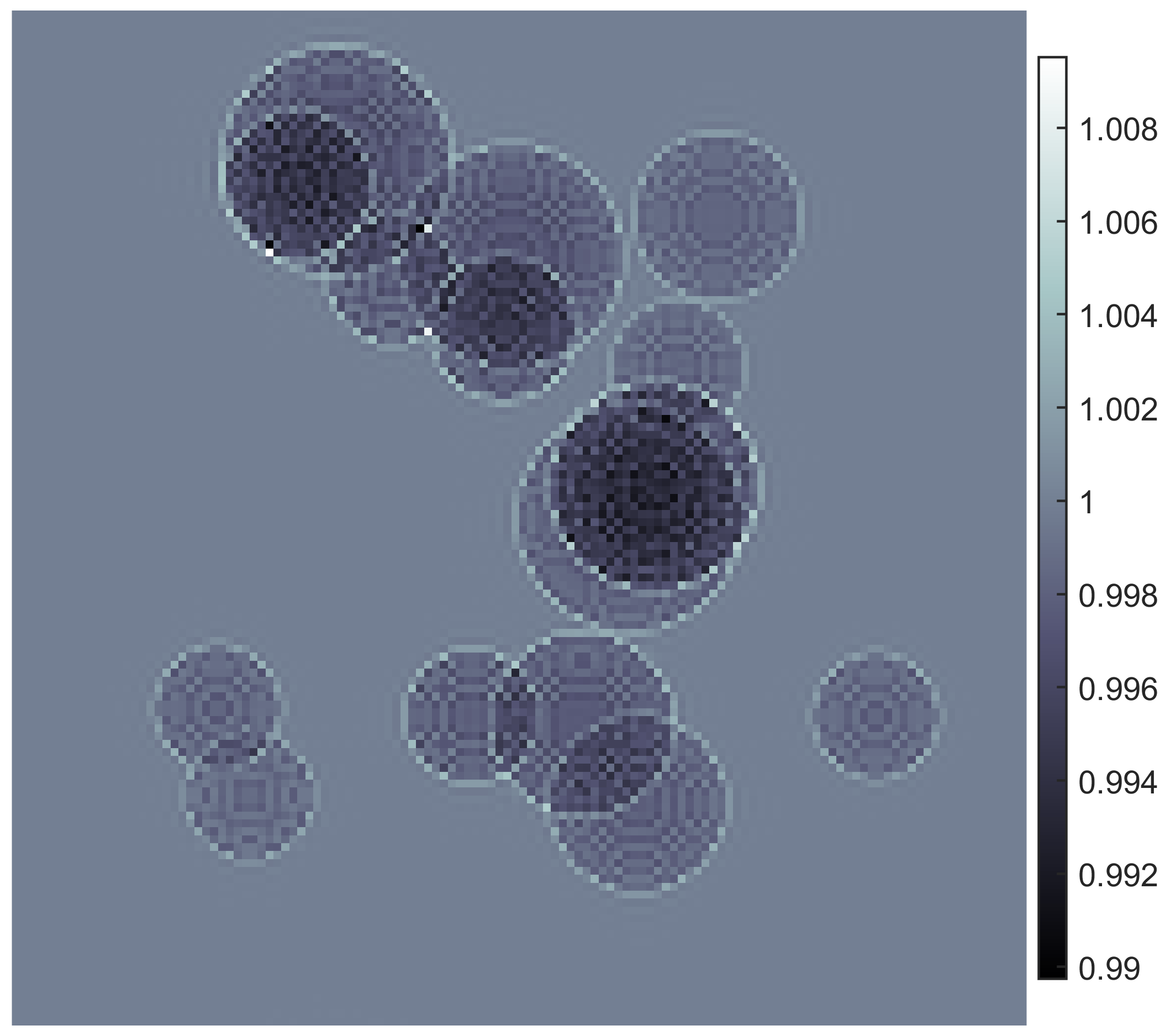}\\
% \centerline{\footnotesize\emph{gradient descent}}
\end{minipage}
\begin{minipage}[t]{0.23\linewidth}
\centering
\includegraphics[width=1\textwidth]{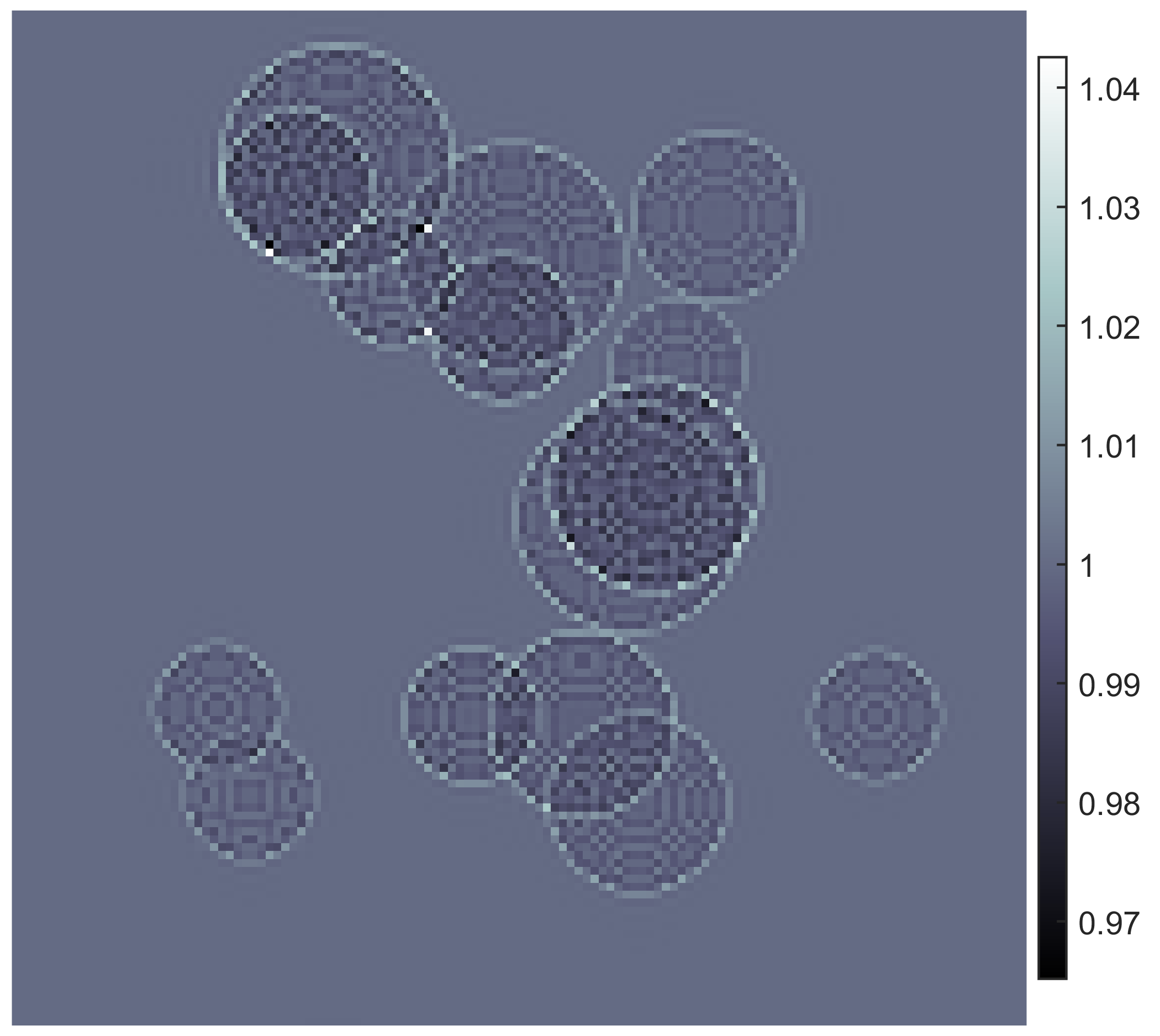}\\
% \centerline{\footnotesize\emph{gradient descent}}
\end{minipage}
\begin{minipage}[t]{0.23\linewidth}
\centering
\includegraphics[width=1\textwidth]{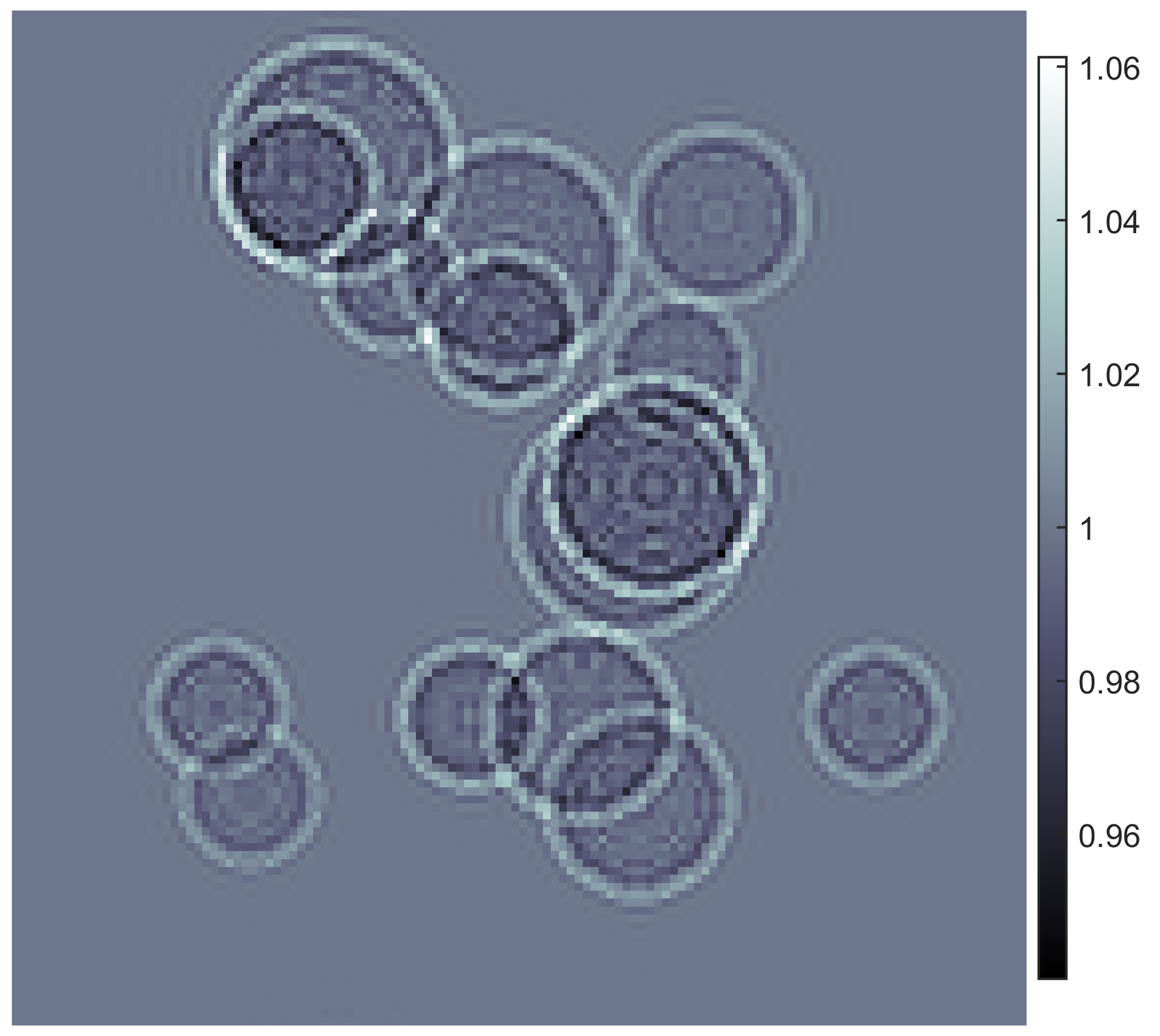}\\
% \centerline{\footnotesize\emph{Bregman TV}}
\end{minipage}
\begin{minipage}[t]{0.23\linewidth}
\centering
\includegraphics[width=1\textwidth]{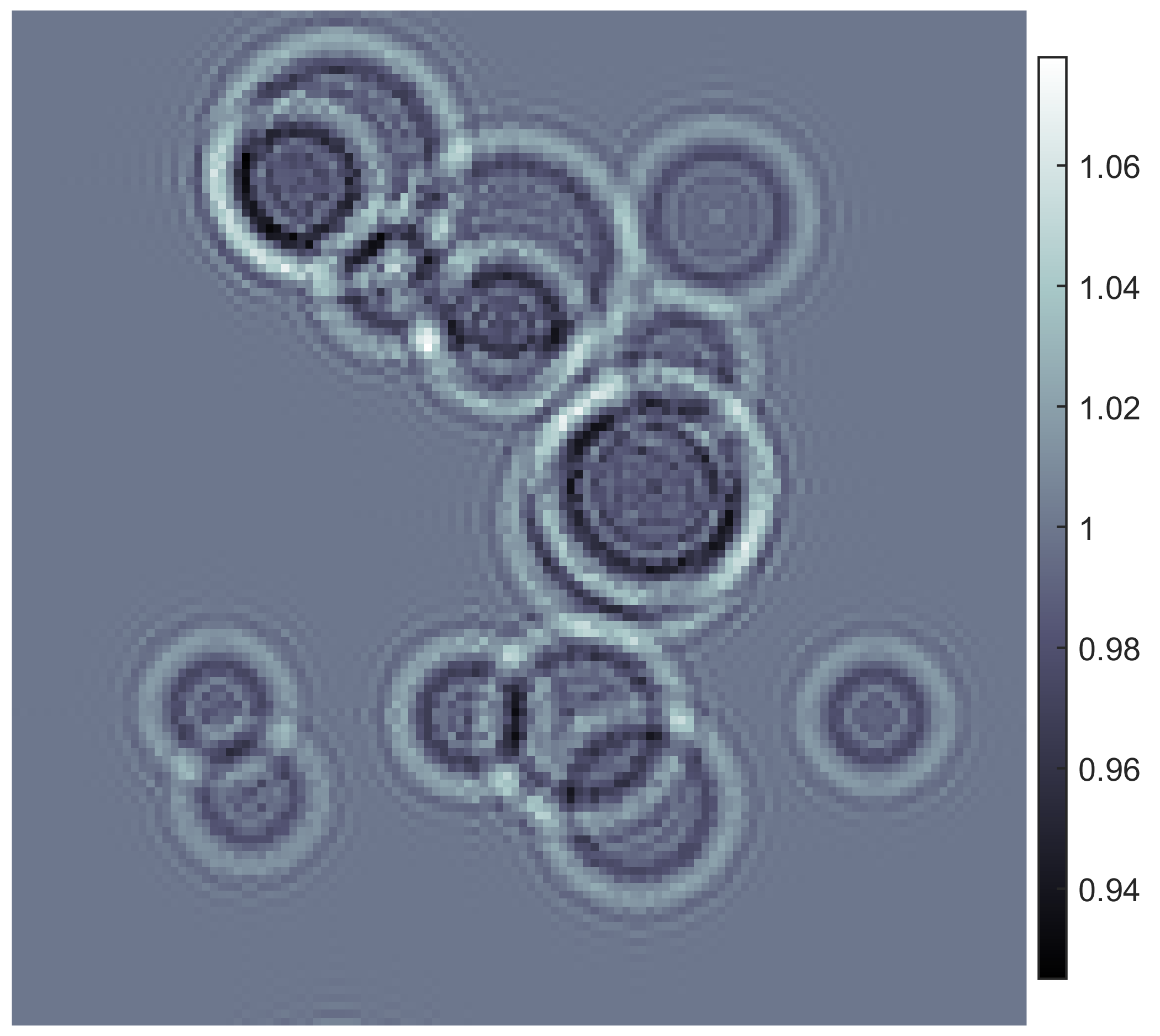}\\
% \centerline{\footnotesize\emph{Bregman TV with phase-guided}}
\end{minipage}
\caption{Forward simulated near-field intensities at a projection angle $\theta = 0$ for
multiple propagation distances.
From left to right: $d = 1.6\times10^{-4}$\,m, $8\times10^{-4}$\,m, $4\times10^{-3}$\,m, and
$1\times10^{-2}$\,m.}
 \label{fig:multidistance_intensity}
\end{figure}

Figure~\ref{fig:multi_recon} presents the reconstruction results obtained using multi-distance measurements, allowing direct comparison with the corresponding single-distance results in Figure~\ref{fig:direct3d_recon}. Table~\ref{tab:single_multi_metrics} complements this visual comparison by reporting the full-volume PSNR and global SSIM for both $\delta$ and $\beta$ under the two acquisition settings. Overall, the multi-distance reconstructions are comparable to or improved over their single-distance counterparts, with the direct 3D methods showing the clearest improvements. Minor variations in individual quantitative metrics are small and do not affect the overall trend. Importantly, the conclusions obtained in the single-distance experiments remain consistent in the multi-distance setting. Bregman TV regularization improves reconstruction stability, whereas the proposed phase-guided Bregman TV method achieves the best overall reconstruction quality.

\begin{figure}[H]
    \centering
    \setlength{\tabcolsep}{1pt}

    \begin{tabular}{c c c c c}
        & \footnotesize Two-step FBP
        & \footnotesize GD 
        & \footnotesize Bregman TV 
        & \footnotesize Phase-guided \\
        
        \rowlabel{$-\delta_{reco}$} &
        \includegraphics[width=0.23\textwidth]{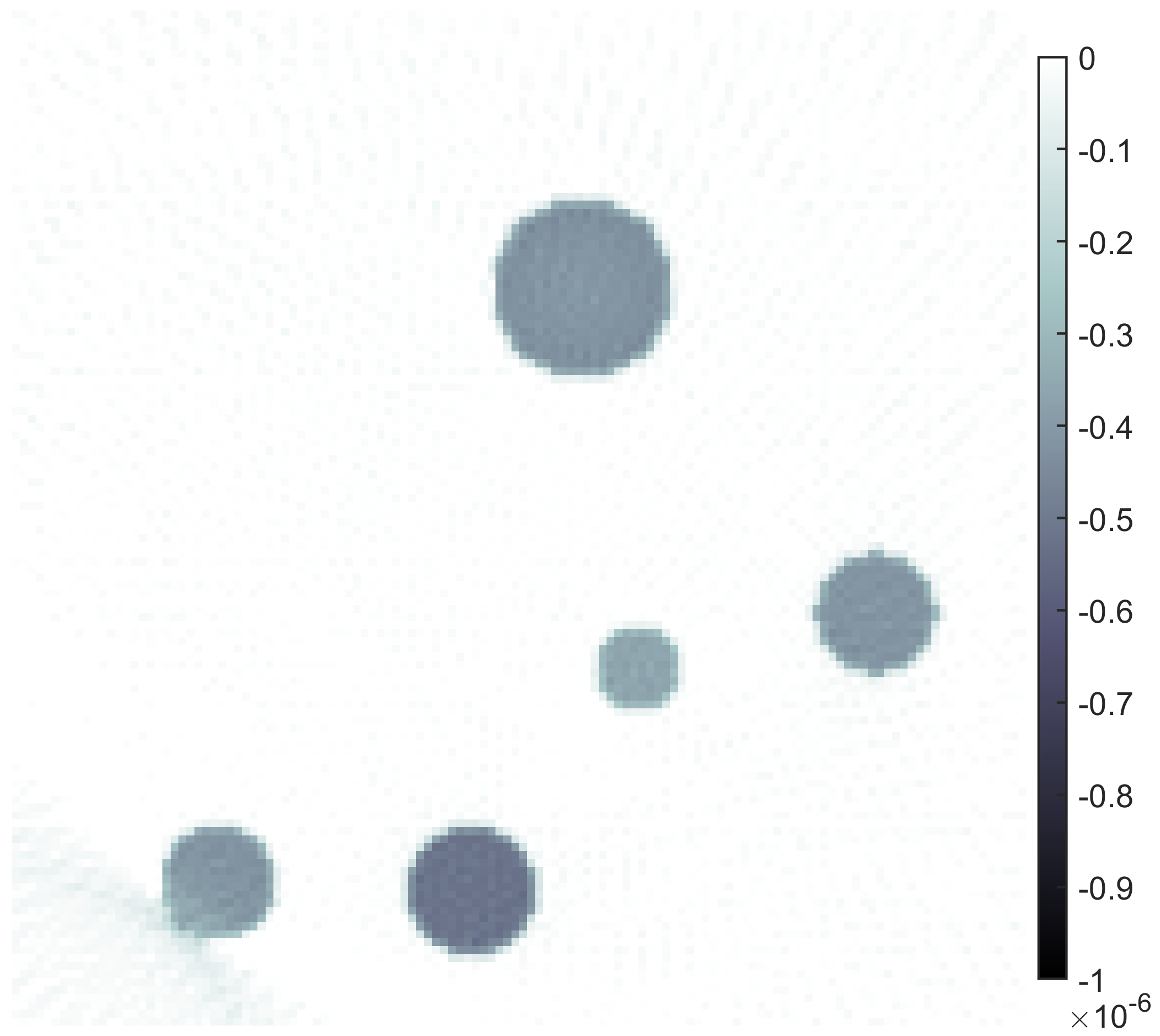} &
        \includegraphics[width=0.23\textwidth]{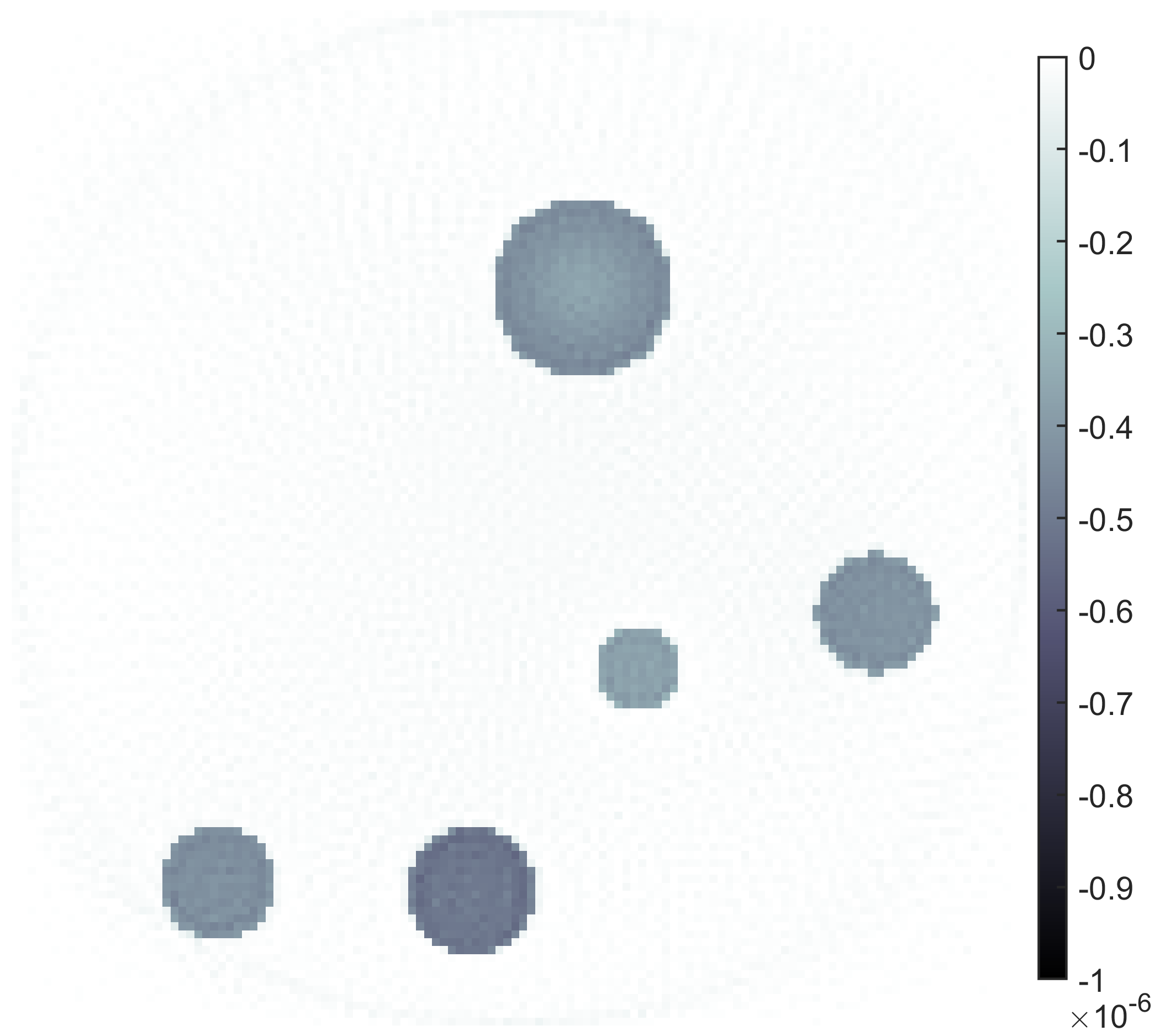} &
        \includegraphics[width=0.23\textwidth]{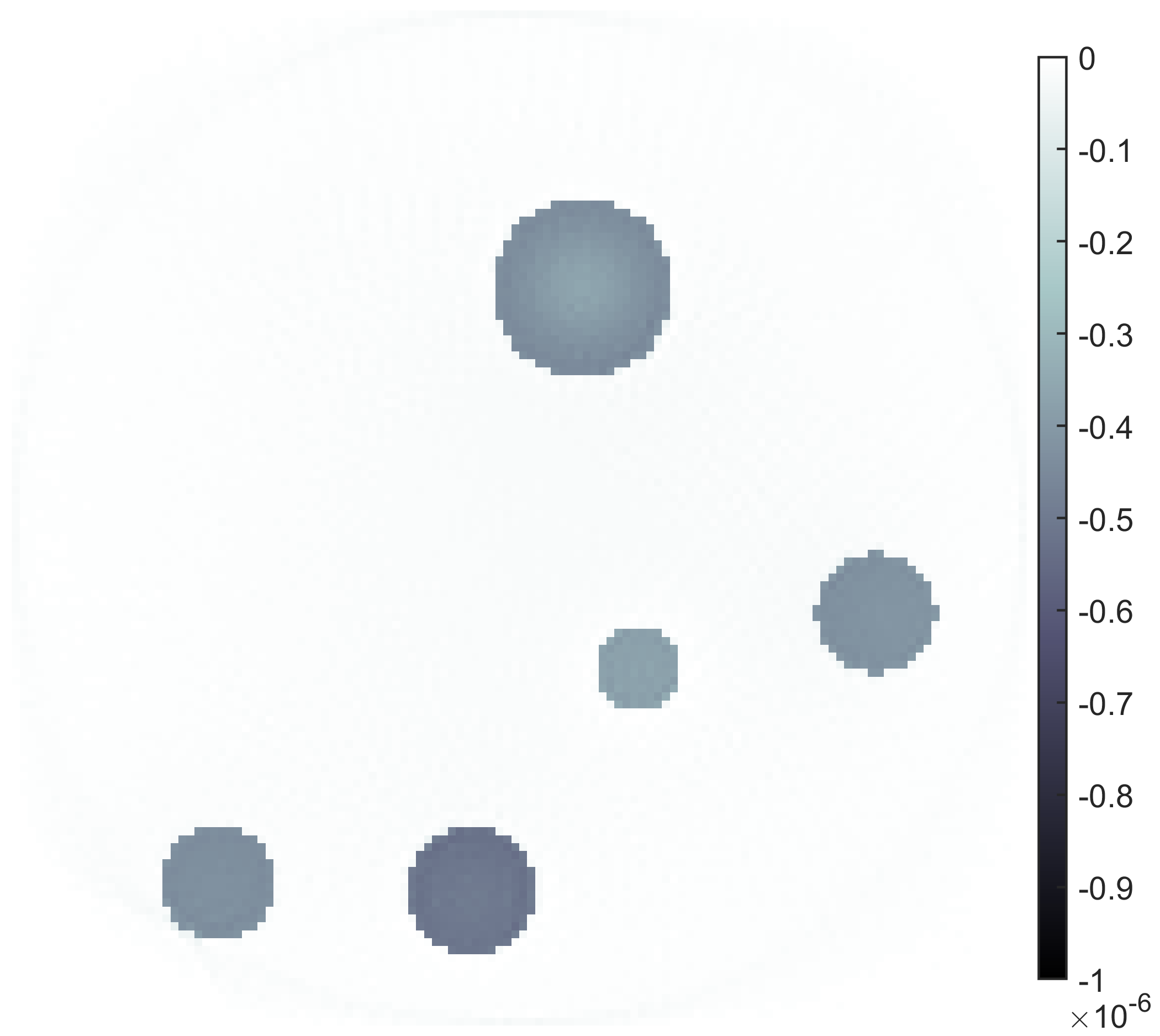} &
        \includegraphics[width=0.23\textwidth]{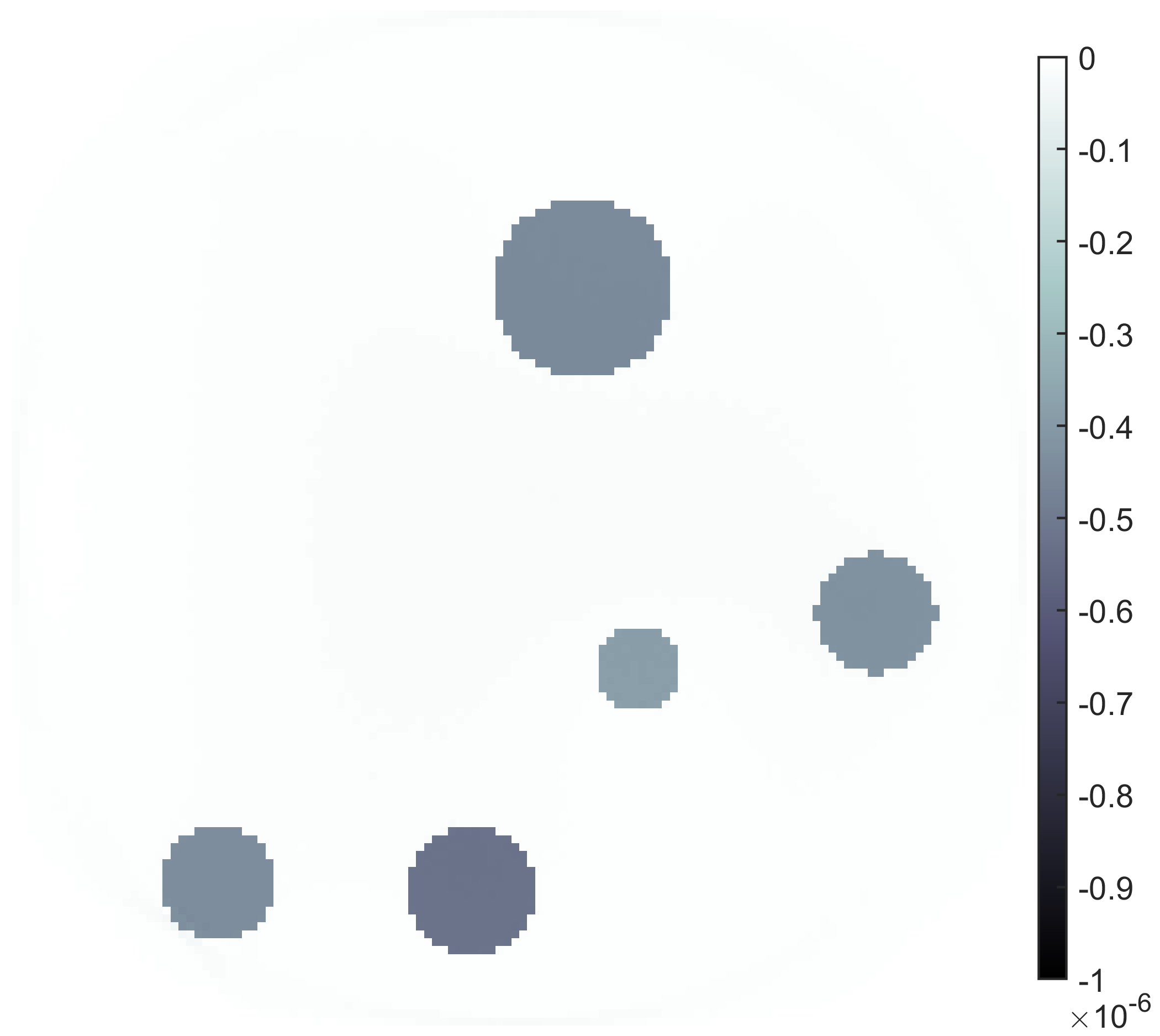} \\

        \rowlabel{diff $\delta$} &
         \includegraphics[width=0.23\textwidth]{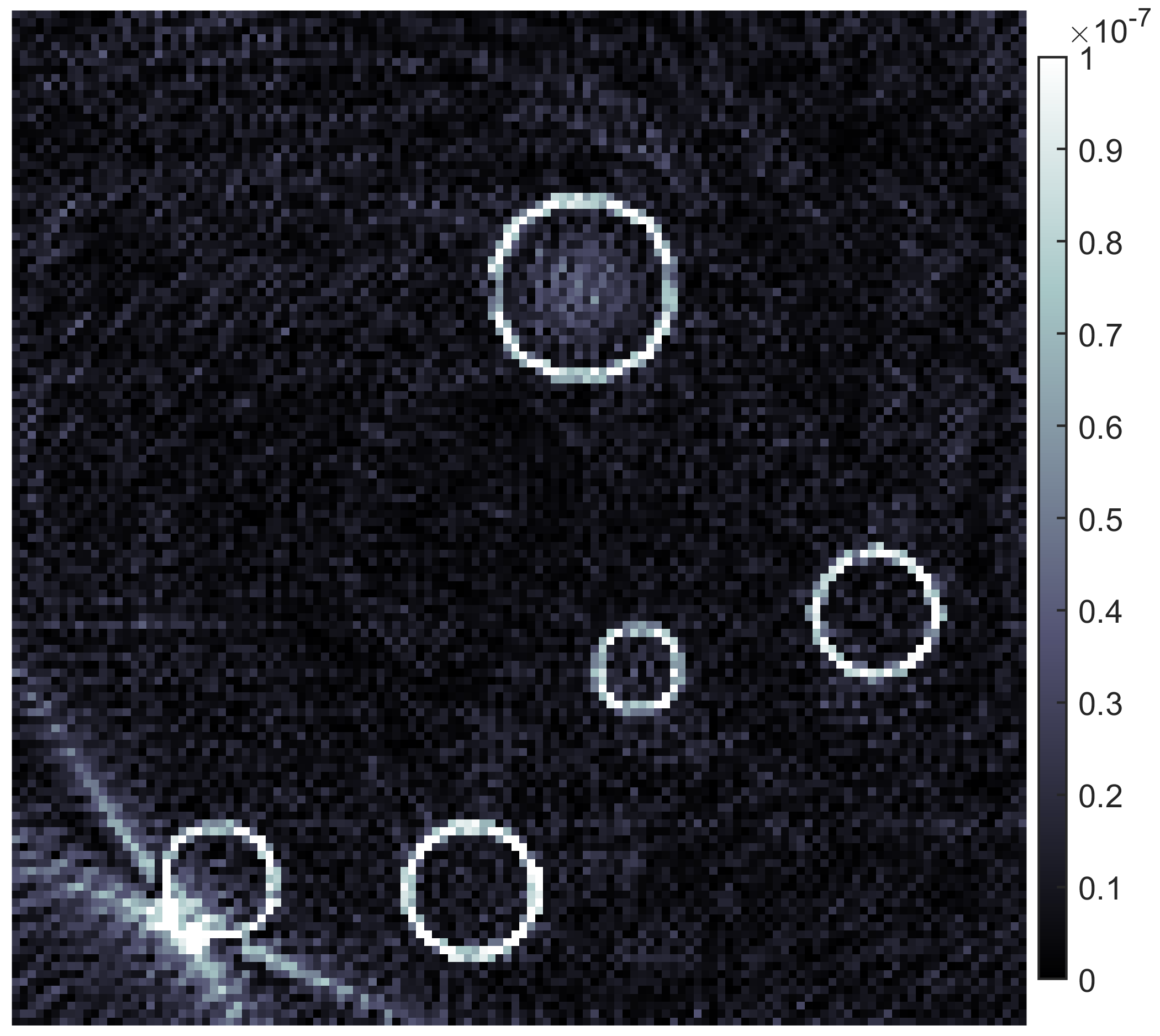} &
        \includegraphics[width=0.23\textwidth]{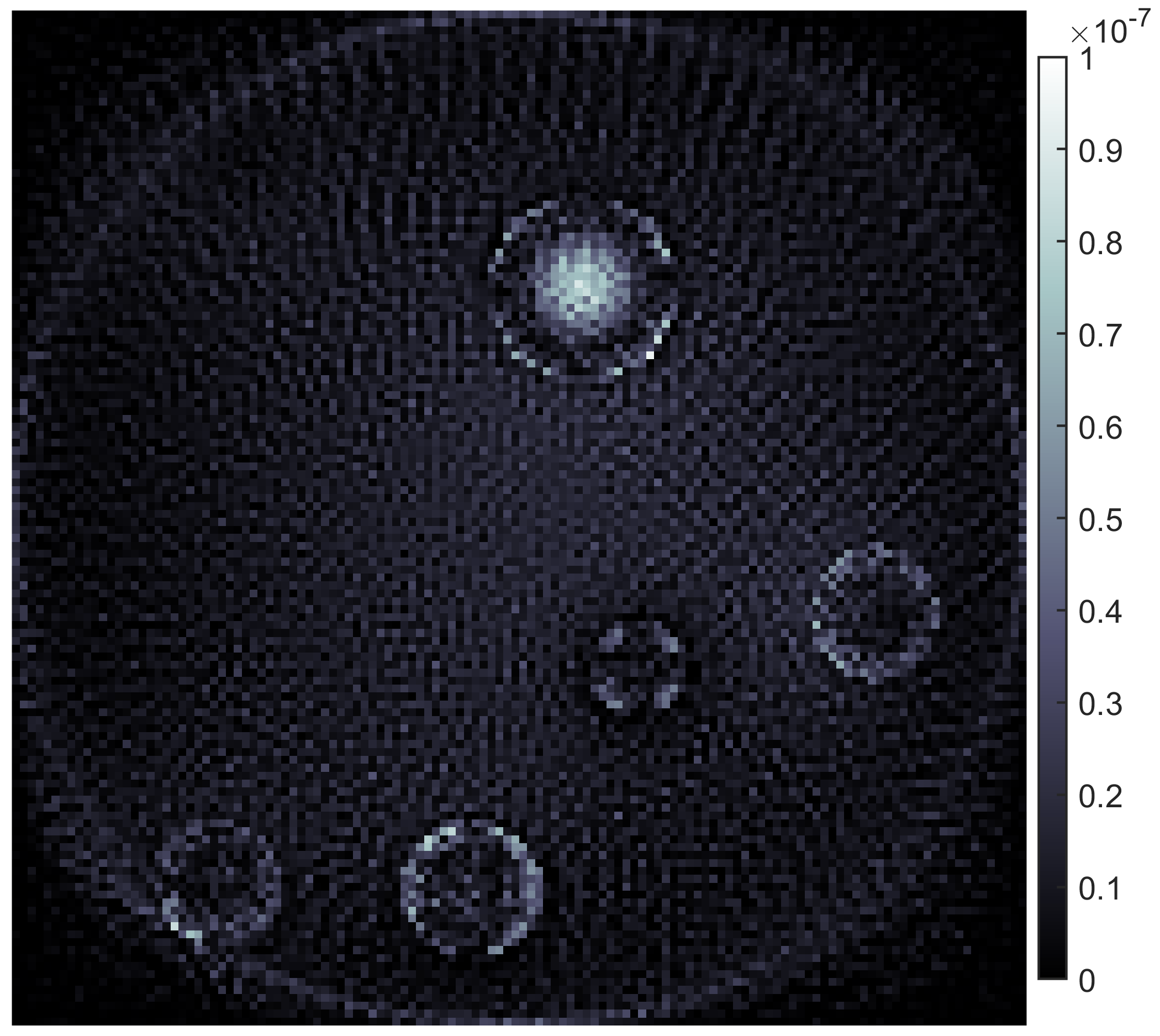} &
        \includegraphics[width=0.23\textwidth]{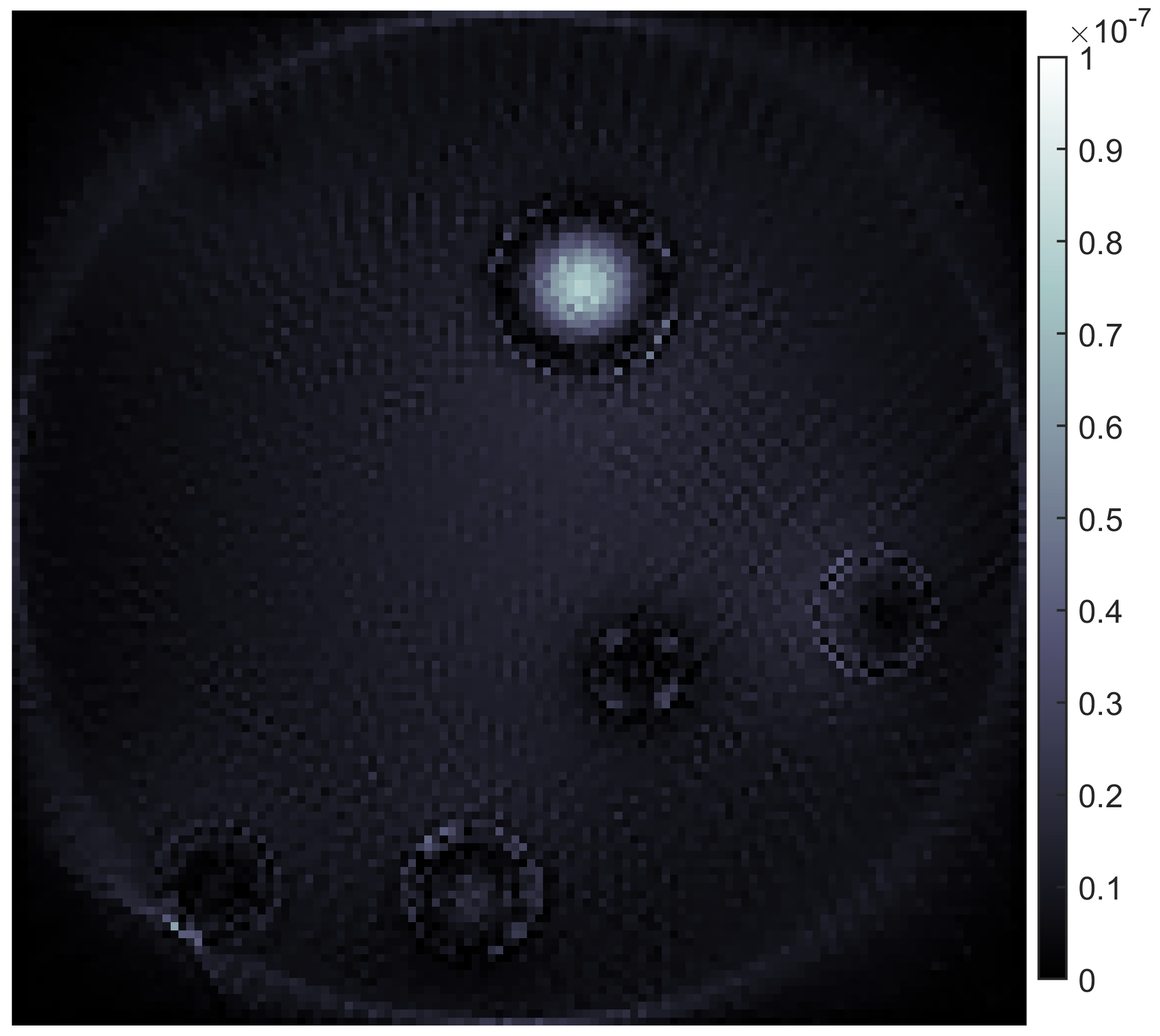} &
        \includegraphics[width=0.23\textwidth]{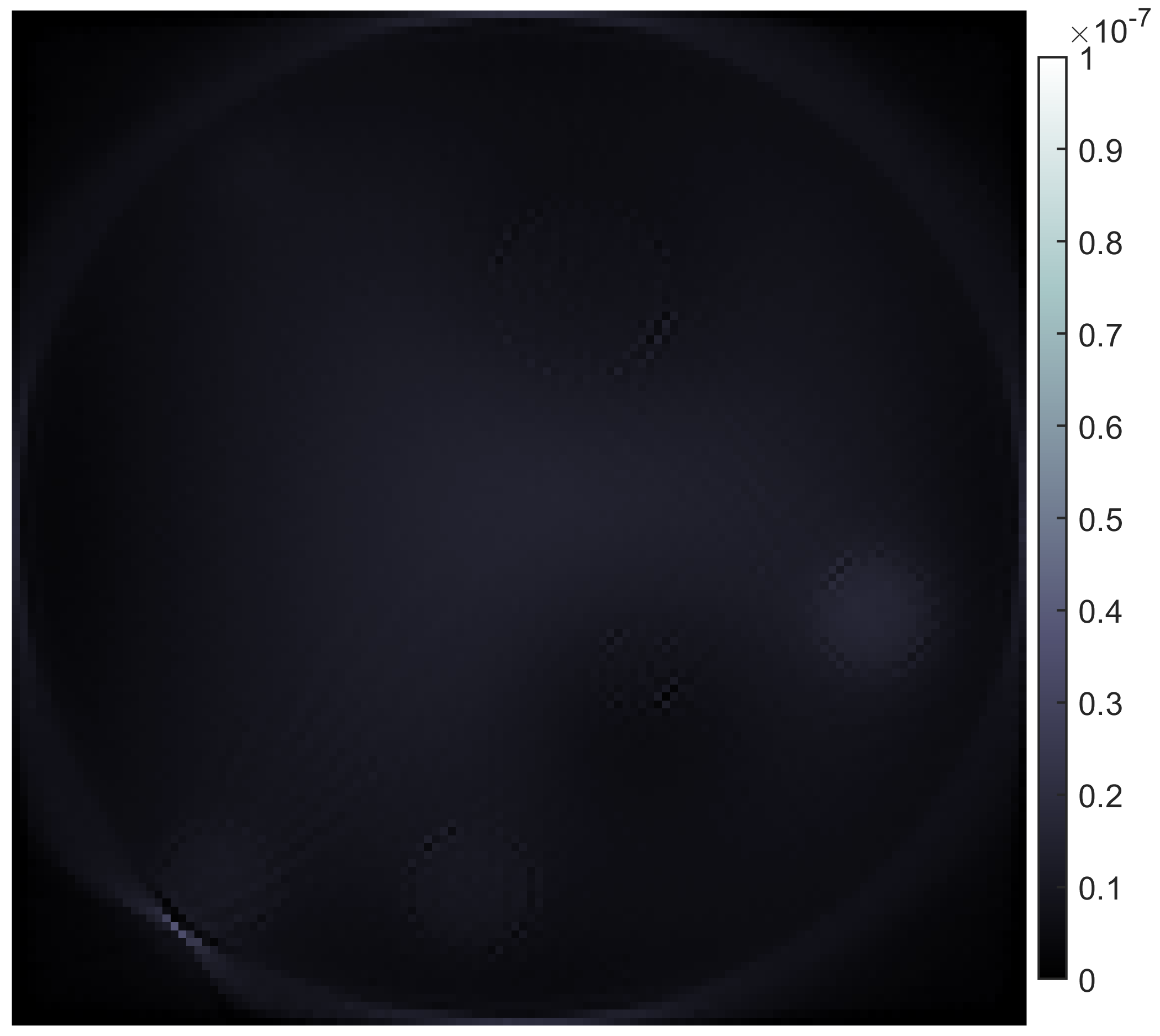} \\

        \rowlabel{$\beta_{reco}$} &
        \includegraphics[width=0.23\textwidth]{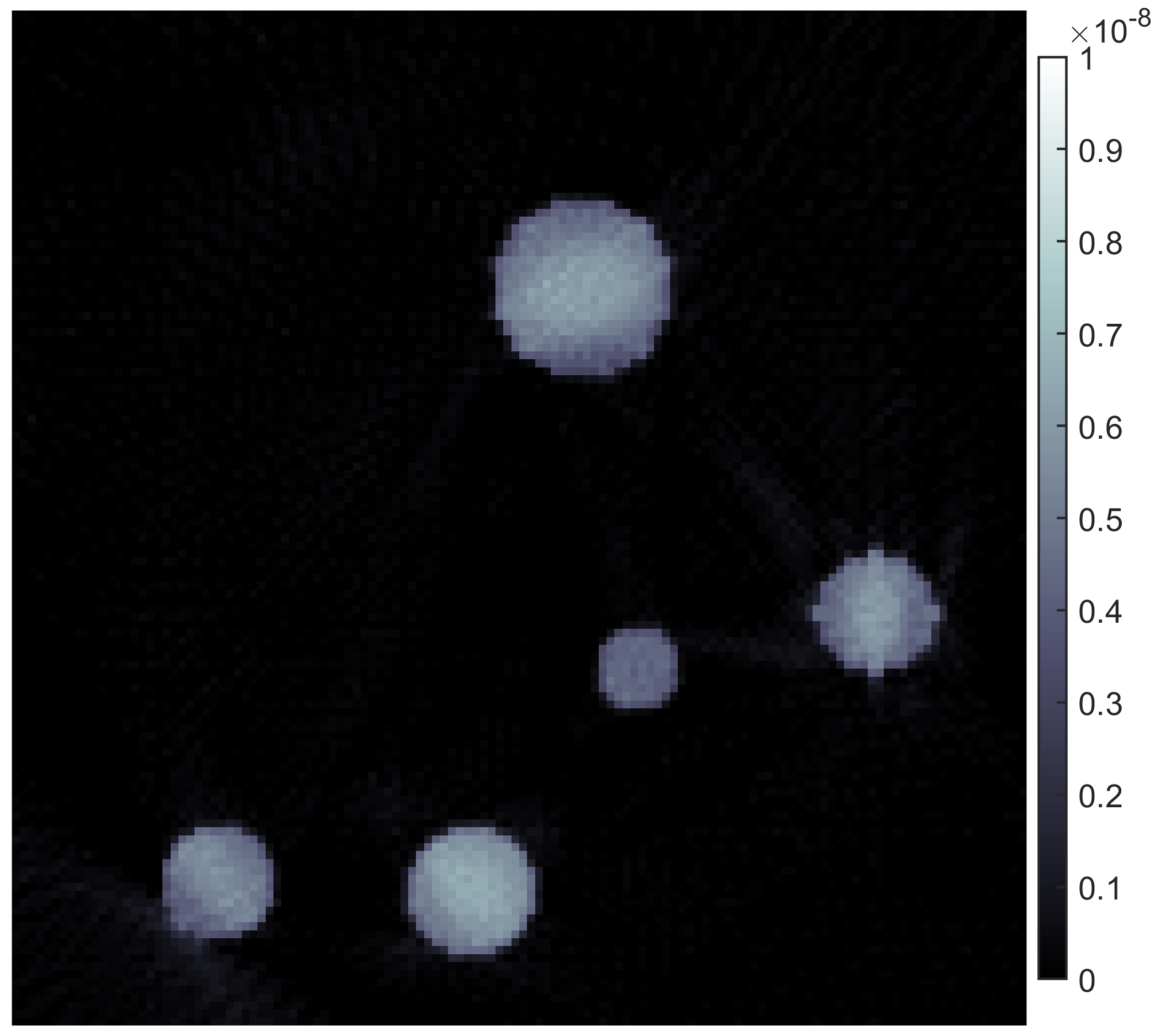} &
        \includegraphics[width=0.23\textwidth]{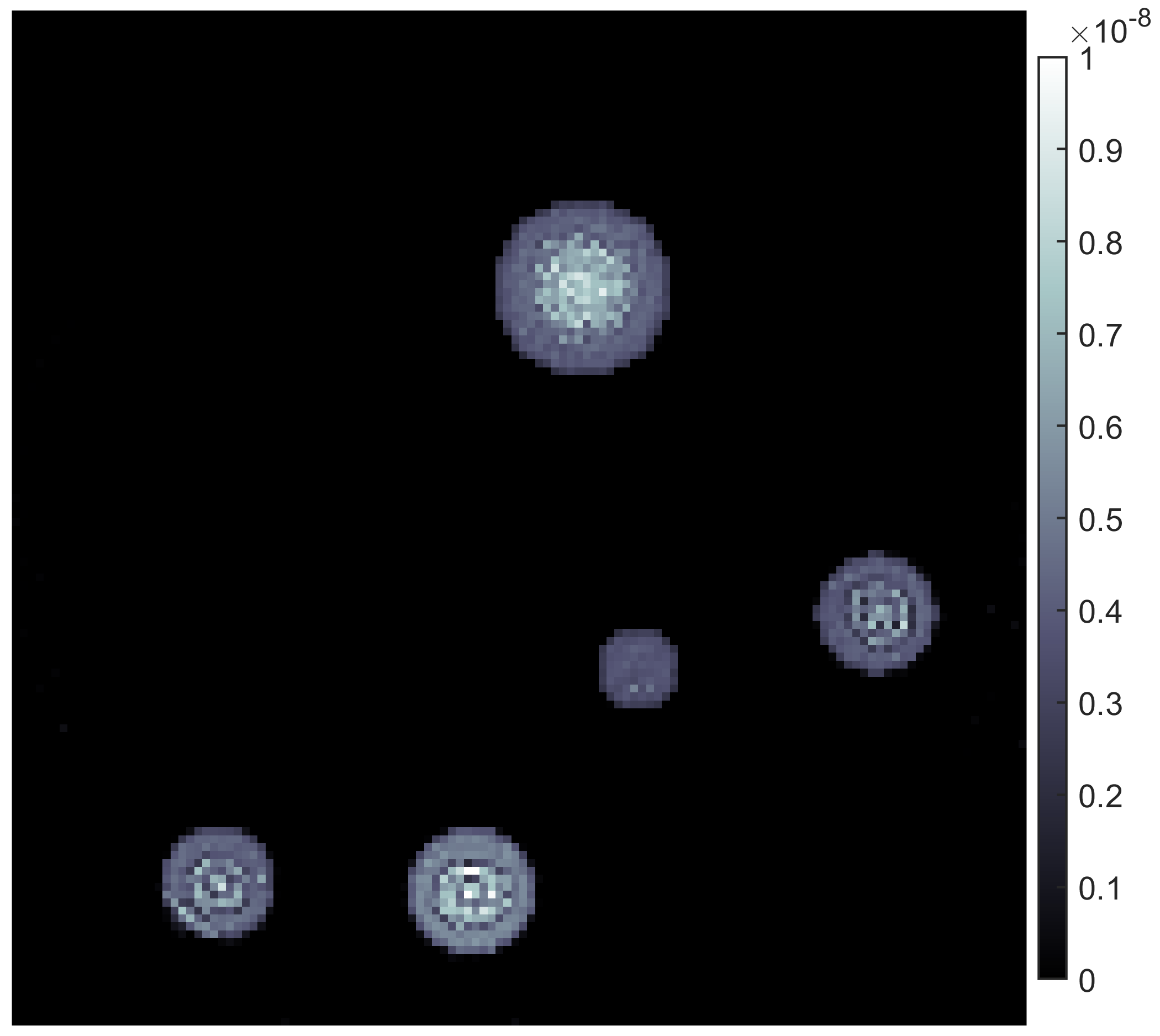} &
        \includegraphics[width=0.23\textwidth]{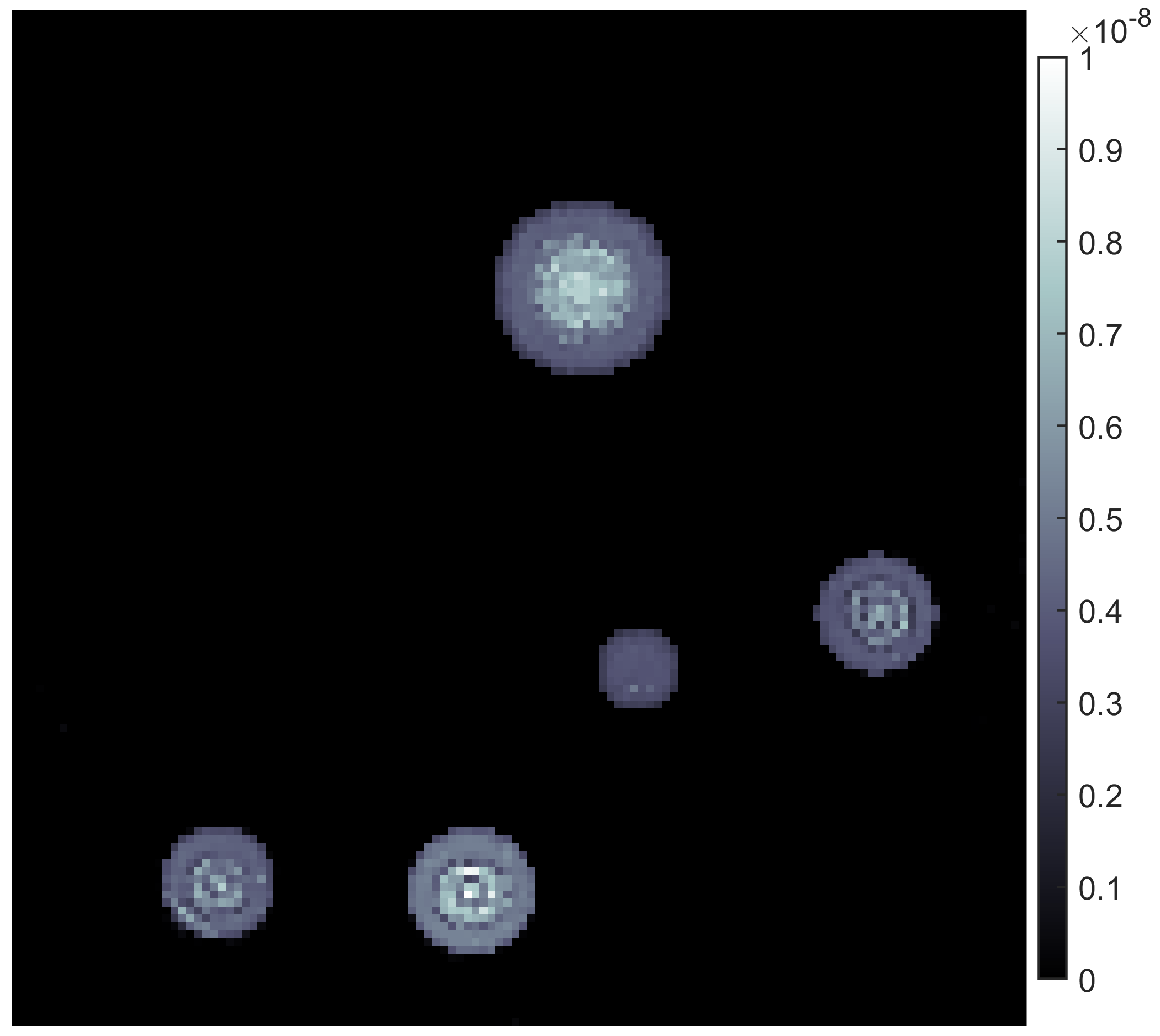} &
        \includegraphics[width=0.23\textwidth]{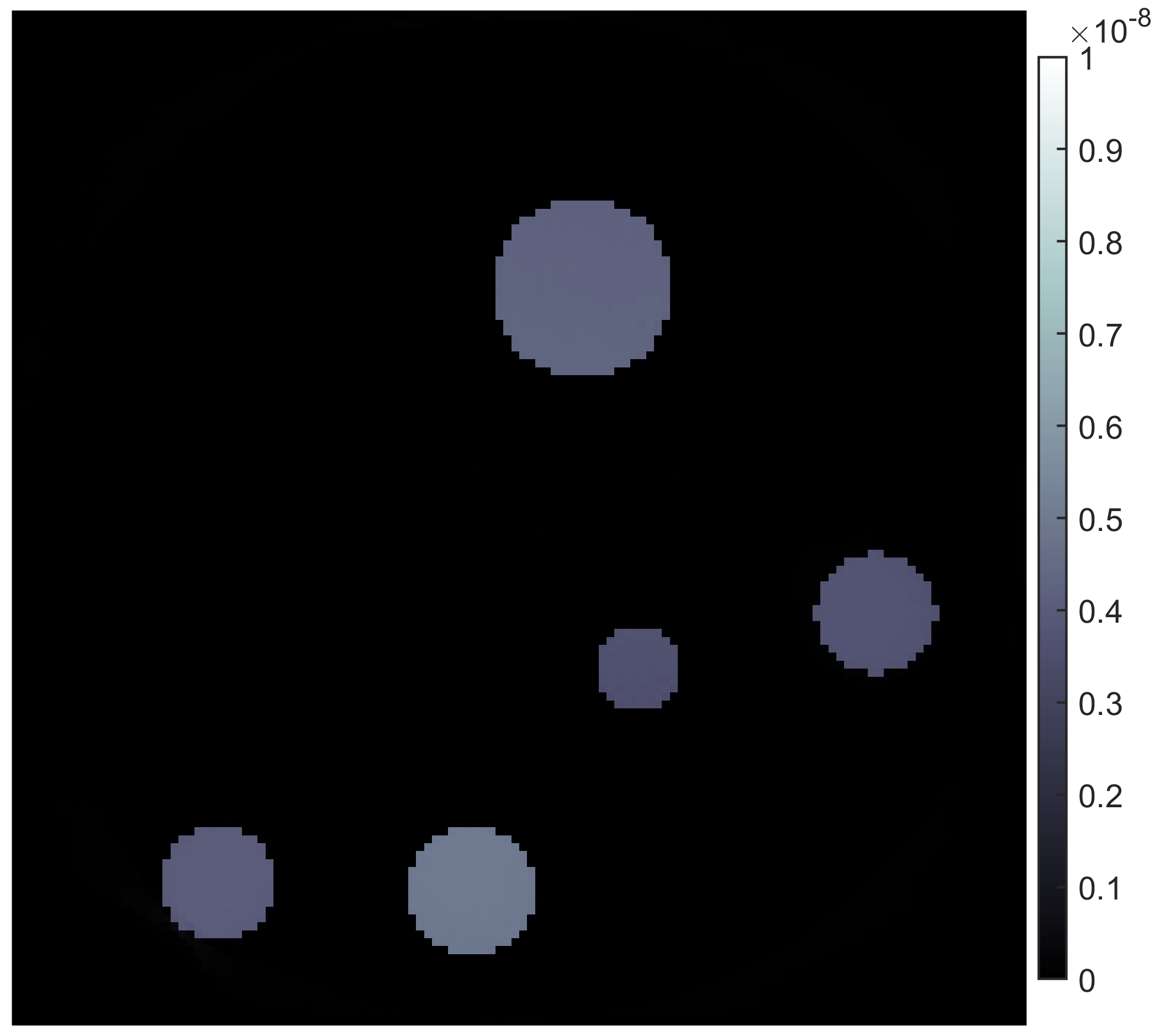} \\

        \rowlabel{diff $\beta$} &
        \includegraphics[width=0.23\textwidth]{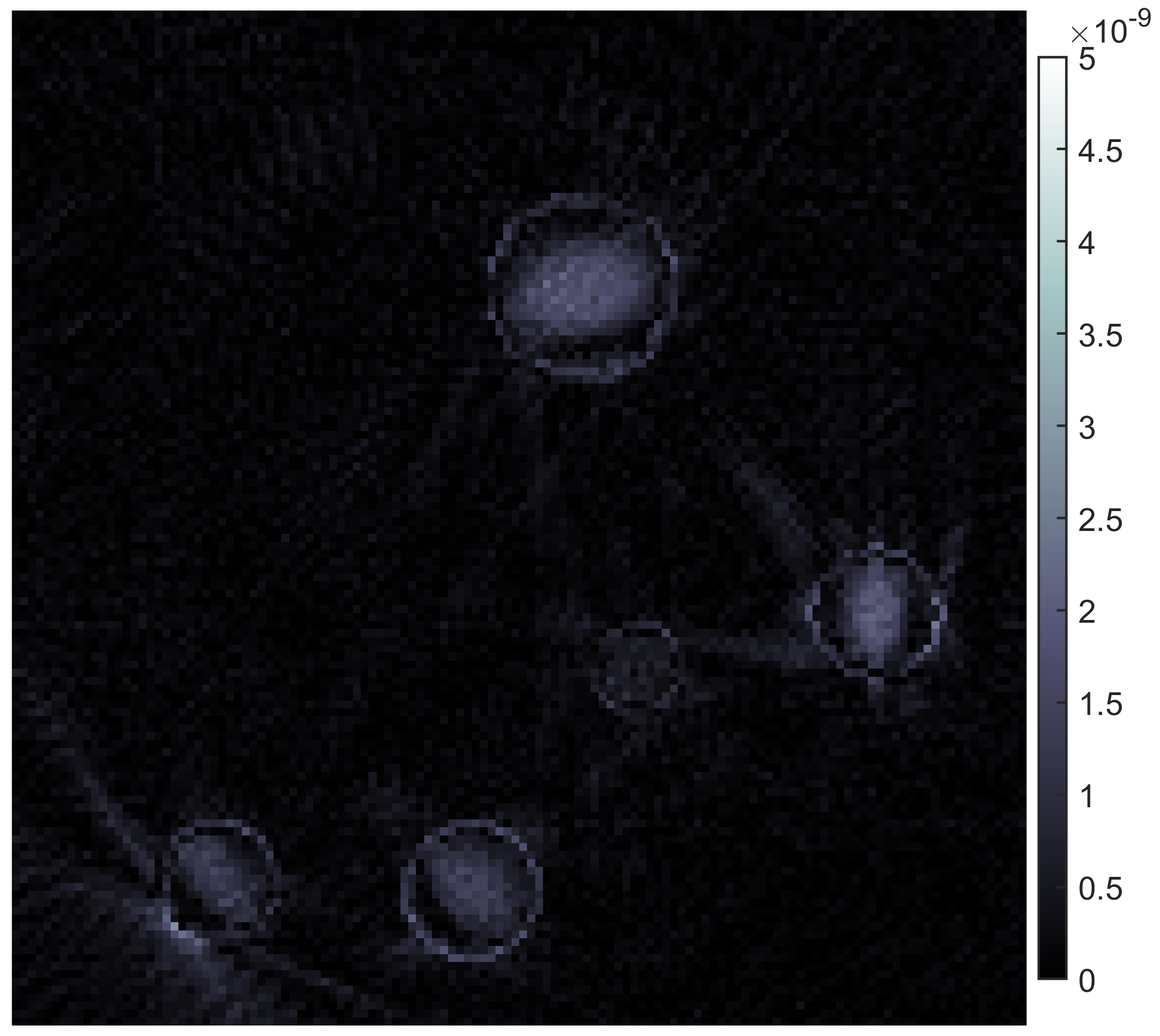} &
        \includegraphics[width=0.23\textwidth]{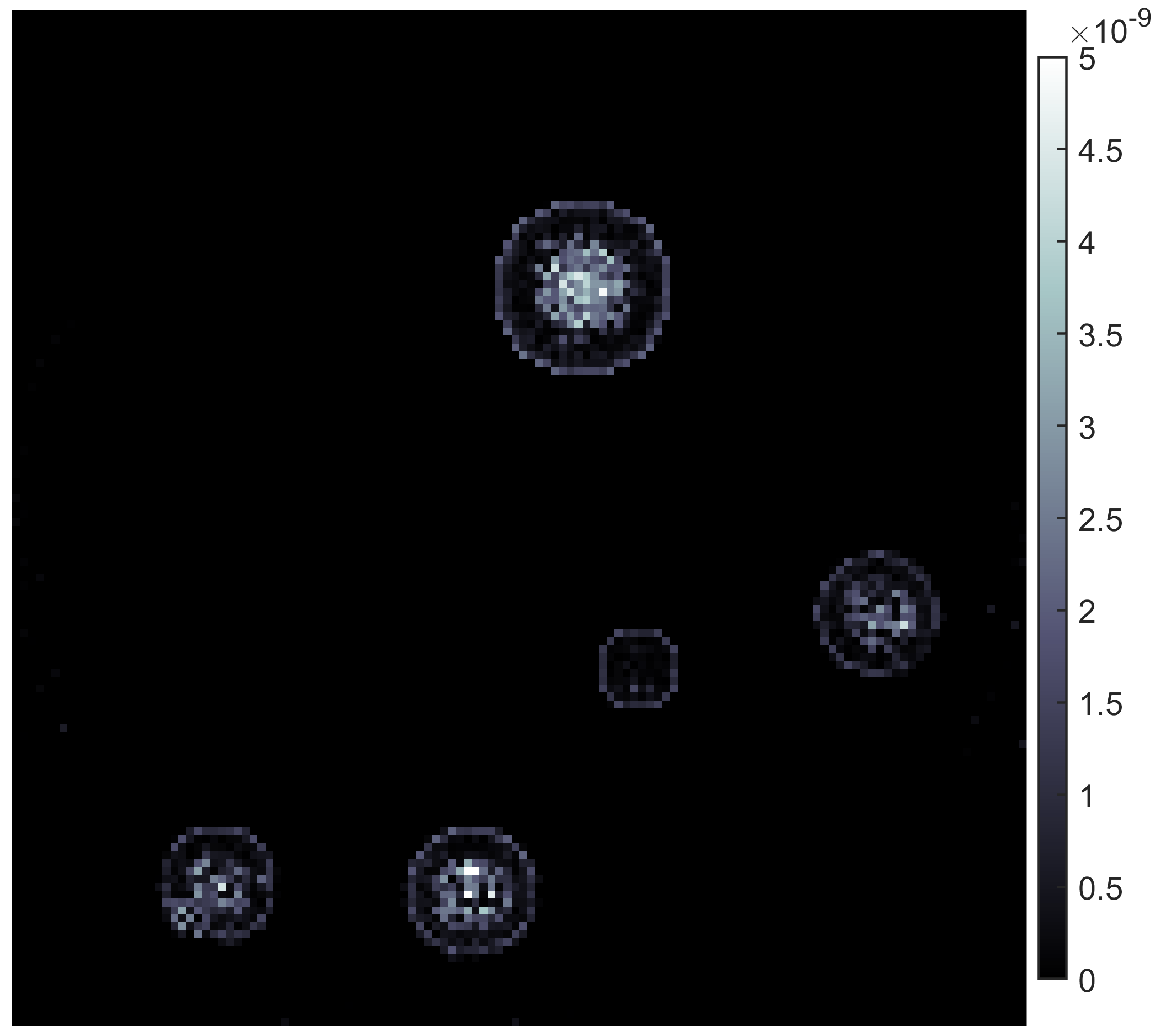} &
        \includegraphics[width=0.23\textwidth]{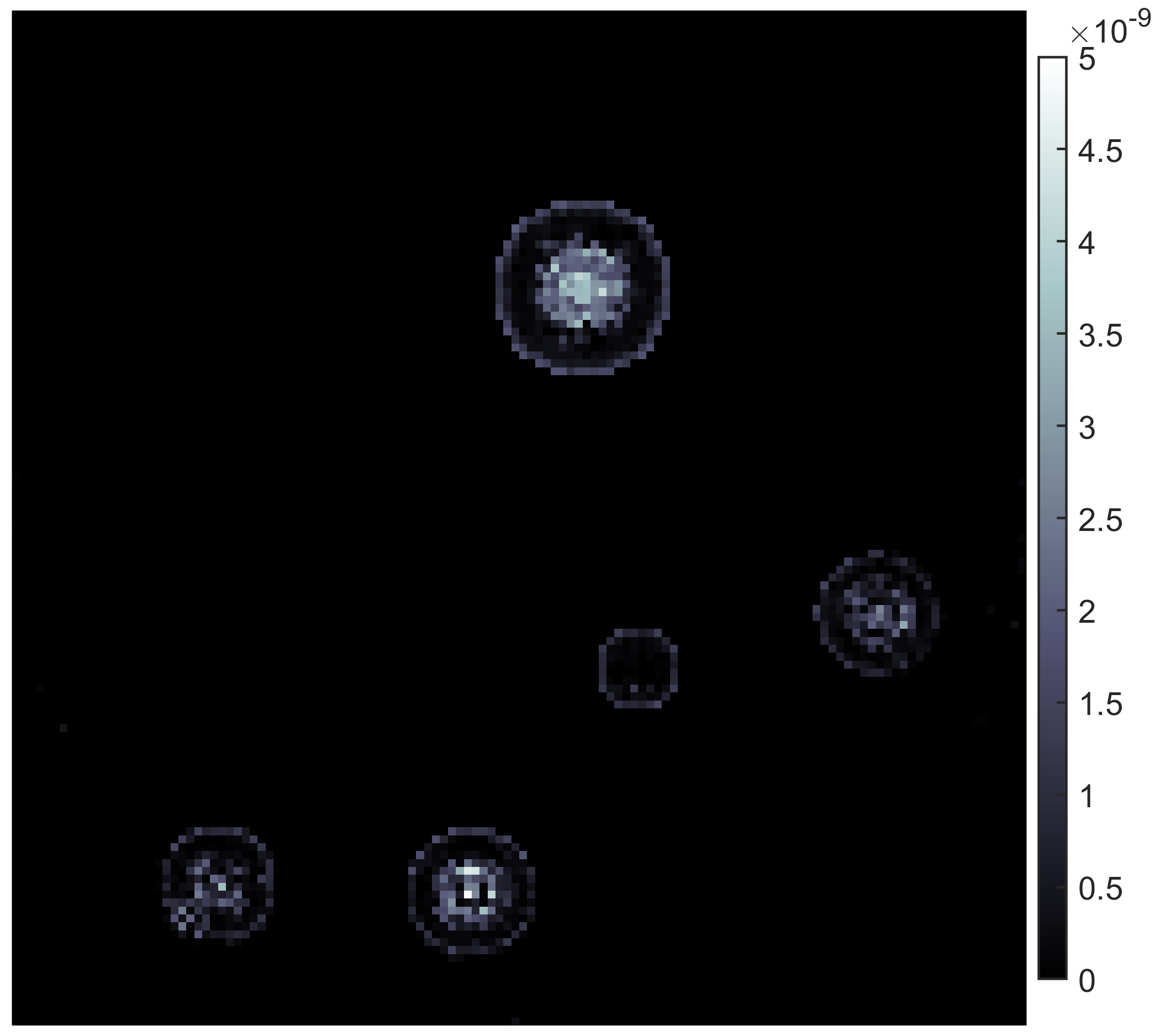} &
        \includegraphics[width=0.23\textwidth]{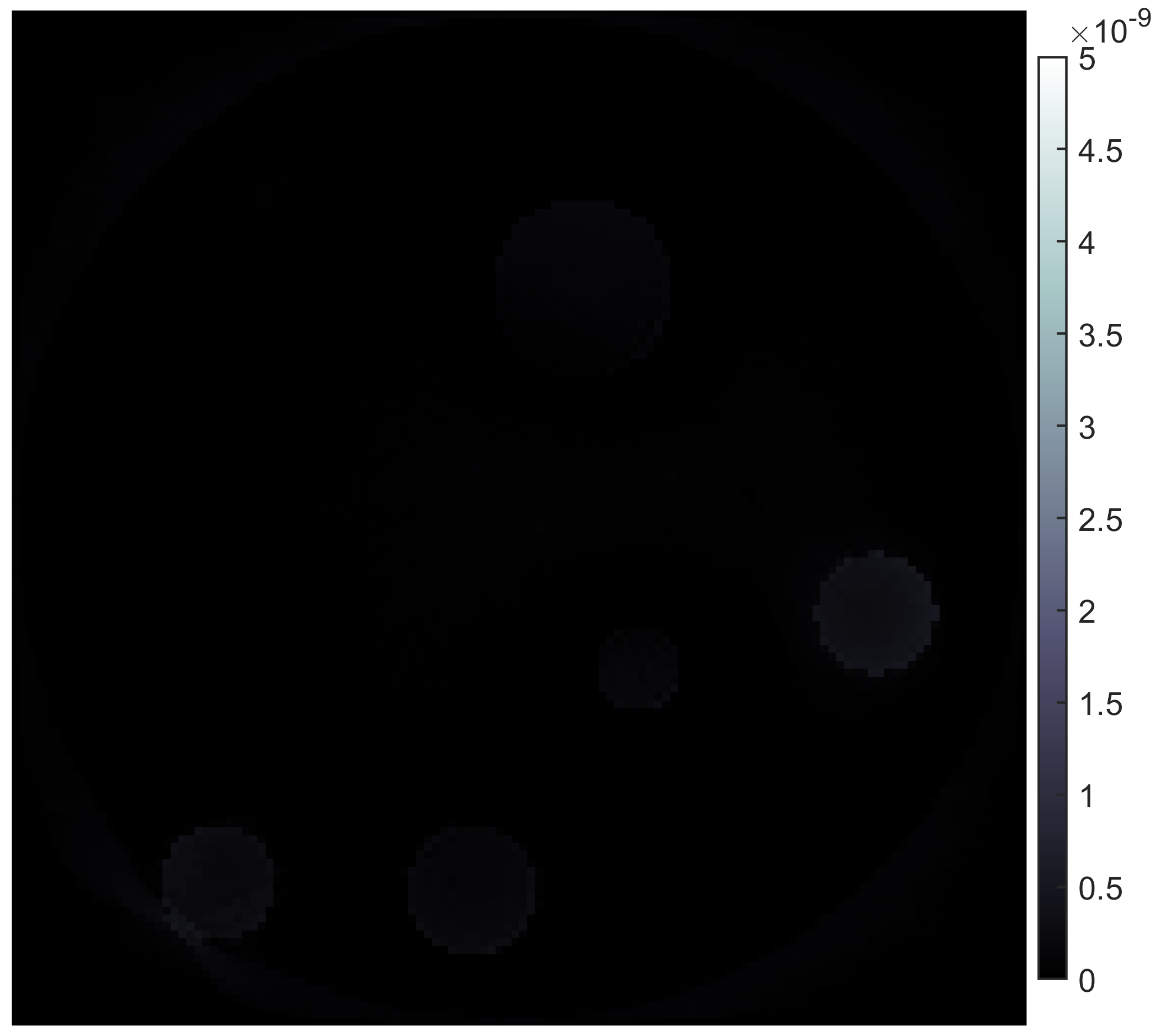} \\
    \end{tabular}
    \caption{3D reconstruction of the relative refractive index (at slice $z=38$) using measurements acquired at multiple propagation distances.}
\label{fig:multi_recon}
\end{figure}

\begin{table}[H]
\centering
\setlength{\tabcolsep}{5pt}
\renewcommand{\arraystretch}{1.15}
\begin{tabular}{llcc@{\hspace{12pt}}cc}
\hline
\multirow{2}{*}{Acquisition}
& \multirow{2}{*}{Method}
& \multicolumn{2}{c}{Phase shift ($\delta$)}
& \multicolumn{2}{c}{Absorption ($\beta$)} \\
\cline{3-6}
& & PSNR & SSIM & PSNR & SSIM \\
\hline
\multirow{4}{*}{\shortstack{Single-\\distance}}
& Two-step FBP
& 33.69 & 0.9606 & 34.68 & \textbf{0.9684} \\
& GD
& 35.69 & 0.9846 & 16.92 & 0.3752 \\
& Bregman TV
& 36.69 & 0.9889 & 17.30 & 0.3963 \\
& Phase-guided
& \textbf{39.39} & \textbf{0.9970}
& \textbf{36.21} & 0.9679 \\
\hline
\multirow{4}{*}{\shortstack{Multi-\\distance}}
& Two-step FBP
& 36.84 & 0.9826 & 32.06 & 0.9141 \\
& GD
& 36.94 & 0.9903 & 32.59 & 0.9591 \\
& Bregman TV
& 37.90 & 0.9933 & 33.43 & 0.9661 \\
& Phase-guided
& \textbf{39.82} & \textbf{0.9977}
& \textbf{45.25} & \textbf{0.9942} \\
\hline
\end{tabular}

\caption{Full-volume quantitative comparison of different reconstruction methods for single-distance and multi-distance acquisitions in terms of PSNR (dB) and global SSIM for the phase shift $\delta$ and absorption index $\beta$. Higher values indicate better reconstruction quality.}

\label{tab:single_multi_metrics}

\end{table}

\subsection{Experimental Validation on Real Data}

We validate the proposed method using experimental near-field holotomography data acquired at the P05 beamline of PETRA~III operated by Helmholtz Center Hereon (DESY, Hamburg). The data were measured with a Fresnel-zone-plate-based NFH setup~\cite{dora2024artifact} and recorded by a scintillator-coupled sCMOS detector (Hamamatsu C12849--101U, $2048\times2048$ pixels, $6.5\,\mu$m pixel size, 16-bit depth). The acquisition parameters were $E=11.0$~keV, $z_{01}=79.45\,\mathrm{mm}$, $z_{02}=19.661\,\mathrm{m}$, $\mathrm{Fr}=7.790\times10^{-5}$, and an exposure time of $t=1.0$~s per hologram. The reconstruction was performed using 90 projection angles obtained by angular subsampling of the original dataset over the available rotation range.

The sample is a spider attachment hair, which contains thin filamentary structures and fine texture. Compared with synthetic experiments, the real data contain measurement noise, residual alignment errors, detector imperfections, and model mismatch. These effects make the absorption component more ill-conditioned and provide a challenging test case. 

We first applied flat-field normalization to the raw measurements and then laterally aligned the projections~\cite{homann2015validity, xu2015three, parkinson2012automatic}. Representative corrected intensity measurements at several projection angles are shown in Figure~\ref{fig:spider_measured}. The measured diffraction patterns exhibit strong holographic contrast together with fine interference fringes generated by the filamentary structures of the sample.
    
\begin{figure}[H]
\centering
\begin{minipage}[t]{0.23\linewidth}
\centering
\includegraphics[width=1\textwidth]{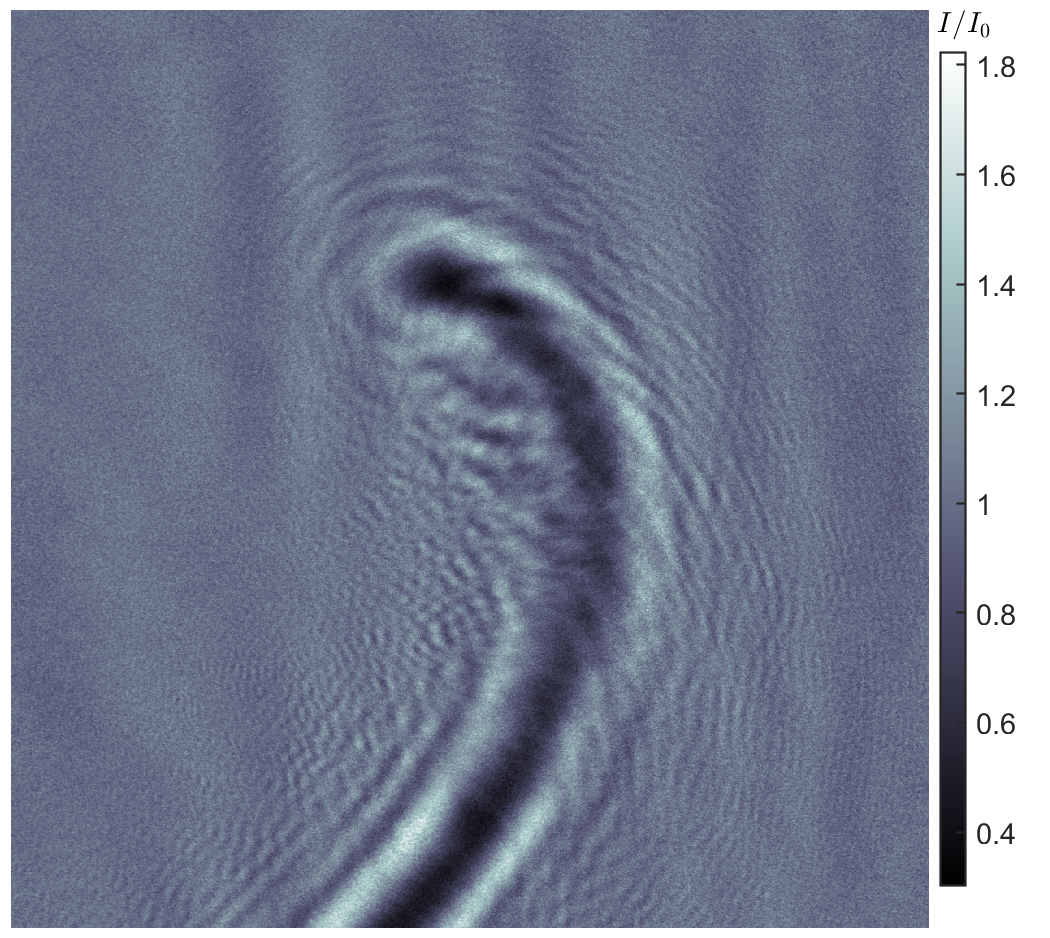}\\
% \centerline{\footnotesize\emph{gradient descent}}
\end{minipage}
\begin{minipage}[t]{0.23\linewidth}
\centering
\includegraphics[width=1\textwidth]{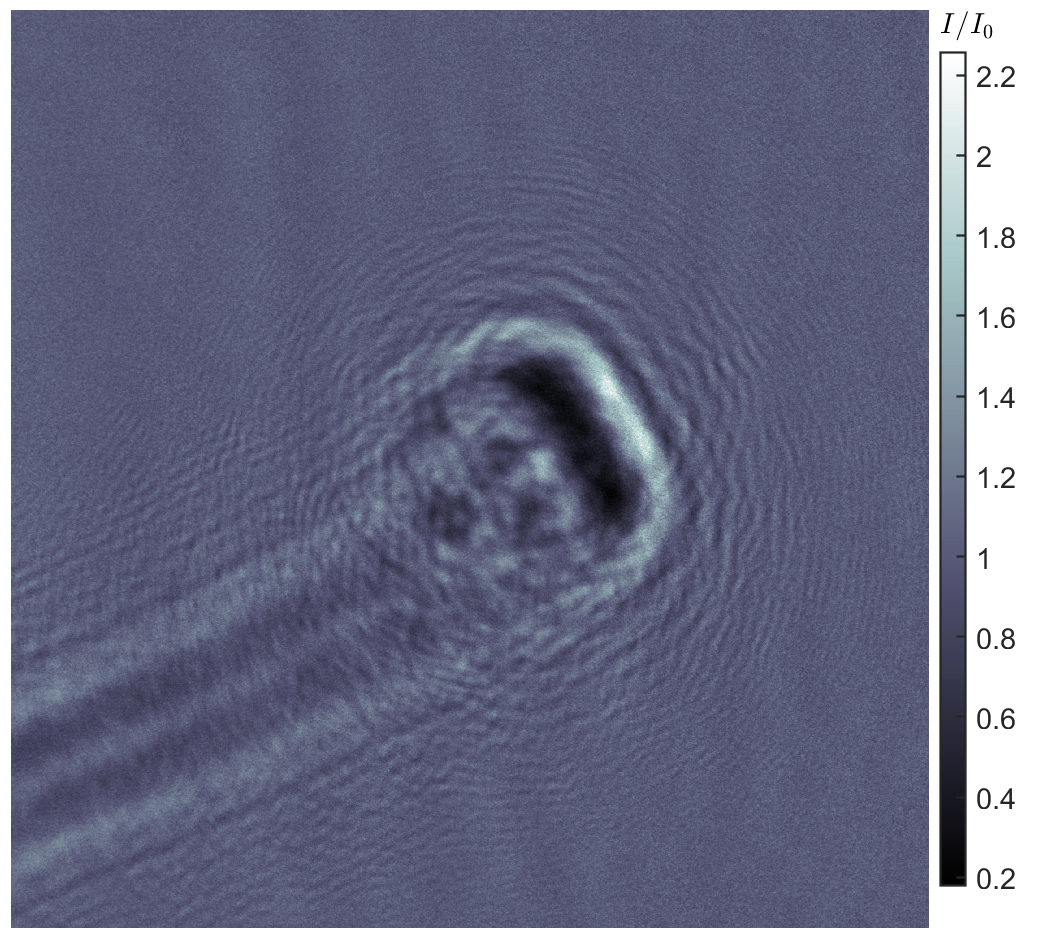}\\
% \centerline{\footnotesize\emph{gradient descent}}
\end{minipage}
\begin{minipage}[t]{0.23\linewidth}
\centering
\includegraphics[width=1\textwidth]{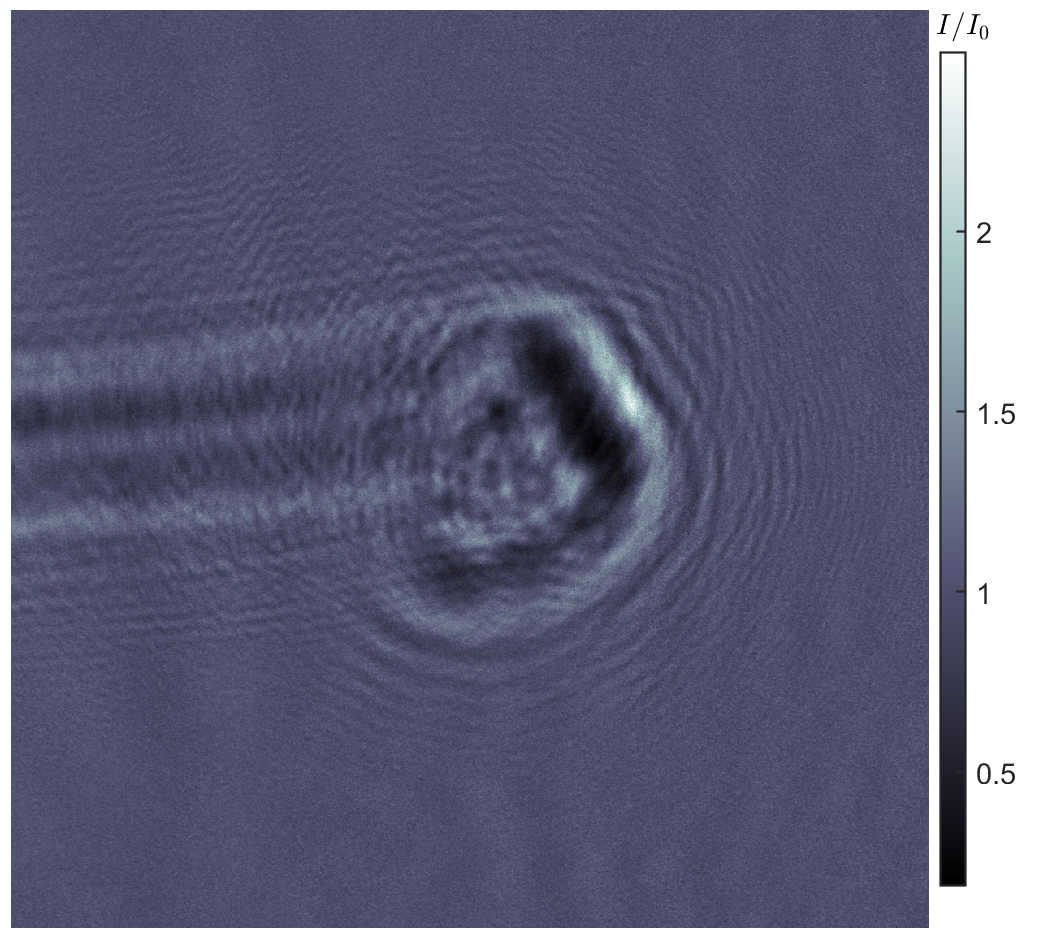}\\
% \centerline{\footnotesize\emph{Bregman TV}}
\end{minipage}
\begin{minipage}[t]{0.23\linewidth}
\centering
\includegraphics[width=1\textwidth]{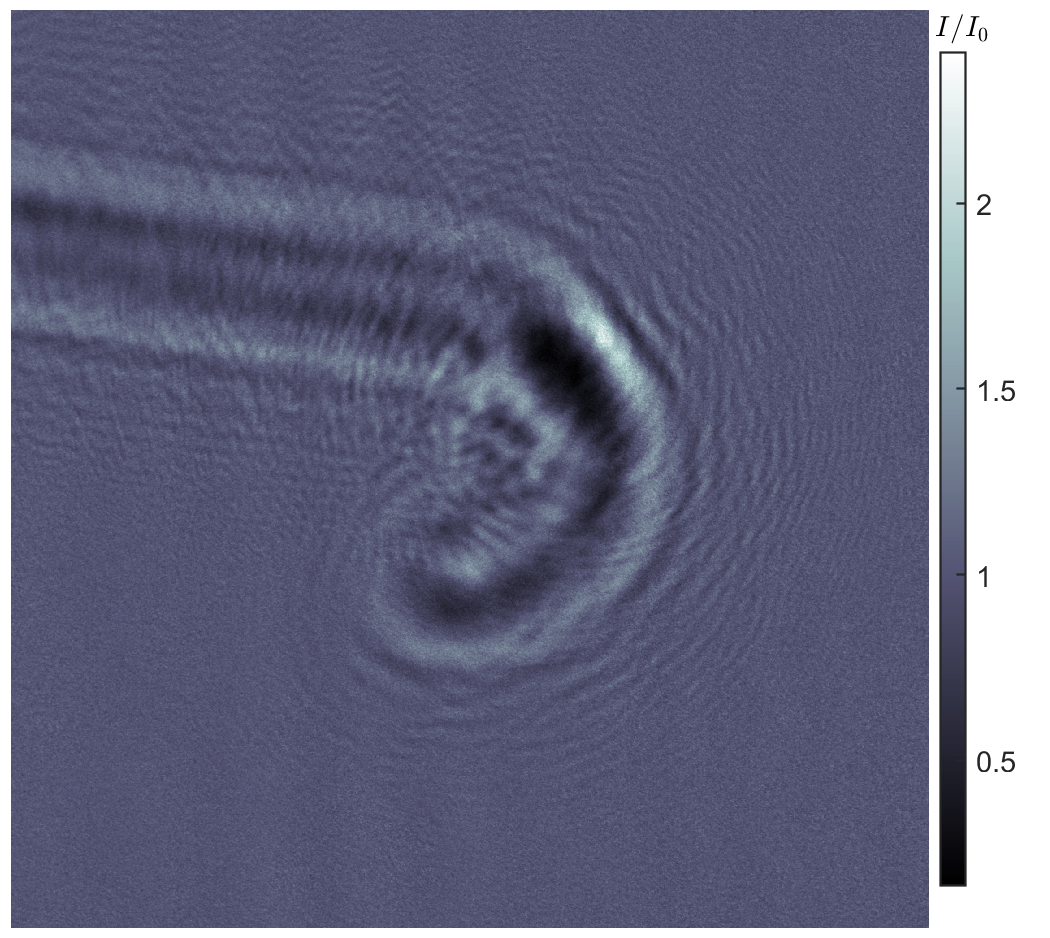}\\
% \centerline{\footnotesize\emph{Bregman TV with phase-guided}}
\end{minipage}
\caption{Representative corrected intensity measurements used as input for the reconstruction, shown at several projection angles.}
\label{fig:spider_measured}
\end{figure}

After flat-field correction, i.e., dividing by the illuminations and projection alignment, we embedded the measured intensity data into an enlarged computational detector domain before reconstruction.
This extension is necessary because Fresnel propagation spreads diffraction features beyond the measured field of view, while an insufficient computational window can introduce boundary artifacts due to the implicit periodicity of the FFT-based propagation model. 
In our implementation, we placed the original measured field of view at the center of the computational domain and extended it by mirror padding. We then applied a smooth fade-out window so that the padded region gradually decays to the flat-field background value~\cite{dora2024artifact}. 
All compared reconstruction methods used the same padded intensity data. 

We performed all reconstructions in a coarse-to-fine manner. In preliminary tests, increasing the direct reconstruction grid beyond $512 \times 512 \times 512$ did not visibly improve the recovered structures, while the memory cost and runtime increased substantially. Therefore, the final results are reported at $512 \times 512 \times 512$ .

\begin{figure}[H]
    \centering
    \setlength{\tabcolsep}{2pt}

    \begin{tabular}{c c c c}
        \footnotesize GD $\Re(\widetilde O)$
        & \footnotesize GD $\Im(\widetilde O)$
        & \footnotesize Phase-guided $\Re(\widetilde O)$
        & \footnotesize Phase-guided $\Im(\widetilde O)$ \\

        \includegraphics[width=0.23\textwidth]{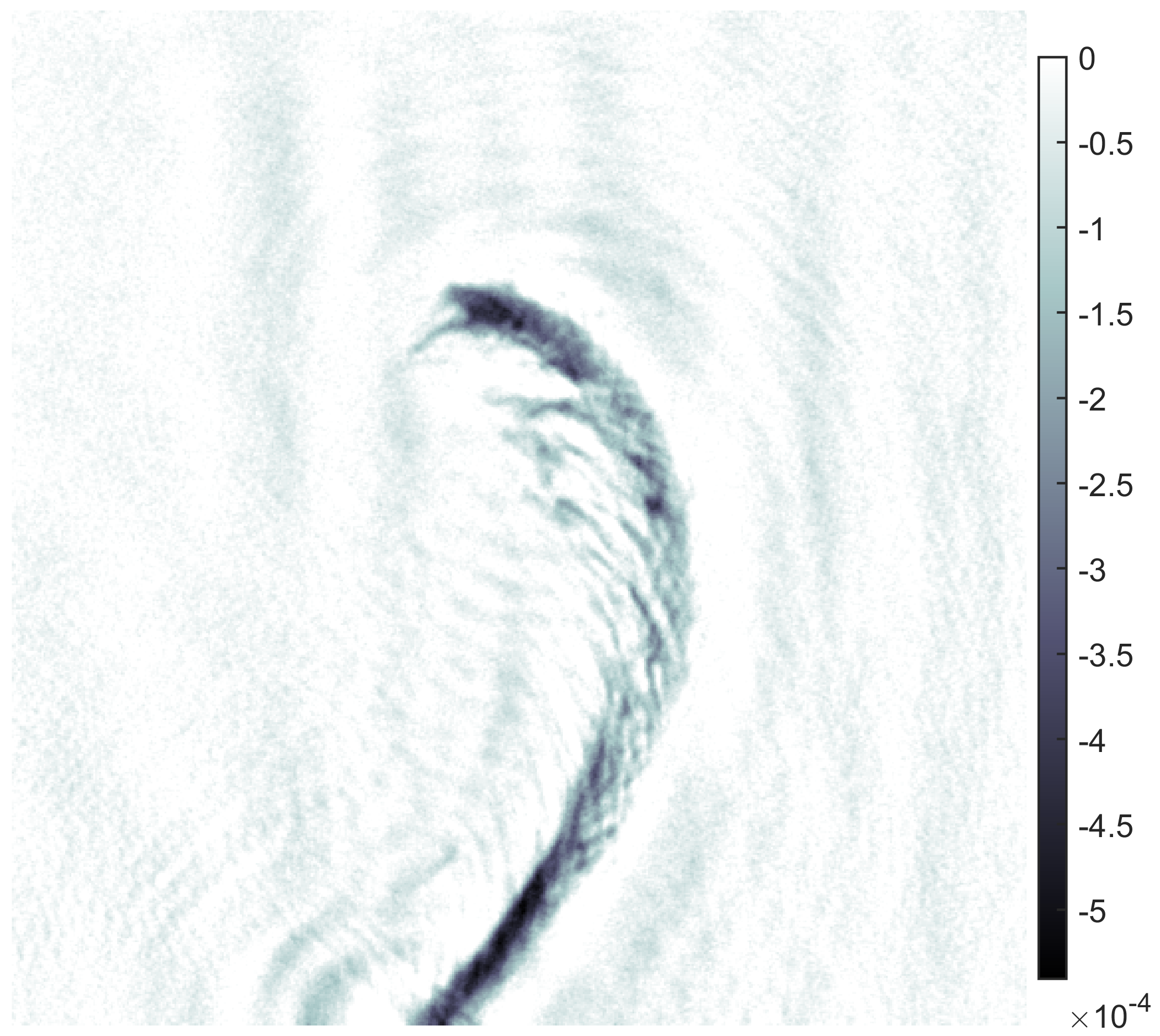} &
        \includegraphics[width=0.23\textwidth]{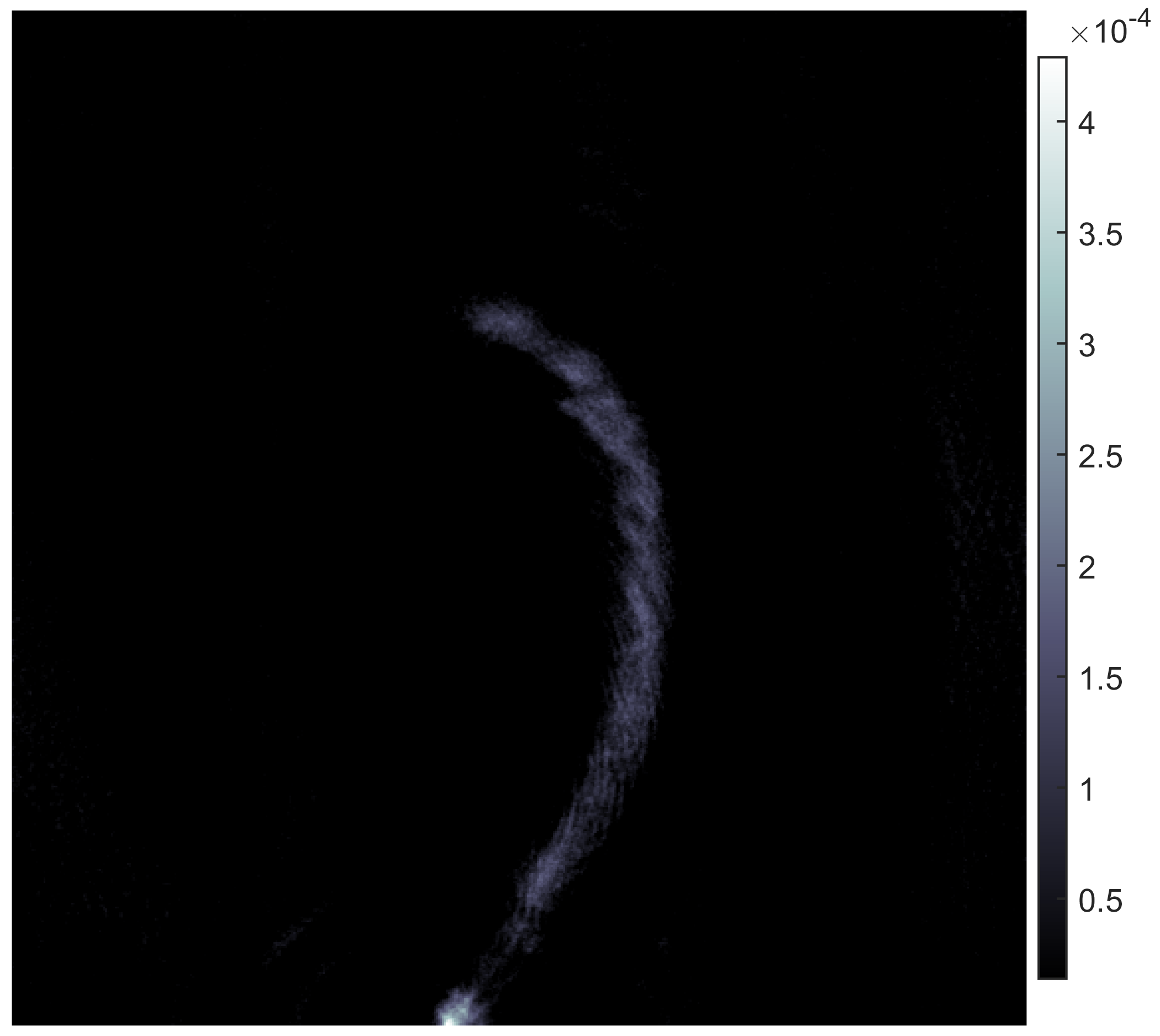} &
        \includegraphics[width=0.23\textwidth]{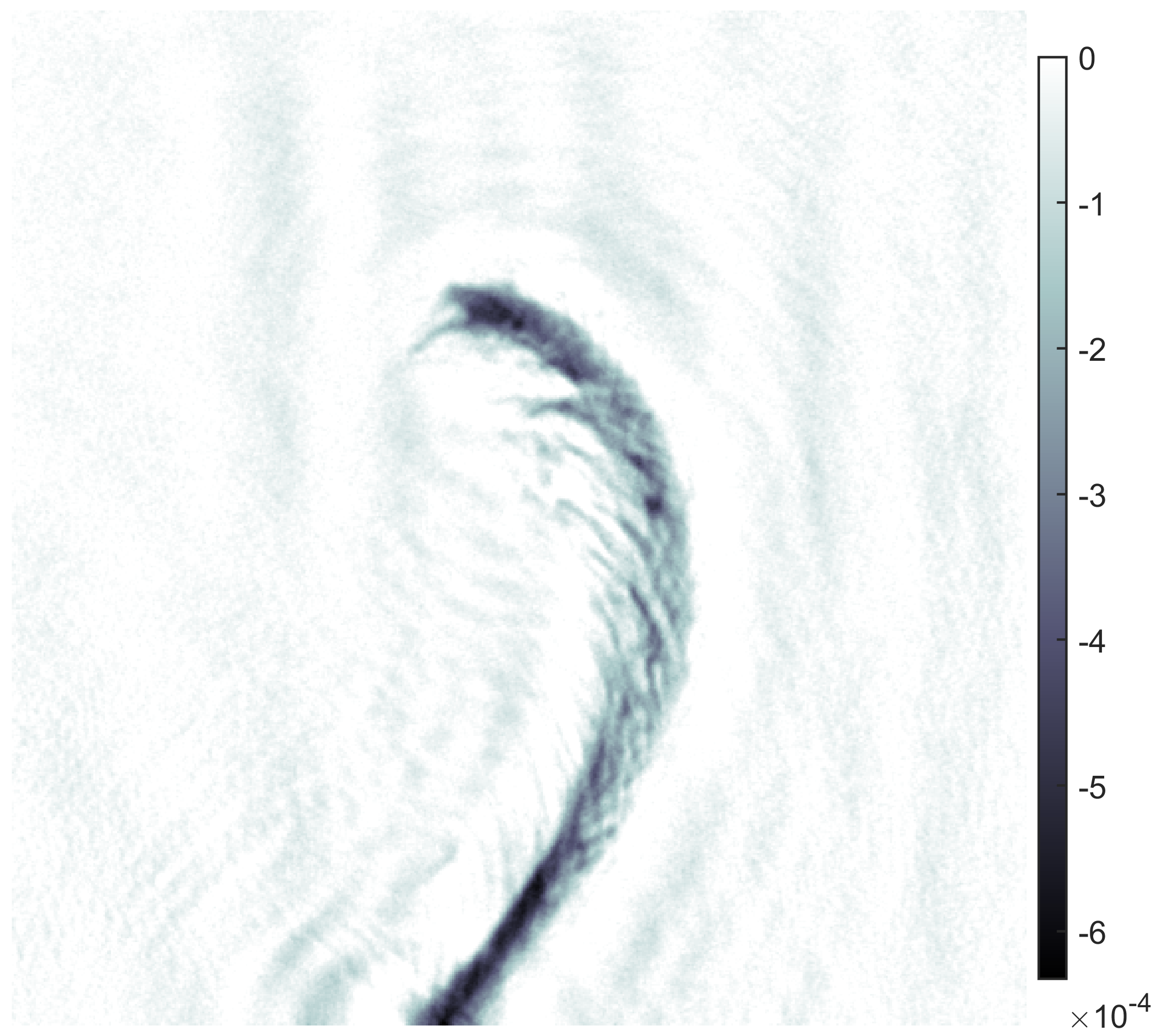} &
        \includegraphics[width=0.23\textwidth]{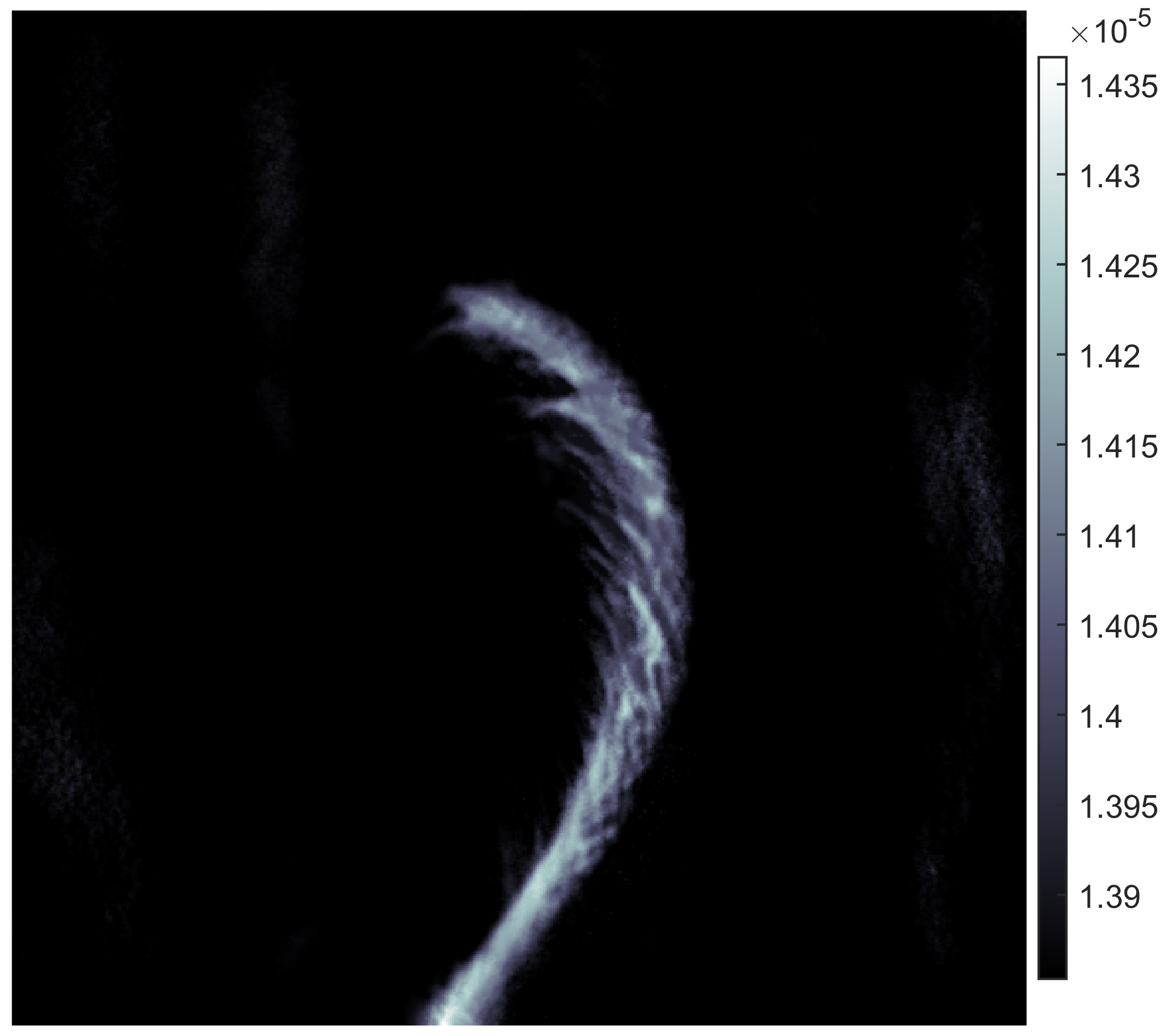}
    \end{tabular}

    \caption{Comparison of projection-domain reconstruction results for a representative projection angle.
    Each subfigure uses its own colorbar range rather than a globally unified scale.}
    \label{fig:spider_projection_reco}
\end{figure}
We first compare the projection-domain reconstruction results obtained using conventional gradient descent and the proposed phase-guided reconstruction strategy. 
Figure~\ref{fig:spider_projection_reco} shows the reconstructed projected refractive components $\mathrm{Re}(\widetilde O)$ and $\mathrm{Im}(\widetilde O)$ for a representative projection angle. Both methods recover the dominant phase structure in the real component, while the phase-guided reconstruction produces visibly cleaner and more coherent structures in the imaginary component.

The reconstructed projection-domain results were then used for tomographic reconstruction. Figure~\ref{fig:spider_3d_compare} compares reconstructed slices of the refractive index components $\delta$ and $\beta$ obtained using different reconstruction strategies. For the phase component $\delta$, all methods recover the overall morphology of the spider hair, although the direct 3D approaches produce sharper boundaries and improved structural continuity. The main differences appear in the absorption component $\beta$, where conventional direct reconstruction suffers from severe streaking and structured artifacts, while the proposed phase-guided method yields a substantially cleaner reconstruction with improved contrast and enhanced recovery of the filamentary structures.

\begin{figure}[H]
    \centering
    \setlength{\tabcolsep}{1pt}

    \begin{tabular}{c c c c c}
        & \footnotesize Two-step FBP
        & \footnotesize GD 
        & \footnotesize Bregman TV 
        & \footnotesize Phase-guided \\
        
        \rowlabel{$-\delta_{reco}$} &
        \includegraphics[width=0.23\textwidth]{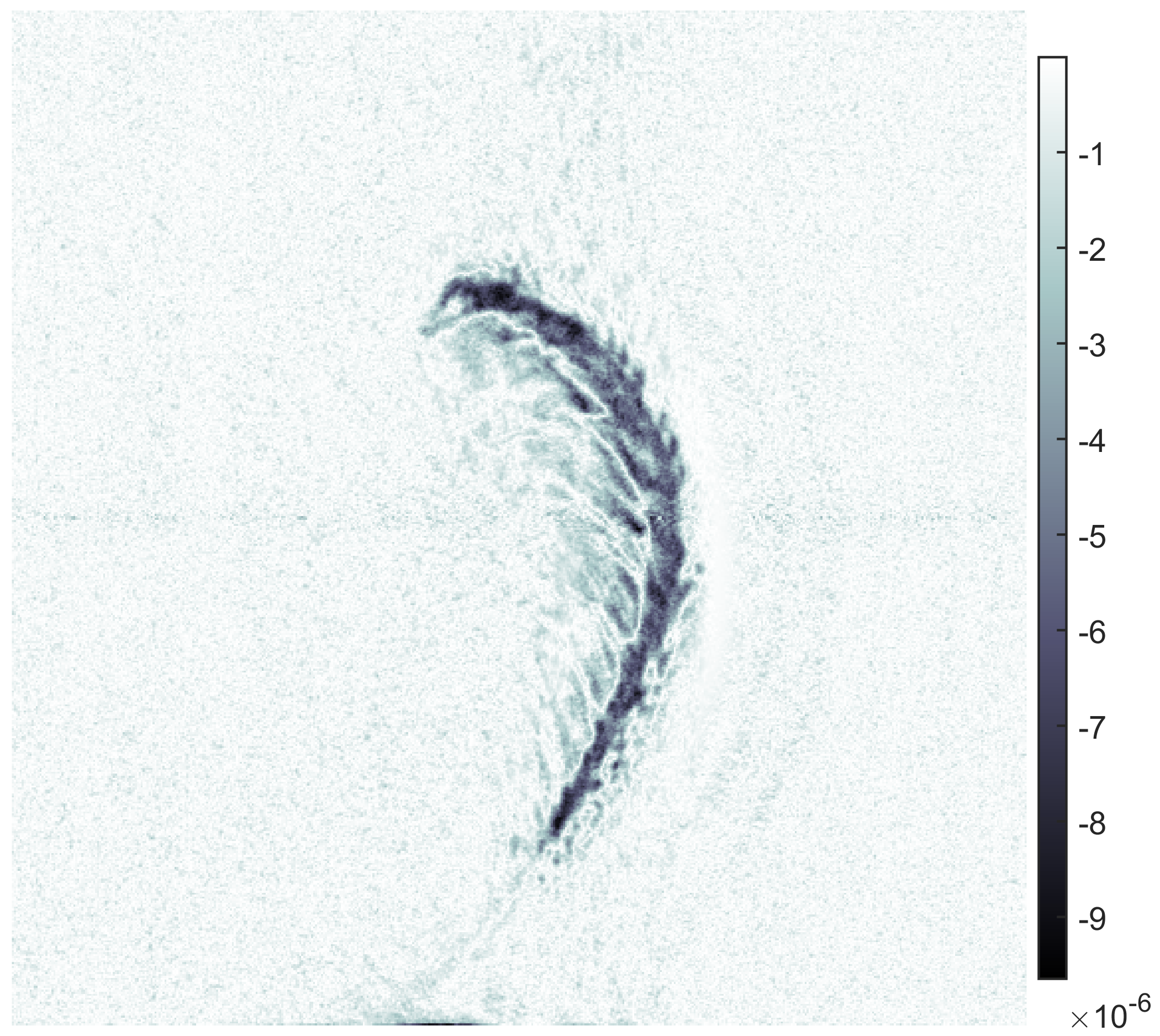} &
        \includegraphics[width=0.23\textwidth]{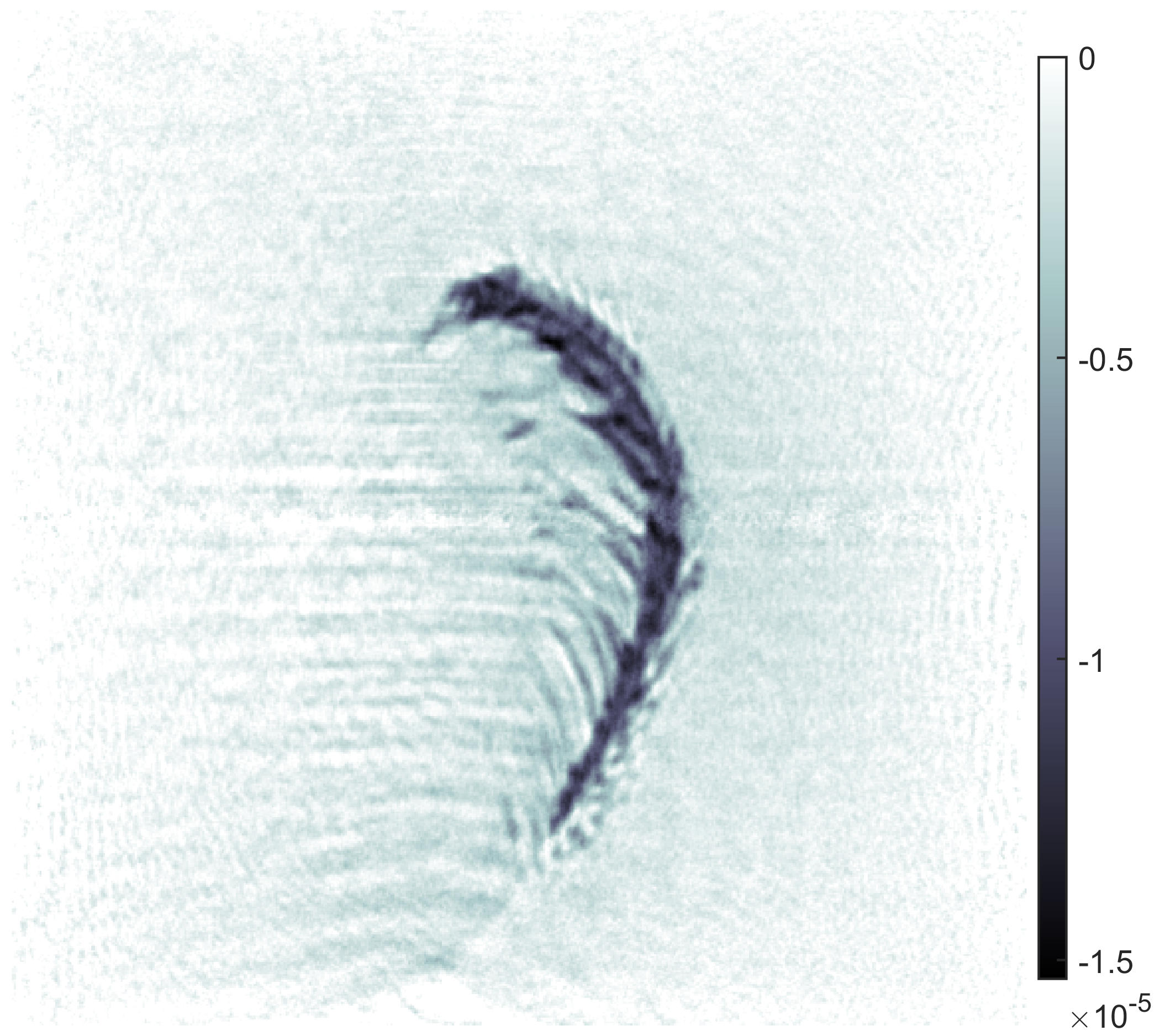} &
        \includegraphics[width=0.23\textwidth]{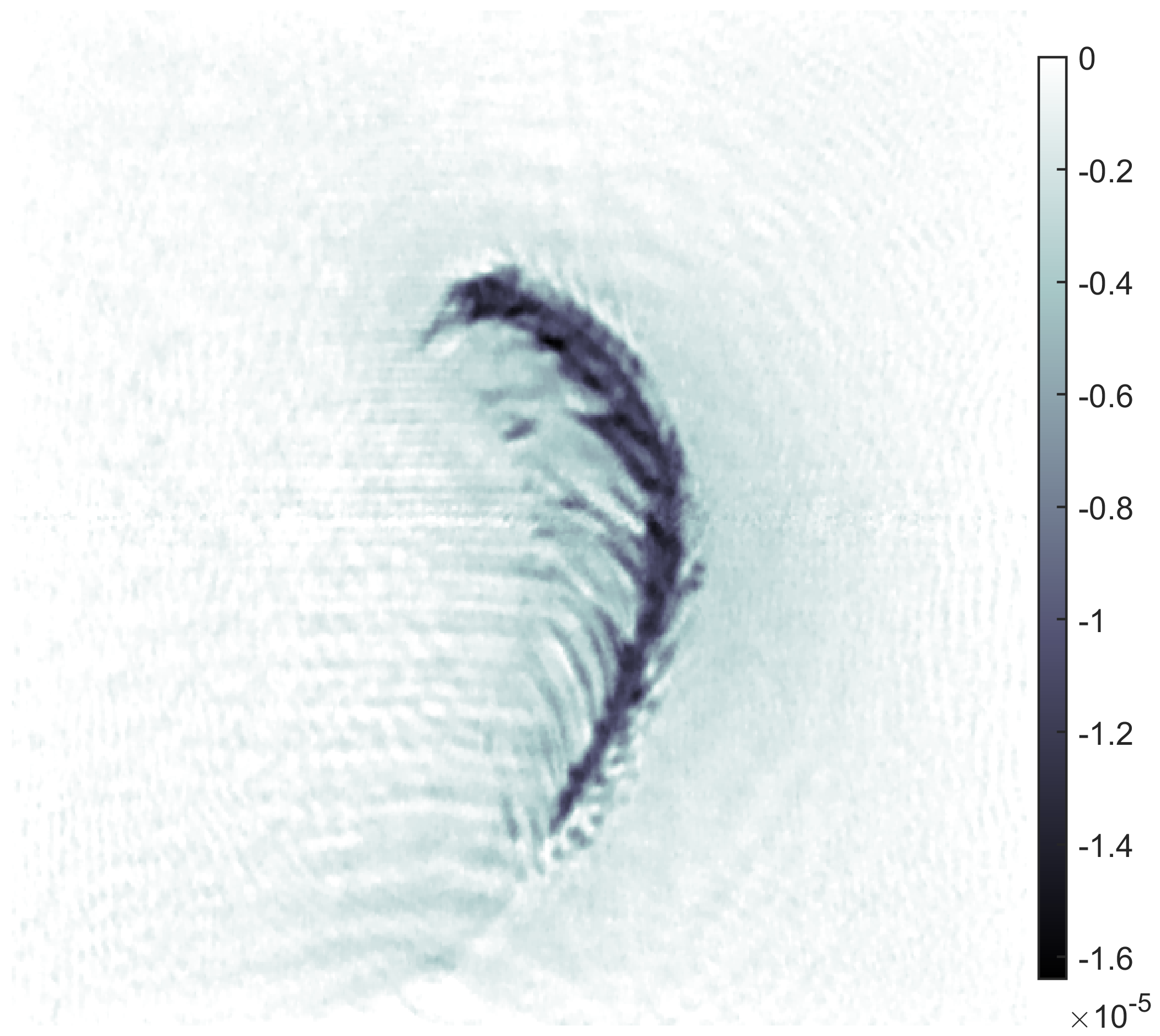} &
        \includegraphics[width=0.23\textwidth]{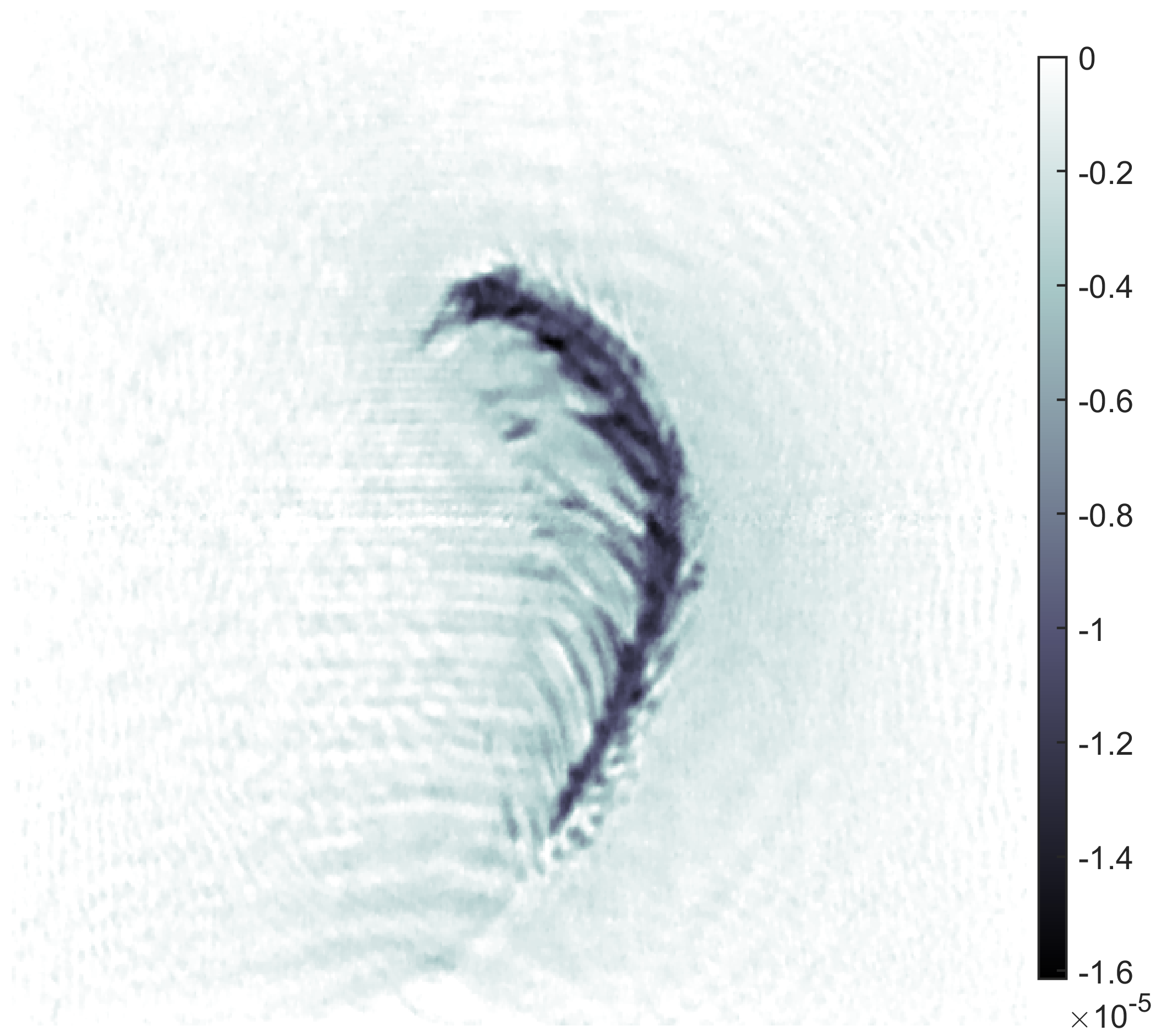} \\

        \rowlabel{$\beta_{reco}$} &
        \includegraphics[width=0.23\textwidth]{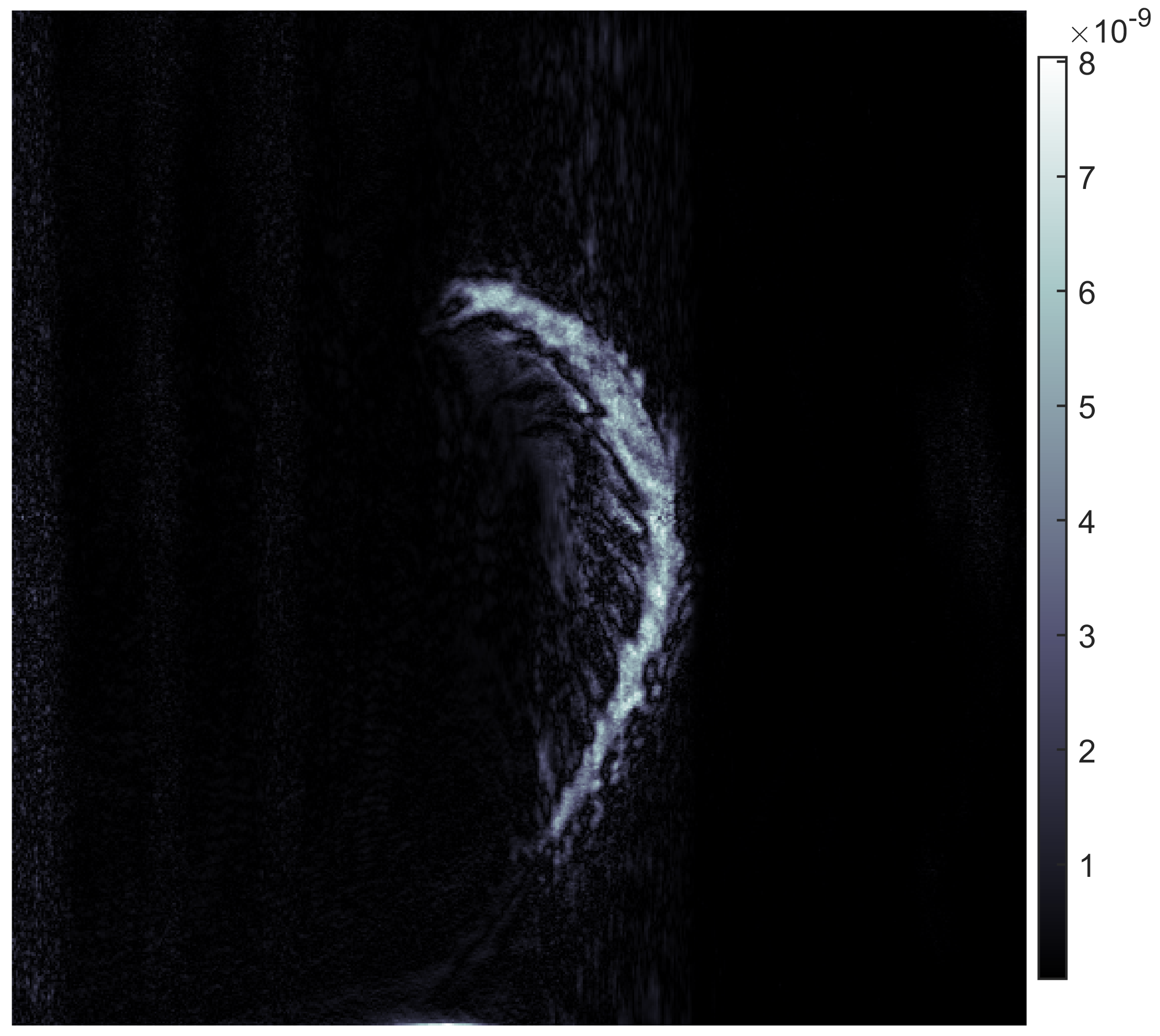} &
        \includegraphics[width=0.23\textwidth]{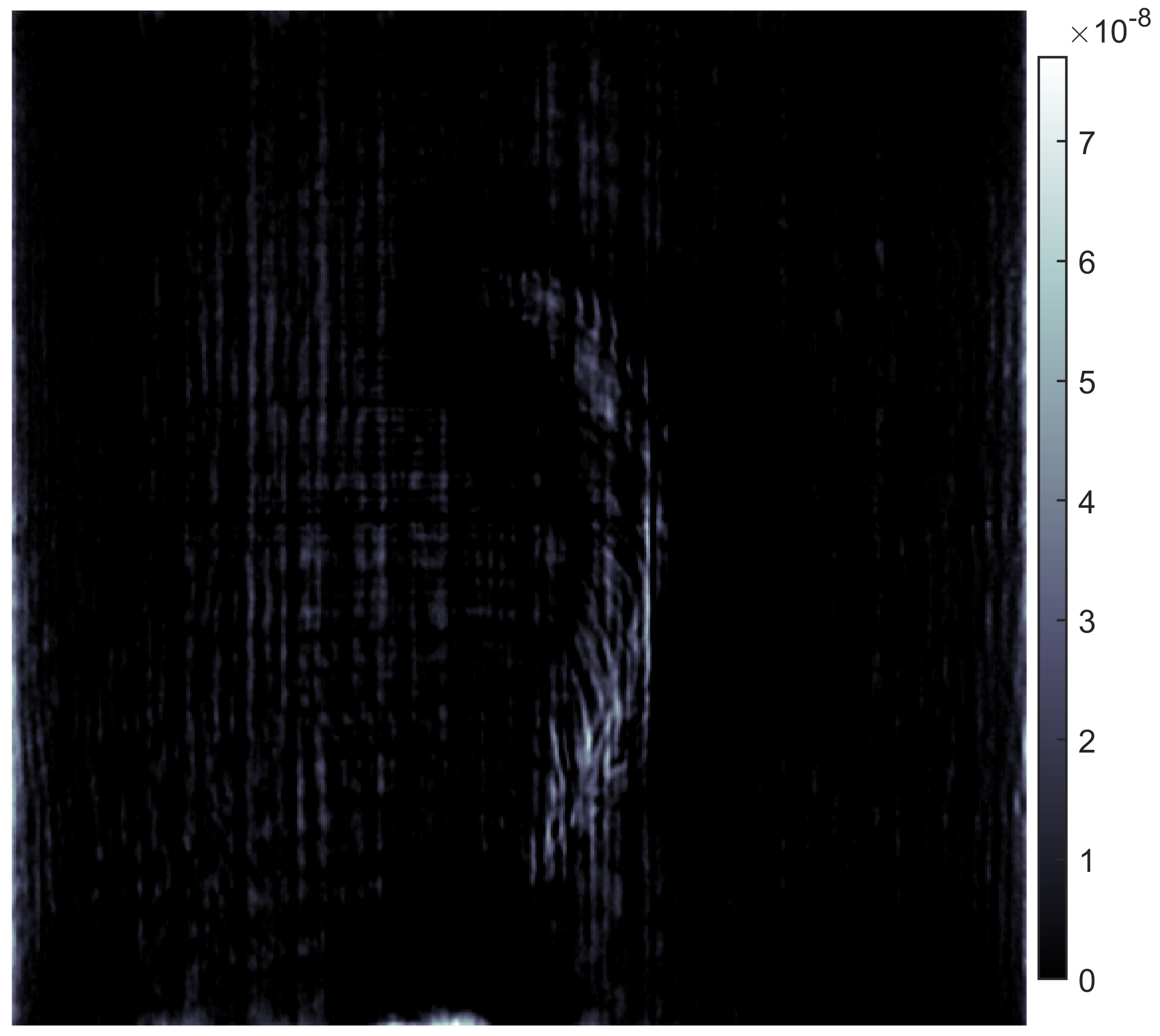} &
        \includegraphics[width=0.23\textwidth]{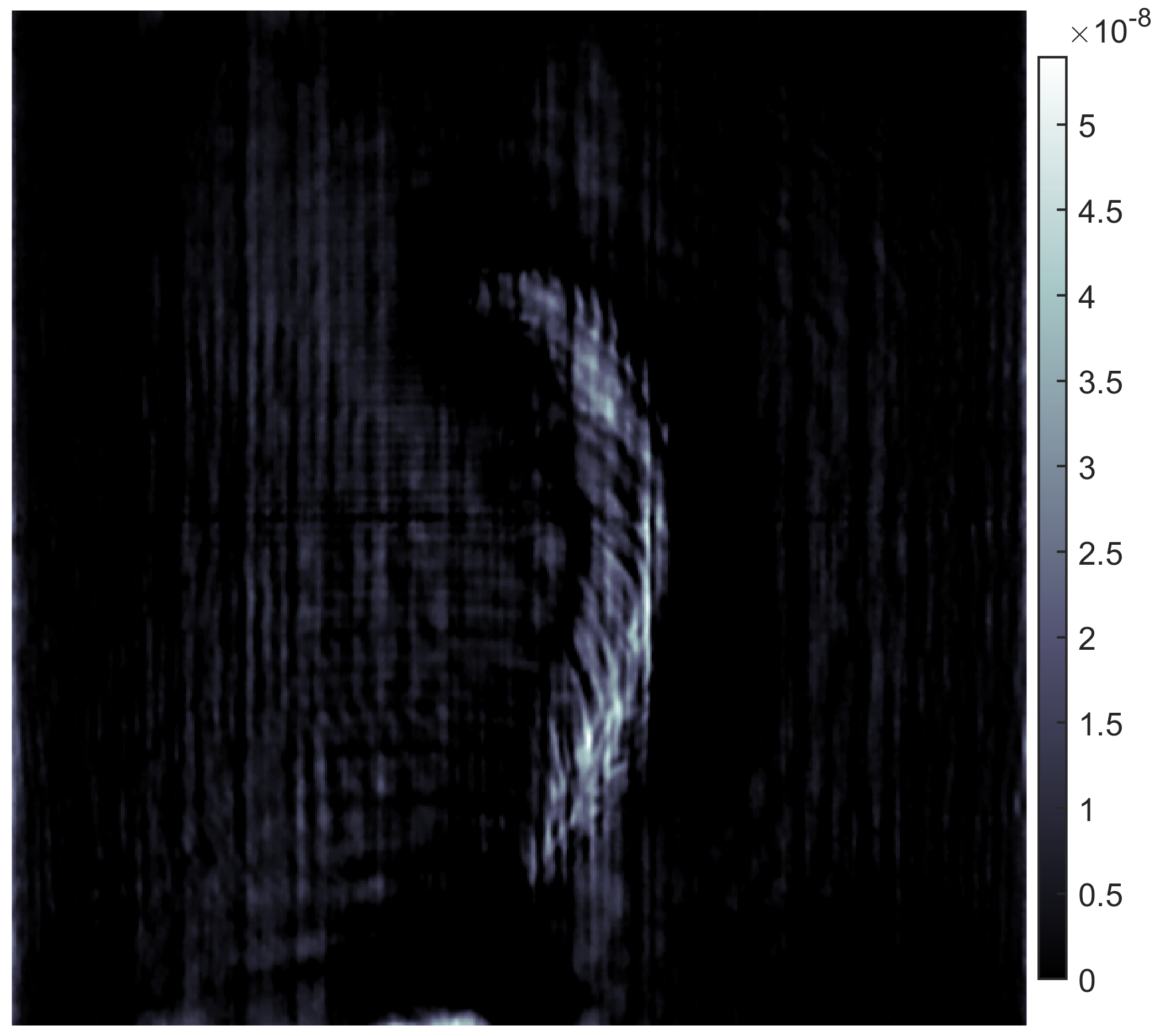} &
        \includegraphics[width=0.23\textwidth]{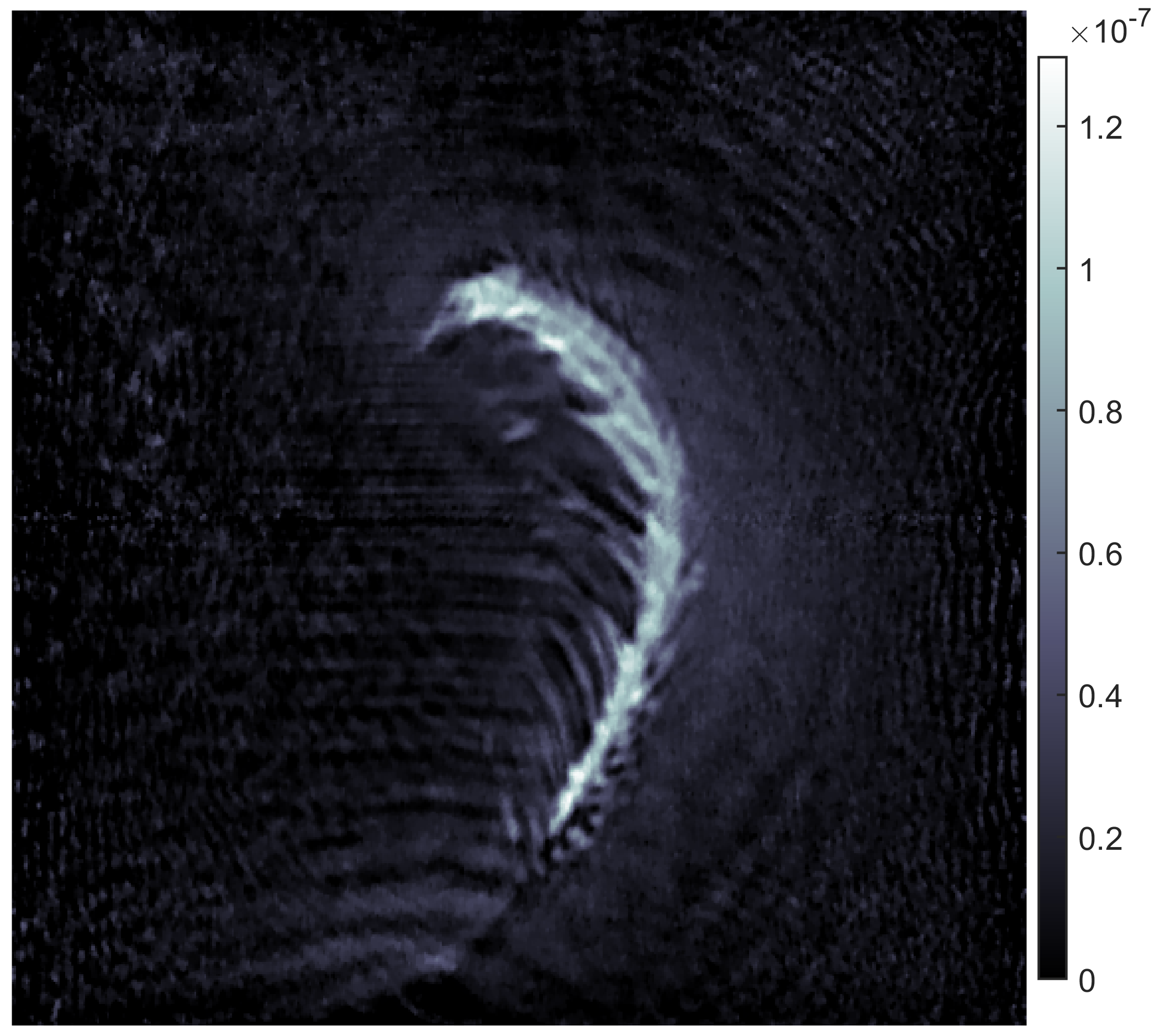} \\
    \end{tabular}
    \caption{
Central $xz$ slices of the reconstructed $\delta$ and $\beta$ volumes obtained using different reconstruction methods.
}
\label{fig:spider_3d_compare}
\end{figure}

Finally, Figure~\ref{fig:spider_render} presents volume renderings of the spider hair reconstructed using the proposed phase-guided Bregman TV framework from four viewing directions. 
The reconstruction mainly preserves the dominant structures shared by both $\delta$ and $\beta$, whereas weaker absorption features remain less pronounced. 
As a result, the reconstructed $\beta$ appears thinner, sparser, and exhibits fewer fine internal textures than $\delta$. 
These visual differences are further enhanced in the volume rendering by the use of the same intensity-to-opacity mapping for both components.

\begin{figure}[t]
    \centering
    \setlength{\tabcolsep}{1pt}

    \begin{tabular}{c c c c c}
        \rowlabel{$\delta_{reco}$} &
        \includegraphics[width=0.23\textwidth]{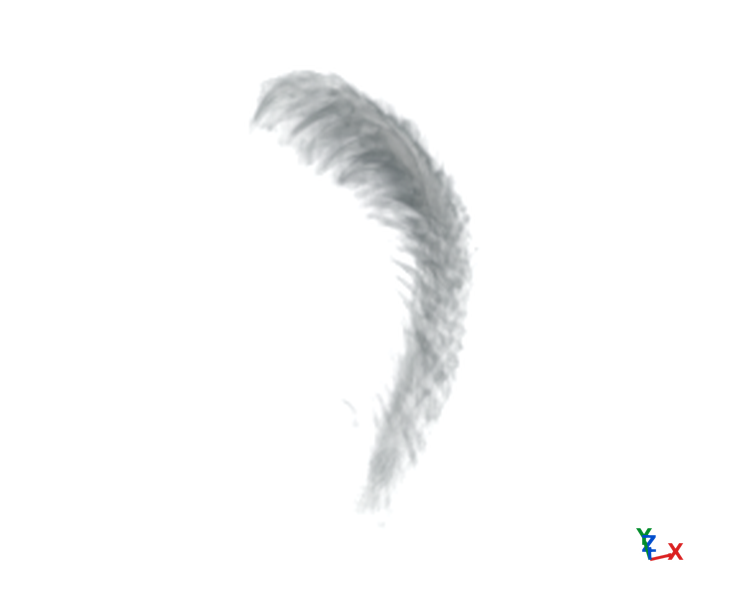} &
        \includegraphics[width=0.23\textwidth]{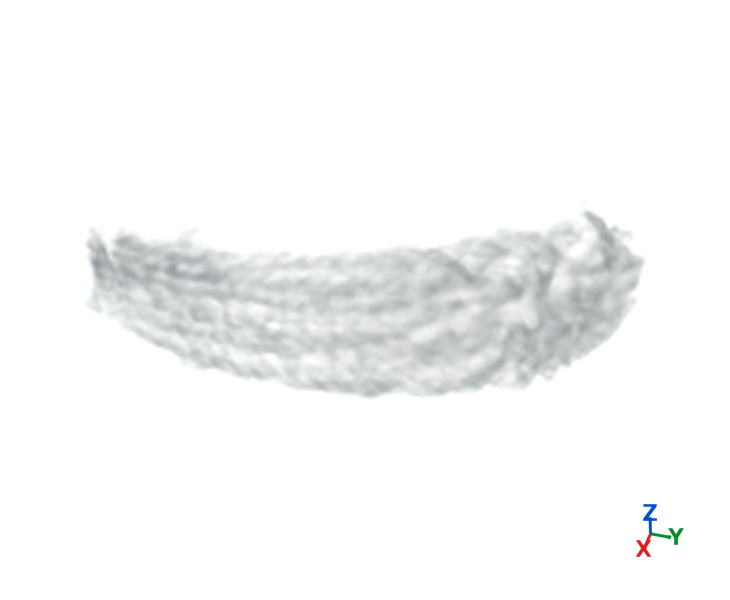} &
        \includegraphics[width=0.23\textwidth]{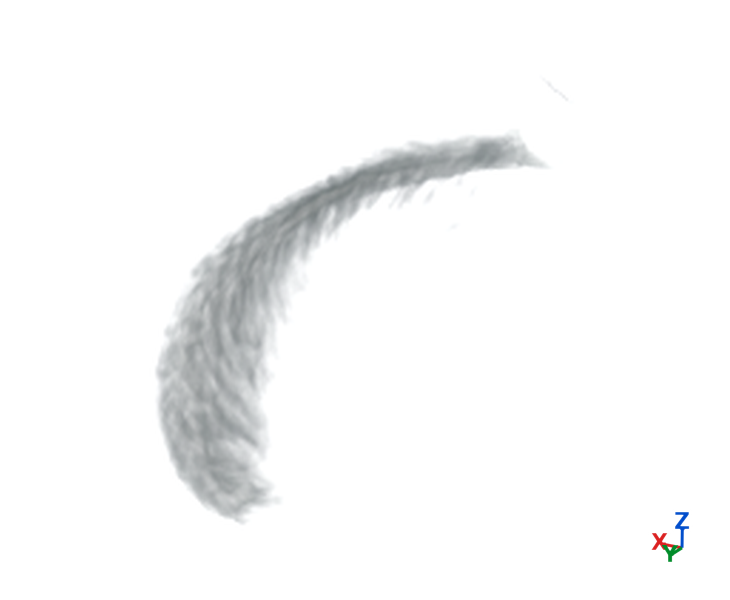} &
        \includegraphics[width=0.23\textwidth]{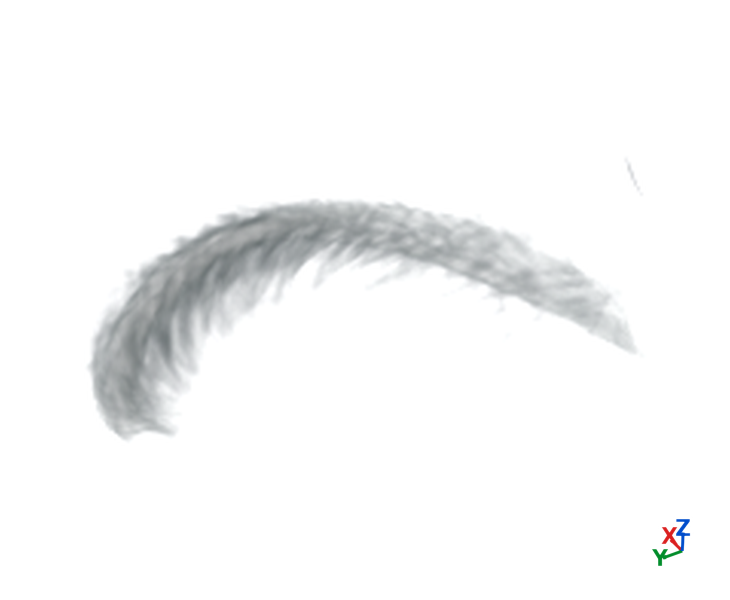} \\

        \rowlabel{$\beta_{reco}$} &
        \includegraphics[width=0.23\textwidth]{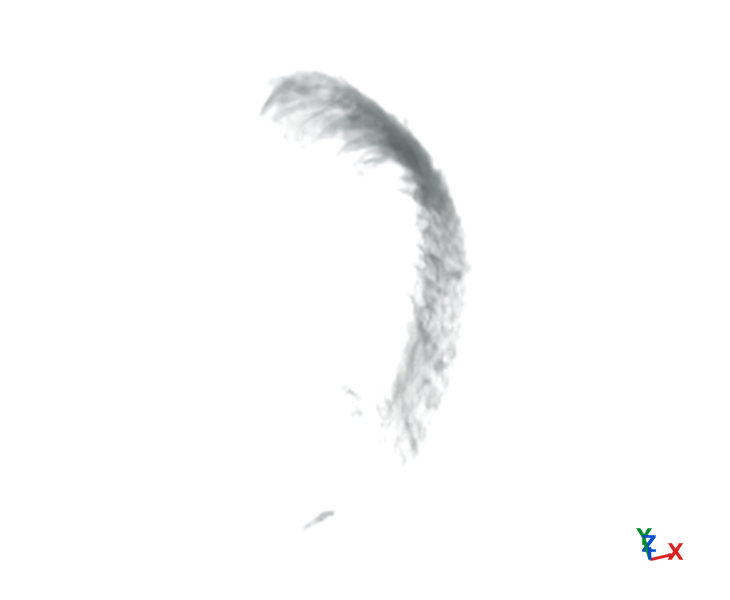} &
        \includegraphics[width=0.23\textwidth]{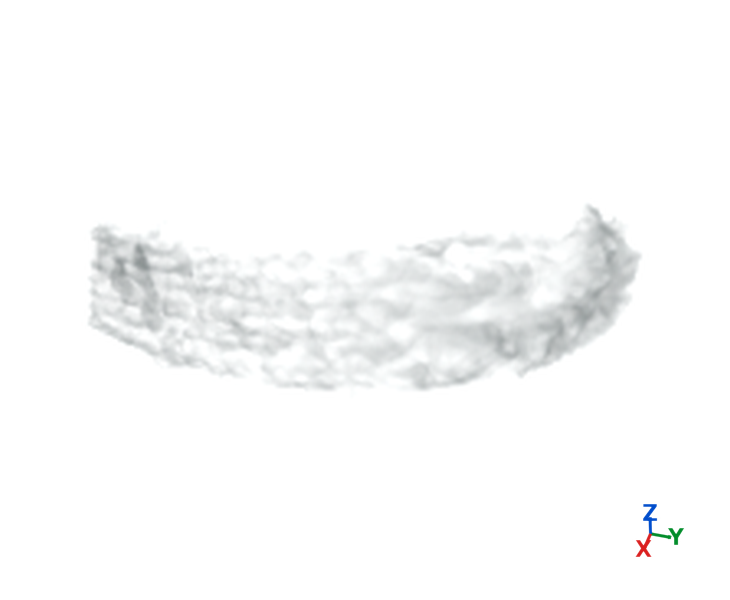} &
        \includegraphics[width=0.23\textwidth]{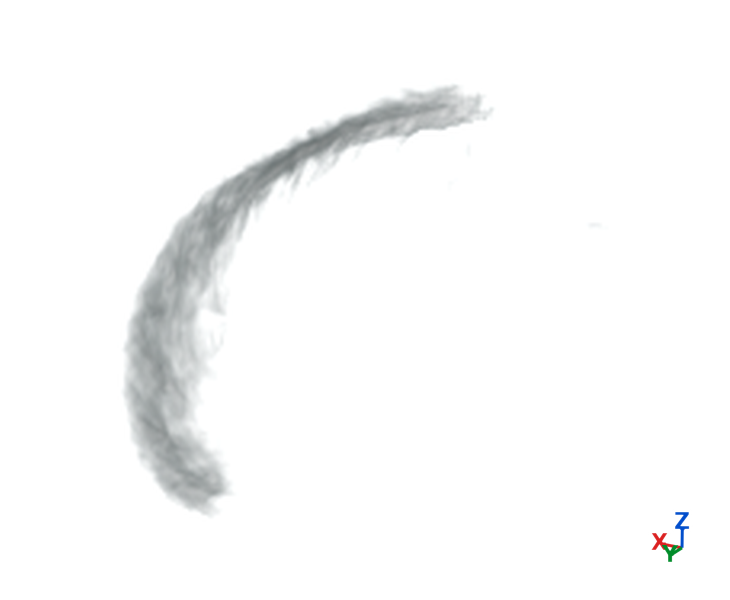} &
        \includegraphics[width=0.23\textwidth]{Figures/spider/spider512/spider3D_render03.png} \\
    \end{tabular}
    \caption{Volume-rendered views of the reconstructed spider attachment hair from different viewing directions. }
\label{fig:spider_render}
\end{figure}

\section{Conclusion}
In this work, we developed a direct variational framework for near-field holotomography with phase-guided Bregman TV regularization for stable reconstruction of both phase and weak absorption contrasts.
From an analytical perspective, the forward operator was shown to be well-posed and Fr\'echet differentiable, moreover nonlinearity conditions leading to local convergence and regularization properties have been verified under suitable assumptions.

Numerical simulations and experimental results demonstrate that the proposed phase-guided regularization substantially improves the recovery of weak absorption features and consistently enhances reconstruction quality and robustness across both two-step and direct holotomographic imaging frameworks.
Together, these results establish phase-guided regularization as an effective strategy for exploiting the complementary information contained in phase and absorption contrasts and provide a strong foundation for quantitative refractive-index imaging in near-field holotomography.

Future work may investigate computational acceleration for large-scale 3D reconstruction, uncertainty quantification, more advanced regularization strategies, and extensions to limited-angle acquisition settings.

\section*{Acknowledgements}
The authors acknowledge support from DESY (Hamburg, Germany), a member of the Helmholtz Association (HGF). This work was supported in part by the Helmholtz Innovation Platform HI-ACTS, by the computational resources of the Maxwell high-performance computing system operated at DESY, and by the CMAT project (Funding ID: 13K25009). Parts of this research were carried out at the PETRA III beamline P05. Beamtime was allocated under proposal I-20170864 (P05). The authors thank Imke Greving and Silja Flenner for fruitful discussions and for their support during the beamtime at P05.

\printbibliography

\clearpage
\appendix

\section{Proof of Lemma \ref{lem:nemytskii}}

We now provide the (lengthy) proof of Lemma \ref{lem:nemytskii}.
\begin{proof}
1. For a fixed $z\in\mathbb C$, the map $x\mapsto g(x,z)=P(x)e^{ikz}$ is measurable, since $P$ is measurable. For a fixed $x\in\Omega'_\pi$, the map $z\mapsto g(x,z)=P(x)e^{ikz}$ is continuous. Thus $g$ satisfies the standard measurability and continuity hypothesis for a Nemytskii kernel and $\mathcal N_g$ is well-defined.

If $u\in\mathcal U$, then $\Im u(x)\ge0$. Consequently,
$$
|e^{ik u(x)}|=e^{-k\Im u(x)}\le 1,\qquad\text{a.e. }x \in \Omega'_\pi,
$$
and
$$
|\mathcal N_g(u)(x)|=|P(x)|\,|e^{ik u(x)}|\le |P(x)|,
$$
from which the boundedness of the norm follows.

2. For any $u,v\in\mathcal U$ and a.e. $x \in \Omega'_\pi$, define
$
\Delta(x):=e^{ik u(x)}-e^{ik v(x)}.
$
It follows that
$$
\Delta(x)=ik\,(u(x)-v(x))\int_0^1 e^{ik\big(v(x)+t(u(x)-v(x))\big)}\,dt.
$$
Since $\Im (v(x)+t(u(x)-v(x)))\ge0$ for a.e. $t\in [0,1]$, we have
$$
|\Delta(x)| \le k\,|u(x)-v(x)|,\qquad\text{a.e. }x \in \Omega'_\pi.
$$
Multiplying $\Delta(x)$ by $P(x)$ and taking the $L^p$-norm yield
$$
\|\mathcal N_g(u)-\mathcal N_g(v)\|_{L^p(\Omega'_\pi)}
= \|P(\,e^{ik u}-e^{ik v}\,)\|_{L^p(\Omega'_\pi)}
\le k\,\|P\|_{L^\infty(\Omega'_\pi)}\,\|u-v\|_{L^p(\Omega'_\pi)},
$$
which proves the Lipschitz continuity.
%Clearly, the estimate holds for $p=\infty$. 

3. Let $u\in \mathcal{U}$ and $h\in L^r(\Omega'_\pi;\mathbb C)$ with a small $\|h\|_{L^r(\Omega'_\pi)}$ such that $u+h\in \mathcal{U}$. 
By Taylor's expansion,
$$
e^z=1+z+|z|^2\int_0^1(1-t)e^{zt}dt,\quad z\in\mathbb{C}.
$$
Then,
$$
{e}^{ik(u+h)} - {e}^{ik u} - ik{h}\,{e}^{ik u}
= {e}^{ik u}\big({e}^{ik h}-1 - ik {h}\big)
= {e}^{ik u}\,(ik)^2\;{h}^2\int_0^1(1-t)\,{e}^{ikh t}\,dt.
$$
Multiplying it by $P$ to obtain the remainder
$$
R(h) := {\mathcal N_g}(u+h) - {\mathcal N_g}(u) - {\mathcal N_g}'(u)[h]
= P {e}^{ik u} (ik)^2 h^2 \int_0^1(1-t){e}^{ik t h}\,dt.
$$
Thus
$$
|R(h)| \le c |h|^2, \quad\text{if} \quad|h| < 1,
$$
where $c$ depends on $P$ and $k$ and the bound on ${e}^{iku}$. 
To consider the case $|h| \ge 1$,
we utilize
$$|R(h)| \le |{\mathcal N_g}(u+h)| + |{\mathcal N_g}(u)| + |{\mathcal N_g}'(u)h| \le c(1+|h|),$$
which leads to
$$|R(h)| \le c|h|, \quad \text{if } |h| \ge 1.$$
Combining two estimates, we obtain
$$|R(h)| \le \begin{cases} c|h|^2, & |h| < 1 \\ c|h|, & |h| \ge 1. \end{cases}$$

Now let us define the set $\Lambda = \{x: |h(x)| \ge 1\}$ and its characteristic function $\lambda = \chi_\Lambda(x)$.
% Then,
% $$\lambda(x) = \begin{cases} 1, & \text{if } x \in \Lambda \\ 0, & \text{if } x \notin \Lambda \end{cases}$$
Then,
$$\|R(h)\|_{L^p(\Omega'_\pi)} \le c\left\|(1-\lambda)\,|h|^2\right\|_{L^p(\Omega'_\pi)} + c\left\|\lambda |h|\right\|_{L^p(\Omega'_\pi)}.$$
Notice that $p < r$, we have 
$$
\|\lambda \,|h|\|_{L^p(\Omega'_\pi)} = \left(\int_{\Omega'_\pi} \lambda \, |h|^p dx\right)^{1/p} 
\le  \left(\int_{\Omega'_\pi} |\lambda \, h|^r dx\right)^{1/p} 
= \|\lambda \, |h|\|_{L^r(\Omega'_\pi)}^{r/p}.
$$
On the other hand, 
$$
\|(1-\lambda)\,|h|^2\|_{L^p(\Omega'_\pi)} 
= \left(\int_{\Omega'_\pi} |(1-\lambda)\,|h|^{2p} dx\right)^{1/p} 
\le
\begin{cases}
    \| h\|_{L^r(\Omega'_\pi)}^{r/p},\quad &\mbox{if } r \le 2p,\\
    c\|h\|_{L^r(\Omega'_\pi)}^{2},\quad &\mbox{if } r > 2p,
\end{cases}
$$
where the last inequality follows from the embedding $L^r(\Omega'_\pi) \hookrightarrow L^{2p}(\Omega'_\pi)$.

Consequently, 
$$
\|R(h)\|_{L^p(\Omega'_\pi)}  
\le c\|h\|_{L^r(\Omega'_\pi)}^{\min(2, \,\frac{r}{p})}+c\|h\|_{L^r(\Omega'_\pi)}^{ r/p}.
$$
Since $r > p$,
$$\frac{\|R(h)\|_{L^p(\Omega'_\pi)}}{\|h\|_{L^r(\Omega'_\pi)}} \le c\|h\|_{L^r(\Omega'_\pi)}^{\min(1, \,\frac{r}{p}-1)}\to 0 \quad \text{as } \|h\|_{L^r(\Omega'_\pi)} \to 0.$$
This finishes the proof.
\end{proof}

\section{Gradient of the Functional $J$}
%\subsection{Gradient Derivation}
\label{Sec: Gradient Derivation}
Let
$\widetilde{O}_\theta$
be the parallel-beam projection at angle $\theta$, $p=\widetilde{O}_\theta f$, $q=P e^{ikp}$, $v=\mathcal D_{\mathrm{Fr}}(q)$,
and
$
s=\widetilde{O}_\theta h
$
be the corresponding projected perturbation. 
It follows from Theorem \ref{thm:forward}, the fact $G(f) = \sqrt{F(f)} =: |v|$, and the chain rule that
$$
G'(f)[h]=\frac{1}{2|v|}
F'(f)[h] =\Re\!\Big( ik\,\frac{\overline v}{|v|}\;\mathcal D_{\mathrm{Fr}}(q\cdot s)\Big),
$$
where $\overline v$ denotes the complex conjugate of $v$.

Denote the residual $r=|v|-y$ and $\langle \cdot, \cdot\rangle$ be the complex $L^2$ inner product, then
$$
J'(f)[h]=
\Big\langle G'(f)[h],\, r \Big\rangle
=\Big\langle \Re\!\Big( ik\,\frac{\overline v}{|v|}\;\mathcal D_{\mathrm{Fr}}(q\cdot s)\Big),\, r \Big\rangle
=\Re\Big\langle\mathcal D_{\mathrm{Fr}}(q\cdot s),\,\overline{ik\,\tfrac{\overline v}{|v|}}\,r\,\Big\rangle.
$$
Since $\overline{ik\,\tfrac{\overline v}{|v|}}\,r=-ik\,\tfrac{v}{|v|}\,r=:b$, we have 
$$
J'(f)[h]=\Re\langle\mathcal D_{\mathrm{Fr}}(q\cdot s), b\rangle
=\Re\langle q\cdot s,\mathcal D_{\mathrm{Fr}}^*(b)\rangle
=\Re\langle s, \overline q\,\mathcal D_{\mathrm{Fr}}^*(b)\rangle,
$$
where $\mathcal D_{\mathrm{Fr}}^*$ denotes the adjoint of the Fresnel propagator.
Using $s=\widetilde{O}_\theta h$,
$$
J'(f)[h]=\Re\langle h, \widetilde{O}_\theta^*(\overline q\,\mathcal D_{\mathrm{Fr}}^*(b))\rangle.
$$
The gradient associated with the projection angle $\theta$ becomes
$$
\nabla J_\theta
=
\widetilde{O}_\theta^*
\Big(
\overline q\,
\mathcal D_{\mathrm{Fr}}^*(b)
\Big).
$$
Integrating over all projection angles yields
$$
\nabla J
=
\int_0^\pi
\nabla J_\theta
\,d\theta.
$$
Recall that
$
f=-\delta+i\beta,
$
the gradients with respect to $\delta$ and $\beta$ read
$$
\nabla_\delta J
=
-\Re(\nabla J),
\qquad
\nabla_\beta J
=
\Im(\nabla J).
$$
Expanding the Fresnel propagators in Fourier form,
$$
\mathcal D_{\mathrm{Fr}}(u)
=
\mathcal F^{-1}\!\big[H_d\cdot\mathcal F(u)\big],
\qquad
\mathcal D_{\mathrm{Fr}}^*(w)
=
\mathcal F^{-1}\!\big[H_{-d}\cdot\mathcal F(w)\big],
$$
where
$
H_{-d}=\overline{H_d}
$
corresponds to Fresnel propagation over the distance $-d$, and using
$$
I_e
=
\left|
\mathcal F^{-1}\!\Big[
H_d\cdot
\mathcal F
\Big(
P e^{ik\widetilde{O}_\theta f}
\Big)
\Big]
\right|^2,
$$
together with
$$
r
=
|v|-y
=
\sqrt{I_e}-\sqrt{I_m},
\qquad
\frac{v}{|v|}r
=
\Bigg(
1-\frac{\sqrt{I_m}}{\sqrt{I_e}}
\Bigg)v,
$$
the gradients with respect to $\delta$ and $\beta$ admit the fully expanded representations
{\footnotesize
\begin{align*}
\frac{\partial J}{\partial \delta}
&=
\int_0^\pi
\Re\!\Bigg[
ik\,\widetilde{O}_\theta^*
\Bigg(
\overline P\,
e^{-ik\overline{\widetilde{O}_\theta f}}
\,
\mathcal F^{-1}\!\Bigg\{
H_{-d}\cdot
\mathcal F\!\Bigg[
\Bigg(
1-\frac{\sqrt{I_m}}{\sqrt{I_e}}
\Bigg)
\mathcal F^{-1}\!\Big[
H_d\cdot
\mathcal F
\big(
P e^{ik\widetilde{O}_\theta f}
\big)
\Big]
\Bigg]
\Bigg\}
\Bigg)
\Bigg]
\,d\theta,
\\[1.2ex]
\frac{\partial J}{\partial \beta}
&=
\int_0^\pi
\Re\!\Bigg[
-k\,\widetilde{O}_\theta^*
\Bigg(
\overline P\,
e^{-ik\overline{\widetilde{O}_\theta f}}
\,
\mathcal F^{-1}\!\Bigg\{
H_{-d}\cdot
\mathcal F\!\Bigg[
\Bigg(
1-\frac{\sqrt{I_m}}{\sqrt{I_e}}
\Bigg)
\mathcal F^{-1}\!\Big[
H_d\cdot
\mathcal F
\big(
P e^{ik\widetilde{O}_\theta f}
\big)
\Big]
\Bigg]
\Bigg\}
\Bigg)
\Bigg]
\,d\theta .
\end{align*}
}
Here the outer $\Re(\cdot)$ extracts the real part so that the expressions yield real-valued fields for each projection angle~$\theta$ before backprojection. The operator $\widetilde{O}_\theta^*$ denotes the adjoint of the projection operator, i.e., the standard backprojection that maps detector-domain quantities back to the 3D object domain.

%\section{Convergence of Iterative Regularization Methods and }
\section{Nonlinearity Conditions }

%\cite{deuflhard1998convergence}. We therefore state the following auxiliary result without proof.
\begin{theorem}[RIC $\Rightarrow$ TCC, \cite{deuflhard1998convergence}]
Assume $G$ is Fr\'echet differentiable on $\mathcal M$ and satisfies the range invariance condition with constant $\eta<1$. Then for any $f,\tilde f\in\mathcal M$,
$$
\|G(f)-G(\tilde f)-G'(\tilde f)(f-\tilde f)\|_Y \le \frac{\eta}{1-\eta}\,\|G(f)-G(\tilde f)\|_Y.
$$
In particular, if $\eta<\tfrac13$, then $G$ satisfies TCC with constant $\bar\eta:=\dfrac{\eta}{1-\eta}\in[0,\tfrac12)$. 
\label{thm:RIC-TCC}
\end{theorem}
\end{document}